\documentclass[12pt,letterpaper, reqno]{amsart}
\usepackage[margin=2.5cm]{geometry}
\usepackage{amscd}
\usepackage{amssymb}
\usepackage{amsthm}
\usepackage{graphicx}
\usepackage{color}
\usepackage[all]{xy}
\usepackage{mathrsfs}
\usepackage{marvosym}
\usepackage{stmaryrd}
\usepackage{srcltx}
\usepackage{mathtools}
\usepackage{scalerel,amssymb}
\usepackage{tikz-cd}
\usepackage{adjustbox}
\usepackage{placeins}
\usepackage{floatrow}
\usepackage{setspace}

\usepackage{todonotes}

\newfloatcommand{capbtabbox}{table}[][\FBwidth]
\definecolor{amaranth}{rgb}{0.9, 0.17, 0.31}
\usepackage[colorlinks=true,citecolor=amaranth,linkcolor=black]{hyperref}%
\usetikzlibrary{shapes,arrows,chains}

\let\oldtocsection=\tocsection
\let\oldtocsubsection=\tocsubsection
\let\oldtocsubsubsection=\tocsubsubsection
\renewcommand{\tocsection}[2]{\hspace{0em}\oldtocsection{#1}{#2}}
\renewcommand{\tocsubsection}[2]{\hspace{1em}\oldtocsubsection{#1}{#2}}
\renewcommand{\tocsubsubsection}[2]{\hspace{0.5em}\oldtocsubsubsection{#1}{#2}}
\newtheorem{theorem}{Theorem}[section]
\newtheorem{lemma}[theorem]{Lemma}
\newtheorem{proposition}[theorem]{Proposition}
\newtheorem{corollary}[theorem]{Corollary}

\newtheoremstyle{defstyle}
  {.6em} 
  {.1em} 
  {} 
  {} 
  {\bfseries} 
  {.} 
  {.5em} 
  {} 
\theoremstyle{defstyle} \newtheorem{definition}[theorem]{Definition}

\newtheorem{construction}[theorem]{Construction}

\newtheorem{example}[theorem]{Example}

\theoremstyle{remark}
\newtheorem{remark}[theorem]{Remark}

\numberwithin{equation}{section}

\newcommand{\A} {\mathbb{A}}
\newcommand{\CC} {\mathbb{C}}

\newcommand{\Hom} {\mathrm{Hom}}

\newcommand{\cO} {\mathcal{O}}

\newcommand{\NN} {\mathbb{N}}
\newcommand{\PP} {\mathbb{P}}
\newcommand{\QQ} {\mathbb{Q}}
\newcommand{\RR} {\mathbb{R}}
\newcommand{\Spec} {\mathrm{Spec}}

\newcommand{\ZZ} {\mathbb{Z}}
\newcommand {\lfor} {\llbracket}
\newcommand {\rfor} {\rrbracket}
\newcommand{\sP} {\mathscr{P}}
\newcommand{\cY} {\mathcal{Y}}
\newcommand{\cD} {\mathcal{D}}
\newcommand{\cL} {\mathcal{L}}
\newcommand{\cZ} {\mathcal{Z}}
\newcommand{\cX} {\mathcal{X}}

\newcommand{\Spf} {\mathrm{Spf}}
\newcommand{\Pic} {\mathrm{Pic}}
\newcommand{\Mov} {\mathrm{Mov}}
\newcommand{\Nef} {\mathrm{Nef}}
\newcommand{\bigslant}[2]{{\raisebox{.2em}{$#1$}\left/\raisebox{-.2em}{$#2$}\right.}}

\begin{document}

\title[The KSBA moduli space of stable log Calabi--Yau surfaces]{\resizebox{\textwidth}{!}{The KSBA moduli space of stable log Calabi--Yau surfaces}}
\author[V.\,Alexeev]{Valery Alexeev}
\address{University of Georgia, Department of Mathematics, Athens, GA 30605}
\email{Valery@math.uga.edu}

\author[H.\,Arg\"uz]{H\"ulya Arg\"uz}
\address{The Mathematical Institute, University of Oxford, Oxford, OX2 6GG, UK}
\email{hulya.arguz@maths.ox.ac.uk}

\author[P.\,Bousseau]{Pierrick Bousseau}
\address{The Mathematical Institute, University of Oxford, Oxford, OX2 6GG, UK}
\email{pierrick.bousseau@maths.ox.ac.uk}

\begin{abstract}
We prove that every irreducible component of the coarse Koll\'ar-Shepherd-Barron and Alexeev (KSBA) moduli space of stable log Calabi--Yau surfaces admits a finite cover by a projective toric variety. This verifies a conjecture of Hacking--Keel--Yu. The proof combines tools from log smooth deformation theory, the minimal model program, punctured log Gromov--Witten theory and mirror symmetry.  
\end{abstract}
\newtheoremstyle{cited}%
  {3pt}
  {3pt}
  {\itshape}
  {}
  {\bfseries}
  {.}
  {.3em}
  {\thmname{#1} \thmnumber{#2}\thmnote{\normalfont#3}}

\theoremstyle{cited}
\newtheorem{citedthm}{Theorem}
\renewcommand*{\thecitedthm}{\Alph{citedthm}}

\maketitle

\setcounter{tocdepth}{2}
\tableofcontents
\setcounter{section}{0}

\section{Introduction}

\subsection{Context and overview}
Constructing moduli spaces along with functorial compactifications, admitting geometrically and combinatorially desirable properties, has been of great interest in algebraic geometry. The fundamental work of Deligne--Mumford \cite{DeligneMumford} and Knudsen \cite{Knudsen} describes a compactification of the moduli space of smooth genus $g$ curves with marked points, using stable curves admitting at worst nodal singularities. More general compactifications, using curves with weighted marked points, were introduced  
by Losev--Manin \cite{LosevManin} in the genus $g=0$ case, and by Hassett \cite{Hassett} in higher genus.
One of the striking features of the compactifications of \cite{LosevManin}, known as Losev--Manin spaces, is that they are toric varieties. In this paper, we investigate the geometry of moduli spaces of higher dimensional analogues of ``stable curves with weighted marked points'', and show that for stable log Calabi--Yau surfaces such moduli spaces admit finite covers by toric varieties. 

The higher dimensional analogue of a stable curve, introduced by Koll\'ar--Shepherd-Barron \cite{KSB} and Alexeev \cite{Valery96}, is a stable pair $(Y,B)$, consisting of a projective variety $Y$ and a $\QQ$-divisor $B$ such that $(Y,B)$ has semi-log-canonical  singularities, and satisfying the \emph{stability condition} requiring $K_Y+B$ to be $\QQ$-Cartier and ample -- see \cite{kollar2023families-of-varieties} for a comprehensive treatment of the KSBA theory.
The geometry of the KSBA moduli space of stable pairs has been investigated so far in a handful of cases, including toric and abelian varieties \cite{Valeryannals}, K3 surfaces \cite{AE, AlexeevEngelThompson, AlexeevThompson, laza2016k3}, elliptic surfaces \cite{AscherBejleri, AscherBejleri-K3}, and for some surfaces of general type \cite{alexeevpardini}. 
Most relevant to the present paper,
Alexeev showed in \cite{Valeryannals} that
for toric varieties the normalization of the coarse 
KSBA moduli space is toric, with associated fan the secondary fan defined as in \cite{GKZ}. 

More generally, for \emph{maximal log Calabi--Yau varieties}, Keel \cite{keel} conjectured around 2012 that the coarse KSBA moduli space should still admit a finite cover by a toric variety. Furthermore, Hacking--Keel--Yu conjectured that such a toric variety should be determined by a generalized version of the ``secondary fan'', derived from mirror symmetry \cite[Conjecture 1.2]{HKY20}.
In this paper, we prove this conjecture in dimension two, that is, for all polarized \emph{maximal log Calabi--Yau surfaces} $(Y,D,L)$. These log Calabi--Yau surfaces consist of a projective surface $Y$ over $\CC$, equipped with a reduced, anticanonical divisor $D\subset Y$, such that $(Y,D)$ has at worst log canonical singularities, and an ample line bundle $L$ on $Y$. 
Additionally, $(Y,D)$ must be \emph{maximal}, meaning that $D$ is non-empty and singular. 

The KSBA moduli space for $(Y,D,L)$, denoted by $\mathcal{M}_{(Y,D,L)}$, is the closure within the moduli space of stable pairs of the locus of stable pairs deformation equivalent to $(Y,D+\epsilon\, C)$, where $C \in |L|$ and $0< \epsilon <\!\!<1$. 
By \cite{MR3779955, MR3671934}, $\mathcal{M}_{(Y,D,L)}$ is a proper Deligne--Mumford stack with a projective coarse moduli space $M_{(Y,D,L)}$.
To investigate the geometry of $M_{(Y,D,L)}$, we follow the general strategy outlined by Hacking--Keel--Yu in \cite[Conjecture 8.13]{HKY20}, which proposes to use a conjectural double mirror construction. To realize this proposal, we first
construct a ``semistable mirror'' to a maximal degeneration of $(Y,D,L)$. This degeneration is determined by choosing a polyhedral decomposition $\sP$ on a ``Symington polytope'' $P$ associated with $(Y,D,L)$. 

The semistable mirror is obtained by first constructing a normal crossing surface with dual intersection complex $(P,\sP)$, followed by a formal smoothing of this surface. We show that such a smoothing is an \emph{open Kulikov degeneration} $\mathcal{X}_{\sP} \to \Delta = \mathrm{Spf}\, \CC\lfor t \rfor$, given by a projective crepant resolution $\mathcal{X}_{\sP} \to \cX^{\mathrm{can}}_{\sP}$ of a formal affine canonical 3-fold singularity
$\cX^{\mathrm{can}}_{\sP}$.
Applying the general construction of the ``secondary fan'' in \cite[\S 2]{HKY20} to the semistable mirror, we obtain a complete fan  $\mathrm{Sec}(\mathcal{X}_{\sP}/\cX^{\mathrm{can}}_{\sP})$ as a coarsening of the Mori fan of $\mathcal{X}_{\sP} \to \cX^{\mathrm{can}}_{\sP}$. Our main result, Theorem \ref{Thm: HKY_conj}, states:

\begin{citedthm}
\label{thm_intro_1}
The coarse KSBA moduli space $M_{(Y,D,L)}$ for a polarized log Calabi--Yau surface $(Y,D,L)$ is the image of a finite surjective morphism \begin{equation}
\nonumber
    S_{\mathcal{X}_{\sP}}^{\mathrm{sec}} \longrightarrow M_{(Y,D,L)}\,,
\end{equation}
where $S_{\mathcal{X}_{\sP}}^{\mathrm{sec}}$ is the
projective toric variety associated to the secondary fan $\mathrm{Sec}(\mathcal{X}_{\sP}/\cX^{\mathrm{can}}_{\sP})$.
\end{citedthm}

We outline the key results involved in the proof in what follows, after briefly explaining why it goes through mirror symmetry, as predicted by \cite[Conjecture 8.13]{HKY20}.

\subsection{Motivation from mirror symmetry}
A stable log Calabi--Yau surface $(Y,D+\epsilon\, C)$ naturally defines a log Calabi--Yau 3-fold $(\overline{Z}, \partial \overline{Z})$ as follows. 
The 3-fold $\overline{Z}$ is obtained by blowing up $C \times \{0\}$ in $Y \times \PP^1$. The anticanonical divisor $\partial{\overline{Z}}$ is defined as the strict transform in $\overline{Z}$ of the anticanonical divisor 
$(Y \times \{0\})\cup (D \times \PP^1) \cup (Y \times \{\infty\})$ in $Y \times \PP^1$. The complement $Z := \overline{Z}\setminus \partial \overline{Z}$ is then a non-compact Calabi--Yau 3-fold. 
The projection $Y \times \PP^1 \rightarrow Y$ induces a map $Z \rightarrow U:=Y \setminus D$, which is an affine conic bundle on $U$ that degenerates over the divisor $C\cap U$.
Heuristically, the KSBA moduli space $\mathcal{M}_{(Y,D,L)}$ can be viewed as a compactification of a ``moduli space of complex structures'' on $Z$. Mirror symmetry then predicts that $\mathcal{M}_{(Y,D,L)}$ should be identified with the ``stringy extended K\"ahler moduli space'' of a mirror Calabi--Yau 3-fold $\cX$ to $Z$ \cite{Mor1993}. As $Z$ is a maximal log Calabi--Yau 3-fold, the homology class of the fiber of the expected SYZ fibration on $Z$ is uniquely determined as the vanishing torus associated to a 0-stratum of the boundary. Consequently, distinct maximal degenerations of $Z$ should correspond to different mirrors related by birational transformations. In this paper, we work directly with $(Y,D,L)$ instead of $Z$.
We provide a precise realization of this heuristic picture, where the ``stringy extended K\"ahler moduli space of $\cX_\sP$'' is rigorously defined as the toric variety $S_{\cX_\sP}^{\mathrm{sec}}$. 

In the situation where $Y$ is a toric variety and $D$ is the toric boundary divisor, mirror symmetry becomes entirely combinatorial, essentially reducing to a duality between fans and momentum polytopes. 
Let $\cX^{\mathrm{can}}$ be the singular toric variety with associated fan $C(P)$, the cone over the momentum polytope $P$ of $(Y, D, L)$. 
A maximal regular subdivision $\sP$ of $P$ defines a fan refining $C(P)$ by adding cones over $\sP$, and so a toric crepant resolution $\mathcal{X}_{\sP} \to \cX^{\mathrm{can}}$, the semistable mirror family to $(Y,D,L)$. By standard toric geometry, the Mori fan of this mirror family to $(Y,D,L)$ coincides with the secondary fan of $(Y,D,L)$ \cite{CLS}. 
Consequently, the KSBA moduli space $\mathcal{M}_{(Y,D,L)}$ for a polarized toric log Calabi--Yau surface can be understood combinatorially, either through the secondary fan or via the Mori fan of the mirror \cite{Valeryannals, zhu2014natural}. 

The presence of a dense torus in $\mathcal{M}_{(Y,D,L)}$ follows in this case from the existence of the basis of $H^0(Y, L)$ given by the monomials $z^m$ on $(\CC^\star)^2 =Y \setminus D$ indexed by the integral points $m \in P_{\ZZ} := P \cap \ZZ^2$. The key feature of this basis is that, for a curve $C$ defined by an equation $\sum_{m \in P_{\ZZ}} a_m z^m=0$ with $a_m \in \CC^\star$ for all $m\in P_{\ZZ}$, the pair $(Y, D+\epsilon\,C)$ is KSBA-stable. 
Indeed, for any vertex $m'$ of $P$, ensuring that $a_{m'} \in \CC^\star$ guarantees 
that $C$ avoids the singular point of $D$ corresponding to $m'$.
As a result, $(Y, D+\epsilon\,C)$ has only semi-log-canonical singularities when $a_m \in \CC^\star$ for all $m \in P_\ZZ$ \cite{Valeryannals}.

To obtain a toric description of the KSBA moduli space for general non-toric log Calabi--Yau surfaces, particularly in identifying a dense torus within the KSBA moduli space, one must construct a basis for $H^0(Y,L)$ that exhibits the same key feature as the monomial basis in the toric case. However, unlike the toric setting, purely combinatorial methods are insufficient for this task.
Nonetheless, mirror symmetry suggests that such a 
basis should consist of ``theta functions'' which are 
constructed using the enumerative geometry of curves in a mirror space \cite{GHK1,  GHKSK3, GHS,  GS2019intrinsic, HKY20, KY23}.
In this paper, we construct a basis of theta functions via a variant of the Gross--Siebert intrinsic mirror symmetry construction \cite{GS2019intrinsic, GScanonical}, applied to the semistable mirror to $(Y,D,L)$. 

When restricted to a torus $(\CC^\star)^2 \subset Y \setminus D$, the theta functions become specific Laurent polynomials with positive integer coefficients. However, no elementary combinatorial description of these Laurent polynomials is currently known. Instead, following \cite{GS2019intrinsic, GScanonical}, we describe them through enumerative geometry, specifically via particular counts of rational curves in the semistable mirror, defined using the theory of punctured Gromov--Witten invariants of Abramovich--Chen--Gross--Siebert \cite{ACGSpunctured}. 

\subsection{Outline of the paper}
In \S\ref{Sec: polarized log cy}, we investigate maximal degenerations of polarized log Calabi--Yau surfaces combinatorially, using real affine geometry. To do so, we consider a ``nice'' toric model of a polarized log Calabi--Yau $(Y,D,L)$, as in previous work of Engel--Friedman \cite{EF21}.
This allows us to systematically describe an associated ``Symington polytope'' $P$ to $(Y,D,L)$ given by an integral affine manifold with singularities \cite{Symington}. 
We introduce the notion of a \emph{good} polyhedral decomposition on $P$,
and we prove the existence of such decompositions in Theorem \ref{Thm: good dec exists}.
A good polyhedral decomposition $\sP$ on $P$ naturally defines a maximal degeneration of $(Y,D,L)$, whose central fiber has ``intersection complex'' $(P,\sP)$, as explained in \S\ref{sec: maximal degenerations}. 

To construct a ``mirror'' to such a maximal degeneration of $(Y,D,L)$, we first construct a normal crossing surface $\mathcal{X}_{\sP,0}$ with ``dual intersection complex'' $(P,\sP)$ in \S\ref{sec: open Kulikov surfaces from log cy}. 
The irreducible components of $\mathcal{X}_{\sP,0}$ arise from $\QQ$-Gorenstein smoothings of toric surfaces with quotient Wahl singularities.
Using Lemmas \ref{Lem: smoothing 1} and \ref{Lem: smoothing 2}, we show that these smoothings  have no local to global obstructions. 
We then appropriately glue these irreducible components together to obtain the normal crossing surface $\mathcal{X}_{\sP,0}$, which exhibits features analogous to the central fiber of a Type III Kulikov degeneration in the context of K3 surfaces \cite{FSIII}. A key difference in our context is that the dual intersection complex of $\mathcal{X}_{\sP,0}$ is a disk rather than a sphere. 
In \S\ref{Sec: open Kulikov surfaces and smoothings}, we introduce the general notion of an \emph{open Kulikov surface}, defined as a normal crossing surface $\mathcal{X}_0$ with a dual intersection complex that is a disk, subject to additional properties -- see Definition \ref{Def:K_disk}. After investigating properties of such surfaces, along with their affinizations, we establish in \S\ref{Sec: semistable mirror} that, with appropriate gluing, the associated
surface $\mathcal{X}_{\sP,0}$ is a $d$-semistable open Kulikov surface. 

Our key results in Theorems \ref{thm: smoothing exists} and \ref{thm: ss smoothings} prove the following result for a general open Kulikov surface $\cX_0$:

\begin{citedthm}
Let $\cX_0$ be a quasi-projective open Kulikov surface.
Then, there exists a quasi-projective formal smoothing $\mathcal{X} \to \Delta$ of $\cX_0$ such that the total space is smooth and $\cX_0\subset \mathcal{X}$ is a reduced normal crossing divisor, that is $\mathcal{X} \to \Delta$ is a semistable degeneration. Moreover, $K_{\mathcal{X}} = 0$ and the affinization morphism $\mathcal{X} \to \cX^{\mathrm{can}}$ is a projective crepant resolution.
\end{citedthm}
We refer to such a smoothing $\mathcal{X} \to \Delta$ as an \emph{open Kulikov degeneration}. By the Minimal Model Program, adapted to the formal setup \cite{MMPformal}, we conclude in \S \ref{Sec: Mori fans} that such a Kulikov degeneration is a relative Mori dream space. 
This yields a complete toric fan, known as the ``MoriFan'', in $\mathrm{Pic}(\mathcal{X}/\cX^{\mathrm{can}}) \otimes \RR$. Applying the general construction of \cite[\S 2]{HKY20}, we obtain the ``secondary fan'' $\mathrm{Sec}(\cX/\cX^{\mathrm{can}})$ as a ``coarsening'' of this Mori fan, by identifying neighboring cones connected by ``M1 flops'' -- for details see \S\ref{Sec: polarized_mirrors}. The associated toric variety is denoted by $S_{\mathcal{X}_{\sP}}^{\mathrm{sec}}$. 

For a polarized log Calabi--Yau surface $(Y,D,L)$ with a choice of $(P,\sP)$, we refer to the open Kulikov degeneration $\cX_\sP \rightarrow \Delta$ obtained by smoothing $\cX_{\sP,0}$ as the ``semistable mirror'' of $(Y,D,L)$. Following the general strategy proposed by Hacking--Keel--Yu in \cite[Conjecture 8.13]{HKY20}, we then
investigate the ``double mirror'' obtained by applying an intrinsic mirror construction 
\cite{GHK1, GHKSK3, GS2019intrinsic, KY23} to the semistable mirror. 
Specifically, we employ a variant of the intrinsic mirror construction of Gross--Siebert \cite{GS2019intrinsic} to the semistable mirror $\mathcal{X}_{\sP} \to \Delta$. 
More generally, we construct an intrinsic mirror for any quasi-projective open Kulikov degeneration $\cX \rightarrow \Delta$.
A key subtlety in this setting is that $\cX \rightarrow \Delta$ is only quasi-projective, whereas the intrinsic mirror construction of \cite{GS2019intrinsic} is formulated for projective degenerations. 
To address this, we first establish in Theorem \ref{thm_proper} that the moduli spaces of punctured log Gromov--Witten invariants of an open Kulikov surface $\mathcal{X}_0$, which define the structure constants in the intrinsic mirror construction, are proper. 
The proof critically relies on the existence of a projective crepant contraction $\cX \to \cX^{\mathrm{can}}$, where $\cX^{\mathrm{can}}$ is affine. 
Consequently, following the approach of \cite{GS2019intrinsic}, we define a commutative and associative graded mirror algebra $\mathcal{R}_\cX$ -- see Theorem \ref{thm: mirror algebra}. 
Taking the Proj of this algebra, we obtain an intrinsic mirror family to $\mathcal{X} \to \Delta$, initially defined over the formal scheme $\mathrm{Spf}\,\CC \lfor NE(\mathcal{X}/\cX^{\mathrm{can}})\rfor$, as in \cite{GS2019intrinsic}. Here, $NE(\mathcal{X}/\cX^{\mathrm{can}})$ stands for the monoid of integral points in the cone of relative effective curve classes of $\mathcal{X} \rightarrow \cX^{\mathrm{can}}$. 

Using techniques introduced in \cite{GHK1, GHKSK3, KY23}, we show in Theorem \ref{Thm_Finiteness} that, for any quasi-projective open Kulikov degeneration $\mathcal{X}\to \Delta$, only finitely many curve classes of punctured log curves contribute to each structure constant of the intrinsic mirror algebra. This shows:
\begin{citedthm}
    The intrinsic mirror family to a quasi-projective open Kulikov degeneration $\mathcal{X}\to \Delta$ naturally extends to a family $\mathcal{Y}_{\mathcal{X}} \to S_{\mathcal{X}}$ over the affine toric variety $S_{\mathcal{X}} := \mathrm{Spec} \,\CC [ NE(\mathcal{X}/\cX^{\mathrm{can}})]$, with fan the nef cone $\mathrm{Nef}(\mathcal{X}/\cX^{\mathrm{can}})$.
\end{citedthm}
The proof of this uses the birational geometry of the crepant resolution $\mathcal{X} \to \cX^{\mathrm{can}}$ to constrain the curve classes appearing in the intrinsic mirror construction. 

By construction, the intrinsic mirror algebra $\mathcal{R}_\cX$ admits a basis of \emph{theta functions} $\{ \vartheta_p \}_{p \in C(P)_{\ZZ}}$, indexed by the integral points of the tropicalization of $\mathcal{X}$ given by the cone $C(P)$ over $P$. Lemma \ref{Lem: support} establishes that the theta functions $\vartheta_p$ corresponding to the interior integral points of $C(P)$ generate an ideal of the mirror algebra $R_\cX$. Following \cite[\S 6.2]{HKY20}, we prove that this ideal defines an anticanonical divisor $\mathcal{D}_{\mathcal{X}}$ in $\cY_\cX$. 
On the other hand, the intrinsic mirror family is naturally polarized by the line bundle $\mathcal{L}_{\cX} = \mathcal{O}_{\mathcal{Y}_{\cX}}(1)$ with a natural basis of sections $\vartheta_p$, for $p\in P_{\ZZ}$. Consequently, as in \cite[\S 6.2]{HKY20}, we have a natural ``theta divisor"  $\mathcal{C}_{\cX}:= \{ \sum_{p \in P_{\ZZ}} \vartheta_p = 0 \} \in |\mathcal{L}_{\cX}|$. 
Hence, applying the intrinsic mirror construction for every projective crepant resolution $\cX' \rightarrow \cX^{\mathrm{can}}$ of the affinization $\cX^{\mathrm{can}}$ of $\mathcal{X}$, we obtain pairs $\left(  \mathcal{Y}_{\cX'}, D_{\cX'} + \epsilon\, \mathcal{C}_{\cX'} \right)$ over the affine toric varieties $\mathrm{Spec} \,\CC [ NE(\cX'/\cX^{\mathrm{can}})]$. 
In \S\ref{sec: polarized} we build upon arguments of \cite[\S 6]{HKY20} to show that these families glue together into a larger family over the movable secondary fan. Furthermore, in Theorem \ref{Thm: ext_bogus}, we establish that this family extends uniquely to a family $(\mathcal{Y}_{\cX}^{\mathrm{sec}},\mathcal{D}^{\mathrm{sec}}_{\cX} + \epsilon\, \mathcal{C}^{\mathrm{sec}}_{\cX}) \to \mathcal{S}^{sec}_{\cX}$, where $\mathcal{S}^{sec}_{\cX}$ is a toric Deligne-Mumford stack defined by a stacky fan $\mathcal{S}ec(\cX/\cX^{\mathrm{can}})$, whose
underlying fan is the secondary fan $\mathrm{Sec}(\cX/\cX^{\mathrm{can}})$. 
Theorems \ref{Thm: KSBA stability} and \ref{thm: diffeomorphism_2} then prove:
\begin{citedthm}
\label{thm D intro}
The intrinsic mirror family
$\left(  \mathcal{Y}, D + \epsilon\, \mathcal{C} \right)$ to a quasi-projective open Kulikov degeneration $\mathcal{X} \to \Delta$ is a family of KSBA stable log Calabi--Yau surfaces over the toric stack $\mathcal{S}^{\mathrm{sec}}_{\cX}$. 
Moreover,
when $\cX=\cX_\sP$ is the semistable mirror of a polarized log Calabi--Yau surface $(Y,D,L)$, the general fiber of this intrinsic mirror family is deformation equivalent to $(Y,D,L)$. 
\end{citedthm}

To prove that the fibers $(Y_t, D_t + \epsilon\, C_t)$ of the intrinsic mirror family are semi-log-canonical, we first consider the case of fibers  
over the dense torus of the 
toric stack $\mathcal{S}_\cX^{\mathrm{sec}}$.
In this case, we first show in Theorem \ref{Thm: restricted flat} that, as in \cite[\S 7]{HKY20}, 
the surface $Y_t$ is irreducible. Then, to prove that  $(Y_t, D_t + \epsilon\, C_t)$ is semi-log-canonical, it suffices to verify that $C_t$ does not intersect the $0$-strata of $D_t$. This follows from the  description in \S\ref{Sec: support} of the restriction of the theta functions to $D_t$.
For an arbitrary fiber $(Y_t, D_t + \epsilon\,C_t)$, with irreducible components $(V_i,D_i+\epsilon\, C_i)$, we show using Theorem \ref{Thm: inductive_mirror} that proving that the pairs $(V_i,D_i+\epsilon\, C_i)$ are semi-log-canonical is sufficient. By the naturalness 
of the intrinsic mirror symmetry, this follows from Theorem \ref{Thm: restricted flat}.

To study the general fiber of the intrinsic mirror family to the semistable mirror $\cX_\sP \rightarrow \Delta$, we show that, up to a base change of $\cX_{\sP} \rightarrow \Delta$, a specific one-parameter sub-family of the intrinsic mirror family constitutes a toric degeneration in the sense of Gross--Siebert \cite{GSannals}. This allows us to first determine the topology of the general fiber, and then to identify it as a log Calabi--Yau surface using a result of Friedman \cite{friedman2015geometry}. Then, by the universal property of the moduli space $\mathcal{M}_{(Y,D,L)}$, we obtain a map of stacks $\mathcal{S}^{\mathrm{sec}}_{\cX_{\sP}} \rightarrow \mathcal{M}_{(Y,D,L)}$, which induces a
corresponding map 
$S^{\mathrm{sec}}_{\cX_{\sP}} \rightarrow M_{(Y,D,L)}$ at the level of coarse moduli spaces. 
To conclude our main result, Theorem \ref{Thm: HKY_conj} demonstrates that this map is finite and surjective. The finiteness follows from \cite[Lemma 8.7]{HKY20} 
by proving that the intrinsic mirror family is non-constant in restriction to the one-dimensional strata of $\mathcal{S}^{\mathrm{sec}}_{\cX_{\sP}}$. 
Finally, the surjectivity is established through a dimension calculation.

\subsection{Related work}
Mirrors to \emph{log smooth} maximal degenerations to (log) Calabi--Yau varieties are constructed by Gross--Siebert in \cite{GS2019intrinsic}. As the maximal degenerations of polarized log Calabi--Yau surfaces we consider in this paper are typically not log smooth, the construction of \cite{GS2019intrinsic} does not apply directly to produce the semistable mirror $\cX \rightarrow \Delta$. Also, the mirror construction of Gross--Hacking--Keel \cite{GHK1} concerns a single ``smooth'' log Calabi--Yau surface, without maximal degenerations, and hence it does not apply either. 
To address this issue, we give a direct construction by deformation theory of the semistable mirror $\cX_\sP \rightarrow \Delta$. 
If we restrict attention to smooth log Calabi--Yau surfaces, we expect the affinization $\cX^{\mathrm{can}}_\sP$ of our semistable mirror $\mathcal{X}_\sP$ to be deformation equivalent to the mirror of \cite{GHK1} -- see Remark \ref{remark_GHK}.

To construct the double mirror to $(Y,D,L)$, that is, the intrinsic mirror to $\cX_\sP \rightarrow \Delta$, we use a version of the Gross--Siebert mirror construction \cite{GS2019intrinsic, GScanonical}, which is possible since $\cX_\sP \rightarrow \Delta$ is log smooth. In \cite{HKY20, HKY22}, the non-Archimedean mirror construction developed by Keel--Yu \cite{KY23} is used in arbitrary dimension to construct intrinsic mirror families over the toric variety defined by the secondary fan of crepant resolutions $\cX \rightarrow \cX^{\mathrm{can}}$, which are algebraic over $\Delta=\mathbb{A}^1$.  
In the present paper, we use instead in dimension two the Gross--Siebert mirror construction, which can be used for crepant resolutions over the formal disk $\Delta=\mathrm{Spf}\, \CC[\![t]\!]$.
To be able to use the non-Archimedean mirror construction in this set-up, it would first be necessary to have an extension of the results of \cite{KY23} for formal families.\footnote{After the second author presented results of this paper at the Homological Mirror Symmetry conference on January 28th 2024 at IMSA/Miami, we were informed that Logan White is currently working on such an extension in his PhD under the guidance of Sean Keel.}

In \cite[Theorem 1.3]{HKY20}, the general construction of the non-Archimedean intrinsic mirror family over the toric variety defined by the secondary fan is used to prove the full Hacking--Keel--Yu conjecture on $\mathcal{M}_{(Y,D,L)}$ for the six stable log Calabi--Yau surfaces $(Y,D+\epsilon\, C)$, where $Y$ is a del Pezzo surfaces of degree $n$, for $1\leq n \leq 6$, endowed with the line bundle $L=-K_Y$, and $D$ a cycle of $n$ $(-1)$-curves. These surfaces are in a sense ``self-mirror'', for example, the dimension of the complex moduli equals the dimension of the K\"ahler moduli, and due to this coincidence one can conveniently take the mirror to any of these del Pezzo's $Y$ as the canonical bundle $K_Y$. In this paper, to prove the Hacking--Keel--Yu conjecture for all log Calabi--Yau surfaces, a key result is the construction of the semistable mirror $\cX_\sP \rightarrow \Delta$, which can then be taken as an input of a general intrinsic mirror construction. Moreover, the semistable mirror is explicit enough to allow us to prove that the ``double mirror'' is related to the original $(Y,D,L)$.
The ongoing work of Gross--Hacking--Keel--Siebert \cite{GHKSK3} on describing compactifications of moduli spaces of K3 surfaces using intrinsic mirror symmetry is also closely related, as open Kulikov degenerations introduced in our paper are non-compact analogues of the Type III Kulikov degenerations of K3 surfaces -- see also \cite{HLL} and \cite{Mori_fan_DNV} in this context.

In \cite{alexeev2024non}, we apply the semistable mirror construction of the present paper to answer a long-standing question in the string theory literature on canonical 3-fold singularities which are M-theory dual to the 5d superconformal field theories constructed from webs of 5-branes with 7-branes. Finally, we will show in future work that the toric variety $S_{\cX_\sP}^{\mathrm{sec}}$ is independent of the choice of the Symington polytope $(P,\sP)$. We will also give a modular interpretation of the finite morphism in Theorem \ref{thm_intro_1}.

\subsection{Acknowledgments} 
Valery Alexeev was partially supported by the NSF grant DMS-2201222.
H\"ulya Arg\"uz was partially supported by the NSF grant DMS-2302116, and Pierrick Bousseau was partially supported by the NSF grant DMS-2302117. We thank the organizers of the ``Homological Mirror Symmetry 2024'' and the ``Geometry of the South'' conferences at IMSA/Miami, where this work was presented, for their hospitality. We also thank Philip Engel for many useful discussions at an early stage of this project. Finally, we thank the anonymous referees for their careful reading and the many suggestions to improve the exposition.

\section{Polarized log Calabi--Yau surfaces and Symington polytopes}
\label{Sec: polarized log cy}

In this section, we first review the definition and basic properties of polarized log Calabi--Yau surfaces in \S\ref{sec: log CY surfaces}.
Next, following \cite{EF21}, we explain 
how to associate to such surfaces ``nice'' choices of ``Symington polytopes'' in \S\ref{Sec: Symington polytopes}. 
The main result of this section is the existence of ``good'' polyhedral decompositions on such polytopes, obtained in Theorem \ref{Thm: good dec exists} in \S\ref{subsec: good polyhedral}. 
Finally, in \S\ref{sec: maximal degenerations}, we outline the
construction of maximal degenerations of polarized log Calabi--Yau surfaces from good polyhedral decompositions on Symington polytopes.

\subsection{Log Calabi--Yau surfaces}
\label{sec: log CY surfaces}
 
In this paper, a \emph{log Calabi--Yau surface} is a log canonical pair $(Y,D)$ over $\CC$, as in \cite[Definition 2.34]{KM}, with $Y$ projective of dimension $2$, $D$ a reduced effective divisor, and $K_Y + D$ trivial, in particular Cartier. Throughout this paper we assume that all log Calabi--Yau surfaces are \emph{maximal}, in the sense that $D$ is non-empty and singular. The following proposition classifies singularities of log Calabi--Yau surfaces.
\begin{proposition}
\label{Prop: germs}
Let $(Y,D)$ be a log Calabi--Yau surface. Then, away from the divisor $D$, the surface $Y$ has at worst du Val (ADE)  singularities, and the germ of $(Y,D)$ at any point $p \in D$ is given by either of the following:
\begin{itemize}
    \item[i)] $(\mathbb{A}^2_{x,y},(x=0))$: both $Y$ and $D$ are smooth locally around $p$.
     \item[ii)] $(\mathbb{A}^2_{x,y},(xy=0))$: $Y$ is smooth and $D$ has an ordinary double point  locally around $p$.
    \item[iii)]   $(\mathbb{A}^2_{x,y}/ \frac{1}{r}(1,a),(xy=0))$, where $r$ and $a$ are coprime integers, with $r\geq 2$ and $1 \leq a \leq r-1$: $Y$ has a cyclic quotient singularity at $p$, obtained as the quotient of $\mathbb{A}^2_{x,y}$ by $\ZZ/r\ZZ$ acting by $x \to \zeta x$, $y \to \zeta^a y$, where $\zeta$ is a primitive r'th root of $1$. Moreover, locally around $p$, $D$ is the image under the quotient by $\ZZ/r\ZZ$ of the union of the coordinate axes in $\mathbb{A}^2_{x,y}$, and consequently $D$ has an ordinary double point at $p$.
\end{itemize}
\end{proposition}

\begin{proof}
See \cite[Propositions 7.2--7.3]{HKY20}.
\end{proof}

Although a log Calabi--Yau surface $(Y,D)$ is in general singular, we can often work with its minimal resolution.

\begin{corollary}
\label{Cor:minimal_resolution}
Let $g' : Y' \to Y$ be the minimal resolution of $Y$, and $D' = (g')^{-1}(D)$ be the pre-image of $D$. Then, $(Y', D')$ is a smooth log Calabi--Yau surface.
\end{corollary}

\begin{proof}
The statement follows from Proposition \ref{Prop: germs} and the toric description of minimal resolutions of cyclic quotient singularities \cite[Chapter 10]{CLS}.  
\end{proof}

\begin{example}
Let $\overline{Y}$ be a projective toric surface, and $\overline{D}$ the union of the toric divisors of $\overline{Y}$. Then, $(\overline{Y},\overline{D})$ is a log Calabi--Yau surface.   
\end{example}

We will consider polarized log Calabi--Yau surfaces $(Y,D,L)$, given by a log Calabi--Yau surface $(Y,D)$ endowed with an ample line bundle $L$. In what follows we describe toric models for such surfaces. For this, we first define interior and corner blow-ups.

\begin{definition}
Let $(Y_1,D_1)$ and $(Y_2,D_2)$ be two log Calabi--Yau surfaces. We say that $g: Y_1 \rightarrow Y_2$ is a \emph{corner blow-up} if $g$ is a birational morphism whose exceptional locus is a union of irreducible components of $D_1$, and
$D_1=g^{-1}(D_2)$.
\end{definition}

\begin{definition}
Let $(Y_1,D_1)$ and $(Y_2,D_2)$ be two log Calabi--Yau surfaces. We say that $h: Y_1 \rightarrow Y_2$ is an \emph{interior blow-up} if $h$ is the blow-up of $Y_2$ at a point on the smooth locus of $D_2$, and $D_1$ is the strict transform of $D_2$ by $h$.
\end{definition}

We can now define the notion of a toric model of a polarized log Calabi--Yau pair.

\begin{definition}
\label{Def:toric_model}
A \emph{toric model} of a polarized log Calabi--Yau surface $(Y,D,L)$ is a diagram of log Calabi--Yau surfaces
\begin{equation}
\label{Eq: toric model}
\tikzstyle{line} = [draw, -latex']
\begin{tikzpicture}
    \node[] at (6,0) (Step 1) { $(\widetilde{Y}, \widetilde{D})$};
      \node[] at (4,-1.6) (Step 2) {$(Y,D)$};
       \node[] at (8,-1.6) (Step 3) {$(\overline{Y}, \overline{D})$};
   \path [line] (Step 1) -- node [text width=0.5cm,midway,above ] {$g$} (Step 2);
 \path [line] (Step 1) -- node [text width=0.2cm,midway,above ] {$h$} (Step 3);
\end{tikzpicture}
\end{equation}
where $g$ is a corner blow-up, $h$ is a composition of interior blow-ups, $(\overline{Y},\overline{D})$ is a toric log Calabi--Yau surface, and $\overline{L}:=h_{\star} g^\star L$ is an ample line bundle on $\overline{Y}$.
\end{definition}

Every Calabi--Yau surface admits a toric model by \cite[Lemma 1.3]{GHK1} applied to the minimal resolution given by Corollary \ref{Cor:minimal_resolution}.

As in \cite[Definition 1.4]{GHKmod}, a log Calabi--Yau surface $(Y,D)$ is called \emph{generic} if its minimal resolution $(Y',D')$ does not contain an internal $(-2)$-curve, that is a smooth rational curve $C$ with $C^2=-2$ and disjoint from $D'$. By \cite[Proposition 4.1]{GHKmod}, every log Calabi--Yau surface is deformation equivalent to a generic one.
For generic log Calabi--Yau surfaces, we will use particularly good toric models, defined as follows. 

\begin{definition}
\label{Def:good_toric_model}
A \emph{good toric model} of a polarized log Calabi--Yau surface $(Y,D,L)$ is a toric model as in Definition \ref{Def:toric_model} such that the following conditions hold:
\begin{itemize}
    \item[i)] $h$ is a composition of interior blow-ups at distinct points of $\overline{D}$.
    \item[ii)] There exists a decomposition 
    \[ g^\star L = \sum_{i=1}^n c_i \, h^\star \overline{D}_i -\sum_{i=1}^n \sum_{j=1}^{N_i} m_{ij} E_{ij} \,,\]
    of $g^\star L$ in terms of the irreducible components $(\overline{D}_i)_{1\leq i\leq n}$ of $\overline{D}$, and of the exceptional divisors $(E_{ij})_{1\leq j\leq N_i}$ of $h$ above $\overline{D}_i$, such that we have 
    $c_i \geq 0$, $m_{ij} \geq 0$, and $ m_{ij} \leq c_i$
    for all $1\leq i \leq n$ and $1\leq j \leq N_i$.
\end{itemize}
\end{definition}

\begin{lemma} \label{Lem: good_toric_model}
Every generic polarized log Calabi--Yau surface $(Y,D,L)$ admits a good toric model.
\end{lemma}

\begin{proof}
This follows from the proof of \cite[Theorem 5.4]{EF21} applied to the smooth log Calabi--Yau pair $(Y',D')$ obtained as the minimal resolution $g' : (Y',D') \rightarrow (Y,D)$ as in 
Corollary \ref{Cor:minimal_resolution}, and to the big and nef line bundle
$L':= (g')^{\star} L$. Note that, while \cite[Theorem 5.4]{EF21} is stated for an ample line bundle, its proof applies more generally to a big and nef line bundle as remarked in \cite[Remark 5.7]{EF21}.
\end{proof}

The following result is well-known for smooth log Calabi--Yau surfaces and is included for reference in the singular case.

\begin{lemma} \label{lem: picard_torsion_free}
Let $(Y,D)$ be a log Calabi--Yau surface. Then, the Picard group $\mathrm{Pic}(Y)$ is torsion-free.
\end{lemma}

\begin{proof}
By the existence of a toric model, $Y$ is isomorphic to a projective toric surface $\overline{Y}$, up to loci of complex dimension one, that is, of real dimension two. Since proper toric varieties have trivial first homology groups, we obtain $H_1(Y,\ZZ)=H_1(\overline{Y},\ZZ)=0$. It follows that $H^2(X,\ZZ)$ is torsion free, and  by the exponential exact sequence, $\mathrm{Pic}(Y)$ is also torsion free.
\end{proof}

\subsection{Symington polytopes}
\label{Sec: Symington polytopes}
Given a generic polarized log Calabi--Yau surface $(Y,D,L)$, one can associate a 
Symington ``polytope'' $P$, which is an integral affine manifold with singularities. This polytope serves as the base of an almost toric fibration $Y \to P$ \cite[\S$2.2$]{Symington}. Engel--Friedman demonstrated in \cite{EF21} that well-chosen Symington polytopes can be obtained from a toric momentum polytope by cutting out non-overlapping triangles. In this section we review this construction, following  \cite[\S$5.1$]{EF21}. 

\begin{construction}
\label{Cons: Symington polytope}
We describe the construction of a Symington polytope 
for a generic polarized log Calabi--Yau surface $(Y,D,L)$ endowed with a good toric model $(\overline{Y},\overline{D},\overline{L})$ as in Definition \ref{Def:good_toric_model}. 
We denote by $\overline{P}$ the toric momentum polytope associated to the ample line bundle $\overline{L}$ on the toric surface $(\overline{Y},\overline{D})$.
The edges $(\overline{P}_i)_{1\leq i\leq n}$ of the boundary $\partial \overline{P}$ of $\overline{P}$ are naturally in one-to-one correspondence with the toric irreducible components $(\overline{D}_i)_{1\leq i\leq n}$ of $\overline{D}$. Moreover, for every $1\leq i\leq n$, the integral length $\ell(\overline{P}_i)$ of the edge $\overline{P}_i$ is given by the intersection number $\overline{L} \cdot \overline{D}_i$.

As in Definition \ref{Def:good_toric_model}, we denote the exceptional curves of the blow-up $h:\widetilde{Y}  \to \overline{Y}$ contracted to $\overline{D}_i$ by $E_{ij}$, for $1\leq j \leq N_i$.
By Definition \ref{Def:good_toric_model}, 
there exists a decomposition 
    \begin{equation}\label{Eq:L_mij} 
    \widetilde{L}:=g^\star L = \sum_{i=1}^n c_i \, h^\star \overline{D}_i -\sum_{i=1}^n \sum_{j=1}^{N_i} m_{ij} E_{ij} \,,\end{equation}
     such that  
    $c_i \geq 0$, $m_{ij} \geq 0$, and $ m_{ij} \leq c_i$
    for all $1\leq i \leq n$ and $1\leq j \leq N_i$.
Moreover, we have 
\begin{equation}
\label{Eq: mij}
\sum_j m_{ij} \leq \ell(\overline{P}_i)    
\end{equation}
where $\ell(\overline{P}_i)$ is the lattice length of $\overline{P}_i$, since $g^*L$ is nef and hence 
\[g^*L \cdot \Tilde{D}_i = \overline{L} \cdot \overline{D}_i - \sum_{j=1}^{N_i}m_{ij} = \ell(\overline{P}_i) - \sum_{j=1}^{N_i}m_{ij} \geq 0 \, ,\]
using that $\overline{L}=h_\star g^\star L = \sum_{i=1}^n c_i \overline{D}_i$.

As $\sum_{i=1}^n c_i \overline{D}_i$ is an effective torus-invariant divisor on $\overline{Y}$ representing the line bundle $\overline{L}$, there exists a corresponding integral point 
$p \in \overline{P}$ at lattice distance $c_i$ from each edge $\overline{P}_i$ of $\overline{P}$.
The momentum polytope $\overline{P}$ admits a polyhedral decomposition, 
\[  \overline{P}  = \bigcup_{i=1}^n \overline{T}_i \, , \]
where $\overline{T}_i$ is the convex hull of $\overline{P}_i$ and the point $p$. Generally $\overline{T}_i$'s are triangles, unless $p \in \partial \overline{P}$, in which case they may degenerate into line segments.

Due to Equation \eqref{Eq: mij}, we can choose non-overlapping line segments $\overline{P}_{ij} \subset \overline{P}_i$ of lattice length $m_{ij}$. 
We denote by $\overline{T}_{ij}$ the triangle with base $\overline{P}_{ij}$ and third vertex $p$, and by $\overline{L}_{ij} \subset \overline{T}_{ij}$ the line segment of points in $ \overline{T}_{ij}$ at lattice distance $m_{ij}$ from $\overline{P}_i$. Note that $\overline{L}_{ij}$ is non-empty since we have $m_{ij}\leq c_i$ by definition of a good toric model.

By the proof of \cite[Proposition 5.8]{EF21}, there exists an integral point $p_{ij}$ on $\overline{L}_{ij}$.
Let $\Delta_{ij}$ be the triangle with base $\overline{P}_{ij}$ and third vertex $p_{ij}$. 
By construction, $\Delta_{ij}$ is an integral triangle, with its base $\overline{P}_{ij}$ of integral length $m_{ij}$, and integral height also equal to  $m_{ij}$. We refer to $m_{ij}$ as the size of $\Delta_{ij}$. 
We then define the \emph{Symington polytope} $P$ associated to $(Y,D,L)$ as the integral affine manifold with singularities $P$ obtained from complement $\overline{P} \setminus \cup_{ij} \mathrm{Int}(\Delta_{ij})$, by gluing the edges adjacent to $p_{ij}$ using a matrix conjugate to
\[\left( \begin{matrix}
1 & 1 \\
0 & 1
\end{matrix} \right)\,.\] 
This process introduces a specific type of integral affine singularity, known as a focus-focus singularity, at each point $p_{ij}$. 
The corresponding monodromy invariant direction is the line $L_{ij}$ which is parallel to the edge $\overline{P}_i$ and passes through $p_{ij}$.
The discriminant locus of the affine structure is given by the union of points 
\[\Delta = \{p_{ij} ~ | ~ 1\leq i \leq n ~ \mathrm{and} ~ 1\leq j \leq N_i  \} \,. \]
\end{construction}

\begin{figure}[h]
\center{\includegraphics{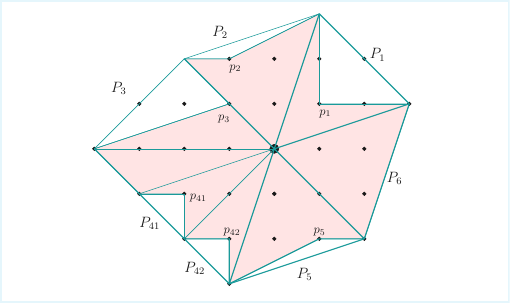}}
\caption{A Symington polytope for a polarized log Calabi--Yau surface $(Y,D,L)$ as in Construction \ref{Cons: Symington polytope}}
\label{figure1}
\end{figure}

We refer to the vertex $p_{ij}$ of $\Delta_{ij}$ as the \emph{interior vertex} of $\Delta_{ij}$, and to the edges of $\Delta_{ij}$ adjacent to $p_{ij}$ as the \emph{interior edges} of $\Delta_{ij}$ -- see Figure \ref{figure1}. 
Note that in some cases, the interior vertices of several triangles $\Delta_{ij}$ may coincide, resulting in more complex integral affine singularities on $Y$ beyond focus-focus singularities.
The following result establishes that the area of the Symington polytope $P$ is given by 
$\frac{1}{2}L^2$, a positive quantity due to the ampleness of the line bundle $L$. 
Consequently, $P$ is non-empty and two-dimensional.

\begin{lemma} \label{lem_area}
Let $P$ be a Symington polytope associated to a polarized log Calabi--Yau surface $(Y,D,L)$ as in Construction \ref{Cons: Symington polytope}. Then, the area of $P$ is equal to $\frac{1}{2}L^2$.
\end{lemma}

\begin{proof}
We use the notation introduced in Construction \ref{Cons: Symington polytope}. First, since $g^\star L = \overline{L}-\sum_{i,j} m_{ij}E_{ij}$, it follows that $L^2=\overline{L}^2 - \sum_{i,j} m_{ij}^2$. 
On the other hand, since $P$ is obtained from $\overline{P}$ by removing the non-overlapping triangles $\Delta_{ij}$ of sizes $m_{ij}$, we obtain $\mathrm{Area}(P)=\mathrm{Area}(\overline{P}) - \frac{1}{2}\sum_{i,j}m_{i,j}^2$. 
Finally, by standard toric geometry \cite[\S 5.3, p111, Corollary]{fulton}, we have $\mathrm{Area}(\overline{P})=\frac{1}{2}\overline{L}^2$, and so the result follows.
\end{proof}

\begin{figure}[h]
\center{\includegraphics{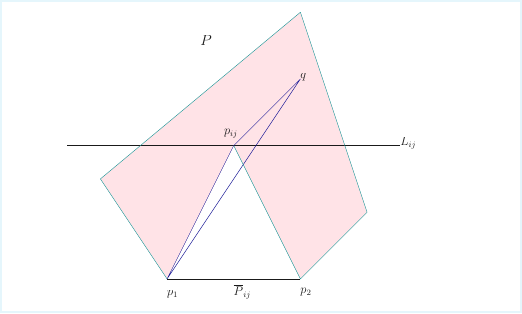}}
\caption{Illustration of the first part of the proof of Lemma \ref{Lem: monod_invariant}}
\label{figure2}
\end{figure}

The following result shows that the monodromy invariant lines have always a non-trivial intersection with the interior of $\overline{P}$.

\begin{lemma} \label{Lem: monod_invariant}
Let $P$ be a Symington polytope associated to a polarized log Calabi--Yau surface $(Y,D,L)$
as in Construction \ref{Cons: Symington polytope}.
Then, for every $1\leq i \leq n$, $1\leq j \leq N_i$, the intersection of the monodromy invariant line $L_{ij}$ with the interior of $P$ is an interval of positive length.
\end{lemma}

\begin{proof}
We first prove the result when $\Delta_{ij}$ is the only cut triangle. 
Let $\ell_i=0$ be the affine equation defining $\overline{P}_i$, so that $L_{ij}$ is the line of equation $\ell_i=m_{ij}$. 
Define $d_{ij}$ the greatest integral distance between the line $L_{ij}$ and a point $q \in \overline{P} \cap \{ \ell_i \geq m_{ij}\}$. To prove Lemma \ref{Lem: monod_invariant}, it suffices to show that $d_{ij}>0$. Indeed, in this situation, let $q$ be a point in $\overline{P}$ with $\ell_i(q)>m_{ij}$, and denote by $p_1$ and $p_2$ the two endpoints of the line segments $\overline{P}_{ij}$. 
Since the points $p_1, p_2, p_{ij}$ are not collinear, at least one of the triangles with vertices $p_1, p_{ij}, q$ or $p_2, p_{ij}, q$ must be non-flat -- see Figure \ref{figure2}. 
Such a non-flat triangle necessarily has a non-trivial intersection with $L_{ij}$ and is contained within $\overline{P}$, as the polygon $\overline{P}$ is convex and contains the points $p_1$, $p_2$, $p_{ij}, q$.

We now prove that $d_{ij}>0$. 
Consider the toric variety $(\overline{Y}',\overline{D}')$ obtained from $(\overline{Y},\overline{D})$ by adding to its fan the ray opposite to the ray corresponding to $\overline{P}_i$. 
These two opposite rays define a generically $\PP^1$-fibration $\overline{Y}' \rightarrow \PP^1$. 
Denote by $(\tilde{Y}', \tilde{D}')$ the corresponding blow-up of $(Y,D)$, with a generically $\PP^1$-fibration $\tilde{Y}' \rightarrow \PP^1$. Then, the pre-image of $E_{ij}$ is an irreducible component $E_{ij}'$ of a fiber of this map. Let $F_{ij}$ be the other irreducible component of this fiber and let $g': \widetilde{Y}' \rightarrow Y$ be the contraction induced by $g$. Define $G_{ij}$ as the curve class of a fiber, so that $G_{ij}=E_{ij}'+F_{ij}$. 
By standard toric geometry 
 \cite[\S 5.3, p112, Corollary]{fulton}, the intersection number $L \cdot g'_\star G_{ij}$ equals the maximal integral distance between the line $\ell_i=0$ and a point in $\overline{P}$. On the other hand, we have $L \cdot g'_\star E_{ij}=m_{ij}$, and so the intersection number
 $L \cdot g'_\star F_{ij}$ is equal to the maximal integral distance between the line $\ell_i=m_{ij}$ and a point in $\overline{P}$, that is, we have $d_{ij}=L \cdot g'_\star F_{ij}$. Since $g'$ is a corner blow-up, we have $g'_\star F_{ij}\neq 0$, and so we conclude that $d_{ij}=L \cdot g'_\star F_{ij}>0$ since $L$ is ample.
 This ends the proof of  Lemma \ref{Lem: monod_invariant} when $\Delta_{ij}$ is the only cut triangle.

 Finally, we prove Lemma \ref{Lem: monod_invariant} in the general case where $\Delta_{ij}$ is not the only cut triangle. If $\Delta_{ij}$ has an interior edge that  is not entirely contained within another cut triangle, then  by the first part of the proof, the corresponding segment of $L_{ij}$ has a non-trivial intersection with $P$. 
 From Construction \ref{Cons: Symington polytope}, the only way both interior edges of $\Delta_{ij}$ are fully contained within other cut triangles is when $p_{ij}$ coincides with the central point $p$. In this case, one can attempt to propagate $L_{ij}$ through the neighboring cut triangles to obtain a non-trivial intersection with $P$. This will always work unless the cut triangles completely cover $\overline{P}$. However, this does not occur, as the Symington polytope $P$ has positive area by Lemma \ref{lem_area} and $L^2>0$ due to the ampleness of $L$. This concludes the proof of  Lemma \ref{Lem: monod_invariant}.
\end{proof}

\subsection{Good polyhedral decompositions}
\label{subsec: good polyhedral}
Let $(Y,D,L)$ be a generic polarized log Calabi--Yau surface  with associated Symington polytope $P$, obtained from a toric polytope $\overline{P}$ as in Construction \ref{Cons: Symington polytope}. In this section, we first define in Definition \ref{Def: good polyhedral decomposition} the notion of a good polyhedral decomposition on $\overline{P}$, which is a particular regular polyhedral decomposition appropriately compatible with the cuts producing the Symington polytope. Then, we prove the existence of such good polyhedral decompositions in Theorem \ref{Thm: good dec exists}.

An integral polyhedral decomposition of $\overline{P}$ is given by a covering $\overline{\sP} = \{ \sigma \}$ of $\overline{P}$ by a finite number of strongly convex polyhedra satisfying the following:
\begin{itemize}
    \item[i)] If $\sigma \in \overline{\sP} $ and $\sigma' \subset \sigma$ is a face, then $\sigma' \in \overline{\sP}$.
    \item[ii)] If $\sigma,\sigma' \in \overline{\sP} $, then $\sigma \cap \sigma'$ is a common face of $\sigma$ and $\sigma'$.
    \item[iii)] Each polyhedron $\sigma$ has all of its vertices at integral points of $\overline{P}$.
\end{itemize}
We denote by $\overline{\sP}^{[k]}$ the set of 
k-dimensional cells of $\overline{\sP}$. A polyhedral decomposition $\overline{\sP}$ on $\overline{P}$ is called \emph{regular} if there exists a convex piecewise linear function $\varphi: \overline{P} \rightarrow  \mathbb{R}$ such that the polyhedra $\{ \sigma\}$ of $\overline{\sP}$ corresponds to the domain of linearity of $\varphi$. In this situation, we refer to $\overline{\sP}$ as the decomposition induced by $\varphi$. In what follows, we describe good regular polyhedral decompositions on $\overline{P}$. We first need the following observation.

\begin{proposition}
\label{Prop: interior edges}
The number of integral points on the two interior edges of a triangle $\Delta_{ij}$, for $1\leq i \leq n$ and $1\leq j \leq N_i$ are equal.
\end{proposition}
\begin{proof}
Without loss of generality, we can assume that the vertex $p_{ij}$ has coordinates $(0,0)$, and that the other two vertices of $\Delta_{ij}$ on $\overline{P}_{ij}$ have coordinates $(a,b)$ and $(a,c)$ respectively, with $b>c$. Since the edge of $\Delta_{ij}$ on the side $\overline{P}_{ij}$ has the same integral length as the height of $\Delta_{ij}$ by Construction \ref{Cons: Symington polytope}, we have $a=b-c$. Then, we obtain $d:=\gcd(a,b)=\gcd(a,c)$, and the sets of integral points on the edges connecting $p_{ij} = (0,0)$ to $(a,b)$ and to $(a,c)$ respectively have the same cardinality, equal to $d+1$:
\[ \left| \left\{ \left( \frac{ka}{d} , \frac{kb}{d} \right) \,\ | \,\ 0 \leq k \leq d  \right\} \right| = \left| \left\{ \left( \frac{ka}{d} , \frac{kc}{d} \right) \,\ | \,\ 0 \leq k \leq d  \right\} \right| = d+1 .\]
\end{proof}

\begin{definition}
The \emph{standard decomposition} of $\Delta_{ij}$ is the polyhedral decomposition of $\Delta_{ij}$ consisting of:
\begin{itemize}
\item[i)] the triangle $\Delta_{ij,0}$ whose set of vertices is
\begin{equation} 
\label{Eq:Delta_ij0}
\left\{  (0,0), \left( \frac{a}{d}, \frac{b}{d}\right) , \left( \frac{a}{d}, \frac{c}{d}\right) \right\} \,,
\end{equation}
\item[ii)] the $(d-1)$ trapezoids $\Delta_{ij,k}$ whose set of vertices are
\begin{equation}
\label{Eq:Delta_ijk}
\left\{  \left( \frac{ka}{d}, \frac{kb}{d}\right) , \left( \frac{(k+1)a}{d}, \frac{(k+1)b}{d}\right),  \left( \frac{ka}{d}, \frac{kc}{d}\right) , \left( \frac{(k+1)a}{d}, \frac{(k+1)c}{d}\right) \right\} , \end{equation}
for $1 \leq k \leq d-1$.
\end{itemize}
\end{definition}

\begin{definition}
\label{Def: good polyhedral decomposition}
 Let $\overline{P}$ be the momentum polytope associated to a toric model $(\overline{Y},\overline{D})$ of a log Calabi--Yau pair $(Y,D)$ with an ample line bundle $L$ and $\Delta_{ij}$ the triangles that we cut out from $\overline{P}$ to obtain the Symington polytope $P$, as explained in the Construction \ref{Cons: Symington polytope}. A \emph{good polyhedral decomposition} on $\overline{P}$ is a regular integral polyhedral decomposition $\overline{\sP}$, satisfying the following:
\begin{itemize}
       \item[i)] Each of the triangles $\Delta_{ij}$ is a union of some polyhedra $\sigma \in \overline{\sP}$ which form a standard decomposition of $\Delta_{ij}$.
\item[ii)] Each polyhedron $\sigma \in \overline{\sP}$ not contained in a triangle $\Delta_{ij}$, is a triangle of size $1$, that is, a triangle whose all edges have integral length $1$.
\end{itemize}
\end{definition}

\begin{theorem}
\label{Thm: good dec exists}
Let $(Y,D,L)$ be a log Calabi--Yau pair with an ample line bundle with associated Symington polytope $P$ obtained from the momentum polytope $\overline{P}$ of a toric model $(\overline{Y},\overline{D})$ by cutting out triangles $\Delta_{ij}$ as in Construction \ref{Cons: Symington polytope}. Then, there exists a good polyhedral decomposition $\overline{\sP}$ on $\overline{P}$. 
\end{theorem}

\begin{proof}
We will construct regular polyhedral decompositions of $\overline{P}$ by successively pulling of integral vertices.
This process involves iteratively modifying a piecewise linear convex function by lowering its value at integral vertices by a small enough positive amount -- see for instance \cite[\S 4.3]{de2010triangulations}.

Let $\overline{\sP}_1$ be the regular polyhedral decomposition of $\overline{P}$ obtained from the trivial decomposition by first pulling the central point $p$, and then pulling the vertices of the triangles $\Delta_{ij}$ -- see Figure \ref{figure3}. Denote by $\varphi_1$ a convex piecewise linear function 
$\varphi_1: \overline{P} \to \mathbb{R}$ inducing $\overline{\sP}_1$.

By Proposition \ref{Prop: interior edges}, there are the same number of integral points on each of the two interior edges of $\Delta_{ij}$'s. Let $\{ a^1_{ij}, \ldots, a^d_{ij} \}$ and $\{ b^1_{ij}, \ldots, b^d_{ij} \}$ respectively denote the set of integral points on the interior edges of $\Delta_{ij}$, except for the interior vertex $p_{ij}$. Let ${ \epsilon_{ij}^1, \ldots , \epsilon_{ij}^d }$ be a set of positive real numbers such that 
the points 
\begin{equation}\label{eq_points_a}
(p_{ij}, \varphi_0(p_{ij}))\,, (a_{ij}^1, \varphi_0(a_{ij}^1) - \epsilon_{ij}^1)\,, \cdots,
\,,(a_{ij}^d, \varphi_0(a_{ij}^d) - \epsilon_{ij}^d)\end{equation}
are exactly the bending points of a piecewise linear function on the interior edge of $\Delta_{ij}$ containing the points $a_{ij}^k$'s. 
Since $\varphi_1$ is affine on $\Delta_{ij}$, the points 
\begin{equation}\label{eq_points_b}
(p_{ij}, \varphi_0(p_{ij}))\,, (b_{ij}^1, \varphi_0(b_{ij}^1) - \epsilon_{ij}^1)\,, \cdots,
\,,(b_{ij}^d, \varphi_0(b_{ij}^d) - \epsilon_{ij}^d)
\end{equation}
are also exactly the bending points of a piecewise linear function on the interior edge of $\Delta_{ij}$ containing the points $b_{ij}^k$'s.
Let $\varphi_2: \overline{P} \rightarrow \RR$ be the convex piecewise linear function whose graph is the lower convex hull of 
the points in \eqref{eq_points_a}, the points in \eqref{eq_points_b}, and
the points $(v,\varphi(v))$, where $v$ are vertices of $\overline{\sP}_1$ 
not contained in triangles $\Delta_{ij}$. Denote by $\overline{\sP}_2$  the regular decomposition of $\overline{P}$ induced by $\varphi_2$. If $\epsilon_{ij}^1, \cdots, \epsilon_{ij}^d$ are small enough for all $1\leq i\leq n$, $1\leq j\leq N_i$, then $\overline{\sP}_2$ satisfies condition i) of Definition \ref{Def: good polyhedral decomposition}.

To obtain a good decomposition that also satisfies condition ii) of Definition \ref{Def: good polyhedral decomposition}, we start with $\overline{\sP}_2$  and then successively pull generically all the integral points $v$ of $\overline{P}$ which are not vertices of 
$\overline{\sP}_2$. The resulting decomposition $\overline{\sP}$, illustrated in Figure \ref{figure3}, coincides with $\overline{\sP}_2$  on each triangle $\Delta_{ij}$, and so also satisfies condition i) of Definition \ref{Def: good polyhedral decomposition}. Moreover, on the complement in $\overline{P}$ of the interior of the triangles $\Delta_{ij}$,  $\overline{\sP}$ 
is a triangulation containing all integral points, 
and so $\overline{\sP}$  satisfies condition ii) of Definition \ref{Def: good polyhedral decomposition}, since by Pick's formula an integral triangle containing no integral point apart from its vertices is necessarily a triangle of size 1.
    
\end{proof}

\begin{figure}[h]
\center{\includegraphics{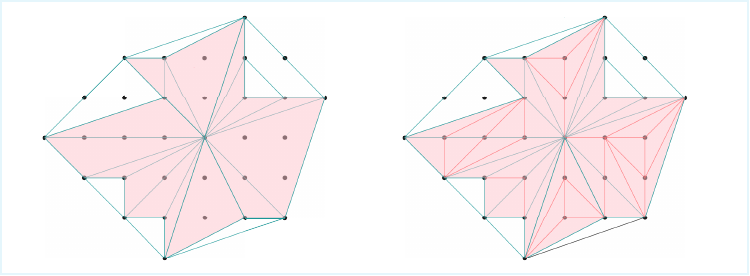}}
\caption{The left hand figure illustrates $\overline{\sP}_2$. To obtain the final decomposition $\overline{\sP}$ the red edges are inserted, as in the right hand figure.}
\label{figure3}
\end{figure}

\subsection{Maximal degenerations}
\label{sec: maximal degenerations}
Given a generic polarized log Calabi--Yau surface $(Y,D,L)$ with an associated Symington polytope $P$, obtained from $\overline{P}$ as described in Construction \ref{Cons: Symington polytope}, we naturally obtain a polyhedral decomposition $\sP$ on $P$ by restricting the good polyhedral decomposition $\overline{\sP}$ on $\overline{P}$ to $P$. In this section, we outline the construction of a \emph{maximal}
degeneration of $(Y,D,L)$, meaning that the central fiber contains a $0$-dimensional stratum. Here, stratum refers to the intersection of irreducible components of the central fiber.  

Let $\varphi$ be a convex piecewise linear function on $\overline{P}$ defining the regular polyhedral decomposition $\overline{\sP}$.
We will construct a one-parameter maximal polarized degeneration 
$(\cY_\varphi, \cD_\varphi, \cL_\varphi) \rightarrow \mathbb{A}^1$, with general fiber deformation equivalent to $(Y,D,L)$.

We begin with the Mumford toric degeneration 
$(\overline{\cY}_\varphi, \overline{\cD}_\varphi, \overline{\cL}_\varphi) \rightarrow \mathbb{A}^1$ 
of the polarized toric surface $(\overline{Y},\overline{D},\overline{L})$ induced by the polyhedral decomposition $\overline{\sP}$ of $\overline{P}$ \cite[Example 3.6]{Gross}.
The intersection complex of the central fiber 
$(\overline{Y}_{\varphi,0},\overline{D}_{\varphi,0},\overline{L}_{\varphi,0})$
is the polyhedral decomposition $\overline{\sP}$ of $\overline{P}$.
The irreducible components of the central fiber are the polarized toric surfaces $Y^f$ with momentum polytopes the 2-dimensional faces $f$ of $\overline{\sP}$ and they are glued together along the toric divisors $Y^e$ corresponding to the edges $e$ of $\overline{\sP}$.

For every $1\leq i\leq n$ and $1\leq j \leq N_i$, we pick a section $s_{ij}$ of  
$(\overline{\cY}_{\varphi}, \overline{\cD}_{\varphi}, \overline{\cL}_{\varphi}) \rightarrow \mathbb{A}^1$, 
whose image in the general fiber is contained in the interior of the irreducible component $\overline{D}_i$ of $\overline{D}$, and 
whose image in the central fiber is a point $x_{ij}$ in the interior of the divisor $Y^{\overline{P}_{ij}}$ corresponding to the edge $\overline{P}_{ij}$ of $\overline{\sP}$ in the boundary of $\overline{P}$, as in 
Construction \ref{Cons: Symington polytope}.

Let $h_\varphi: \widetilde{\cY}_{\varphi} \rightarrow \overline{\cY}_{\varphi}$ be the blow-up of $\overline{\cY}_{\varphi}$ along the sections $s_{ij}$. We denote by $\widetilde{\cD}_\varphi$ the strict transform of $\overline{\cD}_\varphi$, and by $\widetilde{\cL}_\varphi$ the line bundle on $\widetilde{\cY}_\varphi$ defined by 
\[ \widetilde{\cL}_\varphi  = h_{\sP}^\star \overline{\cL}_\varphi -\sum_{i=1}^n \sum_{j=1}^n m_{ij} \mathcal{E}_{ij} \,,\]
where $\mathcal{E}_{ij}$ are the exceptional divisors.
By comparison with Equation \eqref{Eq:L_mij}, we obtain that the general fiber of
$(\widetilde{\cY}_\varphi, \widetilde{\cD}_\varphi, \widetilde{\cL}_\varphi) \rightarrow \mathbb{A}^1$
is deformation equivalent to $(\widetilde{Y},\widetilde{D},\widetilde{L})$.
Whereas $\widetilde{L}$ is a nef line bundle, the restriction $\widetilde{\cL}_{\varphi,0}$ of $\widetilde{\cL}_\varphi$ to the central fiber $(\widetilde{\cY}_{\varphi,0}, \widetilde{\cD}_{\varphi,0})$ is not nef if there exists indices $i$ and $j$ with $m_{ij}>1$.
Indeed, when $m_{ij}>1$, the exceptional divisor $\mathcal{E}_{ij,0}$ on the central fiber is obtained by blowing-up a point on the divisor $Y^{\overline{P}_{ij}}$, which is contained in the toric surface 
$Y^{\Delta_{ij,m_{ij}}}_{\overline{\sP}}$ with momentum polytope given by the trapezoid $\Delta_{ij,m_{ij}}$, as in Equation \eqref{Eq:Delta_ijk}. 
Since $\Delta_{ij,m_{ij}}$ is a trapezoid, 
the toric surface $Y^{\Delta_{ij,m_{ij}}}_{\overline{\sP}}$ is a $\PP^1$-fibration, whose fiber class $F$ satisfies $\overline{\cL}_{\varphi,0} \cdot F=1$.
After the blow-up, the strict transform of the fiber class passing through the blown-up point is a $(-1)$-curve $C_{ij}$ of class $F-\mathcal{E}_{ij,0}$.
Thus, we obtain
$\widetilde{L}_{\varphi,0} \cdot C_{ij}=1-m_{ij}$,
which is negative whenever $m_{ij}>1$.

\begin{figure}[h]
\center{\includegraphics{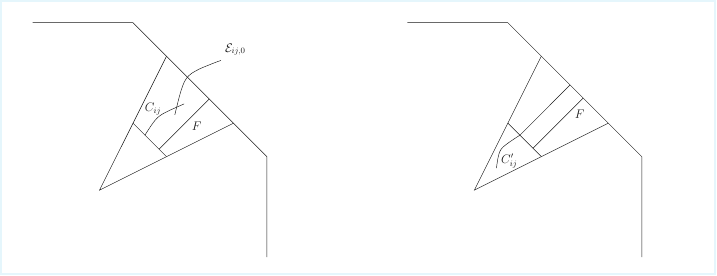}}
\caption{Flop to turn $\widetilde{\cL}_{\varphi,0}$ into a nef line bundle.}
\label{figure4}
\end{figure}

In order to obtain a nef line bundle, we first replace the 3-fold $\widetilde{Y}_{\varphi}$ by its flop along the $(-1,-1)$-curve $C_{ij}$.
On the central fiber, this means blowing up the component corresponding to $\Delta_{ij, m_{ij}-1}$ instead of $\Delta_{ij, m_{ij}}$. If $m_{ij}>2$, then the line bundle is not nef either, because the intersection with the strict transform of the $\PP^1$-fiber containing the blown-up point is $2-m_{ij}$, and one can flop this curve -- see Figure \ref{figure4}. In general, after $m_{ij}-1$ flops, we obtain a family  $(\widetilde{\cY}'_\varphi, \widetilde{\cD}'_\varphi, \widetilde{\cL}'_\varphi) \rightarrow \mathbb{A}^1$,
whose general fiber is still deformation equivalent to $(\widetilde{Y},\widetilde{D},\widetilde{L})$, but with now the line bundle $\widetilde{\cL}'_\varphi$ being relatively big and nef. By \cite[Theorem 4.18]{friedman2015geometry}, a big enough power of $\widetilde{\cL}'_\varphi$ is base point free, and so defines a relatively birational contraction 
\[ g: (\widetilde{\cY}'_\varphi, \widetilde{\cD}'_\varphi) \longrightarrow (\cY_{\varphi},\cD_{\varphi})\,,\]
and a relatively ample line bundle $\cL_{\varphi}$
on $\cY_\varphi$ such that $\widetilde{\cL}'_{\varphi}
=g^\star \cL_\varphi$. 
By construction, the line bundle $\widetilde{\cL}'_\varphi$ is trivial in restriction to every double curve of the central fiber
$\widetilde{\cY}'_{\varphi,0}$ corresponding to an edge of $\overline{\mathscr{P}}$ contained in the interior of triangles 
$\Delta_{ij}$. Moreover, $\widetilde{\cL}'_\varphi$ agrees with $\overline{\mathcal{L}}_{\varphi}$, and so is ample, on every irreducible component of $\widetilde{\cY}'_{\varphi,0}$ corresponding to a 2-dimensional face of $\overline{\sP}$ not contained in the triangles $\Delta_{ij}$. Hence, the irreducible components of $\widetilde{\cY}'_{\varphi,0}$ contracted by $g$ are exactly the ones corresponding to 
the 2-dimensional faces of $\overline{\sP}$ contained in the triangles $\Delta_{ij}$.
Thus, the general fiber of the polarized degeneration 
\begin{equation}
\label{Eq:Yphi}
(\cY_{\varphi},\cD_{\varphi},\cL_\varphi) \longrightarrow \mathbb{A}^1    
\end{equation}
is deformation equivalent to $(Y,D,L)$, and the intersection complex of the central fiber $(\cY_{\varphi,0},\cD_{\varphi,0},\cL_{\varphi,0})$ is given by the Symington polytope $P$ with its polyhedral decomposition $\sP$. In particular, since $\sP$ is a decomposition into triangles of size one, the normalization of every irreducible component of $\cY_{\varphi,0}$ is isomorphic to $\mathbb{P}^2$, and the restriction of $\cL_{\varphi,0}$ to the normalization of each irreducible component of $\cY_{\varphi,0}$ is the ample line bundle $\mathcal{O}_{\mathbb{P}^2}(1)$.

\section{Open Kulikov surfaces and their semistable smoothings}
\label{Sec: open Kulikov surfaces and smoothings}

In this section, we introduce open Kulikov surfaces and establish that $d$-semistable open Kulikov surfaces admit formal semistable smoothings. We also show that these formal smoothings are Mori dream spaces using results of the Minimal Model Program \cite{BCHM10} adapted to the formal setup as in \cite{MMPformal}.

\subsection{Open Kulikov surfaces}
\label{Sec: open Kulikov surfaces}

Open Kulikov surfaces are normal crossing surfaces whose irreducible components are particular types of surfaces called $0$-surfaces, $1$-surfaces, and $2$-surfaces, defined as below.

\subsubsection{$0$-,$1$-, and $2$-surfaces}

Both $1$- and $2$-surfaces will be examples of open log Calabi--Yau surfaces as defined below.

\begin{definition} \label{Def: open log CY}
    An \emph{open log Calabi--Yau surface} $(Y,D)$ is a smooth complex surface $Y$, together with a non-empty reduced normal crossing divisor $D \subset Y$ such that there exists a smooth log Calabi--Yau surface $(\widetilde{Y},\widetilde{D})$ and a non-empty connected curve $Z$ in $\widetilde{Y}$ given by a union of irreducible components of $\Tilde{D}$, such that $Y= \widetilde{Y}  \setminus Z$ and $D= \Tilde{D} \setminus Z$.
    \end{definition}

Let $(Y,D)$ be an open log Calabi--Yau surface as in Definition \ref{Def: open log CY}. Then, for a compactification $(\widetilde{Y}, \widetilde{D})$ of $(Y,D)$ as in Definition \ref{Def: open log CY}, the divisor $\widetilde{D}$ has at least two irreducible components, and so is a cycle of projective lines. 
It follows that the divisor $D$ is a chain $D_0 +\cdots +D_n$, where $D_0 \simeq D_n \simeq \mathbb{A}^1$, and  
$D_i \simeq \PP^1$ for $1 \leq i \leq n-1$.

The terminology used in the following definitions is inspired by \cite{SB2, SB}.

\begin{definition} \label{def_0_surface}
A \emph{$0$-surface} is an open log Calabi--Yau surface $(Y,D)$ such that, denoting
\[(\overline{Y},\overline{D}):=(\Spec\, H^0(Y,\mathcal{O}_Y), \Spec\, H^0(D,\mathcal{O}_D))\,,\]
the affinization of $(Y,D)$, and  $\pi: (Y,D) \rightarrow (\overline{Y},\overline{D})$ the affinization morphism, the following holds:
\begin{itemize}
\item[i)] $(\overline{Y},\overline{D})$ is an affine toric surface with a torus fixed point $x_0$.
\item[ii)] The morphism $\pi$ is birational, projective, and surjective.
\item[iii)] Every $\pi$-exceptional curve is contracted by $\pi$ to $x_0$.
\end{itemize} 
\end{definition}

If $(Y,D)$ is a $0$-surface, then $Y$ contains only finitely many compact irreducible curves. Moreover, $\overline{Y}$ is a cyclic quotient singularity, $Y \rightarrow \overline{Y}$ is a resolution of singularities, and the union of compact curves in $Y$ is the corresponding exceptional divisor.
In particular, we have $C^2<0$ for every compact curve in $Y$.

 \begin{example}
     The fan of the affine toric surface $(\overline{Y}, \overline{D})$ is a strictly convex rational cone $\sigma$ in $\mathbb{R}^2$. A toric $0$-surface is then a toric surface whose fan is a refinement of $\sigma$. 
     Finally, any $0$-surface $(Y,D)$ is obtained from a toric $0$-surface $(Y',D')$ by interior blow-ups. The tropicalization of $(Y,D)$ is obtained by introducing focus-focus singularities moving along the interior rays of the fan of $(Y',D')$, and pushing them to the origin, as in \cite{GHK1}. 
\end{example}

\begin{definition} \label{def_1_surface}
A \emph{$1$-surface} is an open log Calabi--Yau surface $(Y,D)$ 
with affinization $\Spec\, H^0(Y,\mathcal{O}_Y) \simeq \mathbb{A}^1$, such that, denoting by 
$\pi: Y \rightarrow  \mathbb{A}^1$ the affinization morphism and $0\in \mathbb{A}^1$ the image by $\pi$ of the compact components of $D$, the following holds:
\begin{itemize}
\item[i)] The morphism $\pi$ is a projective surjective morphism.
    \item[ii)] $\pi^{-1}(\mathbb{A}^1 \setminus \{0\}) \simeq \PP^1 \times (\mathbb{A}^1 \setminus \{0\})$.
    \item[iii)] $\pi|_{\pi^{-1}(\mathbb{A}^1 \setminus \{0\})}: \PP^1 \times (\mathbb{A}^1 \setminus \{0\}) \rightarrow (\mathbb{A}^1 \setminus \{0\})$ is the projection on the second factor.
\end{itemize}
\end{definition}

If $(Y,D)$ is a $1$-surface, then every compact curve of $Y$ is contained in a fiber of $Y \rightarrow \mathbb{A}^1$. This implies that $Y$ contains a $1$-dimensional family of compact curves, but does not contain a $2$-dimensional family of compact curves.
In particular, we have $C^2 \leq 0$ for every compact curve $C$ in $Y$, and we have $C^2=0$ if and only if $C$ is a fiber of 
$\pi: Y \rightarrow \mathbb{A}^1$.

\begin{example}
    A toric $1$-surface is a toric surface whose fan has support a half-plane in $\mathbb{R}^2$. Any $1$-surface $(Y,D)$
    is obtained from a toric $1$-surface $(Y',D')$ by interior blow-ups and corner blow-downs. The tropicalization of $(Y,D)$ is obtained by introducing focus-focus singularities moving along the interior rays of the fan of $(Y',D')$, pushing them to the origin, as in \cite{GHK1}, and eventually removing some interior rays.
\end{example}

\begin{definition} \label{def_2_surface}
A \emph{$2$-surface} is a smooth log Calabi--Yau surface.
\end{definition}

A $2$-surface is projective, and so is covered by families of compact curves of dimension $\geq 2$.

\subsubsection{Open Kulikov surfaces}

\begin{definition} \label{Def:K_disk}
An \emph{open Kulikov surface} is a reduced normal crossing surface $\cX_0$ satisfying the following conditions:
\begin{itemize}
\item[i)] The dual intersection complex $B$ of $\mathcal{X}_0$ 
is topologically a disk. 
\item[ii)] Let $\{X^i\}_{i\in I}$ be the set of irreducible components of $\mathcal{X}_0$, and denote by $\{v_i\}_{i\in I}$ the corresponding vertices in $B$. For every $i\in I$, denote the union of double curves in $X^i$ by $\partial X^i$. Denote by $\tilde{X}^i$ the normalization of $X^i$ and by $\partial \tilde{X}^i$ the pre-image of $\partial X^i$ in $\tilde{X}^i$.
Then, for each interior vertex $v_i \in \mathrm{Int}(B)$, the pair $(\tilde{X}^i, \partial \tilde{X}^i)$ is a 2-surface, and for each boundary vertex $v_i \in \partial B$, the pair $(\tilde{X}^i, \partial \tilde{X}^i)$ is 
either a $0$-surface or a $1$-surface. Moreover, there exists at least one boundary vertex $v_i \in \partial B$ such that the pair $(\tilde{X}^i, \partial \tilde{X}^i)$ is a $0$-surface. 
 \item[iii)] The surface $\cX_0$ is \emph{combinatorially $d$-semistable}, that is, for every proper irreducible double curve $C$ of $\cX_0$, denoting by $C_i \subset \tilde{X}^i$ and $C_j \subset \tilde{X}_j$ the two connected components of the pre-image of $C$ in the normalization of $\cX_0$, we have   
 \[(C_i|_{\tilde{X}^i})^2+(C_j|_{\tilde{X}_j})^2=-2+2g(C)\,,\] where $g(C)$ is the arithmetic genus of $C$.
\end{itemize}
\end{definition}

\begin{lemma}
\label{Lem: Gorenstein}
An open Kulikov surface $\cX_0$ is Gorenstein and has trivial dualizing line bundle, that is, $\omega_{\cX_0}=\mathcal{O}_{\cX_0}$.
\end{lemma}

\begin{proof}
By Definition \ref{Def:K_disk}, the irreducible components of $\cX_0$ are either open or compact log Calabi--Yau surfaces glued in a normal crossing way along anticanonical divisors, and so the result follows. 
\end{proof}

\begin{remark}
    If one replaces i) in Definition \ref{Def:K_disk} by the condition that $B$ is a sphere, and ii) by the condition that $(\tilde{X}^i, \partial \tilde{X}^i)$ is a $2$-surface for every vertex $v_i$ of $B$, then one obtains the definition of a Type III Kulikov surface, appearing in the context of Type III Kulikov degenerations of K3 surfaces \cite{ AE, FSIII, Kulikov, PP}.
    On the other hand, if one still replaces i) by the condition that $B$ is a sphere, but one rather replaces ii) by the condition that every $(\tilde{X}^i, \partial \tilde{X}^i)$ is a $2$-surface, except one which is a Hirzebruch-Inoue surface, then one obtains the definition of a ``Type III anticanonical pair'', appearing in the context of smoothings of cusp singularities in \cite{Eng18, EF21}.
\end{remark}

The dual intersection complex $B$ of an open Kulikov surface $\cX_0$ is a simplicial complex, containing a vertex for every irreducible component, a 1-simplex for each irreducible component of the double locus, and a 2-simplex for each triple point. By Definition \ref{Def:K_disk} i), the topological realization $P$ of $B$ is homeomorphic to a disk, and we denote by $\sP$ the triangulation into triangles of integral size $1$
induced by the simplicial structure.
Furthermore, Definition \ref{Def:K_disk} iii) implies that the natural integral affine structure on the interior of the triangles of $\sP$ can be naturally extended to the complement of the vertices of $\sP$. This integral affine structure extends over a vertex $v_i$ if the corresponding irreducible component $(X^i,
\partial X^i)$ is a toric pair, but not in general. 
We refer to $(P,\sP)$ as the \emph{tropicalization} of $\cX_0$, where $P$ is viewed as an integral affine disk with singularities.

\subsubsection{Properties of open Kulikov surfaces}
Let $\cX_0$ be an open Kulikov surface with tropicalization $(P,\sP)$.
We denote by $X^v$, $X^e$, and $X^f$ the irreducible components of $\cX_0$, the irreducible components of the double locus of $\cX_0$, and the triple points of $\cX_0$, indexed by the vertices $v$, the edges $e$, and the 2-dimension faces $f$ of $\sP$
respectively.

\begin{lemma} \label{Lem:Pic}
For every open Kulikov surface $\cX_0$, we have  
$H^i(\cX_0,\mathcal{O}_{\cX_0})=0$ for $i>0$, and 
\begin{align*} \Pic(\cX_0)=
H^2(\cX_0,\ZZ) &=
\mathrm{Ker}
\left( 
\bigoplus_{v \in \sP^{[0]}} \Pic(X^v) \rightarrow 
\bigoplus_{e \in \sP^{[1]}} \Pic(X^e)
\right) \\
&=\mathrm{Ker}
\left( 
\bigoplus_{v \in \sP^{[0]}} H^2(X^v,\ZZ) \rightarrow 
\bigoplus_{e \in \sP^{[1]}} H^2(X^e,\ZZ)
\right) \,.\end{align*}
\end{lemma}

\begin{proof}
As the irreducible components $X^v$ of $\cX_0$ are log Calabi--Yau surfaces and the irreducible components $X^e$ of the double locus are projective lines, we have 
$H^i(X^v, \mathcal{O}_{X^v})=
H^i(X^e, \mathcal{O}_{X^e})
=0$ for $i>0$, $\mathrm{Pic}(X^v)=H^2(X^v,\ZZ)$ and $\mathrm{Pic}(X^e)=H^2(X^e,\ZZ)$.
Lemma \ref{Lem:Pic} then follows from the resolution of $\mathcal{O}_{\cX_0}$ given by the complex
\[ 0 \rightarrow \bigoplus_{v \in \sP^{[0]}} \mathcal{O}_{X^v_\sP} 
\rightarrow \bigoplus_{e \in  \sP^{[1]}} \mathcal{O}_{X^e_\sP} 
\rightarrow 
\bigoplus_{f \in  \sP^{[2]}} \mathcal{O}_{X^f_\mathscr{P}}
\rightarrow 0 \,,\]
the corresponding resolution of 
$\mathcal{O}_{\cX_{\sP,0}}^{\star}$, and the fact that the tropicalization $P$ is topologically a disk and so $H^0(P,\ZZ)=\ZZ$, and $H^i(P,\ZZ)=0$ for $i>0$.
\end{proof}

\begin{lemma} \label{lem_rank_pic}
Let $\cX_0$ be an open Kulikov surface with tropicalization $(P,\sP)$. Then, we have 
$\mathrm{rk} \Pic(\cX_0) \geq Q+|P_\ZZ|-3$,
where $Q$ is the charge of $(Y,D)$ as in \cite[Definition 1.1]{friedman2015geometry}, that is, the number of interior blow-ups in a toric model of $(Y,D)$, and $|P_\ZZ|$ is the number of integral points in $P$.
\end{lemma}

\begin{proof}
Denote by $V=|P_\ZZ|$, $E$, $E_{\mathrm{int}}$, $F$ the number of vertices, edges, interior edges, and faces of $\sP$ respectively.
For every verter $v$, we have $\mathrm{rk}\Pic(X^v)=Q^v+E^v-2$, where $Q^v$ is the charge of $X^v$, and $E^v$ is the number of edges adjacent to $v$. Thus, by Lemma \ref{Lem:Pic}, we obtain
\[ \mathrm{rk}\, \mathrm{Pic}(\cX_0) \geq 
 \left(\sum_{v \in \sP^{[0]}} Q_v \right)+ \left(\sum_{v \in \sP^{[0]}} E^v\right)-2V-E_{int}=
Q+ 2E-2V-E_{int}\,.\] 
As $\sP$ is a triangulation of a disk, we have $V-E+F=1$ and $3F=E+E_{int}$, and so
\[2E-2V-E_{int}= 2E-2V-(3F-E)=2E-2V-3(1+E-V)+E =V-3\,,\]
\end{proof}

\begin{remark}
    We will see in Corollary \ref{cor_dim} that we actually have $\mathrm{rk} \Pic(\cX_0)= Q+|P_\ZZ|-3$.
\end{remark}

\subsection{Affinization of open Kulikov surfaces}
\label{Sec: affinization_open_Kulikov_surface}

Let $\cX_0 =\bigcup_{v \in P_\ZZ} X^v$
be an open Kulikov surface, with tropicalization $(P,\sP)$, where we denote by $X^v$ the irreducible component of $\cX_0$ corresponding to the integral point $v \in P_\ZZ$.
Denote by $\cX_0^{\mathrm{can}}$ the affinization of $\cX_0$, that is, $\cX_0^{\mathrm{can}}=\Spec\, H^0(\cX_0, \cO_{\cX_0})$, and by $f_0: \cX_0 \rightarrow \cX_0^{\mathrm{can}}$ the affinization morphism.

\begin{figure}[h]
\center{\includegraphics{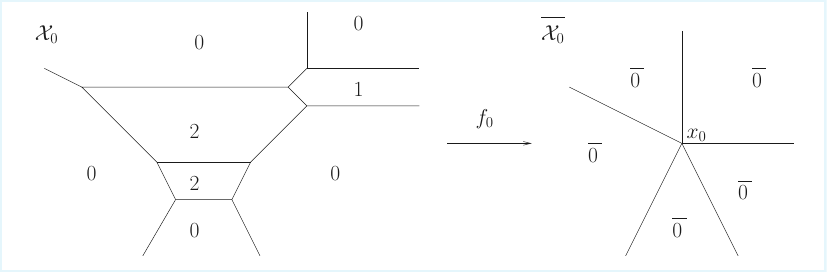}}
\caption{Affinization $f_0: \cX_0 \rightarrow \cX^{\mathrm{can}}_0$ of an open Kulikov surface $\cX_0$.}
\label{figure5}
\end{figure}

\begin{lemma}
\label{Lem: contraction1}
Let $\cX_0$ be an open Kulikov surface.
Then, the affinization morphism $f_0 : \cX_0 \rightarrow \cX^{\mathrm{can}}_0$ is a contraction, that is, a surjective proper morphism with $(f_0)_{\star} \mathcal{O}_{\cX_0}
=\mathcal{O}_{\cX^{\mathrm{can}}_0 }$.
Moreover, we have $R^i (f_0)_{\star} \mathcal{O}_{\mathcal{X}_0}=0$ for all $i>0$.
\end{lemma}

\begin{proof}
We denote by $P_\ZZ(0)$ (resp.\ $P_\ZZ(1)$ and $P_\ZZ(2)$) the set of integral points $v \in P_\ZZ$ such that $X^v$ is a 0-surface (resp.\ a 1-surface and a 2-surface). 
By Definition \ref{Def:K_disk} ii), we have $P_\ZZ(0)\neq \emptyset$, and, for every boundary integral point $v \in (\partial P)_\ZZ$, we have either $v \in P_\ZZ(0)$ or $v \in P_\ZZ(1)$. 
By Definition \ref{Def:K_disk} i), $P$ is topologically a disk, and so the boundary $\partial P$ is topologically a circle. If one fixes an orientation of $\partial P$, we have a corresponding cyclic orientation of the boundary vertices $v \in (\partial P)_\ZZ =P_\ZZ(0) \cup P_\ZZ(1)$. 
We denote by $(v_i)_{i \in \mathbb{Z}/N \mathbb{Z}}$ the cyclically ordered vertices in $P_\ZZ(0)$.
Moreover, $P_\ZZ(2)$ is exactly the set of vertices of $\sP$
contained in the interior of $P$.

For every integral point 
$v \in P_\ZZ$, we denote by $(\overline{X}^v ,\partial \overline{X}^v)$ the affinization of $(X^v, \partial X^v)$, that is,
\[ (\overline{X}^v ,\partial \overline{X}^v):=(\Spec\, H^0(X^v, \mathcal{O}_{X^v}), \Spec\, H^0(\partial X^v, \mathcal{O}_{\partial X^v})) \,.
\]
If $v \in P_\ZZ(2)$, then $X^v$ is compact, and so 
$(\overline{X}^v, \partial \overline{X}^v)$ is just a point. If $v \in P_\ZZ(1)$, then $(\overline{X}^v, \partial \overline{X}^v)=(\mathbb{A}^1,\mathbb{A}^1)$ by Definition \ref{def_1_surface}. 
For $i \in \mathbb{Z}/N \mathbb{Z}$, corresponding to an integral point $v_i \in P_\ZZ(0)$, 
the affinization
$(\overline{X}^{v_i}, \partial \overline{X}^{v_i})$ is an affine toric surface by Definition \ref{def_2_surface}. Moreover,
$\partial \overline{X}^{v_i}$ is the corresponding toric divisor, that is, the union of two irreducible toric divisors $(\partial \overline{X}^{v_i})_{i-1}$ and $(\partial \overline{X}^{v_i})_{i+1}$
corresponding to the two boundary edges of $\sP$ adjacent to $v_i$
and contained in the intervals $[v_{i-1},v_i]$ and $[v_i, v_{i+1}]$ of $\partial P$ respectively. Both $(\partial \overline{X}^{v_i})_{i-1}$ and $(\partial \overline{X}^{v_i})_{i+1}$ are isomorphic to $\mathbb{A}^1$, with $0$ corresponding to the torus fixed point of $\overline{X}^{v_i}$.

The irreducible components of $\cX^{\mathrm{can}}_0$ are the affine toric surfaces $\overline{X}^{v_i}$ for $i \in \ZZ/N\ZZ$. Moreover, the double locus of $\cX^{\mathrm{can}}_0$ is a union of copies of affine lines 
$(\partial X^{v_i})_{i+1} \simeq (\partial X^{v_{i+1}})_{i} \simeq \mathbb{A}^1$, and all intersecting at a single common point $x_0 \in \cX^{\mathrm{can}}_0$, referred below as the \emph{central point} of $\cX^{\mathrm{can}}_0$, see Figure \ref{figure5}.

For every $v \in (\partial P)_\ZZ$, the affinization map $X^v \rightarrow \overline{X}^v$ is proper and surjective, and so $f_0$ is also proper and surjective. 
To compute $R^i (f_0)_{\star} \mathcal{O}_{\cX_0}$, one uses the resolution of $\mathcal{O}_{\cX_0}$ given by the complex 
\begin{equation} \label{eq_resolution}
0 \rightarrow \bigoplus_{v \in \mathscr{P}^{[0]}} \mathcal{O}_{X^v} 
\rightarrow \bigoplus_{e \in  \mathscr{P}^{[1]}} \mathcal{O}_{X^e} 
\rightarrow 
\bigoplus_{f \in  \mathscr{P}^{[2]}} \mathcal{O}_{X^f}
\rightarrow 0 \,,\end{equation}
as in the proof of Lemma \ref{Lem:Pic}. 
We have the following vanishing results:
\begin{itemize} 
    \item[i)] If $v \in P_\ZZ(0)$, then,  $\overline{X}^v$ is an affine toric surface, so with cyclic quotient, and so rational, singularities, and the affinization map $X^v \rightarrow \overline{X}^v$ is a resolution of singularities. It follows that $R^i (f_0)_\star \mathcal{O}_{X^v}=0$ for $i>0$.
    \item[ii)] If $v \in P_\ZZ(1)$, then the affinization map $X^v \rightarrow \overline{X}^v$ is generically a $\PP^1$-fibration, and we have $R^i (f_0)_\star \mathcal{O}_{X^v}=0$ for $i>0$.
    \item[iii)] If $v \in P_\ZZ(2)$ is an interior vertex or if $e$ is an interior edge of $\mathscr{P}$, then $f_0$ is constant on $X^v$ or $X^e$, and so $R^i (f_0)_{\star} \mathcal{O}_{\mathcal{X}^v} = R^i (f_0)_{\star} \mathcal{O}_{\mathcal{X}^e}=0$ for $i>0$ follows from the  
    vanishings $H^i(\mathcal{X}^v, \mathcal{O}_{\mathcal{X}^v})= H^i(\mathcal{X}^e, \mathcal{O}_{\mathcal{X}^e})=0$, which hold since $\mathcal{X}^v$ is a rational surface and we have $\mathcal{X}^e \simeq \PP^1$.  
    \item[iv)] If $e$ is a boundary edge of $\mathscr{P}$, 
    then the induced map 
    $X^e \simeq \mathbb{A}^1 \rightarrow \overline{X}^e \simeq \mathbb{A}^1$ is an isomorphism. In particular, we have $R^i (f_0)_\star \mathcal{O}_{X^e}=0$ for $i>0$. 
\end{itemize}
Hence, the first page of the spectral sequence obtained by applying $R (f_0)_\star$ to the complex \eqref{eq_resolution} reduces to the complex 
\[ 0 \rightarrow \bigoplus_{v \in \mathscr{P}^{[0]}} (f_0)_\star\mathcal{O}_{X^v} 
\rightarrow \bigoplus_{e \in  \mathscr{P}^{[1]}} (f_0)_\star\mathcal{O}_{X^e} 
\rightarrow 
\bigoplus_{f \in  \mathscr{P}^{[2]}} (f_0)_\star\mathcal{O}_{X^f}
\rightarrow 0 \,.\]
The stalk of this complex at the point $x_0$ is isomorphic to the tensor product of the local ring $\cO_{\cX_0^{\mathrm{can}}, x_0}$ with the simplicial chain complex defined by the polyhedral decomposition $\sP$ of the tropicalization $P$. Since $P$ is topologically a disk, we have $H^0(P,\ZZ)=\ZZ$ and $H^i(P,\ZZ)=0$ for $i>0$, and so the stalk of
$R^i (f_0)_{\star} \mathcal{O}_{\mathcal{X}_0}=0$ at $x_0$ is equal to zero for $i>0$. Similarly, at a point of the double locus (resp.\ a smooth point) of $\cX_0^{\mathrm{can}}$, the disk is replaced by an interval (resp.\ a point), and so Lemma \ref{Lem: contraction1} follows.
\end{proof}

\begin{lemma} \label{Lem: contraction2}
Let $\cX_0$ be an open Kuliukov surface. Then, the affine surface $\cX_0^{\mathrm{can}}$ is a degenerate cusp singularity
as in \cite[\S 1]{SB2}. In particular, $\cX_0^{\mathrm{can}}$ is a Gorenstein semi-log-canonical surface with 
$\omega_{\cX_0^{\mathrm{can}}}=\cO_{\cX_0^{\mathrm{can}}}$.
\end{lemma}

\begin{proof}
By Lemma \ref{Lem: contraction1}, the surface $\cX_0^{\mathrm{can}}$ has semi-rational singularities in the sense of \cite[Definition 4.1]{berquist2014semirational}. Hence, $\cX_0^{\mathrm{can}}$ is Cohen-Macaulay by \cite[Corollary 4.4]{berquist2014semirational}.
Moreover, it follows from the proof of Lemma \ref{Lem: contraction1} that the semi-normalization of $\cX_0^{\mathrm{can}}$ is the disjoint union of the affine toric surfaces $\overline{X}^{v_i}$ for $i \in \ZZ/N\ZZ$, that is, of cyclic quotient singularities. In addition, the irreducible components $\overline{X}^{v_i}$ of $\cX_0^{\mathrm{can}}$ are glued together cyclically along their toric boundary divisors. Hence, $\cX_0^{\mathrm{can}}$ is a degenerate cusp singularity as in \cite[\S 1]{SB2}. 
It then follows from  \cite[Lemma 1.1]{SB2} that  $\cX_0^{\mathrm{can}}$ is Gorenstein with 
$\omega_{\cX_0^{\mathrm{can}}}=\cO_{\cX_0^{\mathrm{can}}}$, and from \cite[Theorem 4.21]{KSB} that $\cX_0^{\mathrm{can}}$ is semi-log-canonical.
\end{proof}

\subsection{Smoothing of $d$-semistable open Kulikov surfaces}
\label{Sec: smoothing}

Recall from \cite{friedman2015geometry}
that a normal crossing surface $X$, with singular locus $X_{\mathrm{sing}}$, is called \emph{$d$-semistable} if $\mathcal{E}xt^1(\Omega_X,\mathcal{O}_X) \cong {\mathcal{O}_X}_{\mathrm{sing}}$.
A fundamental result in the study of degenerations of K3 surfaces is the existence of smoothings of $d$-semistable Type III Kulikov surfaces, established by Friedman \cite{Frie83}. 
In this section, we prove in 
Theorem \ref{thm: smoothing exists} the existence of smoothings of $d$-semistable open Kulikov surfaces. 
Technical challenges arise in this setting due to the non-compactness of open Kulikov surfaces. Furthermore, the existence of $(-1)$-forms for Type III Kulikov surfaces, 
which plays a crucial role in the proof of \cite{Frie83}, has no direct analogue in the framework of open Kulikov surfaces.
Therefore, we must develop an alternative geometric argument to establish the vanishing of obstructions to smoothing.

\subsubsection{Log smooth deformation theory}
\label{sec log smooth def theory}
In this section, we briefly review definitions coming from log geometry to establish notation for the following section, where we will use the theory of logarithmic deformations of normal crossing varieties, developed by Kawamata--Nawikawa \cite{KawamataNamikawa} to describe a smoothing of a $d$-semistable open Kulikov surface. We assume basic familiarity with log geometry, in the sense of \cite{KatoLog} -- a comprehensive exposition of log structures can be found in \cite{OgusLog}.

Recall that on a scheme $X$, a \emph{log structure} consists of a sheaf of commutative monoids $\mathcal{M}_X$ and a homomorphism of sheaves $\alpha_X : \mathcal{M}_X \to \mathcal{O}_X$ between multiplicative monoids, inducing an isomorphism $\alpha_X^{-1}(\mathcal{O}_X^{\times})\rightarrow \mathcal{O}_X^{\times}$. 
This isomorphism allows us to identify $\mathcal{O}_X^{\times}$ as a subsheaf of $\mathcal{M}_X$. 
A scheme $X$ endowed with a log structure is called a 
\emph{log scheme}, and is denoted by $X^{\dagger}:=(X,\mathcal{M}_X)$ -- we suppress $\alpha_X$ from the notation when its meaning is clear from context. 
Throughout this paper we assume that all log schemes are fine and saturated \cite[I,\S1.3]{OgusLog}. 

The sheaf $\overline{\mathcal{M}}_{X}:=\mathcal{M}_X/\mathcal{O}_X^{\times}$, known as the \emph{ghost sheaf} or \emph{characteristic sheaf}, encodes the key combinatorial
aspects of the log structure and, in particular, 
determines the \emph{tropicalization} $\Sigma(X)$ of $X$, which is an abstract polyhedral
cone complex -- see \cite{ACGSdecomposition}, \S2.1 for details. In brief, the tropicalization
$\Sigma(X)$ consists of a collection of cones along with face maps between
them. For every 
geometric point $\bar x\rightarrow X$, there is an associated cone $\sigma_{\bar x}:=
\mathrm{Hom}(\overline{\mathcal{M}}_{X,\bar x},\RR_{\geq 0})$. 
Moreover, if $\bar x$ specializes to $\bar y$, there is a generization map
$\overline{\mathcal{M}}_{X,\bar y}\rightarrow \overline{\mathcal{M}}_{X,\bar x}$, which dually induces a map
$\sigma_{\bar x}\rightarrow\sigma_{\bar y}$.

An important log structure for us is the \emph{divisorial log structure}, which is associated with a scheme $X$ with a divisor $D \subset X$. This log structure, denoted by $\mathcal{M}_{(X,D)}$, is defined as the subsheaf of $\mathcal{O}_X$ consisting of functions that are invertible on the complement $X\setminus D$. 
In the situation when $X = \Spec \, \CC[Q] $ is a toric variety endowed with the divisorial log structure induced
 by the toric boundary divisor $D:= \partial X$, the tropicalization $\Sigma(X)=(\RR^n,\Sigma)$ corresponds to the toric fan $\Sigma \subset \RR^n$ for $X$. 

Given a log scheme $Y^{\dagger} = (Y,\mathcal{M}_Y)$ and a morphism of schemes $f:X\to Y$, one can define an induced log structure, also known as \emph{pull-back log structure} on $X$, which we denote by $f^*\mathcal{M}_Y$ -- see \cite[Definition 2.16]{ArguzGGT}. For a monoid $Q$
such that  $Q^\times=\{0\}$, pulling back the divisorial log structure on the affine toric variety $X = \Spec \, \CC[Q]$ to its torus fixed point $\Spec\, \CC$, defines a log structure on $\Spec\, \CC$ given by
\begin{equation} \label{eq_log_point}
Q\oplus\CC^\times\longrightarrow \CC,\quad
(q,a)\longmapsto \begin{cases} a,&q=0\\ 0,&q\neq0\,.\end{cases}
\end{equation}
We denote the corresponding log point by $\mathrm{pt}_Q$.
For $Q:=\NN$, this yields the \emph{standard log point} 
$\mathrm{pt}_{\NN}:=(\Spec \, \CC,\NN \oplus \CC^{\times})$.

A \emph{morphism of log schemes} $f^{\dagger}:X^{\dagger} \rightarrow Y^{\dagger}$ consists of a scheme theoretic morphism $f:X\rightarrow Y$, along with
a map $f^{\flat}:f^{-1}\mathcal{M}_Y\rightarrow \mathcal{M}_X$ that is compatible with $f^\#:f^{-1}\mathcal{O}_Y\rightarrow \mathcal{O}_X$ via the structure homomorphisms
$\alpha_X$ and $\alpha_Y$. 
The tropicalization of log schemes is functorial: 
a morphism of log schemes $f^{\dagger}:X^{\dagger} \rightarrow Y^{\dagger}$ induces a map of cone complexes
$f_{\mathrm{trop}}:\Sigma(X)\rightarrow\Sigma(Y)$. 
Given a one-parameter degeneration of a scheme $X$, with reduced central fiber $X_0$, the divisorial log structure $\mathcal{M}_{(X,X_0)}$ induces a log structure $\mathcal{M}_{X_0}$ on $X_0$. The tropicalization $\Sigma(X_0)$ of the log scheme $(X_0,\mathcal{M}_{X_0})$ is the cone $(C(P), C(\sP))$ over the dual intersection complex $(P,\sP)$ of $X_0$ -- see for instance \cite[Proposition 7.4.i]{Arguzcorals}.

Similarly to scheme-theoretic morphisms,
for which there is the notion of \emph{smoothness} defined using the infinitesimal lifting criterion, there is for log morphisms an analogous notion known as \emph{log smoothness}. 
In the following sections, we employ log smooth deformation theory to study $d$-semistable open Kulikov surfaces. While log smooth deformation theory is developed in a general framework in \cite{FumiharuK}, it suffices for our purposes to consider the normal crossing situation, following \cite{KawamataNamikawa}.

\subsubsection{Smoothing of $\mathcal{X}_0$}

Throughout this section, for every log morphism $f^{\dagger}: X^{\dagger} \rightarrow Y^{\dagger}$ between two log schemes $X$ and $Y$, we denote by $T_{X^{\dagger}/Y^{\dagger}}$
(resp.\ $\Omega_{X^{\dagger}/Y^{\dagger}}$)
the associated relative log tangent (resp.\ cotangent) sheaf, which reduces to a vector bundle in the situation when the morphism $f^{\dagger}$ is log smooth \cite[Proposition 3.10]{KatoLog}.

\begin{lemma}
\label{lem: existence of good log str}
Let $\cX_0$ be a $d$-semistable open Kulikov surface with dual intersection complex $(P, \sP)$.
There exists a log structure $\mathcal{M}_{\mathcal{X}_0}$ on $\mathcal{X}_0$ and a log smooth morphism 
\begin{equation}
\label{Eq: log smooth map}
    {\pi_0}^{\dagger} \colon {\mathcal{X}_0}^\dagger = ( \mathcal{X}_0, \mathcal{M}_{\mathcal{X}_0}) \longrightarrow \mathrm{pt}_\NN \, ,
\end{equation}
to the standard log point, whose tropicalization is the height function 
$h \colon (C(P), C(\mathscr{P})) \to (\RR_{\geq 0}, \Sigma_{\mathbb{A}^1})$,
given by projection onto the second factor.      
\end{lemma}

\begin{proof}
Since $\mathcal{X}_0$ is $d$-semistable, by \cite[Proposition 1.1]{KawamataNamikawa} it admits a log structure which is log smooth over the standard log point $\mathrm{pt}_\NN$. Since by construction $(P,\mathscr{P})$ is the dual intersection complex of $\mathcal{X}_0$, the statement about the tropicalization is immediate by the description of the tropicalization of a map, as explained in \S \ref{sec log smooth def theory}.
\end{proof}

The key ingredient to show the existence of the smoothing of $\mathcal{X}_0$, is the ``vanishing result'' shown in Lemma \ref{lem:obstruction}. Its proof uses the following observations on smooth log Calabi--Yau surfaces and open log Calabi--Yau surfaces.

\begin{lemma}
\label{lem:small}
  Let $(X,D)$ be a smooth log Calabi--Yau surface, and denote by $D=D_1 + \ldots + D_n$ the cyclically ordered irreducible components of $D$. Then, the components $D_3, \ldots , ,D_n$ are linearly independent in $H^2(X,\CC)$.  
\end{lemma}

\begin{proof}
Let $a_3D_3 + \ldots a_n D_n =0$
be a linear relation between $D_3, \dots, D_n$ with coefficients $a_i \in \CC$. 
Taking the intersection with $D_2$, we obtain $a_3 (D_2\cdot D_3)+\ldots+ a_n (D_2\cdot D_n) = 0$. Since $D_1 + \ldots + D_n$ is cyclically ordered, we have
$D_2\cdot D_3 = 1$ and  $D_2\cdot D_i = 0$ for all $3 \leq i \leq n$, and so $a_3=0$. Thus, the linear relation reduces to $a_4D_4 + \ldots+ a_n D_n =0$. Taking the intersection with $D_4$, we similarly deduce that $a_4=0$. Iterating the argument, we eventually get $a_i = 0$ for all $3 \leq i \leq n$.
\end{proof}

\begin{lemma}
\label{lem:small2}
  Let $(X,D)$ be an open log Calabi--Yau surface, and denote by $D=D_0+D_1+ \ldots + D_{n-1}+D_n$ the ordered irreducible components of $D$, so that $D_0$ and $D_n$ are non-compact, and $D_1, \cdots, D_{n-1}$ are compact. Then, the compact components $D_1, \ldots , ,D_{n-1}$ are linearly independent in the cohomology group with compact support $H^2_c(X,\CC)$.  
\end{lemma}

\begin{proof} Analogue to the proof of Lemma \ref{lem:small}, by successively intersecting a relation $\sum_{i=1}^{n-1} a_i D_i=0$ with $D_0$, $D_1$, ..., $D_{n-1}$.
\end{proof}

\begin{lemma}
\label{lem:obstruction}
Let $\cX_0$ be a $d$-semistable open Kulikov surface. Then,
we have 
    \[H^2(\mathcal{X}_0, T_{\mathcal{X}_0^{\dagger}/\mathrm{pt}_\NN
     })= 0\,,\] where $T_{\mathcal{X}_0^{\dagger}/\mathrm{pt}_\NN
     }$ is the log tangent bundle of 
$\mathcal{X}_0^{\dagger}$ over the standard log point $\mathrm{pt}_\NN$.
\end{lemma}

\begin{proof}
By Lemma \ref{Lem: Gorenstein}, the
dualizing sheaf $\omega_{\cX_0}$ of $\cX_0$ is trivial, and so, by Serre duality, 
$H^2(\mathcal{X}_0, T_{\mathcal{X}_0^{\dagger}/\mathrm{pt}_\NN})$
is dual to $H^0_c(\mathcal{X}_0, \Omega_{\mathcal{X}_0^{\dagger}/\mathrm{pt}_\NN})$,
where $H^i_c$ denotes compactly supported coherent cohomology groups -- see \cite{MR297775} or \cite[\href{https://stacks.math.columbia.edu/tag/0G59}{Tag 0G59}]{stacks-project}, and $\Omega_{\mathcal{X}_0^{\dagger}/\mathrm{pt}_\NN}$ is the log cotangent bundle of 
$\mathcal{X}_0^{\dagger}$ over the standard log point $\mathrm{pt}_\NN$. Hence, to prove Lemma \ref{lem:obstruction}, it suffices to show that $H^0_c(\mathcal{X}_0, \Omega_{\mathcal{X}_0^{\dagger}/\mathrm{pt}_\NN})=0$.

To do this, we adapt to our situation the argument given in the proof of \cite[Lemma 5.9]{Frie83} (where the log cotangent bundle is denoted by $\Lambda^1_X$) in the context of degenerations of K3 surfaces -- see also \cite[Theorem 5.2]{log_enriques} for a similar result in the context of degenerations of Enriques surfaces.
We denote by $(P,\sP)$ the dual intersection complex of $\cX_0$, with vertices $v\in \sP^{[0]}$ (resp.\ edges $e\in \sP^{[1]}$ and triangular faces $f\in \sP^{[2]}$) in one-to-one correspondence with the irreducible components $X^v$
(resp.\ double curves $X^e$ and triple points $X^f$) of $\cX_0$.
By \cite[Proof of Theorem 4.2]{KawamataNamikawa} (see also \cite[Corollary 3.6]{Frie83} or \cite[Proof of Lemma 2.21]{EF21} where Equation \eqref{eq_F_kernel} below is taken as the starting point of the discussion), it follows from the definition of the log cotangent bundle that we have a short exact sequence
\begin{equation} \label{eq_short_exact}
0 \rightarrow \Omega_{\mathcal{X}_0}/\tau \rightarrow \Omega_{\mathcal{X}_0^{\dagger}/\mathrm{pt}_\NN} \rightarrow F \rightarrow 0 \,,
\end{equation}
where $\Omega_{\mathcal{X}_0}$ is the usual cotangent sheaf of $\mathcal{X}_0$, $\tau$ is the torsion subsheaf of $\Omega_{\mathcal{X}_0}$, and the cokernel $F$ sits via the Poincar\'e residue map into the short exact sequence 
\[ 0 \rightarrow \mathcal{O}_{\mathcal{X}_0} 
\rightarrow \bigoplus_{v \in \mathscr{P}^{[0]}}\mathcal{O}_{X^v} \rightarrow F \rightarrow 0\,.\] 
The long exact sequence in cohomology with compact support associated to the short exact sequence \eqref{eq_short_exact} is of the form
\begin{equation}\label{eq_long_exact}
0 \rightarrow H^0_c(\cX_0, \Omega_{\cX_0}/\tau) \rightarrow H^0_c(\cX_0, \Omega_{\cX_0^{\dagger}/\mathrm{pt}_\NN}) \rightarrow H^0_c(\cX_0, F) \xrightarrow{\epsilon}  H^1_c(\cX_0, \Omega_{\cX_0}/\tau) \,.\end{equation}
On the other hand, $\tau$ is exactly the kernel of the natural restriction map $\Omega_{\cX_0} \rightarrow \oplus_{v \in \sP^{[1]}} \Omega_{X^v}$ (see \cite[p77]{Frie83}), so $\Omega_{\cX_0}/\tau$ is a subsheaf of $\oplus_{v \in \sP^{[1]}} \Omega_{X^v}$.
Every proper irreducible component $X^v$ of $\cX_0$ is a proper log Calabi--Yau surface, satisfying $H^0(X^v, \Omega_{X^v}) \subset H^1(X^v, \mathbb{C})$ by the Hodge decomposition, and so $H^0(X^v, \Omega_{X^v})=0$ since 
$H^1(X^v, \mathbb{C})=0$. 
Thus, we obtain $H^0_c(\cX_0, \Omega_{\cX_0}/\tau)=0$, and so it remains to prove that the map $\epsilon$ in \eqref{eq_long_exact} is injective.

It follows from the exact sequence
\[ 0 \rightarrow \mathcal{O}_{\cX_0}
\rightarrow
\bigoplus_{v \in \mathscr{P}^{[0]}} \mathcal{O}_{X^v} 
\rightarrow \bigoplus_{e \in  \mathscr{P}^{[1]}} \mathcal{O}_{X^e} 
\rightarrow 
\bigoplus_{f \in  \mathscr{P}^{[2]}} \mathcal{O}_{X^f}
\rightarrow 0 \,,\]
that we have
\begin{equation}\label{eq_F_kernel}
F=\mathrm{Ker}\left(\bigoplus_{e \in  \mathscr{P}^{[1]}} \mathcal{O}_{X^e} 
\rightarrow 
\bigoplus_{f \in  \mathscr{P}^{[2]}} \mathcal{O}_{X^f}\right)\,,\end{equation}
and so 
\[H^0_c(\cX_0, F)=\mathrm{Ker}\left(\bigoplus_{e \in  \mathscr{P}^{[1]}} H^0_c(X^e, \mathcal{O}_{X^e}) 
\rightarrow 
\bigoplus_{f \in  \mathscr{P}^{[2]}} H^0_c(X^f,\mathcal{O}_{X^f})\right)\,.\]
Since  $H^0_c(X^e, \mathcal{O}_{X^e})=0$ if $X^e$ is non-compact, that is, $e \subset \partial P$, and $H^0_c(X^e, \mathcal{O}_{X^e})=\CC$ otherwise, 
$H^0_c(\cX_0, F)$ consists of tuples $(a_e)_{e\in \sP^{[1]}} \in \CC^{\sP^{[1]}}$, such that $a_e=0$ if $e \subset \partial P$, and such that $a_{e_1}+a_{e_2}+a_{e_3}=0$
if $e_1$, $e_2$, $e_3$ are three edges bounding a triangle $f \in \sP^{[2]}$.

On the other hand, for every $e \in \sP^{[1]}$ such that the double curve $X^e$ is proper, $X^e$ is a rational curve, and so $H^0(X^e,
\Omega_{X^e})=0$. Hence, the short exact sequence 
\[0 \rightarrow \Omega_{\cX_0}/\tau \rightarrow \bigoplus_{v \in \sP^{[0]}} \Omega_{X^v} \rightarrow \bigoplus_{e \in \sP^{[1]}} 
\Omega_{X^e} \rightarrow 0\]
implies that
\begin{equation*}
H^1_c(\cX_0, \Omega_{\cX_0}/\tau)
= \mathrm{Ker} 
\left( 
\bigoplus_{v \in \sP^{[0]}} H^1_c(X^v , \Omega_{X^v}) \rightarrow \bigoplus_{e \in \sP^{[1]}} 
H^1_c(X^e, \Omega_{X^e}) 
\right) 
\,.\end{equation*}
As in the proof of \cite[Lemma 5.9]{Frie83}, 
the map $\epsilon$ in \eqref{eq_long_exact} is the restriction of the map 
\begin{align} \label{eq_delta}
    \delta: \bigoplus_{e \in \sP^{[1]}} H^0_c(X^e, \mathcal{O}_{X^e}) &\longrightarrow \bigoplus_{v \in \sP^{[0]}} H^1_c(X^v , \Omega_{X^v}) \\
    (a_e)_{e \in \sP^{[1]}} &\longrightarrow \left(\sum_{e\in \sP^{[1]}} a_e [X^e]_v \right)_{v\in \sP^{[0]}}\,, \nonumber
\end{align}
where $[X^e]_v$ is the class of the compact curve $X^e$
in the surface $X^v$ if $X^e \subset X^v$, that is, if $e$ is an edge incident to the vertex $v$, and $[X^e]_v=0$ else.
We can now prove the injectivity of $\epsilon$. Let 
$(a_e)_{e\in \sP^{[1]}} \in H^0_c(\cX_0,F)$ such that $\delta((a_e)_e)=0$. Since the map $\epsilon$ is the restriction of the map $\delta$ given by \eqref{eq_delta}, we have $\sum_{e\in \sP^{[1]}} a_e [X^e]_v=0$ in $H^2_c(X^v,\CC)$ for every $v \in \sP^{[0]}$. 
For every vertex $v \in \sP^{[0]}$, define the \emph{distance to the boundary} $d(v)$ as $d(v)=0$ if $v \in \partial P$, and else as the smallest integer such that there exists a sequence 
$(e_1, \dots, e_{d(v)})$ of edges $e_i \in \sP^{[1]}$ such that $e_1$ is adjacent to $v$, $e_{i+1}$ is adjacent to $e_i$, and $e_{d(v)}$ is adjacent to a vertex in $\partial P$. 
We will prove that $a_e=0$ for all $e \in \sP^{[1]}$,
by showing by induction on $d(v) \geq 0$ that, for every vertex $v \in \sP^{[0]}$, we have $a_e=0$ for all edges $e \in \sP^{[1]}$ adjacent to $v$. 

For the initialization of the induction, note that if $v \in \sP^{[0]}$ is such that $d(v)=0$, then $v \in \partial P$. Denote by $e_0, e_1, \cdots, e_n$ the edges adjacent to $v$, cyclically ordered so that $e_0, e_n \subset \partial P$. Then, we have $a_{e_0}=a_{e_n}=0$, and so the linear relation $\sum_{i=1}^{n-1} a_{e_i} [X^{e_i}]_v=0$ in $H^2_c(X^v,\CC)$.
By Lemma \ref{lem:small2}, the classes $[X^{e_i}]_v$ for $1\leq i \leq n-1$ are linearly independent in $H^2_c(X^v,\CC)$, and so we conclude that $a_{e_i}=0$ for all $1 \leq i \leq n-1$.

For the induction step, let $d \in \ZZ_{\geq 1}$, and assume that the result holds for all vertices $v'$ with $d(v')<d$. Let $v \in \sP^{[0]}$ such that $d(v)=d$.
Then, there exists an edge $e_1$ adjacent to $v$ which is also adjacent to a vertex $v' \in \sP^{[0]}$ with $d(v')<d$. Let $f \in \sP^{[2]}$ be a triangle with side $e_1$, and denote by $e_2$ the other side adjacent to $v$, and by $e_3$ the other side adjacent to $v'$. By the induction hypothesis, we have $a_{e_1}=a_{e_3}=0$, and so it follows from the condition $a_{e_1}+a_{e_2}+a_{e_3}=0$ that we also have $a_{e_2}=0$. Therefore, the linear relation $\sum_{e \in \sP^{[1]}} a_e [X^e]_v=0$  only involves boundary divisor of $X^v$ distinct from the consecutive components $X^{e_1}$ and $X^{e_2}$. Hence, by Lemma \ref{lem:small}, such boundary divisors are linearly independent in $H^2(X^v,\CC)$, and so $a_e=0$ for all edges $e$ adjacent to $v$. This concludes the proof by induction that the map $\epsilon$ in \eqref{eq_long_exact} is injective, and so the proof of Lemma \ref{lem:obstruction}. 
    
\end{proof}

Finally, we are ready to state the existence of a formal smoothing of $\cX_0$. Recall that a \emph{formal smoothing} of $\cX_0$ is the data of a formal scheme $\cX$ and of a rig-smooth morphism of formal schemes $\pi \colon \mathcal{X} \to \Delta=\mathrm{Spf} \,\CC  \lfor t \rfor$ (i.e.\ a morphism with smooth rigid-analytic generic fiber 
 \cite[\href{https://stacks.math.columbia.edu/tag/0GCN}{Tag 0GCN}]{stacks-project}), such that $\cX_0=\pi^{-1}(0)$.

\begin{theorem}
\label{thm: smoothing exists}
Let $\cX_0$ be a $d$-semistable open Kulikov surface.
There exists a formal smoothing $\pi \colon \mathcal{X} \to \Delta$ of $\mathcal{X}_0$ over the formal disk $\Delta = \mathrm{Spf} \,\CC  \lfor t \rfor $, such that the following holds:
\begin{itemize}
    \item[i)] The total space $\mathcal{X}$ is regular.
      \item[ii)] The central fiber $\cX_0$ is a reduced normal crossing divisor in $\cX$.
       \item[iii)] The log smooth morphism \[ {\pi_0}^{\dagger} \colon {\cX_0}^\dagger = ( \cX_0, \mathcal{M}_{\cX_0}) \longrightarrow \mathrm{pt}_\NN\] in \eqref{lem: existence of good log str} is given by the restriction to the central fiber ${\mathcal{X}_0}$ of the log morphism 
       \[\pi^{\dagger} \colon \mathcal{X}^{\dagger} = (\mathcal{X}, \mathcal{M}_{(\mathcal{X}, \partial  \mathcal{X} )}) \longrightarrow \Delta^{\dagger} = (\Delta, \mathcal{M}_{(\Delta, {0})}) \, , \] 
       induced by the morphism $\pi \colon \mathcal{X} \to \Delta$.     
\end{itemize}
\end{theorem}

\begin{proof}
Since the obstruction group $H^2(\cX_0, T_{\cX_0^{\dagger}/\mathrm{Spec}\CC^{\dagger}
     })$ vanishes by Lemma \ref{lem:obstruction}, the result follows by log smooth deformation theory as in \cite[Corollary 2.4]{KawamataNamikawa}.
\end{proof}

\subsection{Open Kulikov degenerations}
\label{Sec: open Kulikov deg}
We define below open Kulikov degenerations as ``semistable smoothings'' of open Kulikov surfaces.

\begin{definition} \label{Def: open_Kulikov_deg}
An \emph{open Kulikov degeneration} is a formal smoothing
$\pi: \cX \rightarrow \Delta:= \mathrm{Spf}\, \CC[\![t]\!]$ satisfying the following conditions.
\begin{itemize}
    \item[i)] The total space $\cX$ is smooth.
    \item[ii)] The central fiber $\cX_0$ is a reduced normal crossing divisor in $\cX$.
    \item[iii)] The central fiber $\cX_0$ is an open Kulikov surface.
\end{itemize}
\end{definition}

\begin{remark}
    If one replaces condition iii) in Definition \ref{Def: open_Kulikov_deg}
    by ``the central fiber $\cX_0$ is a Kulikov surface", one recovers the notion of Type III Kulikov degeneration of K3 surfaces \cite{FSIII, Kulikov, PP}.
\end{remark}

By definition, the central fiber of an open Kulikov degeneration is a $d$-semistable open Kulikov surface. Conversely, by Theorem \ref{thm: smoothing exists}, every $d$-semistable open Kulikov surface is the central fiber of an open Kulikov degeneration.
Recall from \S\ref{Sec: affinization_open_Kulikov_surface} that every open Kulikov surface $\cX_0$ has an affinization $\cX^{\mathrm{can}}_0$. The following result describes the affinization of open Kulikov degenerations.

\begin{theorem}
\label{thm: ss smoothings}
   Let $\pi \colon \mathcal{X} \to \Delta$ be a semistable smoothing of an open Kulikov surface $\mathcal{X}_0$. Then, the following holds:
   \begin{itemize}
   \item[i)] The affinization $\cX^{\mathrm{can}} \coloneqq \mathrm{Spec}\, H^0(\mathcal{X},\mathcal{O}_{\mathcal{X}})$ is a formal deformation of $\cX^{\mathrm{can}}_0$, and has Gorenstein canonical singularities.
   \item[ii)] The rigid-analytic generic fiber $\cX^{\mathrm{can}}_\eta$ of $\cX^{\mathrm{can}}$ has at worst Du Val (ADE) singularities. 
   \item[iii)] $\cX$ has trivial canonical line bundle, and the affinization morphism $f \colon \mathcal{X} \to \cX^{\mathrm{can}}$ is a crepant resolution of singularities.
   \end{itemize}
\end{theorem}

\begin{proof}
We first prove Theorem \ref{thm: ss smoothings} i) and iii).
By Lemma \ref{Lem:Pic}, we have $H^1(\mathcal{X}_0,\mathcal{O}_{\mathcal{X}_0})=0$, and so it follows from
\cite[Theorem 1.4]{WahlI} (see also \cite[Theorem 3.1]{blowing_down_lifting}) that $f$ is a flat deformation of $f_0$.
To prove that $\cX^{\mathrm{can}}$ has Gorenstein canonical singularities, first note that $\mathcal{X}$ has trivial canonical line bundle -- this can be shown analogously as in \cite[Theorem 5.10, 3)]{Frie83}, where a similar result is proved in the context of degenerations of K3 surfaces. Hence, since $\mathcal{X}$ is a smooth threefold with trivial canonical class,
it follows by \cite[Corollary 1.5]{Kaw88} that $\cX^{\mathrm{can}}$ has Gorenstein canonical singularities. 
Note that \cite[Corollary 1.5]{Kaw88} is stated in the context of projective morphism between normal algebraic varieties. In this context, the result follows from the base point free theorem \cite[Theorem 1.3]{Kaw88}.
The same argument shows that the result for projective morphisms of formal schemes follows from the corresponding version of the base point free theorem
in \cite[\S 11]{MMPformal}.

To prove Theorem \ref{thm: ss smoothings} ii),
denote by $\cX_\eta$ the rigid-analytic generic fiber of $\cX$, which is smooth since $\cX \rightarrow \Delta$ is smoothing. Standard results on birational geometry of surfaces generalize to rigid-analytic surfaces by \cite[\S 5]{ueno}. Hence, it follows from the triviality of the canonical class of $\cX_\eta$ that the birational map $\cX_\eta \rightarrow \cX^{\mathrm{can}}_\eta$ can only contract $(-2)$-curves, and so that $\cX^{\mathrm{can}}_\eta$ has at worst Du Val (ADE) singularities.
\end{proof}

The following result describes the relative Picard group of an open Kulikov degeneration.

\begin{lemma}
\label{lem_picard}
Let $\pi \colon \mathcal{X} \to \Delta$ be a semistable smoothing of an open Kulikov surface $\mathcal{X}_0$. Then, the restriction of line bundles to the central fiber induces an isomorphism 
\[ \Pic(\cX/\cX^{\mathrm{can}}) \simeq \Pic(\cX_0)\,.\]
\end{lemma}

\begin{proof}
 By Lemma \ref{Lem:Pic}, we have $H^1(\cX_0, \cO_{\cX_0})=H^2(\cX_0, \cO_{\cX_0})=0$, and so, by formal deformation theory of line bundles, we have $\Pic(\cX) \simeq \Pic(\cX_0)$. On the other hand, it follows from the explicit description of $\cX_0^{\mathrm{can}}$ given by Lemma \ref{Lem: contraction2} that $\Pic(\cX_0^{\mathrm{can}})=0$, $H^1(\cX_0^{\mathrm{can}}, \cO_{\cX_0^{\mathrm{can}}})=H^2(\cX_0^{\mathrm{can}},\cO_{\cX_0^{\mathrm{can}}})=0$, and so that $\Pic(\cX_0^{\mathrm{can}})=0$.
\end{proof}

We will mainly work with quasi-projective open Kulikov degenerations. The following result gives several equivalent characterizations of this notion.

\begin{lemma} \label{Lem: projective}
Let $\pi: \cX \rightarrow \Delta$ be an open Kulikov degeneration with central fiber $\cX_0$. The
following are equivalent:
\begin{itemize}
    \item[i)]  The open Kulikov surface $\cX_0$ is quasi-projective.
    \item[ii)] The morphism $\pi: \cX \rightarrow \Delta$ is quasi-projective.
    \item[iii)] The affinization morphism $f: \cX \rightarrow \cX^{\mathrm{can}}$ is projective.
\end{itemize}
\end{lemma}

\begin{proof}
The only non-immediate implication is i) implies ii). By Lemma \ref{lem_picard}, 
so there are no obstructions to formal deformations over $\Delta$ of an ample line bundle on $\cX_0$, and so the result follows.
\end{proof}

Finally, we show that the compact irreducible components of the central fiber of a quasi-projective open Kulikov degeneration are positive log Calabi--Yau surfaces.

\begin{definition}
\label{Def: positive}
A smooth log Calabi--Yau surface $(Y,D)$ is called \emph{positive} if the intersection matrix $(D_i \cdot D_j)$ of the irreducible components $D_i$ of $D$ is not negative semi-definite, that is, if there exists integers $a_i$ such that $(\sum_i a_i D_i)^2>0$.     
\end{definition}

\begin{lemma} \label{Lem: big}
Let $(Y,D=\sum_i D_i)$ be a smooth log Calabi--Yau surface. Then, the following are equivalent:
\begin{itemize}
    \item[i)] $(Y,D)$ is positive.
    \item[ii)] The complement $U=Y\setminus D$ is the minimal resolution of an affine surface with at worst Du Val singularities.
    \item[iii)] There exist positive integers $b_j$  
    such that $(\sum_j b_j D_j)\cdot D_i>0$ for all $i$.  
    \item[iv)] There exists $0<c_i<1$ such that $-(K_Y+\sum_i c_i D_i)$ is big and nef.
    \item[v)] $-K_Y$ is big.
\end{itemize}
\end{lemma}

\begin{proof}
The equivalence between i), ii), iii) and iv) is the equivalence between (1.1), (1.2), (1.3) and (1.4) in \cite[Lemma 6.9]{GHK1}.

If $W=-(K_Y+\sum_i c_i D_i)$ is big and nef, then 
$-K_Y=W+\sum_i c_i D_i$ is the sum of a big divisor with an effective divisor, and so $-K_Y$ big. This shows that iv) implies v).
Finally, we show that v) implies iii). If $-K_Y$ is big, then $\sum_i D_i=-K_Y$ is big, and so the result follows from the proof that (1.1) implies (1.2) in the proof of \cite[Lemma 6.9]{GHK1}.
\end{proof}

\begin{theorem} \label{thm: positive}
The normalization of every compact irreducible component $(X^i,\partial X^i)$ of the central fiber $\mathcal{X}_0$ of a quasi-projective open Kulikov degeneration $\pi: \mathcal{X} \rightarrow \Delta$ is a positive smooth log Calabi--Yau surface.
\end{theorem}

\begin{proof}
By assumption, the open Kulikov degeneration $\pi: \cX\rightarrow \Delta$ is quasi-projective. Hence, by Lemma \ref{Lem: projective}, the contraction $f: \cX \rightarrow \cX^{\mathrm{can}}$ is projective.
Thus, by the relative Kodaira Lemma applied to $f: \mathcal{X} \rightarrow \cX^{\mathrm{can}}$, there exists an effective divisor $F$ on $\mathcal{X}$ such that $-F$ is $f$-ample. Write $F=a_i X^i+G$, where $a_i$ is an integer and $G$ is a divisor not containing $X^i$.
Since $X^i$ is a compact irreducible component of $\mathcal{X}_0$, $X^i$ is contracted by $f$ to a point and so it follows from the Negativity Lemma \cite[Lemma 3.39]{KM} that $a_i >0$. 
Restricting $F$ to $X^i$, we obtain:
\[F|_{X^i}=a_i X^i|_{X^i}+G|_{X^i} \,.\]
Since $K_{\mathcal{X}}=0$, it follows by adjunction that the normal bundle to $X^i$ in $\mathcal{X}$ is $K_{X^i}$, and so $X^i|_{X^i}=K_{X^i}$. Hence, we have 
\begin{equation} \label{eq_K}
-K_{X^i}=\frac{1}{a_i}(-F|_{X^i}+G|_{X^i})\,.
\end{equation}
As $X^i$ is contracted by $f$ and $-F$ is $f$-ample, it follows that $-F|_{X^i}$ is ample on $X^i$. On the other hand, as $G$ does not contain $X^i$, $G|_{X^i}$ is effective. Therefore, $-F|_{X^i}+G|_{X^i}$ is the sum of an ample divisor with an effective divisor, and so is big. Hence, $-K_{X^i}$ is big by the Equation \eqref{eq_K} above and using the fact that $a_i>0$. Finally, we conclude that $(X^i,\partial X^i)$ is positive using that v) implies i) in Lemma \ref{Lem: big}
\end{proof}

\subsection{Generic open Kulikov degenerations} \label{Sec: generic open Kulikov}

In this section, we introduce a notion of generic open Kulikov surface, and prove in Theorem \ref{thm_generic} that every open Kulikov surface admits a generic $d$-semistable locally trivial deformation.

Let $\cX_0$ be an open Kulikov surface, with dual intersection complex $(P,\sP)$. 
Let $\hat{\cX}_0$ be a compactification of $\cX_0$ defined by choosing a log Calabi--Yau compactification of every non-compact irreducible component of $\cX_0$. The boundary $\partial{\hat{\cX}}_0 :=\hat{\cX}_0 \setminus \cX_0$ is a cycle of rational curves.
Let $\widetilde{\Lambda}_{\cX_0} \subset \mathrm{Pic}(\hat{\cX}_0)$ be the subgroup of line bundles $L$ on $\hat{\cX}_0$ such that $\mathrm{deg}(L|_C)=0$ for every irreducible component $C$ of $\partial \hat{\cX}_0$. Different compactifications $\hat{\cX}_0$ differ by corner blow-ups, and so the corresponding lattices $\widetilde{\Lambda}_{\cX_0}$ can be canonically identified, and so the lattice $\widetilde{\Lambda}_{\cX_0}$ is canonically associated to $\cX_0$. Moreover, it follows from Lemma \ref{Lem:Pic} that if $\cX_0'$ is a locally trivial deformation of $\cX_0$, then we have a canonical identification $\widetilde{\Lambda}_{\cX_0'} \simeq \widetilde{\Lambda}_{\cX_0}$. We fix an orientation of $\partial \sP$,  which induces an isomorphism $\Pic^0(\partial \hat{\cX}_0) \simeq \CC^\star$, where $\Pic^0(\partial \hat{\cX}_0)$ is the group of line bundles on $\partial \hat{\cX}_0$, which are of degree zero in restriction to every irreducible component of $\partial \hat{\cX}_0$.
Every locally trivial deformation $\cX_0'$ of $\cX_0$ defines a \emph{period point} $\psi_{\cX_0'} \in \Hom(\widetilde{\Lambda}_{\cX_0},\CC^\star)$, given by
\[ \psi_{\cX_0'}(L)=[L|_{\partial \hat{\cX_0}}] \in \Pic^0(\partial{\hat{\cX}}_0) \simeq \CC^\star\]
for every $L \in \widetilde{\Lambda}_{\cX_0}$.
For every interior vertex $v$ of $\sP$, corresponding to a compact irreducible component $X^v$ of $\cX_0$, define $\xi_v =\sum_{v'\in \sP^{[0]}}(D_{vv'}-D_{v'v})\in \widetilde{\Lambda}_{\cX_0}$, where $D_{vv'}=X^v \cap X^{v'}$. For example, when $\cX_0$ admits a  smoothing $\cX$, we have  $\xi_v =[\cO_{\cX}(-X^v)|_{\cX_0}]$. 

\begin{lemma} \label{lem_d_semistable}
	A locally trivial deformation $\cX_0'$ of $\cX_0$ is d-semistable if and only if $\psi_{\cX_0'}(\xi_v)=1$ for every interior vertex $v$ of $\sP$.
\end{lemma}

\begin{proof}
	The normal crossing surface $\cX_0'$ is $d$-semistable if $T^1_{\cX'_0}=\mathcal{E}xt^1(\Omega^1_{\cX_{0,\psi}}, \cO_{\cX_{0,\psi}}) \simeq \cO_{\cX'_{0,\mathrm{sing}}}$, where $\cX'_{0,\mathrm{sing}}$ is the double locus of $\cX'_0$, see \cite{Frie83, FSIII}. The sheaf $T^1_{\cX'_0}$ is an invertible sheaf on $\cX'_{0,\mathrm{sing}}$ and its class lies in $\operatorname{Pic}^0(\cX'_{0,\mathrm{sing}}) = \Hom(H_1(\Gamma(\cX'_{0,\mathrm{sing}}), \mathbb C^*)) \simeq (\mathbb C^*)^g$, where $\Gamma(\cX'_{0,\mathrm{sing}})$ is the dual graph of $\cX'_{0,\mathrm{sing}}$ and $g$ is its genus. The group $H_1(\Gamma(\cX'_{0,\mathrm{sing}}))$ is isomorphic to $\mathbb Z^g$, with a basis consisting of the loops $\Gamma(\partial X^v)$ corresponding to the boundaries of the compact components $X^v$. As in \cite[\S 4A]{AE}, the restriction $T^1_{\cX'_{0,\psi}}|_{\partial X^v}$ is $\psi_{\cX_0'}(\xi_v)\in \operatorname{Pic}^0(\partial X^v)=\mathbb C^*$. It follows that $T^1_{\cX'_0} = 1$ in $\Pic^0(\cO_{\cX_0', \mathrm{sing}})$ if and only if all $\psi_{\cX_0'}(\xi_v)=1$.
\end{proof}

Denote by $\Xi$ the saturation in $\widetilde{\Lambda}_{\cX_0}$ of the lattice spanned by the classes $\xi_v$ for all interior vertices $v$
of $\sP$. 
Let $\Lambda_{\cX_0}:=\widetilde{\Lambda}_{\cX_0}/\Xi$, so that by Lemma \ref{lem_d_semistable}, the torus $\Hom(\Lambda_{\cX_0},\CC^\star)$ contains period points of $d$-semistable locally trivial deformations of $\cX_0$.

\begin{lemma} \label{lem_universal_family}
There exists a family $\widetilde{\cX}_0 \rightarrow \Hom(\Lambda_{\cX_0},\CC^\star)$ of $d$-semistable locally trivial deformations of $\cX_0$, such that the induced period map 
$\Hom(\Lambda_{\cX_0},\CC^\star) \rightarrow \Hom(\Lambda_{\cX_0},\CC^\star)$
is the identity map.
\end{lemma}

\begin{proof}
Choosing a toric model for each irreducible component of $\cX_0$, we can define as in \cite[\S 7B]{AE} a corresponding glueing complex 
	$\mathcal{G}: \ZZ^{E+Q} \rightarrow (\ZZ^2)^V$, where $E$ is the number of edges of $\sP$ and $V$ is the number of faces of $\sP$, and a universal family of locally trivial deformations of $\cX_0$ over the torus  $\Hom(H^0(\mathcal{G}),\CC^\star)$.
	Arguing as in the proof of \cite[Theorem 7.9]{AE}, we have an isomorphism
	\begin{align*}\widetilde{\Lambda}_{\cX_0} &\xrightarrow{\quad\sim \quad} H^0(\mathcal{G}) \\
	L &\longmapsto ((L \cdot C_e)_e, (L \cdot E_i)_i)\,,\end{align*}
    where $C_e$ are the irreducible components of the double locus of $\hat{\cX}_0$ corresponding to the edges $e$ of $\sP$, and $E_i$ are the exceptional curves of the toric models. Moreover,
    as in the proof of 
     \cite[Proposition 7.11]{AE}, the induced isomorphism of tori
	$\Hom(H^0(\mathcal{G}),\CC^\star) \simeq \Hom(\widetilde{\Lambda}_{\cX_0},\CC^\star)$
	 agrees with the period map of the universal family. We obtain $\widetilde{\cX}_0 \rightarrow T$ by restriction of the universal family to the subtorus $\Hom(\Lambda_{\cX_0},\CC^\star)$ induced by the surjection $\widetilde{\Lambda}_{\cX_0} \rightarrow \Lambda_{\cX_0}=\widetilde{\Lambda}_{\cX_0}/\Xi$. 
\end{proof}

\begin{definition}
A $d$-semistable open Kulikov surface $\cX_0$ is called \emph{generic} if $\psi_{\cX_0}(\gamma) \neq 1$ for every $\gamma \in \Lambda_{\cX_0} \setminus \{0\}$. An open Kulikov degeneration $\cX \rightarrow \Delta$ is called \emph{generic} if its central fiber $\cX_0$ is a generic $d$-semistable open Kulikov surface.
\end{definition}

\begin{theorem} \label{thm_generic}
	Every open Kulikov surface admits a generic $d$-semistable locally trivial deformation.
\end{theorem}

\begin{proof}
This follows from the existence of a family of $d$-semistable locally trivial deformations of $\cX_0$ as in Lemma \ref{lem_universal_family}.
\end{proof}

\begin{lemma} \label{lemma_no_internal}
	Let $\cX_0$ be a generic $d$-semistable open Kulikov surface. Then, no irreducible component of $\cX_0$ contains an internal $(-2)$-curve.
\end{lemma}

\begin{proof}
For every vertex $v$ of $\sP$, denote by $\Lambda_{X^v} \subset \Pic(X^v)$ the subgroup of line bundles $L$ on $X^v$ such that $\mathrm{deg}(L|_C)=0$ for every irreducible component $C$ of $\partial X^v$. The period point $\psi_{X^v} \in \Hom(\Lambda_{X^v},\CC^\star)$ of $X^v$ is defined by $\psi_{X^v}(L)=[L|_{\partial X^v}]$, as in \cite{GHKmod}. 

By Lemma \ref{Lem:Pic}, every line bundle $L \in \Lambda_{X^v}$ uniquely extends to a line bundle $\widetilde{L} \in \widetilde{\Lambda}_{\cX_0}$, which is trivial in restriction to every irreducible component $X^{v'}$ with $v' \neq v$.
Assume that $X^v$ contains an internal $(-2)$-curve $E$. Then, the line bundle $\cO_{X^v}(E)$ uniquely extends to a line bundle $\cO_{\cX_0}(E)$ on $\cX_0$,  which is trivial in restriction to every irreducible component $X^{v'}$ with $v' \neq v$. For every interior vertex $v'$, we have $\xi_{v'} \cdot E=0$, whereas
$\cO_{\cX_0}(E) \cdot E=-2$. So, we have  $\cO_{\cX_0}(E) \notin \Xi$, that is, $\cO_{\cX_0}(E)$ is non-zero in $\Lambda_{\cX_0}=\widetilde{\Lambda}_{\cX_0}/\Xi$.
We have $\psi_{X^v}(\cO_{X^v}(E)) = \psi_{\cX_0}(\cO_{\cX_0}(E))$, 
but $\psi_{X^v}(\cO_{X^v}(E))=1$, so $\psi_{\cX_0}(\cO_{\cX_0}(E))=1$, in contradiction with the assumption that $\cX_0$ generic.
\end{proof}

\begin{lemma} \label{lem_generic}
	Let $\cX \rightarrow\Delta$ be a generic open Kulikov degeneration. Then, the following holds:
    \begin{itemize}
    \item[i)] The generic rigid-analytic fiber $\cX_\eta$ of $\cX$ does not contain any $(-2)$-curve.
    \item[ii)] The affinization induces an isomorphism $\cX_\eta \simeq \cX_\eta^{\mathrm{can}}$.
    \item[iii)] $\cX^{\mathrm{can}} \rightarrow \Delta$ is a formal smoothing of $\cX_0^{\mathrm{can}}$. 
    \end{itemize}
\end{lemma}

\begin{proof}
	Assume by contradiction that $\cX_\eta$ contains a $(-2)$-curve $E_\eta$. Let $E$ denote the closure of $E_\eta$ in $\cX$. Since $\cX$ is regular, $E$ is a Cartier divisor on $\cX$. Moreover, $E$ does not contain any irreducible component of $\cX_0$, and so is flat over $\Delta$. Thus,  $E$ is a relative effective Cartier divisor as in  \cite[\href{https://stacks.math.columbia.edu/tag/056P}{Tag 056P}]{stacks-project}. Hence, by \cite[\href{https://stacks.math.columbia.edu/tag/056Q}{Tag 056Q}]{stacks-project}, the pullback $E_0$ of $E$ to $\cX_0$ is a Cartier divisor in $\cX_0$, corresponding to the line bundle $\cO_{\cX_0}(E_0)=\cO_\cX(E)|_{\cX_0}$. By construction, we have $\cO_{\cX_0}(E_0)
    \in \Lambda_{\cX_0}$, $\psi_{\cX_0}(\cO_{\cX_0}(E_0))=1$ and
	 $\cO_{\cX_0}(E_0) \cdot E_0=-2$. Moreover, for every interior vertex $v$, we have $\cO_{\cX}(X^v) \cdot E_\eta=0$, and so $\cO_{\cX_0}(X^v) \cdot E_0 =0$. It follows that  $\cO_{\cX_0}(E_0) \notin \Xi$, that is, $\cO_{\cX_0}(E_0)$ is non-zero in $\Lambda_{\cX_0}=\widetilde{\Lambda}_{\cX_0}/\Xi$, in contradiction with the assumption that $\cX_0$ is generic. This proves Lemma \ref{lem_generic} i). Items ii)-iii) then follow from i) by Theorem \ref{thm: ss smoothings}.
\end{proof}

\begin{remark}
Let $\cX \rightarrow \Delta$ be a generic open Kulikov degeneration. By Lemma \ref{lem_generic}, $\cX^{\mathrm{can}}$ is an affine formal smoothing, and so it follows from Elkik's theorem
\cite[Theorem 7, p582]{elkik1973}
-- see also \cite[Proposition 3.3.1]{temkin2008desingularization},
that $\cX^{\mathrm{can}}$ is algebraizable over $\Spec\, \CC[\![t]\!]$.
Moreover, $\cX \rightarrow \cX^{\mathrm{can}}$ is a formal blow-up, and so is algebraizable in the category of algebraic spaces over $\Spec\, \CC[\![t]\!]$ by
\cite[Theorem 3.2]{artin1970}. Since $H^2(\cX_0, \mathcal{O}_{\cX_0})=0$, and $\cX \rightarrow \cX^{\mathrm{can}}$ is projective, it follows that $\cX \rightarrow \cX^{\mathrm{can}}$ is in fact algebraizable in the category of schemes over $\Spec\, \CC[\![t]\!]$.
\end{remark}

\subsection{Mori fans of quasi-projective open Kulikov degenerations}
\label{Sec: Mori fans}

In this section, we show that the total space of a quasi-projective open Kulikov degeneration is a (relative) Mori dream space, as in \cite{KH} for the absolute case and \cite{AW,OhtaR} for the relative case. Consequently, we attach to any quasi-projective open Kulikov degeneration the corresponding Mori fan.

\begin{definition}
Let $f: V \to W$ be a projective morphism between formal schemes $V,W$. Then, $V$ is called a \emph{Mori dream space over W}, if the following conditions are satisfied:
\begin{itemize}
    \item[i)] $V$ is $\QQ$-factorial, 
    and $\Pic(V/W)$ is a finitely generated abelian group.
    \item[ii)] The relative nef cone $\Nef(V/W)$ is generated by finitely many semi-ample divisors.
    \item[iii)] There is a finite collection of small $\QQ$-factorial modifications $f_i: V \dashrightarrow V_i $ such that each $V_i$ satisfies ii) and the Movable cone of relative moving divisors $\Mov(V/W)$ is a union of the cones $f_i^*(\Nef(V_i/W))$.
\end{itemize}
\end{definition}

\begin{proposition}
\label{Prop: Mori dreams}
The total space of a quasi-projective open Kulikov degeneration $\mathcal{X} $ is a Mori dream space over $\cX^{\mathrm{can}}$.  
\end{proposition}

\begin{proof}
  This follows from Theorem \ref{thm: ss smoothings}, using \cite[Corollary 1.3.2]{BCHM10} adapted to the formal setup -- see \cite{MMPformal}.
\end{proof}

In what follows, we describe the Mori fan associated with an open Kulikov degeneration. For any Mori dream space 
$V$ over $W$, where $V$ and $W$ are normal algebraic varieties with a projective morphism $f: V \to W$, there is an associated fan in $\mathrm{Pic}(V/W)_{\RR}$, known as the Mori fan. 
Its support is given by the cone of $f$-effective divisors, $\mathrm{Eff}(V/W)$ \cite{KH, OhtaR}. 
Using formal Mori theory, as developed in  \cite{MMPformal}, one can similarly construct a Mori fan for a Mori dream space $V$ over $W$, where $V$ and $W$ are formal schemes and $f: V\rightarrow W$ is a projective morphism.

The maximal dimensional cones of the Mori fan are in one-to-one correspondence with the birational contractions $g: V \dashrightarrow  V'$ over $W$ where $V'$ is $\mathbb{Q}$-factorial. For such a birational contraction, the corresponding maximal cone of the Mori fan is given by
$g^*\mathrm{Nef}(V'/W) + \mathrm{Ex}(g)$,
where $\mathrm{Ex}(g)$ is the cone in $\mathrm{Pic}(V/W)_{\RR}$ spanned by all $g$-exceptional divisors. 
For a birational projective morphism $f: V \to W$, we have $\mathrm{Eff}(V/W) = \mathrm{Pic}(V/W)_{\RR}$, and so the Mori fan is a complete fan in $\mathrm{Pic}(V/W)_{\RR}$ in this case.
Given an open Kulikov degeneration, with affinization $f \colon \mathcal{X} \to \cX^{\mathrm{can}}$,  the space $\mathcal{X}$ is a Mori dream space over $\cX^{\mathrm{can}}$ by Proposition \ref{Prop: Mori dreams}.
Consequently, we have an associated complete Mori fan in $\mathrm{Pic}(\cX/\cX^{\mathrm{can}})$, denoted by 
$\mathrm{MF}(\mathcal{X}/\cX^{\mathrm{can}})$ in the following sections. Additionally, we note that we have 
$\mathrm{Pic}(\cX/\cX^{\mathrm{can}}) \simeq \Pic(\cX_0)$ by Lemma \ref{lem_picard}.

\section{Semistable mirrors of polarized log Calabi--Yau surfaces}
\label{Sec: semistable mirror}

In this section, we construct in 
\S \ref{sec: open Kulikov surfaces from log cy} 
a quasi-projective open Kulikov surface $\mathcal{X}_{\mathscr{P},0}$, associated to a generic polarized log Calabi--Yau $(Y,D,L)$ and a choice of Symington polytope $P$ endowed with a good polyhedral decomposition $\mathscr{P}$. We describe $\mathcal{X}_{\mathscr{P},0}$ as a union of suitable ``deformations'' of toric varieties determined by the polyhedral decomposition $\mathscr{P}$. The necessary results in deformation theory of log Calabi--Yau surfaces are first established in \S \ref{sec: deformation theory}.
Finally, we define in \S\ref{Sec: semistable mirror of log CY} the semistable mirror to $(Y,D,L)$ as an 
open Kulikov degeneration $\cX_\sP \rightarrow \Delta$ obtained as the total space of a semistable smoothing of the open Kulikov surface $\mathcal{X}_{\mathscr{P},0}$.

\subsection{Deformation theory of log Calabi--Yau surfaces}
\label{sec: deformation theory}

The results in deformation theory of surfaces established below will be used in the next section \S\ref{sec: open Kulikov surfaces from log cy} in the construction of the open Kulikov surface 
$\cX_{\sP,0}$.

\begin{lemma} \label{Lem: smoothing 1}
Let $X$ be a normal projective surface with isolated singularities and let $D$ be an effective divisor on $X$. If we have  $H^2(X, (T_X(-D))^{\vee \vee})=0$, then there are no local-to-global obstructions to deformations of the pair $(X,D)$.
\end{lemma}

\begin{proof}
The result is well-known if $D$ is empty, see for example \cite[Theorem 4.10]{Rim}, \cite[Lemma 1]{Manetti}, or \cite[Proposition 6.4]{Wahl}. On the other hand, if $D$ is non-empty but $X$ is smooth, a proof can be found in \cite[Proposition 3.5]{friedman2015geometry}. Our proof below of the general case is based on a combination of the techniques used for these two known special cases.

Let $U$ be the non-empty open subset of $X$ where both $X$ and $D$ are smooth, and denote by $j: U \hookrightarrow X$ the inclusion. By assumption, the complement $X \setminus U$ consists of isolated points in the surface $X$, and so is of codimension two.
Consider the complex $A:=(j_{\star} T_X|_U  \rightarrow j_\star N_{D|U})$, where $T_X|_U$ is the restriction to $U$ of the tangent sheaf $T_X$, placed in degree $0$, and $N_{D|U}=\mathcal{O}_{D \cap U}(D\cap U)$ is the normal sheaf to $D\cap U$ in $U$, placed in degree $1$. Using 
as in the proof of \cite[Theorem 4.10]{Rim} a \v{C}ech cover of $X$ by affine subsets centered on the isolated singularities and with double and triple intersections contained in $U$, we deduce that if the hypercohomology group $\mathbb{H}^2(X,A)$
vanishes, then there are no local-to-global obstructions to deformations of $(X,D)$. When $X$ is smooth, details of this argument can be found in \cite[Proposition 8]{SV}.

Therefore, it remains to show that $H^2(X, (T_X(-D))^{\vee \vee})=0$ implies that $\mathbb{H}^2(X,A)=0$. 
Consider the short exact sequence \[0 \rightarrow \mathcal{O}_U(-D \cap U) \rightarrow \mathcal{O}_U \rightarrow \mathcal{O}_{D\cap U} \rightarrow 0\]
defining $D\cap U$ in $U$. 
Since $T_{X}|_U$ is locally free, 
we obtain by tensoring with $T_{X}|_U$
and then applying $j_\star$, an exact sequence of the form 
\[ 0 \rightarrow j_\star T_{X}|_{U}(-D\cap U) \rightarrow j_\star T_{X}|_U \rightarrow j_\star T_X|_{D\cap U } \rightarrow P \rightarrow 0 \,,\]
where $P$ is a sheaf supported on $X \setminus U$. 
Since $T_{X}|_{U}(-D\cap U)$ is locally free on $U$, we have 
$j_\star T_{X}|_{U}(-D\cap U)=(T_X(-D))^{\vee \vee}$. 
We split the above long exact sequence into two short exact sequences:
\begin{equation}
0 \rightarrow (T_X(-D))^{\vee \vee} 
\rightarrow
j_\star T_{X}|_{U}
\rightarrow  L \rightarrow 0
\end{equation}
and 
\begin{equation}
0 \rightarrow L
\rightarrow j_\star T_{X}|_{D \cap U}  \rightarrow P \rightarrow 0\,.
\end{equation}
Consider the complex 
$B:=(L \longrightarrow j_{\star} N_{D|U})$
obtained by restricting to $L$ the natural morphism $(j_{\star} T_X|_{D \cap U} \longrightarrow j_{\star} N_{D|U} )$.
Restriction to $D$ induces a map of complexes $A \rightarrow B$.
By construction, we have an exact sequence of complexes 
\[ 0 \rightarrow (T_X(-D))^{\vee \vee} \rightarrow A \rightarrow B \rightarrow 0 \,,\]
and so a corresponding long exact sequence in hypercohomology containing 
\[ \cdots \rightarrow H^2(X,(T_X(-D))^{\vee \vee}) \rightarrow \mathbb{H}^2(X,A) \rightarrow \mathbb{H}^2(X,B) \rightarrow \cdots\,.\]
Hence, it suffices to show that $\mathbb{H}^2(X,B)=0$ in order to show the desired result that $H^2(X,(T_X(-D))^{\vee \vee})=0$ implies $\mathbb{H}^2(X,A)=0$.

The kernel $\mathcal{H}^0(B)$ of $B$ is supported in dimension one, so we have 
$H^2(X,\mathcal{H}^0(B))=0$, and the cokernel $\mathcal{H}^1(B)$ of $B$ is supported in $X \setminus U$, that is in dimension zero, and so we have also 
$H^1(X,\mathcal{H}^1(B))=0$. 
It follows from the hypercohomology spectral sequence that $\mathbb{H}^2(X,B)=0$, and this concludes the proof.
\end{proof}

\begin{lemma} \label{Lem: smoothing 2}
Let $X$ be a normal projective surface with isolated quotient singularities and let $D \in |-K_X|$ be an effective reduced anticanonical divisor. Assume that we have $H^1(X,\mathcal{O}_X)=0$. Then, there are no local-to-global obstructions to deformations of $(X,D)$.
\end{lemma}

\begin{proof}
The following argument is inspired by the proof of 
\cite[Theorem 21]{Manetti}.
By Lemma \ref{Lem: smoothing 1}, it suffices to show that $H^2(X, (T_X(K_X))^{\vee \vee})=0$. 
A normal surface being Cohen-Macaulay, the reflexive canonical sheaf $\mathcal{O}_X(K_X)$ is the dualizing sheaf of $X$, and so by Serre duality we have 
\[ H^2(X, (T_X(K_X))^{\vee \vee})=
\mathrm{Hom}((T_X(K_X))^{\vee \vee}, \mathcal{O}_X(K_X))^\vee \,.\]
Let $j: U \hookrightarrow X$ be the open immersion of the smooth locus of $X$. 
Since $X$ is normal and $X\setminus U$ is of codimension two, the functors $j_\star$ and $j^{\star}$ are inverse equivalences between the category of coherent reflexive sheaves on $X$ and the coherent reflexive sheaves on $U$ \cite[\href{https://stacks.math.columbia.edu/tag/0EBJ}{Lemma 0EBJ}]{stacks-project}. As both $(T_X(K_X))^{\vee \vee}$ and $\mathcal{O}_X(K_X)$ are reflexive, it follows that 
\[ \mathrm{Hom}((T_X(K_X))^{\vee \vee}, \mathcal{O}_X(K_X)) = 
\mathrm{Hom} (j^{*}(T_X(K_X))^{\vee \vee}, j^{*}\mathcal{O}_X(K_X))\,.\]
As $U$ is the smooth locus of $X$, the restrictions to $U$ of $T_X$ and $\mathcal{O}_X(K_X)$ are locally free, and so 
\[\mathrm{Hom} (j^{*}(T_X(K_X))^{\vee \vee}, j^{*}\mathcal{O}_X(K_X)) 
= \mathrm{Hom} (\mathcal{O}_U, j^{*} \Omega_X)=H^0(U, j^{*}\Omega_X)\,.\]
Using $j_{\star}$ to come back on $X$, we have 
\[H^0(U, j^{*}\Omega_X)=H^0(X, j_{*}
j^{\star} \Omega_X)=H^0(X, \Omega_X^{\vee \vee}) \,.\]
Finally, as $X$ has quotient singularities, it follows from Hodge theory for orbifolds \cite[\S 1]{Steenbrink} that $\dim H^0(X,\Omega_X^{\vee\vee}) =\dim H^1(X,\mathcal{O}_X)$. Hence, we have
\[\dim H^2(X, (T_X(K_X))^{\vee \vee}) =\dim H^1(X,\mathcal{O}_X)\,,\]
and so $H^2(X, (T_X(K_X))^{\vee \vee}) =0$ if $H^1(X,\mathcal{O}_X)=0$.
\end{proof}

\subsection{Open Kulikov surfaces from polarized log Calabi--Yau surfaces}
\label{sec: open Kulikov surfaces from log cy}

Let $(Y,D,L)$ be a generic polarized log Calabi--Yau surface and $P$ a Symington polytope of $(Y,D,L)$ defined by a good toric model $(\overline{Y}, \overline{D}, \overline{L})$
with toric momentum polytope $\overline{P}$, as in \S\ref{Sec: polarized log cy}. In this section, we describe the construction, for every good polyhedral decomposition $\overline{\sP}$ of $\overline{P}$ as in \S\ref{subsec: good polyhedral}, of a quasi-projective open Kulikov surface 
$\cX_{\sP,0}$, which is unique up to locally trivial deformations.

Throughout this section we denote by $M \cong \ZZ^2$ the character lattice for $\overline{Y}$, so that $\overline{P} \subseteq M_{\RR}$, where $M_{\RR} = M \otimes_{\ZZ} \RR$ denotes the real vector space associated to $M$. 
We first consider the Mumford toric degeneration 
$\pi_{\overline{\sP}}: \cX_{\overline{\sP}} \rightarrow \mathbb{A}^1$
associated to the toric polytope $\overline{P}$ endowed with the polyhedral decomposition $\overline{\mathscr{P}}$. As in \cite[\S3]{NS},
$\mathcal{X}_{\overline{\mathscr{P}}}$ is the 3-dimensional toric variety with fan $C(\overline{\sP})$
in $M_{\RR} \times \RR_{\geq 0}$ consisting of the cones
\[ C(\sigma) = \RR_{\geq 0} \cdot (\sigma \times \{ 1 \} ) \subset  M_{\RR} \times \RR_{\geq 0} \]
over the cells $\sigma \in \overline{\mathscr{P}}$.
Moreover, the toric morphism
$\pi_{\overline{\mathscr{P}}} \colon \mathcal{X}_{\overline{\mathscr{P}}} \to \mathbb{A}^1$
is induced by the projection $M_\RR \times \RR_{\geq 0} 
\rightarrow \RR_{\geq 0}$ onto the second factor.
By Definition \ref{Def: good polyhedral decomposition} of a good polyhedral decomposition, $\overline{\mathscr{P}}$ is a regular decomposition, and so $\mathcal{X}_{\overline{\mathscr{P}}}$ is a quasi-projective toric variety.

Since all cells $\sigma \in \overline{\mathscr{P}}$ are bounded, the so called \emph{asymptotic cone} in \cite{NS} consists of a single point, the origin in $M_{\RR} \times \RR$, and hence the fibers
$\pi_{\overline{\sP}}^{-1}(t)$ for $t \neq 0$
are isomorphic to the algebraic torus $(\CC^*)^2$. 
The central fiber $\mathcal{X}_{\overline{\mathscr{P}},0}\coloneqq \pi_{\overline{\sP}}^{-1}(0)$ is a union of irreducible components $X_{\overline{\mathscr{P}}}^v$, indexed by the vertices $v \in \overline{\mathscr{P}}^{[0]}$, which are toric surfaces with fan the star of $v$ in $\overline{\sP}$:
\[ \mathcal{X}_{\overline{\mathscr{P}},0} = \bigcup_{v \in \overline{\mathscr{P}}^{[0]}} X_{\overline{\mathscr{P}}}^v \,.\]
For every vertex $v \in \overline{\mathscr{P}}^{[0]}$, we denote by $\partial X_{\overline{\mathscr{P}}}^v$
the union of intersections of $X_{\overline{\mathscr{P}}}^v$
with the other irreducible components of $\mathcal{X}_{\overline{\mathscr{P}},0}$, which is also the union of toric divisors of $X_{\overline{\mathscr{P}}}^v$.

In this section, we construct by deformation of 
$\cX_{\overline{\mathscr{P}},0}$ a normal crossing surface 
$\cX_{\mathscr{P},0}$, whose irreducible components are log Calabi--Yau surfaces $X_{\mathscr{P}}^v$ indexed by the vertices $v \in \mathscr{P}^{[0]}$.  The construction proceeds in four steps.

STEP I: Let $pr: \overline{\mathscr{P}}^{[0]} \rightarrow \mathscr{P}^{[0]}$ be the natural projection coming from the description of the Symington of $P$ as a quotient of $\overline{P} \setminus \cup_{ij} \mathrm{Int}(\Delta_{ij})$.
As a first step, we construct for every vertex $v \in \mathscr{P}^{[0]}$, a toric surface $\widetilde{X}^v_{\mathscr{P}}$ by deformation of the union $\bigcup_{v' \in pr^{-1}(v)} X^{v'}_{\overline{\mathscr{P}}}$
of irreducible components of $\mathcal{X}_{\overline{\mathscr{P}},0}$.
When $pr^{-1}(v)$ is of cardinality one, with $pr^{-1}(v)=\{v'\}$, we simply set $\widetilde{X}^v_{\mathscr{P}}:
=X^{v'}_{\overline{\mathscr{P}}}$. 

When $pr^{-1}(v)$ is of cardinality strictly bigger than one, then $pr^{-1}(v)=\{v'_1, \dots, v'_r\}$, 
where for every $1\leq k \leq r-1$, $v'_k$ and $v'_{k+1}$ are opposite integral points on the interior edges of a triangle $\Delta_{i_k j_k}$. 
For each $1 \leq k \leq r$, let $\Sigma_k$ be the fan of $X^{v_k'}_{\overline{\sP}}$, that is, the star of $\overline{\sP}$ at $v_k'$. 
Let $f_k$ be the primitive integral direction vector pointing towards $p_{i_kj_k}$ of the interior edge of $\Delta_{i_kj_k}$ containing $v_k'$. Denote by $e_k$ the primitive direction vector of the edge of $\overline{\sP}$ connecting $v_k'$ to $v_{k+1}'$, pointing towards $v_{k+1}'$. 
Each fan $\Sigma_k$ contains rays of directions $f_k$, $e_k$ if $k<r$, and $-e_{k-1}$ if $k>1$.

Now, let $T_{k}$ be the unique element of $SL(2,\ZZ)$ such that $T_k(f_{k+1})=f_k$ and $T_k(e_k)=e_k$. 
Define $\widetilde{\Sigma}_k$ as the fan obtained by applying $T_1 \cdots T_{k-1}$ to $\Sigma_k$. Then, the ray of direction $f_k$ in $\Sigma_k$ becomes a ray of direction $f_1$ in $\widetilde{\Sigma}_k$. 
Since all these rays are parallel, there exists a polyhedral decomposition $\mathfrak{F}$ of $\RR^2$, with vertices $w_1, \dots, w_r$, such that the star of $\mathfrak{F}$ at each vertex 
$w_k$ is precisely $\widetilde{\Sigma}_k$, and the vertices $w_k$ and $w_{k+1}$ are connected by an edge of direction 
$T_1 \cdots T_{k-1}(e_k)$. An explicit example of this construction is given in Example \ref{example_step_1} below.

The polyhedral decomposition $\mathfrak{F}$ defines a toric degeneration over $\A^1$, as in \cite[\S 3]{NS}. 
Since each fan $\widetilde{\Sigma}_k$ is obtained from $\Sigma_k$ by applying an element of $SL(2,\ZZ)$, it defines the same toric variety $X^{v_k'}_{\overline{\sP}}$. Thus, the central fiber of the toric degeneration is $\bigcup_{k=1}^r X^{v_k'}_{\overline{\mathscr{P}}}$.
We define the toric surface $\widetilde{X}^v_\sP$ as the general fiber of the toric degeneration, whose fan is the asymptotic fan of the polyhedral decomposition $\mathfrak{F}$.
The toric degeneration effectively smooths out the double locus of $\bigcup_{k=1}^r X^{v_k'}_{\overline{\mathscr{P}}}$, see
Figure \ref{figure6}.

\begin{figure}[h]
\center{\includegraphics{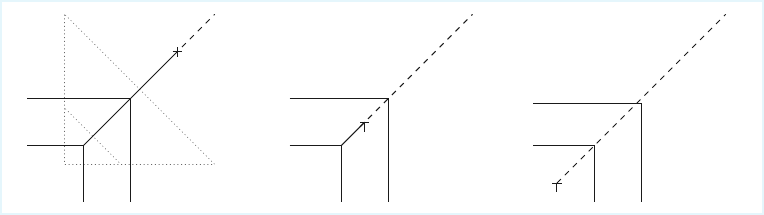}}
\caption{Deformation of the toric intersection complex 
of $\cX_{\overline{\sP}, 0}$ into the non-toric intersection complex of $\cX_{\sP, 0}$. The black line turning into the dashed-line represents the part of the double locus of $\cX_{\overline{\sP},0}$ becoming smooth in $\cX_{\sP,0}$.}
\label{figure6}
\end{figure}

\begin{example}\label{example_step_1}
Consider Figure \ref{figure7}, where the three vertices $v'_1$, $v'_2$, $v'_3$ of $\overline{\sP}$ are identified to a single vertex $v$ of $\sP$. 
In this example, we have $f_1=(0,1)$, $f_2=(-1,1)$, $f_3=(-1,0)$, and $e_1=(1,0)$, $e_2=(0,1)$. The matrices $T_1, T_2 \in SL(2,\ZZ)$ such that $T_1(f_2)=f_1$, $T_1(e_1)=e_1$ and $T_2(f_3)=f_2$, $T_2(e_2)=e_2$ are given by 
\[ T_1 = \begin{pmatrix}
    1 & 1 \\
    0 & 1
\end{pmatrix}\,\,\,\, \text{and}\,\,\,\,T_2=\begin{pmatrix}
    1 & 0\\
    -1 & 1
\end{pmatrix}\,.\]
The resulting polyhedral decomposition $\mathfrak{F}$ of $\RR^2$ is represented on the left of Figure \ref{figure8}. The local fan around $w_1$ coincides with the local fan around $v_1'$, whereas the local fan around $w_2$ is obtained from the local fan around $v_2'$ by applying $T_1$, and the local fan around $w_3$ is obtained from the local fan around $v_3'$ by applying $T_1 T_2$. The corresponding asymptotic fan, which gives the fan of the toric surface $\widetilde{X}^v_\sP$, is represented on the right of  Figure \ref{figure8}. 
\end{example}

\begin{figure}[h]
\center{\includegraphics{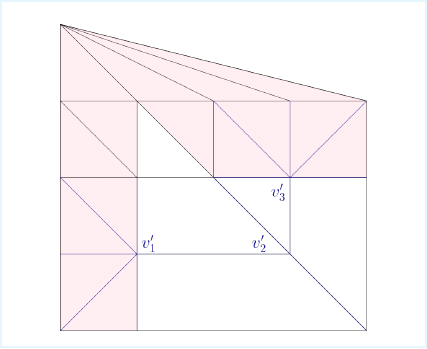}}
\caption{The polygon $\overline{P}$ and the polyhedral decomposition $\overline{\sP}$ in Example \ref{example_step_1}}.
\label{figure7}
\end{figure}

\begin{figure}[h]
\center{\includegraphics{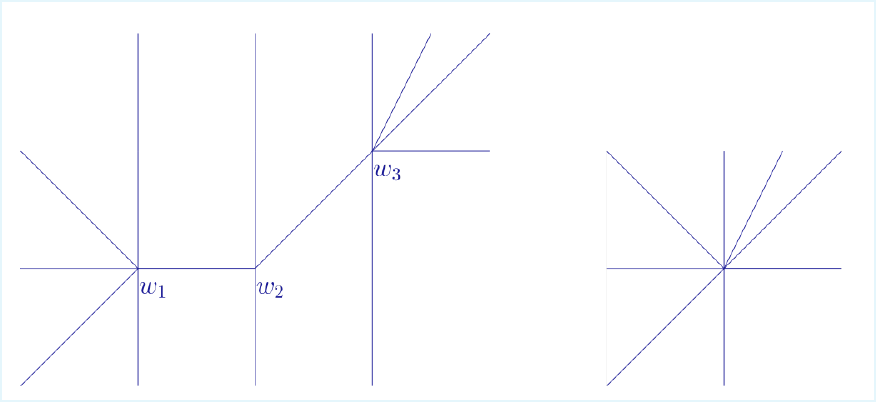}}
\caption{On the left, the polyhedral decomposition $\mathfrak{F}$ in Example \ref{example_step_1}. On the right, the corresponding asymptotic fan of $\widetilde{X}^v_\sP$.}
\label{figure8}
\end{figure}

STEP II: As a second step, for each vertex $v \in \mathscr{P}^{[0]}$, we describe a deformation of the toric pair $(\widetilde{X}_{\sP}^{v}, \partial \widetilde{X}_\sP^{v})$ into a log Calabi--Yau pair $(X_\sP^v, \partial X_\sP^v)$. 
Recall from Construction \ref{Cons: Symington polytope} that the Symington polytope $P$ is obtained from $\overline{P}$ by cutting out triangles $\Delta_{ij}$, each having one edge on the boundary of $\overline{P}$ and an interior vertex $p_{ij}$.
For every $i,j$ such that $v=pr(p_{ij})$,
there is a face $F_{ij}$ of $\overline{\mathscr{P}}$
containing $v$ and contained in the triangle $\Delta_{ij}$. The face $F_{ij}$
defines a 2-dimensional cone in the star of $\overline{\mathscr{P}}$ at $v$, that is, in the fan of $X_{\overline{\mathscr{P}}}^{v}$, and so corresponds to a toric fixed point $x_{ij}$ of $X_{\overline{\mathscr{P}}}^{v}$.
This point is not smoothed by the deformation in Step I, and so also defines a toric fixed point in the toric surface $\widetilde{X}_{\mathscr{P}}^{v}$ that we continue to denote by $x_{ij}$.
We denote by $(V_{ij}, \partial V_{ij})$
the analytic germ of the surface pair $(\widetilde{X}_{\mathscr{P}}^{v}, \partial \widetilde{X}_{\mathscr{P}}^{v})$ around $x_{ij}$.

Since the triangles $\Delta_{ij}$ have equal lattice height and lattice base,
it follows from \cite[\S 7.3]{Altmann} that $(V_{ij}, \partial V_{ij})$
is the germ of the origin in an affine toric variety of the form 
\[\mathbb{A}^2_{u,v}/\frac{1}{n^2}(1,an-1)\,,\] with 
$n \geq 1$ and $\mathrm{gcd}(a,n)=1$.
If $n=1$, then $(V_{ij}, \partial V_{ij})$ is the germ of the origin in $(\mathbb{A}^2, \{uv=0\})$.
If $n>1$, then $V_{ij}$ is an elementary T-singularity as in \cite[Proposition 3.10]{KSB}, also known as a Wahl singularity.
Using the change of variables $x=u^n$, $y=v^n$, $z=uv$, this singularity can be equivalently described as $\bigslant{(xy=z^n)}{\frac{1}{n}(1,-1,a)}$, and admits $\mathbb{Q}$-Gorenstein smoothings by \cite[Definition 3.7]{KSB}. 
More precisely, its universal $\mathbb{Q}$-Gorenstein deformation space is irreducible, one-dimensional, and given by 
\[ \bigslant{(xy=z^n+t)}{\frac{1}{n}(1,-1,a,0)}\,.\]
By taking $\{z=0\}$ as deformation of $\partial V_{ij}$, we view this smoothing of $V_{ij}$ as a one-parameter deformation of the pair $(V_{ij}, \partial{V}_{ij})$.
If $n=1$, we consider the one-parameter deformation of $(\mathbb{A}^2_{u,v}, \{uv=0\})$ given by $(\mathbb{A}^2_{u,v}, \{uv=t\})$, that is by just smoothing the normal crossing toric divisor.
Thus, for every $n$, we obtain a natural one-parameter $\QQ$-Gorenstein deformation of $(V_{ij},\partial V_{ij})$, whose general fiber is the germ of a smooth point with a smooth connected divisor.

Using the deformation theory results proved in \S\ref{sec: deformation theory}, we will glue these local $\mathbb{Q}$-Gorenstein deformations of the germs $(V_{ij},\partial V_{ij})$ into a global deformation of $(\widetilde{X}_{\mathscr{P}}^v, \partial \widetilde{X}_{\mathscr{P}}^v)$. 
Let $Q_v$ be the number triangles $\Delta_{ij}$ with $pr(p_{ij})=v$. 
Since a projective toric variety $X$ always satisfies $H^1(X,\mathcal{O}_X)=0$, we obtain by Lemma \ref{Lem: smoothing 2} that there exists a
$Q_v$-dimensional
$\mathbb{Q}$-Gorenstein deformation of
$(\widetilde{X}_{\mathscr{P}}^{v}, \partial \widetilde{X}_{\mathscr{P}}^{v})$
induced by the one-dimensional $\mathbb{Q}$-Gorenstein deformations of the $Q_v$ germs $(V_{ij}, \partial V_{ij})$. Moreover, the base of this $Q_v$-dimensional deformation is irreducible, since it is the case for each of the one-dimensional deformations of the germs $(V_{ij}, \partial V_{ij})$.
We denote by $(X_{\mathscr{P}}^{v}, \partial X_{\mathscr{P}}^{v})$ a general fiber of this deformation. It is a smooth, possibly non-compact, log Calabi--Yau surface of charge $Q_v$, well-defined up to smooth deformations since the base of the deformation is irreducible. If $v$ is an interior vertex of $P$, then $(X_{\mathscr{P}}^{v}, \partial X_{\mathscr{P}}^{v})$ is a smooth projective log Calabi--Yau surface, that is, a $2$-surface as in Definition \ref{def_2_surface}. 

If $v$ is a boundary vertex of $P$, then
$(X_{\mathscr{P}}^{v}, \partial X_{\mathscr{P}}^{v})$ is a smooth non-compact
log Calabi--Yau surface, described explicitly in Lemmas \ref{Lem: non-compact1}-\ref{Lem: non-compact2} below.
The boundary $\partial P$ of $P$ is naturally identified with the intersection complex of the divisor $D$ polarized by the restriction $L|_D$ of $L$ to $D$. As $D$ is either a cycle of rational curves or a nodal curve of arithmetic genus one, it follows that $\partial P$ is topologically a circle, and is of integral length $\mathrm{deg}L|_D$. 
Moreover, there are integral points $v_x \in (\partial P)_\ZZ$ in one-to-one correspondence with the 0-dimensional strata $x$ of $D$. If $x$ and $y$ are two 0-dimensional strata on the same irreducible component $D_i$ of $D$, then there is a corresponding interval $I_{D_i} \subset \partial P \simeq S^1$ with endpoints $v_x$ and $v_y$, of integral length $\mathrm{deg}L|_{D_i}$, and containing no other vertex of the form $v_z$. 
For every $v \in (\partial P)_\ZZ$, we denote by
\[ (\overline{X}^v_{\mathscr{P}}, \partial \overline{X}^v_{\mathscr{P}}) := (\Spec\, H^0( X^v_{\mathscr{P}},
\mathcal{O}_{ X^v_{\mathscr{P}}}),
\Spec\, H^0( \partial X^v_{\mathscr{P}},
\mathcal{O}_{\partial X^v_{\mathscr{P}}}))
\]
the affinization of $(X^v_{\mathscr{P}}, \partial X^v_{\mathscr{P}})$.
By Proposition \ref{Prop: germs}, 
the germ of $(Y,D)$ at every 0-dimensional stratum $x$ of $D$ is a cyclic quotient surface singularity, which can be described torically by a fan consisting of a single rational polyhedral cone $\sigma_x \subset \RR^2$.
The affine toric surface $\Spec\, \CC[\sigma_x \cap \ZZ^2]$ is the dual quotient singularity of the quotient singularity $\Spec\, \CC[\sigma_x^\vee \cap \ZZ^2]$ describing the germ of $(Y,D)$ at $x$.

\begin{lemma} \label{Lem: non-compact1}
Let $(Y,D,L)$ be a generic polarized log Calabi--Yau surface, with a polyhedral decomposition $\sP$ of a Symington polytope $P$ defined as above.
For every $0$-dimensional stratum $x$ of $D$, $\overline{X}^{v_x}_{\mathscr{P}}$ is an affine surface obtained by a flat deformation of the affine toric surface 
$\Spec\, \CC[\sigma_x \cap \ZZ^2]$, and $\partial \overline{X}^{v_x}_{\mathscr{P}}$ is a flat locally trivial deformation of the union of the two toric divisors of $\Spec\, \CC[\sigma_x \cap \ZZ^2]$. Moreover, the affinization map
$X^{v_x}_{\mathscr{P}} \rightarrow \overline{X}^{v_x}_{\mathscr{P}}$ is a surjective projective birational morphism, and every exceptional curve is contracted to the $0$-dimensional stratum of $\partial \overline{X}^{v_x}_{\mathscr{P}}$. In particular,  $(X^{v_x}_{\mathscr{P}}, \partial X^{v_x}_{\mathscr{P}})$ is a $0$-surface as in Definition \ref{def_0_surface}.
\end{lemma}

\begin{proof}
    Assume first that there is no interior vertex $p_{ij}$ of a triangle $\Delta_{ij}$ as in \S \ref{Sec: Symington polytopes} such that $pr(p_{ij})=v_x$. Then, by Step I, $X^{v_x}_{\mathscr{P}}$ is the toric surface whose fan is given by the star of $\mathscr{P}$ at $v_x$. Moreover, this fan is naturally a refinement of the strictly convex cone $\sigma_x$ giving a local description of $P$ near $v_x$. It follows that in this case the affinization $\cX^{\mathrm{can}}_{\mathscr{P},0}$ is exactly the affine toric surface $\Spec\, \CC[\sigma_x \cap \ZZ^2]$, and the affinization map is a toric morphism induced by a refinement of fans, and so is surjective, projective, and birational. 

    In general, if there are points $p_{ij}$ such that $pr(p_{ij})=v_x$, then, by Step II of the construction of $\mathcal{X}_{\mathscr{P},0}$ in \S \ref{sec: open Kulikov surfaces from log cy}, $X^{v_x}_{\mathscr{P}}$ is a flat deformation, obtained by partial smoothing of the toric boundary, of the toric surface $X^{v_x}_{\overline{\mathscr{P}}}$ whose fan is given by the star of $\mathscr{P}$ at $v_x$, and whose affinization is given by $\overline{X}^{v_x}_{\overline{\mathscr{P}}}=\Spec\, \CC[\sigma_x \cap \ZZ^2]$. 
    Since $X^{v_x}_{\overline{\mathscr{P}}}$ is a toric surface, we have $H^1(X^{v_x}_{\overline{\mathscr{P}}}, \mathcal{O}_{X^{v_x}_{\overline{\mathscr{P}}}})=0$. Hence, by \cite[Theorem 1.4]{WahlI} -- see also \cite[Theorem 3.1]{blowing_down_lifting}, the affinization $X^{v_x}_{\mathscr{P}} \rightarrow \overline{X}^{v_x}_{\mathscr{P}}$ is a flat deformation of the toric affinization $X^{v_x}_{\overline{\mathscr{P}}} \rightarrow \overline{X}^{v_x}_{\overline{\mathscr{P}}}=\Spec\, \CC[\sigma_x \cap \ZZ^2]$. 
    Finally, every exceptional curve is contracted to the $0$-dimensional stratum of $\partial \overline{X}^{v_x}_{\mathscr{P}}$ by Lemma \ref{Lem: monod_invariant}.
\end{proof}

\begin{lemma} \label{Lem: non-compact2}
Let $(Y,D,L)$ be a generic polarized log Calabi--Yau surface, with a polyhedral decomposition $\sP$ of a Symington polytope $P$ defined as above.
For every vertex $v \in (\partial P)_\ZZ$ such that $v \neq v_x$ for all $0$-dimensional strata $x$ of $D$, we have $ \overline{X}^{v}_{\mathscr{P}} \simeq \mathbb{A}^1$. Moreover, the affinization map 
$X^{v}_{\mathscr{P}} \rightarrow \overline{X}^{v}_{\mathscr{P}} \simeq \mathbb{A}^1$ is a surjective projective morphism with general fiber $\PP^1$, and every exceptional curve is contracted to $0 \in \mathbb{A}^1$. In particular,  $(X^{v_x}_{\mathscr{P}}, \partial X^{v_x}_{\mathscr{P}})$ is a $1$-surface as in Definition \ref{def_1_surface}.
\end{lemma}

\begin{proof}
Assume first that there is no interior vertex $p_{ij}$ of a triangle $\Delta_{ij}$ as in \S \ref{Sec: Symington polytopes} such that $pr(p_{ij})=v$. Then, by Step I of the construction of $\mathcal{X}_{\mathscr{P},0}$ in \S \ref{sec: open Kulikov surfaces from log cy}, $X^{v}_{\mathscr{P}}$ is the toric surface whose fan is given by the star of $\mathscr{P}$ at $v$. 
Since $v \neq v_x$ for all $0$-dimensional strata $x$ of $D$, the boundary $\partial P$ of $P$ is straight at $v$, and so the fan of $X^{v}_{\mathscr{P}}$ contains two opposite rays coming from the two edges of $\partial P$ containing $v$, and all the other rays are contained in a half-plane delimited by these two opposite rays. It follows that there exists a surjective projective toric morphism $X^{v}_{\mathscr{P}} \rightarrow \mathbb{A}^1$ with general fiber $\PP^1$. In particular, the affinization  $\overline{X}^{v}_{\mathscr{P}}$ of $X^{v}_{\mathscr{P}}$ is isomorphic to $\mathbb{A}^1$, and 
$X^{v}_{\mathscr{P}} \rightarrow \mathbb{A}^1$ is the affinization morphism.

In general, if there are points $p_{ij}$ such that $pr(p_{ij})=v$, then, $X^{v}_{\mathscr{P}}$ is a flat deformation, obtained by partial smoothing of the toric boundary, of the toric surface $X^{v}_{\overline{\mathscr{P}}}$ whose fan is given by the star of $\mathscr{P}$ at $v$. As in the previous case, we have a surjective projective toric morphism $X^{v}_{\overline{\mathscr{P}}} \rightarrow \mathbb{A}^1$ with general fiber $\PP^1$. 
Since $X^{v}_{\overline{\mathscr{P}}}$ is a toric surface, we have $H^1(X^{v}_{\overline{\mathscr{P}}}, \mathcal{O}_{X^{v}_{\overline{\mathscr{P}}}})=0$, and so by \cite[Theorem 1.4]{WahlI}, the morphism $X^{v}_{\overline{\mathscr{P}}} \rightarrow \mathbb{A}^1$ 
deforms to a surjective projective morphism $X^{v}_{\mathscr{P}} \rightarrow \mathbb{A}^1$ with general fiber $\PP^1$. This morphism is necessarily the affinization map of $X^{v}_{\mathscr{P}}$.
Finally, every exceptional curve is contracted to $0\in \mathbb{A}^1$ by Lemma \ref{Lem: monod_invariant}.
\end{proof}

STEP III: 
In a third step we glue together the log Calabi--Yau surfaces $X_{\mathscr{P}}^{v}$ for $v \in \mathscr{P}^{[0]}$ obtained in Step II along the divisors corresponding to the edges of $\mathscr{P}$, to obtain a normal crossing surface $\mathcal{X}_{\sP,0}$, well-defined up to locally trivial deformations. We show that the normal crossing surface $\cX_{\sP,0}$ is a  quasi-projective open Kulikov surface as in Definition \ref{Def:K_disk}.

\begin{lemma} \label{Lem: quasi-projective}
Let $(Y,D,L)$ be a generic polarized log Calabi--Yau surface, with a polyhedral decomposition $\sP$ of a Symington polytope $P$ defined as above.
Then, the normal crossing surface $\cX_{\sP,0}$ is a quasi-projective open Kulikov surface.
\end{lemma}

\begin{proof}
By construction, the dual intersection complex of $\cX_{\sP,0}$ is $(P,\sP)$. Since $P$ is a disk, condition i) in Definition \ref{Def:K_disk} is satisfied.
By Lemmas \ref{Lem: non-compact1} and \ref{Lem: non-compact2}, the irreducible components of $\cX_0$ are $0$-, $1$-, and $2$-surfaces, and so condition ii) in Definition \ref{Def:K_disk} holds as well.
Finally, since $\sP$ is a triangulation into triangles of size one, $\cX_0$ is combinatorially $d$-semistable as in condition iii) of Definition \ref{Def:K_disk} by \cite[Proposition 3.6]{Eng18}. 
Thus, the normal crossing surface $\cX_0$
is an open Kulikov surface. 

To prove quasi-projectivity, recall that, by Definition \ref{Def: good polyhedral decomposition}, the good polyhedral decomposition $\overline{\mathscr{P}}$ is regular.
Consequently, there exists a convex piecewise linear function $\varphi$ on $\overline{P}$ that induces $\overline{\mathscr{P}}$.
Let $L_\varphi$ denote the corresponding ample line bundle on $\mathcal{X}_{\overline{\mathscr{P}}}$.
From the toric description of the deformation in Step I of the construction of $\cX_{\sP,0}$ in \S\ref{sec: open Kulikov surfaces from log cy}, the restrictions of $L_\varphi$ to $\bigcup_{v' \in pr^{-1}(v)} X^v_{\overline{\mathscr{P}}}$ deform into ample line bundles $\widetilde{L}_{\varphi}^v$ on the toric surfaces $\widetilde{X}^v_{\mathscr{P}}$.
Since any toric surface $X$ satisfies $H^2(X,\mathcal{O}_X)=0$,
deformation theory for line bundles (see for example the proof of \cite[Theorem 2.5.23 (ii)]{sernesi}) ensures that,  in Step II of the construction of $\cX_{\sP,0}$ in \S\ref{sec: open Kulikov surfaces from log cy}, 
the ample line bundles $L^{v}_{\varphi}$ on 
the toric surfaces $\widetilde{X}^{v}_{\mathscr{P}}$ deform 
into ample line bundles $L^{v}_{\varphi}$ on 
the log Calabi--Yau surfaces $X^{v}_{\mathscr{P}}$. 
Furthermore, for every edge $e \in \sP^{[1]}$ adjacent to two vertices $v, v' \in \sP^{[1]}$, the line bundles $L^{v}_{\varphi}$ and $L^{v'}_{\varphi}$
are isomorphic when restricted to $X^e_\sP$.
Thus, by the description of $\Pic(\cX_{\sP,0})$ in Lemma \ref{Lem:Pic},
the ample line bundles $L^{v}_{\varphi}$ glue together to form a global ample line bundle on $\cX_{\sP,0}$.
\end{proof}

\subsection{Semistable mirrors of polarized log Calabi--Yau surfaces}
\label{Sec: semistable mirror of log CY}
Let $(Y,D,L)$ be a generic polarized log Calabi--Yau surface and $(P,\sP)$ a Symington polytope with polyhedral decomposition defining a maximal degeneration of $(Y,D,L)$ as in \S \ref{sec: maximal degenerations}. 
By Theorem \ref{thm_generic}, up to a locally trivial deformation, we can assume that the open Kulikov surface $\mathcal{X}_{\sP,0}$ is $d$-semistable and generic. Therefore, by Theorem \ref{thm: smoothing exists}, there exists a smoothing $\pi_\sP: \mathcal{X}_{\sP} \to \Delta$ of $\mathcal{X}_{\sP,0}$, which in addition is an open Kulikov degeneration. From now on we refer to this smoothing as the \emph{semistable mirror} to $(Y,D,L)$.

Although a smoothing of $\mathcal{X}_{\mathscr{P},0}$ is not unique, it is unique up to deformation.
Throughout the remainder of this paper, we will always consider smoothings of $\mathcal{X}_{\mathscr{P},0}$ up to deformation. Therefore, to simplify notation, we will call $\mathcal{X}_{\mathscr{P}}$ ``the'' smoothing of $\mathcal{X}_{\mathscr{P},0}$, and refer to it as the \emph{semistable mirror} to the polarized log Calabi--Yau surface $(Y,D,L)$, with given Symington polytope $P$ endowed with a good polyhedral decomposition $\mathscr{P}$.  
In the following section \S\ref{sec:mirror algebras}, 
we extend the intrinsic mirror symmetry construction of Gross--Siebert to encompass quasi-projective open Kulikov degenerations such as $\cX_\sP \rightarrow \Delta$. 
Ultimately, we show that starting with a polarized log Calabi--Yau $(Y,D,L)$, constructing its semistable mirror, and then taking the intrinsic mirror to the semistable mirror indeed yields a family of polarized log Calabi--Yau's deformation equivalent to the initial $(Y,D,L)$, as shown in \S\ref{sec: KSBA moduli spaces}. Thus, the semistable mirror construction can be viewed as the reverse of a generalized intrinsic mirror construction of Gross--Siebert for families of log Calabi--Yau's. 

\begin{remark}
    Consider the case where $(Y,D)$ is a toric pair, where $Y$ is a toric surface and $D$ is its anti-canonical boundary divisor, and denote by $P$ the momentum polytope corresponding to the ample line bundle $L$ on $Y$. 
    Then,  the affinization $\cX^{\mathrm{can}}_{\sP} \rightarrow \Delta$ of the semistable mirror is the completion along $0 \in \A^1$ of the three dimensional Gorenstein affine toric variety with fan the cone over $P$ and which admits a natural toric morphism to $\A^1$.
    The semistable mirror
    $\mathcal{X}_{\mathscr{P}} \rightarrow \Delta$ is the completion along $0$ of the toric crepant resolution of this affine toric singularity defined by the regular triangulation $\sP$ of $P$ into triangles of size one.   
\end{remark}

\begin{remark} \label{remark_GHK}
Let $(Y,D,L)$ be a log Calabi--Yau surface with an ample line bundle $L$. 
Since $L$ is ample, we have $L \cdot \beta \in \NN$ for every effective curve class $\beta \in NE(Y)$. Hence, we have a map of monoids $NE[Y] \to \NN$ defined by $\beta \mapsto L \cdot \beta$, inducing a formal curve $\Spf\, \CC[\NN] \to \Spf\, \CC[NE(Y)]$.
If $Y$ is smooth, then the central fiber $\cX^{\mathrm{can}}_{\mathscr{P},0}$ is the $n$-vertex, obtained by gluing $\A^2$'s pairwise along $\A^1$'s as in \cite{GHK1}. In this case, $\cX^{\mathrm{can}}_{\mathscr{P}}$ is expected to be deformation equivalent to the restriction to $\Spf\,\CC[\NN] \to \Spf\, \CC[NE(Y)]$ of the mirror family $\cX^{\mathrm{can}}_{GHK} \to \Spf\, \CC[NE(Y)]$ to $(Y,D)$ obtained by Gross--Hacking--Keel in \cite{GHK1} by smoothing the $n$-vertex.
 \end{remark}

\section{Mirror algebras for open Kulikov degenerations}
\label{sec:mirror algebras}

Gross and Siebert developed an ``Intrinsic Mirror Symmetry'' program, which in particular explains how to construct an algebra of functions associated to degenerations of Calabi--Yau varieties \cite{GS2019intrinsic,GScanonical}. In the following section, after shortly reviewing this construction, we explain how to adapt it to our situation, to construct an algebra of functions, referred to as the ``mirror algebra'', for the mirror to an open Kulikov degeneration.

\subsection{Intrinsic mirror symmetry for projective degenerations} 
\label{sec: intrinsic degenerations}
The intrinsic mirror symmetry construction particularly applies to a projective morphism $g \colon \mathcal{X} \to S$ from a smooth variety $\mathcal{X}$ with $K_\cX=0$ to a non-singular, non-complete curve $S$ with single closed point $0\in S$, such that the central fiber $\mathcal{X}_0=g^{-1}(0)$ is a reduced simple normal crossing divisor on $\mathcal{X}$ -- this is a special situation of a so called ``log smooth'' morphism as in \cite[\S 1.1, pg 7]{GS2019intrinsic}. In particular, the schemes $\mathcal{X}$ and $S$ are viewed as ``log schemes'', with a log structure defined by the data of the divisors $\mathcal{X}_0 \subset \mathcal{X}$ and $0\in S$ respectively.

Associated to the total space $\mathcal{X}$ is a cone complex, referred to as the tropicalization of $\mathcal{X}$ and denoted by $\Sigma(\mathcal{X})$ -- see \cite[\S2.1.4]{ACGSdecomposition} or \cite[\S 2.3.1]{HDTV}. This cone complex can be viewed as the dual intersection complex of the pair $(\mathcal{X}, \mathcal{X}_0)$, which contains a $k$-dimensional cone for each non-empty, connected intersection of $k$ components of $\mathcal{X}_0$ of codimension $k$. The mirror algebra for $g \colon \mathcal{X} \to S$ constructed in \cite{GS2019intrinsic} has a basis $\{ \vartheta_p \}_{p \in \Sigma(\mathcal{X})_{\ZZ}}$, indexed by the integral points of the tropicalization of $\mathcal{X}$, denoted by $\Sigma(\mathcal{X})_{\ZZ}$.

The mirror algebra depends on a choice of a monoid $P$, together with a monoid ideal $I$, defined as follows. Let $P \subset N_1(\mathcal{X}/S)$ be a saturated and finitely generated monoid containing $NE(\mathcal{X}/S)$, where $N_1(\mathcal{X}/S)$ is the group of relative curve classes, modulo numerical equivalence and $NE(\mathcal{X}/S)$ is the monoid of integral points in the cone in $N_1(\mathcal{X}/S) \otimes \RR$ generated by effective curve classes. Let $I \subset P$ be a monoid ideal such that $P \setminus I$ is finite, that is, such that $S_I(\mathcal{X}):=\mathbb{C}[P]/I$ is an Artinian algebra.
Then, the mirror algebra is an $S_I(\mathcal{X})$-algebra structure on the free $S_I(\mathcal{X})$-module
\[R_I(\mathcal{X})=\bigoplus_{p\in \Sigma(X)_\ZZ} S_I(\mathcal{X}) \vartheta_p \,.\]
For every $p,q \in \Sigma(\mathcal{X})_{\ZZ}$, the product is given by 
\begin{equation}
\label{Cpqr}
\vartheta_p \cdot \vartheta_q = \sum_{r \in \Sigma(\mathcal{X})_{\ZZ}} \sum_{\beta \in P} N_{pqr}^{\beta}(\mathcal{X}) t^{\beta} \vartheta_r     
\end{equation}
as explained in \cite{GS2019intrinsic}, where $N_{pqr}^{\beta}(\mathcal{X})$ is a punctured Gromov--Witten invariant of $\mathcal{X}$, enumerating punctured log maps, that are given by rational curves of class $\beta$ with three marked points $x_p,x_q,x_r$ with contact orders $p,q,-r$ with $\mathcal{X}_0$. In addition, the image of $x_r$ is required to be a given point in the stratum $Z_r$ of $\mathcal{X}$ corresponding to the smallest cone of $\Sigma(\mathcal{X})$ containing $r$. Note that such curves, with negative order of tangency $-r \neq 0$ live in the log scheme $\mathcal{X}_0$, where the log structure is obtained by restricting the log structure on $\mathcal{X}$.

\subsection{Mirror algebras}
We cannot directly apply the intrinsic mirror symmetry construction of \cite{GS2019intrinsic} to an open Kulikov degeneration $\pi \colon \mathcal{X} \to \Delta$ for three reasons. First, $\mathcal{X}$ is not an algebraic variety over an algebraic curve, but only a formal scheme over the formal disk $\Delta$. Second, the map $\pi: \mathcal{X} \rightarrow \Delta$ is not projective. Third, the central fiber $\cX_0$ is a normal crossing but possibly not simple normal crossing. We will address these issues to define punctured Gromov--Witten invariants $N_{pqr}^{\beta}(\mathcal{X})$ for $\mathcal{X}$. The first issue will be addressed by reformulating the constructions of \cite{GS2019intrinsic} using only the special fiber $\mathcal{X}_0$, which is a quasi-projective variety. To resolve the second issue, we will use instead that the contraction $f: \mathcal{X} \rightarrow \cX^{\mathrm{can}}$ is projective. Finally, the third issue is addressed in Remark \ref{remark: snc}.

Let $\pi: \mathcal{X} \rightarrow \Delta$ be a quasi-projective open Kulikov degeneration. 
As $\pi$ is a smoothing of $\cX_0$, there is a natural log structure on $\cX_0$, log smooth over the standard log point
$\mathrm{pt}_\NN$ -- see e.g.\ \cite[\S 3.2]{temkin2023height}
for foundational aspects of log structures on formal schemes.
We denote by $\mathcal{X}_0^\dagger$ the corresponding log scheme. Its tropicalization $\Sigma(\cX_0^\dagger)$ is the cone $(C(P), C(\sP))$ over the dual intersection complex $(P,\sP)$.

We will define the structure constants $N_{pqr}^{\beta}(\mathcal{X})$ of the mirror algebra to $\pi: \mathcal{X} \rightarrow \Delta$ using moduli spaces of stable punctured maps to $\mathcal{X}_0^\dagger$. 
A punctured map to $\mathcal{X}_0^\dagger$ is a log morphism $\nu^\dagger: C^\dagger \rightarrow \mathcal{X}_0^\dagger$, where $C^\dagger$ is a log curve satisfying the conditions in \cite[Definition 2.13]{ACGSpunctured}. 
In particular, the scheme $C$ underlying $C^\dagger$ is a projective curve with at worst nodal singularities, and there are distinguished points $x_i$ on $C$ called punctured points. Moreover, each punctured point $x_i$ has a contact order $p_i$, which is either an element of $C(P)_\ZZ$ or of $-C(P)_\ZZ$. A punctured point $x_i$ is called a marked point if $p_i \in C(P)_\ZZ$. Finally, a punctured map is called stable if the underlying map of schemes $\nu: C \rightarrow \mathcal{X}_0$ is a stable map.
We refer to \cite{ACGSpunctured} for the details of the theory of stable punctured maps to a log scheme.

\begin{remark}\label{remark: snc}
    The theory of punctured maps is developed in  \cite{ACGSpunctured}
    under the assumption that the target has a log structure defined with respect to the Zariski topology (and not with respect to the \'etale topology). Correspondingly, the mirror construction of Gross--Siebert \cite{GS2019intrinsic, GScanonical} is formulated under the assumption that the central fiber is simple normal crossing instead of just normal crossing. If the open Kulikov surface $\cX_0$ is not simple normal crossing, there is an edge $e$ in $\sP$ connecting a vertex $v$ to itself. Taking the cone $C(\widetilde{\sP})$ over a rational triangulation $\widetilde{\sP}$ refining $\sP$ and containing as vertices  some rational points in the interior of every such edges $e$ defines a log blow-up $\widetilde{\cX}$ of $\cX$, so that the central fiber $\widetilde{\cX}_0$ is now simple normal crossing, and the Gross--Siebert construction (with the modifications described below to deal with the non-compactness of $\cX_0$) applies to such a degeneration $\widetilde{\cX} \rightarrow \Delta$. Since this log blow-up is harmless for all our arguments, we will no longer mention it explicitly for the remaining of this paper.
\end{remark}

A type $\tau$ of punctured maps is a choice of a curve class $\beta \in NE(\mathcal{X}_0)$, a genus, a number of punctured points, and for each punctured point a choice of contact order. For every type $\tau$, there is a moduli space $\mathcal{M}(\mathcal{X}_0^\dagger,\tau)$ of stable punctured maps to $\mathcal{X}_0^\dagger$ of type $\tau$. By \cite{ACGSpunctured}, $\mathcal{M}(\mathcal{X}_0^\dagger,\tau)$ is a Deligne-Mumford stack, endowed with a natural log structure.

Let $r \in C(P)_\ZZ$. For $r \neq 0$, denote by $Z_r$ the stratum of $\mathcal{X}_0^\dagger$ corresponding to the smallest cone of $C(P)$ containing $r$. For $r=0$, set $Z_0:=\mathcal{X}_0^\dagger$.
As in \cite[\S 3.1]{GS2019intrinsic}, one can define a moduli space $\mathscr{P}(\mathcal{X}_0^\dagger,r)$ of punctures in $\mathcal{X}_0$ with contact order $-r$, which is a log stack with underlying stack $Z_r \times B \mathbb{G}_m$. 
Following \cite[Definition 3.5]{GS2019intrinsic}, for every 
type $\tau$ of genus zero containing a punctured point $x_{\mathrm{out}}$ with contact order $-r$, 
there is a log evaluation morphism at $x_{\mathrm{out}}$
\[ \mathrm{ev}_{x_{\mathrm{out}}}: \mathcal{M}(\mathcal{X}_0^\dagger,\tau) \longrightarrow \mathscr{P}(\mathcal{X}_0^\dagger,r) \,.\]
Fix a morphism of log stacks
\begin{equation} \label{eq_W}
W=\mathrm{pt}_Q/\mathbb{G}_m \longrightarrow \mathscr{P}(\mathcal{X}_0^\dagger,r) \end{equation}
as in \cite[\S 7.2]{GS2019intrinsic} (see also \cite[Step I of the proof of Theorem 6.4]{GScanonical}), where $\mathrm{pt}_Q$ is a log point as in \eqref{eq_log_point}. Forgetting the log and stacky data, this is just the data of a point $z_W \in Z_r$. Then, we define 
\begin{equation*}
\mathcal{M}(\mathcal{X}_0^\dagger,\tau)_W:= \mathcal{M}(\mathcal{X}_0^\dagger,\tau) \times_{\mathscr{P}(\mathcal{X}_0^\dagger, r)} W \,,\end{equation*}
where the fiber product is taken in the category of fine and saturated log stacks.
If all the punctured points of the type $\tau$ apart from $x_{\mathrm{out}}$ are marked points, that is, with contact orders in $C(P)$, then, as in \cite[Definition 7.9]{ACGSpunctured}, the smoothness of $\cX$ over $\CC$ implies that $\mathcal{M}(\mathcal{X}_0^\dagger,\tau)_W$ carries a virtual fundamental class 
\begin{equation} \label{eq_virtual}
[\mathcal{M}(\mathcal{X}_0^\dagger,\tau)_W]^{\mathrm{virt}}\end{equation}
of dimension $n-2$, where $n$ is the number of marked points in $\tau$.
To define invariants from these virtual fundamental classes, it is essential to establish that the moduli spaces $\mathcal{M}(\mathcal{X}_0^\dagger,\tau)_W$ are proper. This properness is not immediate since $\mathcal{X}_0$ is not projective, but nevertheless holds by the following result. The proof relies on the existence of a projective contraction $f: \mathcal{X} \rightarrow \cX^{\mathrm{can}}$ for an open Kulikov degeneration $\pi: \mathcal{X} \rightarrow \Delta$.

\begin{theorem} \label{thm_proper}
Let $\pi: \mathcal{X} \rightarrow \Delta$ be a quasi-projective open Kulikov degeneration, with tropicalization 
$(C(P), C(\sP))$.
Let $r \in C(P)_{\mathbb{Z}}$ and $\tau$ be a genus zero type of punctured maps to $\mathcal{X}_0^\dagger$, of class $\beta \in NE(\mathcal{X}_0/\cX^{\mathrm{can}}_0) \cong NE(\mathcal{X}/\cX^{\mathrm{can}})$,
with a punctured point $x_{\mathrm{out}}$ of contact order $-r$, and whose all other punctured points are marked points. Then, for every $W$ as in \eqref{eq_W}, the moduli space $\mathcal{M}(\mathcal{X}_0^\dagger,\tau)_W$ is proper.
\end{theorem}

\begin{proof}
Let $\nu \colon C \to \mathcal{X}_0^\dagger$ be a punctured map defining a point in the moduli space $\mathcal{M}(\mathcal{X}_0^\dagger,\tau)_W$. 
By definition of $\mathcal{M}(\mathcal{X}_0^\dagger,\tau)_W$, the image curve $\nu(C)$ contains the point $z_W \in Z_r \subset \cX_0$. Thus, denoting by $f: \cX \rightarrow \cX^{\mathrm{can}}$ the affinization morphism, we have $f(\nu(C)) = \{f(z_W)\}$, and so $\nu(C) \subset f^{-1}(\{f(z_W)\})$. Since every punctured map defining a point in the moduli space $\mathcal{M}(\cX_0^\dagger,\tau)_W$ factors through the subvariety $V:= f^{-1}(\{f(z_W)\})$ of $\cX_0$, it follows that 
$\mathcal{M}(\cX_0^\dagger,\tau)_W \simeq \mathcal{M}(V^\dagger,\tau)_W$, 
where $V^\dagger$ is the log scheme obtained by restricting the log structure of $\cX_0^\dagger$ to $V$.
Since $f \colon \mathcal{X} \to \cX^{\mathrm{can}}$ is projective, the fiber $V=f^{-1}(\{f(z_W)\})$ of $f$ is a projective subvariety of $\mathcal{X}_0$, and so the properness of $\mathcal{M}(V^\dagger,\tau)_W$
follows by 
\cite[Corollary 3.19]{ACGSpunctured}.
\end{proof}

We can now define the structure constants $N_{pqr}^{\beta}(\mathcal{X})$ in our setup, following \cite[Definition 7.9]{GS2019intrinsic}.
For every $p,q, r\in C(P)_\ZZ$, and $\beta \in NE(\mathcal{X}/\cX^{\mathrm{can}})$, we denote by $\tau_{pqr}^\beta$ the type of punctured maps to $(\cX^{\mathrm{can}}_0)^\dagger$ of genus zero, class $\beta$, with two marked points with contact order $p$ and $q$ and one punctured point with contact order $-r$. As the type $\tau_{pqr}^\beta$ has $n=2$ marked points, the virtual fundamental class $[\mathcal{M}(\mathcal{X}_0^\dagger,\tau)_W]^{\mathrm{virt}}$ in \eqref{eq_virtual} has dimension $n-2=0$. Moreover, by Theorem \ref{thm_proper}, the moduli space $\mathcal{M}(\mathcal{X}_0^\dagger,\tau)_W$ is proper and so the degree of $0$-cycles on $\mathcal{M}(\mathcal{X}_0^\dagger,\tau)_W$ is well-defined.

\begin{definition} \label{Def: structure constants}
Let $\pi: \mathcal{X} \rightarrow \Delta$ be a quasi-projective open Kulikov degeneration. For every $p,q,r\in C(P)_\ZZ$, $\beta \in NE(\mathcal{X}/\cX^{\mathrm{can}})$,
and $W$ as in \eqref{eq_W}, 
the \emph{structure constant} $N_{pqr}^{\beta}(\mathcal{X})_W \in \mathbb{Q}$ is defined by 
\[   
N_{pqr}^{\beta}(\mathcal{X})_W = 
     \mathrm{deg} [\mathcal{M}(\mathcal{X}_0^\dagger,\tau_{pqr}^\beta)_W]^{\mathrm{virt}} \,.
\]
\end{definition}
As shown in \cite[\S 7.2]{GS2019intrinsic}, the structure constants $N_{pqr}^{\beta}(\mathcal{X})_W$ are independent of the choice of $W$ as in \eqref{eq_W}. 
Consequently, we will omit $W$ from the notation and refer to the structure constants simply as $N_{pqr}^{\beta}(\mathcal{X})$.

\begin{remark}
    In the set-up of \cite{GS2019intrinsic}, for a degeneration $g:\cX \rightarrow S$ 
    as in \S\ref{sec: intrinsic degenerations}, the structure constants $N_{pqr}^\beta(\cX)$ are initially defined in \cite[Definition 3.21]{GS2019intrinsic} using a choice of a point $z$ in the interior of the stratum $Z_r$. For $r=0$, we have $Z_0 = \mathcal{X}$ in \cite{GS2019intrinsic}, and so the interior of $Z_0$ is the complement $\cX \setminus \cX_0$ of the central fiber $\cX_0$. 
    However, in our setting, where we only have a formal scheme $\cX \rightarrow \Delta$, it is not meaningful to choose a point $z$ away from $\cX_0$. 
    As a result, the definition of the structure constants $N_{pqr}^\beta(\cX)$ given in \cite[Definition 3.21]{GS2019intrinsic} cannot be directly applied for $r=0$ in our context.
    Fortunately, a more general definition, using a choice of log point $W$ as in \eqref{eq_W} is provided in \cite[\S 7.2]{GS2019intrinsic} (see also \cite[Step I of the proof of Theorem 6.4]{GScanonical}). This is actually the definition used in the proof of the associativity of the mirror algebra. Our definition of the structure constants $N_{pqr}^\beta(\cX)$ follows this more general approach, which can be applied even if $r=0$ since the point $z_W$ can then be chosen arbitrarily.
\end{remark}

\begin{definition}
Let $\pi: \mathcal{X} \rightarrow \Delta$ be a quasi-projective open Kulikov degeneration.
For every monoid ideal $I \subset NE(\mathcal{X}/\cX^{\mathrm{can}})$ such that 
$NE(\mathcal{X}/\cX^{\mathrm{can}}) \setminus I$ is finite, the mirror algebra $R_I(\mathcal{X})$
is the free $S_I(\mathcal{X})=\CC[NE(\mathcal{X}/\cX^{\mathrm{can}})]/I$-module
\begin{equation}
\label{eq: mirror algebra}
    R_I(\mathcal{X})=\bigoplus_{p\in C(P)_\ZZ} S_I(\mathcal{X}) \vartheta_p \,,
\end{equation}
endowed with the $S_I(\mathcal{X})$-linear product given
for every $p,q \in C(P)_{\ZZ}$ by
\begin{equation}
\vartheta_p \cdot \vartheta_q = \sum_{r \in C(P)_{\ZZ}} \sum_{\beta \in NE(\mathcal{X}/\cX^{\mathrm{can}}) \setminus I} N_{pqr}^{\beta}(\mathcal{X}) t^{\beta} \vartheta_r     
\end{equation}
where the structure constants $N_{pqr}^{\beta}(\mathcal{X}) \in \QQ$ are as in Definition \ref{Def: structure constants}.
\end{definition}

We call the basis $\{\vartheta_p\}_{p\in C(P)_\ZZ}$ of $R_I(\mathcal{X})$ the \emph{theta basis} of the mirror algebra.

\begin{theorem} \label{thm: mirror algebra}
Let $\pi: \mathcal{X} \rightarrow \Delta$ be a quasi-projective open Kulikov degeneration.
For every monoid ideal $I \subset NE(\mathcal{X}/\cX^{\mathrm{can}})$ such that 
$NE(\mathcal{X}/\cX^{\mathrm{can}}) \setminus I$ is finite, the mirror algebra $R_I(\mathcal{X})$ is a commutative, associative $S_I(\mathcal{X})$-algebra with unit given by $\vartheta_0$.
\end{theorem}

\begin{proof}
All the arguments in the proof of \cite[Theorem 1.9]{GS2019intrinsic}
for projective degenerations $g: \mathcal{X} \rightarrow S$ as in \S\ref{sec: intrinsic degenerations} can be adapted to the set-up of open Kulikov degenerations to prove Theorem \ref{thm: mirror algebra}. 
Indeed, the proof of \cite[Theorem 1.9]{GS2019intrinsic} for a projective degeneration 
$g: \mathcal{X} \rightarrow S$
is based on the study of the moduli spaces $\mathcal{M}(\mathcal{X},\tau)_W$, where $\tau$ is a genus zero type of punctured maps with three marked points and an additional punctured point. For an open Kulikov degeneration $\pi: \mathcal{X} \rightarrow \Delta$, the corresponding moduli spaces $\mathcal{M}(\mathcal{X}_0^\dagger,\tau)_W$ are proper by Theorem \ref{thm_proper} and so the non-compactness of $\mathcal{X}_0$ is not an issue.
\end{proof}

\subsection{Structure constants from the canonical scattering diagram} 
\label{Sec: canonical}
Let $\pi: \cX \rightarrow \Delta$ be a quasi-projective open Kulikov degeneration, with tropicalization $(C(P), C(\sP))$.
By \cite{GScanonical}, the structure constants $N_{pqr}^\beta(\cX)$ can be computed in terms of broken lines in the canonical scattering  diagram for $(\cX, \cX_0)$. The canonical scattering diagram
$\mathfrak{D}_{\cX}$ is a collection of walls in $C(P)$, which are cones $C(W)$ over line segments $W$ of rational slopes contained in 2-dimensional faces $\sigma_W$ of $\sP$ in $P$. Every wall $C(W)$ is endowed with a kink $\kappa_W \in NE(\cX/\cX^{\mathrm{can}})$.
If $W$ is contained in an interior edge $e$ of $\sP$, then the kink $\kappa_W$ is equal to the class $[C_e]$ of the corresponding proper curve $C_e$ in the double locus of $\cX_0$. Else, we have $\kappa_W=0$. Each wall $C(W)$ is also endowed with a power series $f_W \in 
\QQ[\Lambda_{W}][\![NE(\cX/\cX^{\mathrm{can}})]\!]$, where $\Lambda_W$ is the 1-dimensional lattice of integral tangent vectors to $W$. The power series $f_W$
is a generating series of genus $0$ punctured 
Gromov--Witten invariants $N_\tau$, 
indexed by wall types $\tau$ and virtually counting \emph{$\mathbb{A}^1$-curves},
\begin{equation} \label{Eq: wall_function}
f_W =\exp \left( \sum_{\tau \in \mathcal{T}_W} k_\tau N_{\tau} z^{m_\tau} t^{\beta_\tau}
\right)\,,\end{equation}
where $\mathcal{T}_W$ is the set of wall types $\tau$ whose leg contains $W$ -- see \cite{GScanonical} for details.

A broken line is a piecewise linear embedding 
$b: (-\infty, 0] \rightarrow C(P)$, oriented by increasing values of the parameter, with finitely many domains of linearity $b_j$, each one contained in a 3-dimensional face of $C(\sP)$, and decorated with a monomial $c_{b_j} z^{m_{b_j}}$, where $c_{b_j} \in \QQ[\![NE(\cX/\cX^{\mathrm{can}})]\!]$ and $m_{b_j} \in \Lambda_{b_j}$, where $\Lambda_{b_j}$ is the 1-dimensional lattice of integral tangent vectors to $b_j$, and with $-m_{b_j}$ pointing in the direction of the broken line. For the initial domain of linearity $b_{\mathrm{in}}$, we have 
$c_{b_{\mathrm{in}}}=1$ and $m_{b_{\mathrm{in}}}$ is called the initial direction of the broken line. Moreover, a bending from a domain of linearity $b_j$ to a domain of linearity $b_{j+1}$ occurs when the broken line crosses a wall $W$, and the monomial
$c_{b_{j+1}}z^{m_{b_{j+1}}}$ is a monomial in the expansion of 
\begin{equation} \label{Eq: wall_crossing}
c_{b_j} z^{m_{b_j}} \left(t^{\kappa_W} f_W\right)^{\langle n_W, m_{b_j} \rangle}\,,\end{equation}
where $n_W$ is the primitive normal vector to $C(W)$, positive on the half-space containing $b_j$. 
Finally, we denote by $b_{\mathrm{out}}$ the final domain of linearity of $b$, and $c_{b_\mathrm{out}} z^{m_{b_\mathrm{out}}}$ the corresponding monomial.
By \cite[Theorem 6.1]{GScanonical}, for every $p,q,r \in C(P)_\ZZ$, we then have
\begin{equation} \label{Eq: broken_line_product} 
N_{pqr}^\beta(\cX)=\sum_{b, b'} c_{b_{\mathrm{out}}} c_{b'_{\mathrm{out}}}\,,
\end{equation}
where the sum is over pairs of broken lines $b$, $b'$
of initial directions $p$, $q$, ending at a common general point close to $r \in C(P)$, and with $m_{b_\mathrm{out}}+m_{b_{\mathrm{out}}'}=r$.
Alternatively, one could consider only  $P$ at height 1 in $C(P)$ by working with the radial projections on $P$ of the broken lines, called jagged paths \cite[\S 4.5]{GHS}.

\begin{lemma} \label{Lem: N_independent}
Let $\cX \rightarrow \Delta$ be a quasi-projective open Kulikov degeneration, with tropicalization 
$(C(P), C(\sP))$. Then, for every $p,q,r \in C(P)_\ZZ$ and $\beta \in NE(\cX/\cX^{\mathrm{can}})$, the structure constants $N_{pqr}^\beta$ only depend on the log central fiber $\cX_0^\dagger$.
\end{lemma}

\begin{proof}
As reviewed above, by \cite[Theorem 6.1]{GScanonical}, the structure constants 
$N_{pqr}^\beta(\cX)$ can be computed from the canonical scattering diagram, which depends only on $\cX_0^\dagger$ by \cite[Proposition 3.7]{GScanonical}.
\end{proof}

\begin{remark}
    Lemma \ref{Lem: N_independent} does not immediately follow from Definition \ref{Def: structure constants}. 
    Although $\mathcal{M}(\cX_0^\dagger, \tau)_W$ depends only on $\cX_0^\dagger$, the definition of the relative obstruction theory in \cite[Definition 7.9]{GS2019intrinsic}, and consequently the definition of the virtual fundamental class $[\mathcal{M}(\cX_0^\dagger, \tau)_W]^{\mathrm{virt}}$, depend a priori on $\cX$.
\end{remark}

\section{Polarized mirrors of open Kulikov degenerations}
\label{sec: polarized}
 
After classifying in \S \ref{Sec: flops} the flops of relative Picard rank one between generic quasi-projective open Kulikov degenerations, we study in \S \ref{Sec: constants_flops} how these flops affect the structure constants $N_{pqr}^\beta(\cX)$ of the mirror algebra.
This is used in \S \ref{Sec: finiteness} to prove a key finiteness property: the mirror algebra of a quasi-projective open Kulikov degeneration $\pi:\cX \rightarrow \Delta$ can be defined over $\CC[NE(\cX/\cX^{\mathrm{can}})]$, and not just over the Artinian rings $\CC[NE(\cX/\cX^{\mathrm{can}})]/I$.
In \S \ref{Sec: polarized_mirrors}-\ref{sec: torus}, following \cite{HKY20}, we introduce the secondary fan 
$\mathrm{Sec}(\cX/\cX^{\mathrm{can}})$ as a coarsening of the Mori fan $\mathrm{MF}(\cX/\cX^{\mathrm{can}})$. We further establish the existence of a natural stacky fan $\mathcal{S}ec (\cX/\cX^{\mathrm{can}})$, whose underlying fan is
$\mathrm{Sec}(\cX/\cX^{\mathrm{can}})$. 
Finally, we show that the Proj of the mirror algebras of the various projective crepant resolutions of $\cX^{\mathrm{can}}$ glue together, extending to a polarized mirror family $(\cY_\cX^{\mathrm{sec}
}, \cD_\cX^{\mathrm{sec}} + \epsilon \, \mathcal{C}_\cX^{\mathrm{sec}
}) \rightarrow \mathcal{S}_\cX^{\mathrm{sec}}$ over the toric Deligne--Mumford stack $\mathcal{S}_\cX^{\mathrm{sec}}$ with fan $\mathcal{S}ec(\cX/\cX^{\mathrm{can}})$.

\subsection{Flops of relative Picard rank one}
\label{Sec: flops}
Let $\pi: \cX \rightarrow \Delta$ be a generic quasi-projective open Kulikov 
degeneration. By Proposition \ref{Prop: Mori dreams}, 
$\pi: \cX \rightarrow \cX^{\mathrm{can}}$ is a Mori dream space, and so $\cX $ admits finitely many small $\QQ$-factorial modifications over $\cX^{\mathrm{can}}$.
Let $b: \cX \dashrightarrow \cX'$ be such a small $\QQ$-factorial modification of $\cX$ over $\cX^{\mathrm{can}}$. Since $K_{\cX}=K_{\cX'}=0$, we will simply refer to such a $b$ as a \emph{flop}.
The nef cones
$\mathrm{Nef}(\cX/\cX^{\mathrm{can}})$ and $\mathrm{Nef}(\cX'/\cX^{\mathrm{can}})$ are maximal dimensional cones of the Mori fan $\mathrm{MF}(\cX/\cX^{\mathrm{can}})$. 

We say that  $b: \cX \dashrightarrow \cX'$ is a flop of \emph{relative Picard rank one} 
if the intersection $\mathrm{Nef}(\cX/\cX^{\mathrm{can}}) \cap \mathrm{Nef}(\cX'/\cX^{\mathrm{can}})$ is a common codimension one face of $\mathrm{Nef}(\cX/\cX^{\mathrm{can}})$ and $\mathrm{Nef}(\cX'/\cX^{\mathrm{can}})$, that is, if the face $NE(b)$ of $NE(\cX/\cX^{\mathrm{can}})$ spanned by the curves contracted by $b$ is a ray.
We classify the possible flops of relative Picard rank one in Theorem \ref{Thm: exceptional_curves} below.
To do this, we first review the elementary flops referred to as Type I and Type II in \cite[pg. 12]{SB}, and called M1 flops and M2 flops in \cite[Definition 7.15]{AE}. We will follow the terminology of \cite[Definition 7.15]{AE}.

\begin{definition} \label{def_M1_curve}
Let $\cX \rightarrow \Delta$ be an open Kulikov degeneration. 
An \emph{interior exceptional curve} in $\cX$ is a curve $C \simeq \PP^1$, that is contained within a unique irreducible component $X^v$ of $\cX_0$, and satisfies $C|_{X^v}^2=-1$.
\end{definition}

\begin{definition} \label{def_M1_flop}
Let $\cX \rightarrow \Delta$ be a quasi-projective open Kulikov degeneration.
An \emph{M1 flop} of $\cX$ is a flop $b: \cX \dashrightarrow \cX'$ whose exceptional locus in $\cX$ consists of a single interior exceptional curve.
A \emph{multiple M1 flop} is a flop $b: \cX \dashrightarrow \cX'$ whose exceptional locus in $\cX$ consists of a finite disjoint union of interior exceptional curves. 
\end{definition}

\begin{definition} \label{def_M2_curve}
Let $\cX \rightarrow \Delta$ be an open Kulikov degeneration. A \emph{$(-1,-1)$-double curve} in $\cX$ is an irreducible component $C$ of the double locus of $\cX_0$  such that $C \simeq \PP^1$, and, denoting by $C_1$ and $C_2$ the irreducible components of the preimage of $C$ in the normalization of $\cX_0$, and $X_1$, $X_2$ the irreducible components of the normalization of $\cX_0$ containing $C_1$ and $C_2$ respectively, we have $C_1|_{X_1}^2=C_2|_{X_2}^2=-1$.
\end{definition}

\begin{definition} \label{def_M2_flop}
Let $\cX \rightarrow \Delta$ be a quasi-projective open Kulikov degeneration.
An \emph{M2 flop} of $\cX$ is a flop $b: \cX \dashrightarrow \cX'$ whose exceptional locus in $\cX$ consists of a single $(-1,-1)$-double curve.
A \emph{multiple M2 flop} is a flop $b: \cX \dashrightarrow \cX'$ whose exceptional locus in $\cX$ consists of a finite disjoint union of $(-1,-1)$-double curves.
\end{definition}

\begin{remark}
     M0 flops, whose exceptional locus is an internal $(-2)$-curve, are also introduced in \cite[pg. 12]{SB} and \cite[Definition 7.15]{AE}. 
     However, since this paper primarily focuses on \emph{generic} open Kulikov degenerations, 
     which do not contain internal $(-2)$-curves, M0-flops do not occur in our setting.
\end{remark}

The following result will be used to prove Theorem \ref{Thm: exceptional_curves} on the exceptional locus of flops of relative Picard rank one.

\begin{lemma} \label{lem2}
	Let $\cX \rightarrow \Delta$ be a quasi-projective open Kulikov degeneration. Let $E$ be an interior exceptional curve and $C$ a $(-1,-1)$-double curve.
	Then, the classes of $E$ and $C$ in $NE(\cX/\cX^{\mathrm{can}})$ are not proportional.
\end{lemma}

\begin{proof}
	By Lemma \ref{lem_picard}, the morphism $\cX \rightarrow \cX^{\mathrm{can}}$ is projective, implying that the classes $[E]$ and $[C]$ of $E$ and $C$ in 
	$NE(\cX/\cX^{\mathrm{can}})$ are both non-zero. 
	Consider the tropicalization $(P,\sP)$ of $\cX_0$, and let $e$ be the edge of $\sP$ corresponding to $C$, with adjacent vertices $v_1$ and $v_2$. Denote by $f_3$ and $f_4$ the faces of $\sP$ that meet along $e$. 
    Since $\sP$ is a planar  polyhedral decomposition, it is impossible for both $v_1=v_2$ and $f_3=f_4$ to hold simultaneously.
    As a consequence, there must be at least three components $X^v$ of $\cX_0$ such that $X^v\cdot C \neq 0$. On the other hand, at most two  components $X^v$ of $\cX_0$ satisfy $X^v\cdot E \neq 0$, and so $[C]$ and $[E]$ cannot be proportional.
\end{proof}

\begin{theorem} \label{Thm: exceptional_curves}
Let $\pi: \cX \rightarrow \Delta$ be a generic quasi-projective open Kulikov degeneration, and let $b: \cX \dashrightarrow \cX'$ be a flop of $\cX$ of relative Picard rank one. Then, either $b$ is a multiple M1 flop, or a multiple M2 flops. In particular, $\cX' \rightarrow \Delta$ is also a generic quasi-projective open Kulikov degeneration.
\end{theorem}

\begin{proof}
Since $\cX_0$ is generic,  Lemma \ref{lemma_no_internal} ensures that no irreducible component of $\cX_0$ contains an internal $(-2)$-curve. Consequently, by the proof of \cite[Lemma 3.2]{SB2}, the connected components of the exceptional locus of $b$ in $\cX$ must be either interior exceptional curves or $(-1,-1)$-double curves.
Moreover, because $b$ is assumed to be of relative Picard rank one, the classes of all these curves are proportional. 
However, by Lemma \ref{lem2}, the curves classes of an interior exceptional curve and of a $(-1,-1)$-double curve cannot be proportional. 
It follows that the connected components of the exceptional locus of $b$ consist either entirely of interior exceptional curves, or entirely of $(-1,-1)$-double curves. 
Finally, the genericity of $\cX' \rightarrow \Delta$  follows from the local calculations in \cite[\S 7C]{AE}, which imply that the periods of open Kulikov surfaces remain invariant under both M1 and M2 flops.
\end{proof}

We show below that if the irreducible components $X^v$ of $\cX_0$ are normal, and the double intersections $X^v \cap X^{v'}$ are connected, then any flop of relative Picard rank one must be either an M1 flop or an M2 flop. In general, it may not always be possible to factorize a multiple M1 flop or a multiple M2 flop into a sequence of M1 and M2 flops defined as above within the quasiprojective category, though it is always possible analytically.

\begin{theorem}
Let $\pi: \cX \rightarrow \Delta$ be a generic quasi-projective open Kulikov degeneration.
Assume that every irreducible component $X^v$ of the central fiber $\cX_0$ is normal, and that all the double intersections $X^v \cap X^{v'}$ are connected. 
Then, every flop $b: \cX \dashrightarrow \cX'$ of relative Picard rank one is either an M1 flop or an M2 flop. 
\end{theorem}

\begin{proof}
By Theorem \ref{Thm: exceptional_curves}, 
it suffices to show that two disjoint interior exceptional curves or two disjoint $(-1,-1)$-double curves cannot have positively proportional curve classes.

Let us first consider the case of two disjoint $(-1,-1)$-double curves $C_1$ and $C_2$. As the irreducible components of $\cX_0$ are normal and the double intersections are connected, there exists $v$ and $w \neq v$ such that $C_1=X^v \cap X^w$. Moreover, $X^v$ and $X^w$ are exactly the irreducible components of $\cX_0$ that have negative intersection with $C_1$. If $C_2$ is positively proportional to $C_1$, then $C_2$ also intersects negatively $X^v$ and $X^w$, and so $C_2=X^v \cap X^w=C_1$, contradicting that $C_1$ and $C_2$ are disjoint.

It remains to treat the case of two disjoint interior exceptional curves $C_1$ and $C_2$. As the irreducible components of $\cX_0$ are normal, there exists $v$ and $w \neq v$ such that $C_1 \subset X^v$, $C_1 \cdot X^v=-1$, and $C_1 \cdot X^w=1$. Moreover, $X^v$ and $X^w$ are exactly the irreducible components of $\cX_0$ that have non-zero intersection with $C_1$. If $C_2$ is positively proportional to $C_1$, then we also have $C_2 \cdot X^v =-1$, hence $C_2 \subset X^v$, and $C_2 \cdot X^w=1$. Since the double intersection $X^v \cap X^w$ is connected by assumption, the line bundle $\cO_{X^v}(C_1-C_2)$ is trivial 
 in restriction to every irreducible component of $\partial X^v$, and so, by the description of $\Pic(\cX_0)$ given by Lemma \ref{Lem:Pic}, defines a line bundle on $\cX_0$. By Lemma \ref{Lem:Pic}, we also have $H^2(\cX_0, \mathcal{O}_{\cX_0})=0$, and so this line bundle formally deforms into a line bundle $L$ on $\cX$. As $C_1$ and $C_2$ are disjoint curves, we obtain that $L \cdot C_1=-1$ and $L\cdot C_2=1$. 
 Since the classes $C_1$ and $C_2$ are assumed to be proportional, we deduce that $C_1=-C_2$, in contradiction with the strict convexity of the cone of curves $NE(\cX/\cX^{\mathrm{can}})$.
\end{proof}

\subsection{Structure constants under flops}
\label{Sec: constants_flops}

In this section, we describe the behavior of 
structure constants $N_{pqr}^\beta(\cX)$ under flops. Parallel results are obtained in \cite{GHKSK3}
in the context of Type III degenerations of K3 surfaces, and we refer to \cite{GHKSK3} for a more detailed exposition.

Since a flop $b: \cX \dashrightarrow \cX'$ is an isomorphism in codimension one, the Picard groups of $\cX$ and $\cX'$ are naturally identified.
Consequently, we obtain a natural isomorphism
\begin{equation}
\label{Eq: h}
h \colon    N_1(\mathcal{X}/\cX^{\mathrm{can}}) \longrightarrow  N_1(\mathcal{X}'/\cX^{\mathrm{can}})\,, 
\end{equation}
where $N_1(\mathcal{X}/\cX^{\mathrm{can}})$ denotes the group of numerical equivalence classes of one cycles in $\mathcal{X}$, whose irreducible components are curves contracted to a point under the map $f \colon \mathcal{X} \to \cX^{\mathrm{can}}$. 
For instance, if $C$ is a flopping curve in $\cX$ for a multiple M1 or M2 flop, then we have $h([C])=-[C']$, where $C'$ is the flopped curve in $\cX'$.
Additionally, there is a natural identification $C(P)_\ZZ \simeq C(P')_\ZZ$ between the integral points of the tropicalization $C(P)$ and $C(P')$ of $\cX \to \Delta$ and $\cX' \to \Delta$ respectively.

We now describe the transformation of structure constants $N_{pqr}^\beta(\cX)$ under multiple M1 or M2 flops.

\begin{theorem} \label{thm_flop_M1}
Let $\cX \rightarrow \Delta$ be a quasi-projective open Kulikov degeneration and $b: \cX \dashrightarrow \cX'$ be a multiple
M1 flop. 
Then, for every $p,q,r \in C(P)_{\ZZ} \cong C(P')_{\ZZ} $ and $\beta \in NE(\mathcal{X} / \cX^{\mathrm{can}})$, we have the equality $N_{pqr}^{\beta}(\mathcal{X}) = N_{pqr}^{h(\beta)}(\mathcal{X}')$,
where $h$ is the isomorphism given in Equation \eqref{Eq: h}.
\end{theorem}

\begin{proof}
It follows from the properness of moduli spaces given by Theorem \ref{thm_proper} that the results proved in \cite[\S 8]{gross2023remarks}
on punctured Gromov--Witten invariants of Type III degenerations of K3 surfaces also hold by the same arguments for the punctured Gromov--Witten $N_{pqr}^\beta(\cX)$ of quasi-projective open Kulikov degenerations. 
By \cite[Theorem 8.15]{gross2023remarks}, the punctured Gromov--Witten invariants $N_{pqr}^\beta(\cX)$ can be decomposed as a product of punctured Gromov--Witten invariants of the irreducible components of $\cX_0$ and of toric surfaces appearing as irreducible components of a resolution of a large enough base change of $\cX \rightarrow \Delta$. 
On the other hand, a multiple M1 flop only changes the irreducible components of $\cX_0$ by interior blow-ups or interior blow-downs.
Thus, the statement follows from \cite[Theorem 8.16]{gross2023remarks} describing the change of the punctured Gromov--Witten invariants of a log Calabi--Yau surfaces under interior blow-up. We refer to \cite{GHKSK3} for more details.
\end{proof}

\begin{remark}
An M1 flop induces an isomorphism at the level of polyhedral complexes $(P,\sP)$. At the level of the integral affine structure with singularities on $P$,
a focus-focus singularity moves along its monodromy invariant direction from the vertex $v$, corresponding to the component $X^v$ containing the flopping curve $C$, to the vertex $v'$, corresponding to the component $X^{v'}$ containing the flopped curve $C'$.
\end{remark}

Let $\cX \rightarrow \Delta$ be a quasi-projective open Kulikov degeneration and $b: \cX \dashrightarrow \cX'$ be a multiple
M2 flop along disjoint $(-1,-1)$-double curves $C_i$ in $\cX$.  
For every $i$, the double curve $C_i$ corresponds to an edge $e_{C_i}$ of $\sP$, which is a diagonal in a quadrilateral $Q_i$. Since the double curves $C_i$ are disjoint, the quadrilaterals $Q_i$ do not overlap.
The polyhedral decomposition $\sP'$ of $P$ defined by $\cX'$ is obtained from $\sP$ by flipping $e_{C_i}$ in $Q_i$ for every $i$, that is, by removing the diagonal $e_{C_i}$ of $Q$ and introducing the opposite diagonal $e_{C_i'}$ of $Q$.
Let $c_{Q_i} \in C(\sP)_\ZZ$ be the integral primitive generator of the ray in $C(\sP)$ that passes through the intersection point of $e_{C_i}$ and $e_{C_i'}$ in $P \subset C(P)$. Geometrically, $c_{Q_i}$ corresponds to the divisorial valuation defined by the exceptional divisor $E_i$ of the blow-up $\cZ$ of the ordinary double point obtained by contracting $C_i$.
Denote by $\widetilde{C(\sP)}$ the dual intersection complex of $\cZ$, that is, the cone complex obtained from $C(\sP)$ by adding the cones over $e_{C_i'}$. For every $i$, define
\begin{equation} \label{Eq: phi_E}
\varphi_{E_i}: C(P) \longrightarrow \RR
\end{equation}
as the unique piecewise linear function, linear on each cone of $\widetilde{C(\sP)}$, satisfying $\varphi_{E_i}(p)=0$ for every $p \in P_\ZZ$ and $\varphi_{E_i}(c_{Q_i})=1$ -- see Figure \ref{figure9}.

\begin{figure}[h]
\center{\includegraphics{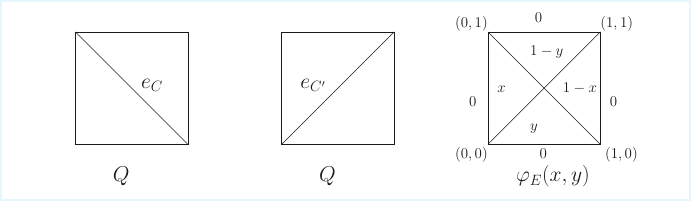}}
\caption{Change of the polyhedral decomposition $\sP$ under a M2 flop on the left. The function $\varphi_E$ on the right.}
\label{figure9}
\end{figure}

\begin{theorem} \label{thm_flop_M2}
Let $\cX \rightarrow \Delta$ be a quasi-projective open Kulikov degeneration and $b: \cX \dashrightarrow \cX'$ be a multiple
M2 flop along disjoint $(-1,-1)$-double curves $C_i$ in $\cX$.  
Then, for every $p,q,r \in C(P)_{\ZZ} \cong C(P')_{\ZZ} $ and $\beta \in NE(\mathcal{X} / \cX^{\mathrm{can}})$, we have the equality
\[ N_{pqr}^{\beta}(\mathcal{X}) = N_{pqr}^{h(\beta)+\sum_i m_{pqr,i}[C_i]}(\mathcal{X}') \, ,\]
where $h$ is the isomorphism given in Equation \eqref{Eq: h}, and 
\[m_{pqr,i}= \varphi_{E_i}(p)+\varphi_{E_i}(q)-\varphi_{E_i}(r)\,,\]
where $\varphi_{E_i}$ is the piecewise linear function defined in Equation \eqref{Eq: phi_E}.
\end{theorem}

\begin{proof}
The multiple M2 flop is defined by a diagram 
\[\begin{tikzcd}
&
\cZ \arrow[dr, "p'"] \arrow{ld}[swap]{p}
&\\
\cX& 
&
\cX'\,.
\end{tikzcd}\]
As $\cZ$ is a log modification of $\cX$,
induced by the subdivision $\widetilde{C(\sP)}$ of $C(\sP)$,  
we have by \cite[\S 10]{johnston2022birational}
\[N_{pqr}^\beta(\cX)=\sum_{\substack{\widetilde{\beta}\in NE(\cZ/\cX^{\mathrm{can}})\\  p_\star \widetilde{\beta}=\beta}} N_{pqr}^{\widetilde{\beta}}(\cZ)\,.\]
In fact, arguing as in the proof of \cite[Lemma 8.7]{gross2023remarks}, there exists a unique $\widetilde{\beta}$ such that $N_{pqr}^{\widetilde{\beta}}(\cZ) \neq 0$, and so we have $N_{pqr}^\beta(\cX)=N_{pqr}^{\widetilde{\beta}}(\cZ)$.
Similarly, we have $N_{pqr}^{\widetilde{\beta}}(\cZ)=N_{pqr}^{p'_\star \widetilde{\beta}}(\cX')$.
Hence, in order to prove Theorem \ref{thm_flop_M2}, it suffices to show that $p'_\star \widetilde{\beta} = h(\beta)+\sum_i m_{pqr,i}C_i$.
This formula holds since, by geometry of the flop, we have $p'_\star \widetilde{\beta} = h(\beta)+\sum_i (E_i \cdot \widetilde{\beta}) C_i$, and
we have $E_i \cdot \widetilde{\beta}=\varphi_{E_i}(p)+\varphi_{E_i}(q)-\varphi_{E_i}(r)$ by 
\cite[Corollary 1.14]{GS2019intrinsic}.
\end{proof}

We now aim to formulate how the structure constants $N_{pqr}^\beta(\cX)$ transform under general flops. 
To achieve this, we first apply \cite[Construction 1.14]{GHS}.
Let $\cX \rightarrow \Delta$ be a quasi-projective open Kulikov degeneration, with tropicalization $(C(P), C(\sP))$. 
Following \cite[Construction 1.14]{GHS}, we define 
\[ \mathbb{B}:= C(P) \times N_1(\cX/\cX^{\mathrm{can}})_\RR\]
and equip $\mathbb{B}$ with the polyhedral decomposition \[\sP_{\mathbb{B}}:=\{\tau \times N_1(\cX/\cX^{\mathrm{can}})_\RR\,|\, \tau \in C(\sP)\}\,.\] 
Additionally, there exists a natural lift of the integral affine structure defined on the complement in $C(P)$ of codimension $2$ cells of $C(\sP)$ to an integral affine structure on the complement in $\mathbb{B}$ of the codimension $2$ cells of $\sP_{\mathbb{B}}$, see \cite[Eq 1.9]{GHS} : given two maximal cells $\sigma$, $\sigma'$ of $C(\sP)$ that intersect along a codimension cell $\rho$, and a chart $f: \sigma \cup \sigma' \rightarrow \RR^n$ of the integral affine structure on $C(P)$, a corresponding chart of the integral affine structure on $\mathbb{B}$ is given by 
\begin{align} \label{eq_B}
    (\sigma \cup \sigma')\times N_1(\cX/\cX^{\mathrm{can}})_\RR &\longrightarrow \RR^n \times N_1(\cX/\cX^{\mathrm{can}})_\RR \\
    (x,q) &\longmapsto \begin{cases}
        (f(x),q) \,\,\,\,\,\,\,\,\,\, \mathrm{when}\, x \in \sigma \\
        (f(x), q+[C_\rho] \delta(x)) \,\,\mathrm{when}\, x\in \sigma'\,,
    \end{cases} \nonumber
\end{align} 
where $\delta$ is the signed integral distance to $\rho$ that is positive on $\sigma'$, and $[C_\rho] \in N_1(\cX/\cX^{\mathrm{can}})$ is the curve class of the double curve in $\cX_0$ corresponding to the edge $\rho \cap P$ of $\sP$.
There is a natural linear projection $\pi: \mathbb{B} \rightarrow C(P)$ making $\mathbb{B}$ a $N_1(\cX/\cX^{\mathrm{can}})_\RR$-torsor over $C(P)$.
Finally, we denote by 
\begin{equation} \label{eq_varphi}
\varphi: C(P) \longrightarrow \mathbb{B}\end{equation}
the zero section. For the integral affine structure on $\mathbb{B}$ defined by \eqref{eq_B}, $\varphi$ is not a linear map but only a piecewise linear map. 

A non-Archimedean interpretation of $\mathbb{B}$, $\pi$, and $\varphi$ is provided in \cite[\S 2.1]{HKY20} in terms of the essential skeleton of the universal $\mathbb{G}_m$-torsor over the non-Archimedean generic fiber of $\cX \rightarrow \Delta$.
According to \cite[\S 2.1]{HKY20},
every flop $b: \cX \dashrightarrow \cX'$ induces a natural piecewise linear isomorphism 
\[ \nu_b: \mathbb{B} \xrightarrow{\sim} \mathbb{B}'\]
between the corresponding $N_1(\cX/\cX^{\mathrm{can}})_\RR \simeq N_1(\cX'/\cX^{\mathrm{can}})_\RR$ torsors. This isomorphism does not preserve the polyhedral decompositions in general.
In particular, the difference \begin{equation} \label{eq_diff}
\varphi'-\nu_b \circ \varphi
\end{equation}
can be regarded as an ordinary $N_1(\cX/\cX^{\mathrm{can}})_\RR$-valued function on $C(P)_\ZZ \simeq C(P')_\ZZ$. The following lemmas compute this function for multiple M1 and M2 flops.

\begin{lemma} \label{Lem_M1_flop}
Let $\cX \rightarrow \Delta$ be a quasi-projective open Kulikov degeneration and $b: \cX \dashrightarrow \cX'$ be a multiple
M1 flop. Then, we have $\varphi' - \nu_b \circ \varphi=0$.
\end{lemma}

\begin{proof}
Since the exceptional locus of a multiple M1 flop contains no stratum of $\cX_0$, the result follows from \cite[Lemma 2.10]{HKY20} and \cite[Lemma 2.11]{HKY20}.
\end{proof}

\begin{lemma} \label{Lem_M2_flop}
Let $\cX \rightarrow \Delta$ be a quasi-projective open Kulikov degeneration and $b: \cX \dashrightarrow \cX'$ be a multiple
M2 flop along disjoint $(-1,-1)$-double curves $C_i$ in $\cX$. Then, we have 
\[ \varphi' - \nu_b \circ \varphi = \sum_i \varphi_{E_i}[C_i]\,, \]
where $\varphi_{E_i}$ are the piecewise linear functions defined by Equation \eqref{Eq: phi_E}.
\end{lemma}

\begin{proof}
The result follows from \cite[Lemma 2.10]{HKY20} calculating $\varphi' - \nu_b \circ \varphi$ in terms of the exceptional divisors of a diagram resolving the flop.
\end{proof}

\begin{theorem} \label{thm_flop_general}
   Let $\cX \rightarrow \Delta$ be a generic quasi-projective open Kulikov degeneration and $b: \cX \dashrightarrow \cX'$ be a flop.
    Then, for every $p,q,r \in C(P)_{\ZZ} \cong C(P')_{\ZZ} $ and $\beta \in NE(\mathcal{X} / \cX^{\mathrm{can}})$, we have the equality
\[ N_{pqr}^{\beta}(\mathcal{X}) = N_{pqr}^{h(\beta)+n_{pqr}}(\mathcal{X}') \, ,\]
where 
\[ n_{pqr}=(\varphi'(p)-\nu_b(\varphi(p)))+
(\varphi'(q)-\nu_b(\varphi(q)))-(\varphi'(r)-\nu_b(\varphi(r))) \in N_1(\cX/\cX^{\mathrm{can}}) \simeq N_1(\cX'/\cX^{\mathrm{can}})\,.\]
\end{theorem}

\begin{proof}
Since any flop can be expressed as a composition of flops of relative Picard rank one, and Theorem \ref{Thm: exceptional_curves} establishes that flops of relative Picard rank one of a generic quasi-projective open Kulikov degeneration are either multiple M1 or M2 flops, 
it suffices to prove Theorem \ref{thm_flop_general} under the assumption that $b$ is a multiple M1 or M2 flop.
For multiple M1 flops, the result follows by combination of Theorem 
\ref{thm_flop_M1} and Lemma \ref{Lem_M1_flop}. 
Similarly, for multiple M2 flops, the result follows from Theorem 
\ref{thm_flop_M2} in conjunction with Lemma \ref{Lem_M2_flop}.
\end{proof}

\subsection{The finiteness property}
\label{Sec: finiteness}

In this section, we prove a finiteness result for the punctured Gromov--Witten invariants $N_{pqr}^\beta(\cX)$, which implies that the mirror algebra of a quasi-projective open Kulikov degeneration $\pi:\cX \rightarrow \Delta$ can be defined over $\CC[NE(\cX/\cX^{\mathrm{can}})]$, and not just over the Artinian rings $\CC[NE(\cX/\cX^{\mathrm{can}})]/I$. We first introduce the notion of a good divisor in Definition \ref{Def: good_divisor} and prove the finiteness result under the assumption of existence of a good divisor in Theorem \ref{thm_finiteness_new}. 
Then, we show in Lemmas \ref{Lem: nef}-\ref{lem_finiteness_2} that a good divisor always exist after flops, and we deduce from this the general finiteness result in Theorem \ref{Thm_Finiteness}.

Let $\pi: \cX \rightarrow \Delta$ be a quasi-projective open Kulikov degeneration, with tropicalization $(C(P),C(\sP))$.
Let \[F=\sum_{p \in P_\ZZ} a_p D_p\] 
be a divisor on $\mathcal{X}$ supported on $\mathcal{X}_0$, with $a_p \in \ZZ$. 
We denote by $\varphi_F : C(P)\rightarrow \RR$ the unique function on $C(P)$ which is linear on each cone of $C(\sP)$ and such that $\varphi_F (p)=a_p$ for all $p \in P_\ZZ$.

\begin{lemma}
\label{Balancing}
Let $\pi: \cX \rightarrow \Delta$ be a quasi-projective open Kulikov degeneration, with tropicalization $(C(P), C(\sP))$.
For every divisor $F$ on $\mathcal{X}$ supported on $\mathcal{X}_0$, and for every $p,q,r\in C(P)_\ZZ$ and $\beta \in NE(\cX/
\cX^{\mathrm{can}})$ such that $N_{pqr}^\beta(\cX) \neq 0$, we have 
\begin{equation}
\varphi_F(p)+\varphi_F(q)-\varphi_F(r)=F \cdot \beta \,.
\end{equation}
\end{lemma}

\begin{proof}
See \cite[Corollary 1.14]{GS2019intrinsic}.
\end{proof}

We define below a notion of good divisor. 
Applying Lemma \ref{Balancing} to a good divisor $F$ will allow us to constrain the curve classes $\beta$ such that $N_{pqr}^\beta(\cX) \neq 0$.

\begin{definition}
\label{Def: good_divisor}
Let $\pi: \cX \rightarrow \Delta$ be a quasi-projective open Kulikov degeneration. An effective divisor $F$ on $\cX$ is said to be \emph{good} if $F$ is supported on $\cX_0$ and satisfies the following conditions. 
\begin{itemize}
\item[i)] The divisor $-F$ is nef, and for every proper irreducible curve $C \subset \mathcal{X}_0$ we have $-F \cdot C \geq 1$, except for finitely many disjoint curves $C$ which are either interior exceptional curves, as in Definition \ref{def_M1_curve}, or $(-1,-1)$-double curves, as in Definition \ref{def_M2_curve}. 
\item[ii)] The map $f: \cX \rightarrow \cX^{\mathrm{can}}$ factors as $\cX \xrightarrow{g} \mathcal{Z} \xrightarrow{h} \cX^{\mathrm{can}}$, where $g$ contracts the finitely many curves $C$ in $\cX$ with $C \cdot F=0$. Moreover, the divisor $g_\star F$ is Cartier and $-g_\star F$ is $h$-ample. 
\end{itemize}
\end{definition}

Let $\pi: \cX \rightarrow \Delta$ be a quasi-projective open Kulikov degeneration and $F$ a good effective divisor on $\cX$.
By Definition \ref{Def: good_divisor} ii) of a good divisor, the curves $C$ with $F \cdot C=0$
are exactly the curves contracted by a birational morphism, and so they span a face $G$ of the cone $NE(\cX/\cX^{\mathrm{can}}) \otimes \RR$. 
We denote by $J \subset NE(\cX/\cX^{\mathrm{can}})$ the corresponding monoid ideal consisting of curve classes $\beta \in NE(\cX/\cX^{\mathrm{can}})$ such that $\beta \notin G$. 
By Definition \ref{Def: good_divisor} i) of a good divisor, the curve classes $\beta \in J$ are exactly the curve classes $\beta \in NE(\cX/\cX^{\mathrm{can}})$ such that $-F \cdot \beta \geq 1$.

As reviewed in \S \ref{Sec: canonical}, the structure constants $N_{pqr}^\beta(\cX)$ can be computed in terms of the canonical scattering diagrams $\mathfrak{D}_\cX$, which consists of walls $(W, f_W)$ as in Equation \eqref{Eq: wall_function}.
In the following two lemmas, we show that the existence of a good effective divisor strongly constrains the canonical scattering diagram $\mathfrak{D}_\cX$.

\begin{lemma} \label{Lem: finiteness_0}
Let $\pi: \cX \rightarrow \Delta$
be a quasi-projective open Kulikov degeneration, 
and let $F$
be a good effective divisor on $\cX$.
Let $(W,f_W)$ be a wall of the canonical scattering diagram $\mathfrak{D}_\cX$ with $f_W \neq 1 \mod J$. Then, $W$ coincides with an edge $e$ of $\sP$ and 
\begin{equation} \label{Eq: wall_mod_J}
f_W = \prod_{i\in I_{v,e}} (1+t^{[E_i]} z^{m_{v,e}}) \prod_{j\in I_{v',e}} (1+t^{[E_j]} z^{m_{v',e}})\mod J\,,\end{equation} 
where $v$ and $v'$ are the vertices adjacent to $e$, $(E_i)_{i \in I_{v,e}}$ 
(resp. $(E_j)_{j \in I_{v',e}}$) is the set of
interior exceptional curves $E_i$ in $X^v$ (resp\,. $X^{v'}$) intersecting $X^e$, and where $m_{v,e}$ (resp.) is the integral primitive direction of $e$ pointing towards $v$ (resp\,. $v'$).  
\end{lemma}

\begin{proof}
We consider $\mathbb{A}^1$-curves $C$ of class $[C] \notin J$, that is, $[C] \in G$.
Since $C$ is connected, and the curves in Definition \ref{Def: good_divisor} i) spanning $G$ are disjoint and rigid, it follows that $C$ must be either a multiple of an interior exceptional curve or a multiple of a double curve. 
However, the formal neighborhood of a double curve is toric, whereas the tropicalization of an $\mathbb{A}^1$-curve has a unique leg. 
By the toric balancing condition for punctured curves, this implies that $C$ cannot be a multiple of a double curve. 
Consequently, $C$ is necessarily a multiple of an interior exceptional
curve $E_i$, which belongs to an irreducible component 
$X^v$ of $\cX_0$ and intersects a component $X^e \subset X^v$ of the double locus of $\cX_0$.
Denote by $f_{W,E_i}$ the contribution to $f_W$ of the multiple covers of $E_i$.
The formal neighborhood of $E_i$ in the log scheme $\cX_0^\dagger$ is isomorphic to the formal neighborhood of the exceptional curves appearing in the proof of \cite[Lemma 4.21]{HDTV}.
Therefore, applying \cite[Lemma 4.21]{HDTV}, we obtain
\[f_{W,E_i}=1+t^{[E_i]} z^{m_{v,e}}\,, \]
where $m_{v,e}$ is the integral primitive direction of $e$ pointing towards $v$.
This concludes the proof of Lemma \ref{Lem: finiteness_0}.
\end{proof}

\begin{lemma} \label{Lem: finiteness_new}
Let $\pi: \cX \rightarrow \Delta$
be a quasi-projective open Kulikov degeneration, 
and let $F$
be a good effective divisor on $\cX$.
For every $k \in \ZZ_{\geq 1}$, there exists finitely many walls $(W,f_W)$ in the canonical scattering diagram $\mathfrak{D}_\cX$ with $f_W \neq 1 \mod J^k$. Moreover, for every such a wall, $f_W \mod J^k$ is a polynomial.
\end{lemma}

\begin{proof}
We prove the result by induction on $k$.
For $k=0$, the result follows from Lemma \ref{Lem: finiteness_0}. For the induction step, let us assume that the result holds for $k-1$, and show that it holds for $k$. For this, first note that every wall $(W,f_W)$ of $\mathfrak{D}_\cX$ with $f_W \neq 1 \mod J^k$
and $f_W =1 \mod J^{k-1}$ has an initial point $x \in P$. By consistency of $\mathfrak{D}_\cX$, either $x \in P_\ZZ$, or $x$ is an intersection point of walls which are non-trivial modulo $J^{k-1}$. By the induction hypothesis, there are finitely many walls non-trivial modulo $J^{k-1}$, and so finitely many possible points $x$ which are initial points for a wall non-trivial modulo $J^k$. Hence, it suffices to prove that for such a point $x$, there are finitely many walls $(W,f_W)$ with initial point $x$ such that $f_W \neq 1 \mod J^k$, and that for every such a wall, $f_W \mod J^k$ is a polynomial.

If $x$ is contained in the interior of a 2-dimensional face of $\sP$, then $\mathfrak{D}_\cX$ is given locally around $x$ by a consistent scattering diagram in $\mathbb{R}^2$. Moreover, all the walls passing through $x$ are trivial modulo $J$ by Lemma \ref{Lem: finiteness_0}, and so the walls non-trivial modulo $J^k$ are uniquely determined by consistency from the walls non-trivial modulo $J^{k-1}$, and so the finiteness and polynomiality of the walls non-trivial modulo $J^k$ follows from the induction hypothesis and \cite[Theorem 6.38]{Gross}.

If $x$ is contained in the interior of an edge of $\sP$, then $\mathfrak{D}_\cX$ is still given locally around $x$ by a consistent scattering diagram in $\mathbb{R}^2$. The only difference with the previous case is that some walls passing through $x$ might be trivial modulo $J$.
However, these walls are of the explicit form given by Lemma \ref{Lem: finiteness_0}, and so the result still holds by \cite[Proposition 6.47]{Gross}.

It remains to consider the case where $x=v$ is a vertex of $\sP$, corresponding to an irreducible component $X^v$ of $\cX_0$.
It follows from \cite[Theorem 8.15]{gross2023remarks} that  
the punctured Gromov--Witten invariants contributing to the walls $(W, f_W)$ with initial point $x=v$ and non-trivial modulo $J^k$ can be expressed explicitly in terms of punctured Gromov--Witten invariants $N_\tau^\beta$ of $(X^v, \partial X_v)$ and of the walls non-trivial modulo $J^{k-1}$ passing through $v$.
More precisely, it follows from \cite[Lemma 8.9]{gross2023remarks} that, up to applying corner blow-ups to $(X^v, \partial X^v)$, the invariants $N_\tau^\beta$ are counts of curves $C$ in $X^v$ intersecting $\partial X^v$ at fixed smooth points $x_1, \dots, x_q$, with multiplicities $m_1, \dots, m_q$, and at a smooth non-fixed point $x_{\mathrm{out}}$ with multiplicity $m_{\mathrm{out}}$.
If $q \neq 0$, then, as in the proof of \cite[Lemma 4.1]{GPS}, given the points $x_1$,..., $x_q$, there are only finitely many possibilities for the point $x_{\mathrm{out}}$, depending in a non-trivial way on $x_1$,..., $x_q$. In particular, for generic $x_1, \dots, x_q$, none of the intersection points of $C$ with $\partial X^v$ are intersection of an interior exceptional curve with $\partial X^v$, and so $C$ does not contain any interior exceptional curve. If $q=0$ and $x_{\mathrm{out}}$ is an intersection point of an interior exceptional curve $E$ with $\partial X^v$, then $\mathcal{O}(C-k E)$ has trivial restriction to $\partial X^v$. Since the invariant $N_\tau^\beta$ is invariant under deformation of $(X^v, \partial X^v)$, it follows from \cite[Proposition 4.1]{GHKmod} that $C=kE$ if  $N_\tau^\beta \neq 0$.

In all cases, we conclude that, if $N_\tau^\beta \neq 0$, then either $\beta$ is a multiple of an interior exceptional curve, or $\beta$ is represented by a curve in $X^v$ containing no interior exceptional curve and no component of $\partial X^v$. In the first case, the corresponding walls are given as in Lemma \ref{Lem: finiteness_0}, and so the finiteness and polynomiality is clear. In the second case, using the notations of Definition \ref{Def: good_divisor}, the class $\beta$ is uniquely determined by $g_\star \beta$, and $-F \cdot \beta = -g_\star F \cdot g_\star \beta$. For a wall non-trivial modulo $J^k$, we have $-F \cdot \beta \leq k$, and so $-g_\star F \cdot g_\star \beta \leq k$. Since $-g_\star F$ is $h$-ample by Definition \ref{Def: good_divisor} of a good divisor, there are finitely many such classes $\beta$. Combined with the induction hypothesis, this implies the finiteness and polynomiality of the walls non-trivial modulo $J^k$.
\end{proof}

We now prove a key finiteness result. Similar results have been previously established in different contexts in \cite[Corollary 6.11]{GHK1} and \cite[Proposition 15.10]{KY23}.

\begin{theorem} \label{thm_finiteness_new}
Let $\pi: \cX \rightarrow \Delta$
be a quasi-projective open Kulikov degeneration,
with tropicalization $(C(P),C(\sP))$, 
and let $F$
be a good effective divisor on $\cX$.
Then, for every $p, q, r \in C(P)_\ZZ$, there exists finitely many curve classes $\beta \in NE(\cX/
\cX^{\mathrm{can}})$ such that $N_{pqr}^\beta(\cX) \neq 0$.
\end{theorem}

\begin{proof}
First note that, by Lemma \ref{Balancing}, if $N_{pqr}^\beta(\cX) \neq 0$, then $-F \cdot 
\beta$ is fixed to be an integer $k$ depending only on $p,q,r$. As reviewed in 
\S \ref{Sec: canonical}, $N_{pqr}^\beta(\cX)$ can be computed in terms of broken lines in $C(P)$, or equivalently in terms of jagged paths in $P$. 
Moreover, these broken lines have fixed asymptotic directions determined by $p$ and $q$, and fixed final direction determined by $r$.
Therefore, the corresponding jagged paths have fixed initial points, given by the radial projections of $p$ and $q$ on $P$, and fixed endpoint, given by the radial projection of $r$ on $P$.

Let $\sP'$ be the coarsening of $\sP$ obtained by removing the edges $e$ such that the corresponding double curve $X^e$ of $\cX_0$ is contained in $G$. Such curves are distinct, and so $\sP'$ is just obtained by removing diagonals of some quadrilaterals of $\sP$. 
By Equation \eqref{Eq: wall_crossing}, each time a jagged path crosses an edge of $\sP'$, the corresponding kink is non-trivial modulo $J$, and so the intersection with $-F$ of the curve class attached to the jagged path increases by $1$. Since $-F \cdot \beta=k$ is fixed by $p,q,r$, one deduces that the number of times a jagged path contributing to  $N_{pqr}^\beta(\cX)$ crosses an edge of $\sP'$ is uniformly bounded. Moreover, if $(W, f_W)$ is a wall generically contained in a 2-dimensional face of $\sP'$, then $f_W=1 \mod J$ by Lemma \ref{Lem: finiteness_0}, and so, by Equation \eqref{Eq: wall_crossing}, the number of times a jagged path crosses such a wall is also uniformly bounded. Hence, we have in fact a uniform bound on the number of times a jagged path can cross a wall. In addition, by Lemma \ref{Lem: finiteness_0}, there are only finitely many walls non-trivial modulo $J^k$, and they all contain only finitely many terms modulo $J^k$. Since the jagged paths have fixed asymptotic monomials $z^p$ or $z^q$, we obtain at every wall-crossing finitely many monomials.
Since the number of wall-crossings is uniformly bounded, we conclude that there exists only finitely many possible jagged paths contributing to structures constants of the form $N_{pqr}^\beta(\cX)$ with fixed $p,q,r$, and so in particular only finitely many curve classes $\beta$ with $N_{pqr}^\beta(\cX) \neq 0$.
\end{proof}

We will prove the existence of good effective divisors using the following general result in birational geometry.

\begin{lemma}
\label{Lem: nef}
Let $\cX^{\mathrm{can}}$ be a quasi-projective variety with canonical singularities. 
Then, there exists $\cX \xrightarrow{g} \cZ \xrightarrow{h} \cX^{\mathrm{can}}$, where $g$ and $h$ are projective birational morphisms, 
such that the following conditions hold.
\begin{itemize}
\item[i)]$\mathcal{X}$ has $\QQ$-factorial terminal singularities. 
\item[ii)] $f:= h \circ g: \cX \rightarrow \cX^{\mathrm{can}}$ is crepant.
\item[iii)]$g: \cX \rightarrow \cZ$ is an isomorphism in codimension one.
\item[iv)] There exists an effective $f$-nef Cartier divisor $F$ on $\cX$, supported on the exceptional locus of $f$, such that $g_\star F$ is a Cartier divisor 
on $\cZ$ and $-g_\star F$ is $h$-ample.
\end{itemize}
\end{lemma}

\begin{proof}
First recall that by general MMP, \cite[Corollary 1.4.3]{BCHM10} -- see also 
\cite[Corollary 2.6]{Kawakita},
there exists a crepant $\QQ$-factorial terminalization of $\cX^{\mathrm{can}}$. Moreover, 
any two such crepant $\QQ$-factorial terminalizations are isomorphic in codimension one -- see e.g.\ \cite{Kaw2008}, and so have the same number $e(\cX^{\mathrm{can}})$ of exceptional divisors.

We prove the result by induction on $e(\cX^{\mathrm{can}})$. If $e(\cX^{\mathrm{can}})=0$, then,
for any crepant $\QQ$-factorial terminalization $f: \cX \rightarrow \cX^{\mathrm{can}}$, one can trivially take $\cZ=\cX^{\mathrm{can}}$, $g=f$, $h=\mathrm{Id}$, and $F=0$.
Assume by induction that the result holds for 
$e(\cX^{\mathrm{can}}) <n$ with $n \in \mathbb{Z}_{\geq 1}$, and consider $\cX^{\mathrm{can}}$ with $e(\cX^{\mathrm{can}})=n$.
Since $n \geq 1$, by \cite[Proposition 2.4]{Kawakita} -- see also \cite[Corollary 1.4.3]{BCHM10}, there
exists a crepant projective birational morphism $f': \cX' \rightarrow \cX^{\mathrm{can}}$, with $\cX'$ $\QQ$-factorial, such that $f'$ has exactly one exceptional divisor $E$, and $-E$ is $f'$-nef. 

Since $f'$ is crepant birational, and $-E$ is $f'$-nef, it follows from the base point free theorem that $-E$ is $f'$-semi-ample.
Thus, $f'$ factors as $\cX' \xrightarrow{g'} \cZ' \xrightarrow{h'} \cX^{\mathrm{can}}$, where $g'$ contracts exactly the irreducible curves $C$ with $E \cdot C =0$, and $\mathcal{O}_{\cX'}(-E)$ is the $g'$-pullback of a $h'$-ample line bundle on $\cZ'$. We claim that $g'$ is an isomorphism in codimension one.
Otherwise, it would contract the unique $f'$-exceptional divisor $E$, that is, $g_\star'(-E)=0$, which is effective.
However, since $E$ is $g'$-nef and $-E$ is not effective, 
this contradicts the negativity lemma \cite[Lemma 3.39]{KM}. 
Hence, $g'$ is an isomorphism in codimension one, $- g_\star' E$
is Cartier and $h'$-ample, and $e(\cZ')=e(\cX^{\mathrm{can}})-1$.

By the induction hypothesis applied to $\cZ'$, there exists a crepant $\QQ$-factorial terminalization $f''= h'' \circ g: \cX \rightarrow \cZ'$, where $g: \cX \rightarrow \cZ$ is an isomorphism in codimension one, $h'': \cZ \rightarrow \cZ'$, and there exists a $f''$-nef Cartier divisor $F''$ on $\cX$, supported on the exceptional locus of $f''$, such that $g_\star F''$ is a Cartier divisor on $\cZ$ and $-g_\star F''$ is $h''$-ample. To prove the result for $\cX^{\mathrm{can}}$, we set $f:= f'' \circ h': \cX \rightarrow \cX^{\mathrm{can}}$, 
and $h:= h' \circ h'': \cZ \rightarrow \cX^{\mathrm{can}}$.
As $-g_\star' E$ is $h'$-ample and $-g_\star F''$ is $h''$-ample, we have that $-g_\star F'' - \lambda g_\star' E$ is $h$-ample for $\lambda$ large enough, and it suffices to take $F:= 
(f'')^\star (-g_\star F'' - \lambda g_\star' E)$ to conclude the proof.

\end{proof}

\begin{lemma}
\label{lem_finiteness_2}
    Let $\pi \colon \cX \to \Delta$ be a quasi-projective open Kulikov degeneration. Then, there exists a quasi-projective open Kulikov degeneration $\pi' \colon \cX' \rightarrow \Delta$, with affinization $\cX^{\mathrm{can}}$, and admitting a good effective divisor.
\end{lemma}

\begin{proof}
The result follows immediately from Lemma \ref{Lem: nef}, and from the fact that, if $g: \cX \rightarrow \cZ$ is an isomorphism in codimension one, then the exceptional locus of $g$ is a configuration of curves as in Definition \ref{Def: good_divisor} i).
\end{proof}

\begin{lemma} \label{Lem_finiteness}
Let $\pi \colon \cX \rightarrow \Delta$ be a generic quasi-projective open Kulikov degeneration, with tropicalization $(C(P),C(\sP))$.
Then, for every $p,q, r \in C(P)_{\ZZ}$, there exists finitely many $\beta \in NE(\cX/\cX^{\mathrm{can}})$ such that $N_{pqr}^{\beta}(\mathcal{X}) \neq 0$.
\end{lemma}

\begin{proof}
By Lemma \ref{lem_finiteness_2}, there exists a quasi-projective open Kulikov degeneration
$\cX' \rightarrow \Delta$ with affinization $\cX^{\mathrm{can}}$, and 
with a good effective divisor as in Definition \ref{Def: good_divisor}.
In particular, it follows by Theorem \ref{thm_finiteness_new} that Lemma \ref{Lem_finiteness} holds for $\cX' \rightarrow \Delta$. Hence, Lemma \ref{Lem_finiteness} also holds for $\cX \rightarrow \Delta$ by Theorem 
\ref{thm_flop_general}.
\end{proof}

Denote by $h: C(P) \rightarrow \RR_{\geq 0}$ the height function, given by projection onto the second factor, coming from the definition of $C(P)$ as the cone over $P$, so that $P=h^{-1}(1)$.

\begin{lemma}
\label{Grading}
Let $\pi \colon \cX \rightarrow \Delta$ be a quasi-projective open Kulikov degeneration, with tropicalization $(C(P),C(\sP))$.
For every $p,q,r \in C(P)_\ZZ$ and $\beta \in NE(\cX/\cX^{\mathrm{can}})$ such that $N_{pqr}^\beta(\cX) \neq 0$, we have $h(r)=h(p)+h(q)$.
\end{lemma}

\begin{proof}
This follows from Lemma \ref{Balancing} applied to $F=\mathcal{X}_0$. Indeed, we have $\varphi_{\mathcal{X}_0}=h$ and $\mathcal{X}_0 \cdot \beta =0$ for all $\beta \in NE(\mathcal{X})$.
\end{proof}

\begin{theorem}
\label{Thm_Finiteness}
Let $\pi \colon \cX \rightarrow \Delta$ be a generic quasi-projective open Kulikov degeneration, with tropicalization $(C(P),C(\sP))$.
Then, for every $p,q \in C(P)_{\ZZ}$, there exists finitely many $r \in C(P)_{\ZZ}$ and $\beta \in NE(\cX/\cX^{\mathrm{can}})$ such that $N_{pqr}^{\beta}(\mathcal{X}) \neq 0$.
\end{theorem}

\begin{proof}
By Lemma \ref{Grading}, for every $p,q, r \in C(P)_\ZZ$ such that there exists $\beta \in NE(\mathcal{X})$ with $N_{pqr}^{\beta}(\mathcal{X}) \neq 0$, we have $h(r)=h(p)+h(q)$. 
Given $p, q \in C(P)_\ZZ$, the set of integral points $h^{-1}(h(p)+h(q)) \cap C(P)_\ZZ$ is finite, and so it follows that there exists only finitely many $r \in C(P)_\ZZ$ such that there exists $\beta \in NE(\mathcal{X})$ with $N_{pqr}^{\beta}(\mathcal{X}) \neq 0$.
On the other hand, by Lemma \ref{Lem_finiteness}, for every $p,q,r \in C(P)_\ZZ$, there are only finitely many $\beta \in NE(\mathcal{X})$ with $N_{pqr}^{\beta}(\mathcal{X}) \neq 0$, and this concludes the proof of Theorem \ref{Thm_Finiteness}.
\end{proof}

\subsection{Polarized mirror family to an open Kulikov degeneration}
\label{Sec: polarized_mirrors}
\subsubsection{Polarized mirror family}
Let $\pi: \cX \rightarrow \Delta$ be a generic quasi-projective open Kulikov degeneration with tropicalization $(C(P), C(\sP))$.
It follows from Theorem \ref{Thm_Finiteness} that the product of theta functions given by Equation \eqref{Cpqr} defines an algebra structure on the $\CC[NE(\cX/\cX^{\mathrm{can}})]$-module 
\[ \mathcal{R}_\cX = \bigoplus_{p \in C(P)_{\ZZ}} \CC[NE(\cX/\cX^{\mathrm{can}})]\,  \vartheta_p \,.\]
Moreover, by Lemma \ref{Grading}, $\mathcal{R}_\cX$ is a graded algebra for the grading defined by the height function $h: C(P)_\ZZ \rightarrow \ZZ_{\geq 0}$. We refer to $\mathcal{R}_\cX$ as the \emph{mirror algebra} of $\cX$.

\begin{definition}
\label{Def:polarized mirror}
Let $\pi: \cX \rightarrow \Delta$ be a quasi-projective 
open Kulikov degeneration. The \emph{polarized mirror family} $(\cY_\cX, \mathcal{L}_\cX)$ of $\cX$ is the family
\[ \cY_\cX :=\mathrm{Proj}\, \mathcal{R}_\cX \longrightarrow 
S_\cX:= \Spec\, \CC[NE(\cX/\cX^{\mathrm{can}})]\]
endowed with the sheaf $\mathcal{L}_\cX :=\mathcal{O}_{\cY_\cX}(1)$.
\end{definition}

\subsubsection{Torus action}
To prove properties of the polarized mirror family $(\cY_\cX,\cL_\cX)$, we will make use as in \cite{GHK1, HKY20, KY23} of a torus action, defined as follows.
Denote by $T_0$ the torus $\Spec\, \CC[\ZZ^{P_\ZZ}]$. The character lattice $\ZZ^{P_\ZZ}$ of $T_0$ can be naturally identified with the group of divisors of $\cX$ supported on $\cX_0$, since the points of $P_\ZZ$ are in a natural one-to-one correspondence with the irreducible components of $\cX_0$. In particular, for every divisor $F$ supported on $\cX_0$, we have a corresponding one-parameter subgroup $\CC^\star \subset T_0$. There is a natural action of $T_0$ on $S_\cX =\Spec\, \CC[NE(\cX/\cX^{\mathrm{can}})] $ such that, for every divisor $F$ supported on $\cX_0$, the corresponding one-parameter subgroup acts with weights $F \cdot \beta$ on the monomials $t^\beta$.

\begin{lemma} \label{Lem: torus_action}
Let $\pi: \cX \rightarrow \Delta$ be a quasi-projective open Kulikov degeneration. There exists a unique action of the torus $T_0=\Spec\, \CC[\ZZ^{P_\ZZ}]$ on the mirror family \[ \cY_\cX =\mathrm{Proj}\, \mathcal{R}_\cX \longrightarrow 
S_\cX= \Spec\, \CC[NE(\cX/\cX^{\mathrm{can}})]\,,\]
lifting the natural action on $S_\cX$, such that for every divisor $F$ on $\cX$ supported on $\cX_0$, the corresponding one-parameter subgroup acts on the monomial $t^\beta \vartheta_p$ with weight $\varphi_F(p)+F\cdot \beta$, where $\varphi_F$ is as in \S \ref{Sec: finiteness}.
\end{lemma}

\begin{proof}
This follows immediately from Lemma \ref{Balancing}-- see also \cite[Construction 1.24]{GS2019intrinsic} in the general context of the intrinsic mirror construction.
\end{proof}

\subsubsection{Central fiber, finite generation, and theta divisor}

We first describe the central fiber of the polarized mirror family. Then, using the action of the torus $T_0$, we prove the finite generation of the mirror algebra, and we show that the sheaf $\cL_\cX$ is actually a relatively ample line bundle on the mirror family $\cY_\cX \rightarrow S_\cX$. Finally, we define a natural \emph{theta divisor} $\mathcal{C}_\cX \in |\cL_\cX|$.

\begin{lemma} \label{lem_central_fiber}
Let $\pi \colon \cX \rightarrow \Delta$ be a generic quasi-projective open Kulikov degeneration with tropicalization $(C(P), C(\sP))$. 
Let $(Y_{0_\cX}, L_{0,\cX})$ be the central fiber of the polarized mirror $\cY_\cX \rightarrow S_\cX$ over the unique torus fixed point $0_\cX$ of the affine toric variety  $S_\cX$.
Then, the connected components $(Y_{0_\cX}^v, L_{0,\cX}^v)$ of the normalization of $(Y_{0_\cX}, L_{0,\cX})$ are in one-to-one correspondence with the integral points $v \in P_\ZZ$ and are all isomorphic to 
$(\PP^2,\cO_{\PP^2}(1))$.
Moreover, $(Y_{0_\cX}^v, L_{0,\cX}^v)$ are glued to form $(Y_{0_\cX}, L_{0,\cX})$ in such a way that the resulting intersection complex is given by $(P,\sP)$. 
\end{lemma}

\begin{proof}
This follows from \cite[\S 2.1]{GHS}, where the central fibers of mirror families given by the intrinsic mirror construction are described explicitly.
\end{proof}

\begin{lemma}
\label{Lem: mumtor}
    Let $\pi: \mathcal{X} \to \Delta$ be a generic quasi-projective open Kulikov degeneration and let $F$ be a good effective divisor on $\mathcal{X}$. Let $G$ be the face of $NE(\cX/\cX^{\mathrm{can}})$ spanned by the finitely many curves $C$ such that $F\cdot C = 0$, and $J=NE(\cX/\cX^{\mathrm{can}}) \setminus G$. Then, the restriction of the family $\mathcal{Y}_{\cX} \to \mathcal{S}_{\cX} $ to the toric stratum
    \[ \Spec \, \CC [G] \cong  \Spec \, \CC [NE(\cX/\cX^{\mathrm{can}})]/J \subset S_{\cX} \]
    is a Mumford toric degeneration with central fiber $Y_{0_\cX}$.
\end{lemma}

\begin{proof}
As shown in \S \ref{Sec: finiteness}, the wall-crossing transformations induced by the walls of the canonical scattering diagram are trivial modulo $J$, except for the contributions of the kinks corresponding to the finitely many double curves $C$ of $\cX_0$ with $F\cdot C=0$. The results follows by \cite[\S 2.1]{GHK1} -- see also \cite[\S 6.2.1]{Gross} and \cite[\S 2]{ArguzHearth}.
\end{proof}

For every face $H$ of $NE(\cX/\cX^{\mathrm{can}})$, the affine toric variety $\Spec\, \CC[H]$
is a toric stratum of the affine toric variety $\Spec\, \CC[NE(\cX/\cX^{\mathrm{can}})]$.
We denote by $T_H:= \Spec\, \CC[H^{\mathrm{gp}}]$ the corresponding open dense torus in 
$\Spec\, \CC[H]$. 

\begin{lemma} \label{Lem: finitely_generated_0}
Let $\pi: \cX \rightarrow \Delta$ be a generic quasi-projective open Kulikov degeneration. Then, for every face $H$ of $NE(\cX/\cX^{\mathrm{can}})$, the restriction $\cY_\cX|_{T_H} \rightarrow T_H$ of the mirror family to $T_H$ is a finite type morphism, that is, $R_\cX \otimes \CC[H^{\mathrm{gp}}]$ is a finitely generated $\CC[H^{\mathrm{gp}}]$-algebra.
\end{lemma}

\begin{proof}
Since $f: \cX \rightarrow \cX^{\mathrm{can}}$ is a relative Mori dream space by Proposition \ref{Prop: Mori dreams}, every face $H$ of $NE(\cX/\cX^{\mathrm{can}})$ corresponds to a
birational contraction 
$f_H: \cX \rightarrow \cX_H$ such that $H=NE(\cX/\cX_H)$.
By Theorem \ref{thm_flop_general}, the restriction $\cY_\cX|_{T_H} \rightarrow T_H$ of the mirror family to $T_H$ is independent of the choice of the crepant resolution $\cX \rightarrow \cX_H$. 
Moreover, by Lemma \ref{Lem: nef}, there exists a crepant resolution $\cX \rightarrow \cX_H$ s for which a good divisor $F$ exists on $\cX$ over $\cX_H$. 
To prove Lemma \ref{Lem: finitely_generated_0}, we will show that for such a particular crepant resolution,  $\mathcal{R}_\cX \otimes \CC[H]$ is a finitely generated $\CC[H]$-algebra, which in turn implies that $R_\cX \otimes \CC[H^{\mathrm{gp}}]$ is a finitely generated $\CC[H^{\mathrm{gp}}]$-algebra. 

Let $G$ be the face of $H$ spanned by the finitely many curves $C$ such that $F \cdot C=0$, and denote
by $J:=H \setminus G$ the corresponding monoid ideal. By Lemma \ref{Lem: mumtor}, $\mathcal{R}_\cX \otimes \CC[H]/J$ is a Stanley-Reisner algebra over $\CC[H]/J \simeq \CC[G]$, and is thus finitely generated over $\CC[H]/J$. In particular, there exists a finite set $\{\vartheta_p\}_{p \in I}$ of theta functions generating $\mathcal{R}_\cX \otimes \CC[H]/J$ as $\CC[H]/J$-algebra. We will show that 
$\{\vartheta_p\}_{p \in I}$ actually generate $\mathcal{R}_\cX \otimes \CC[H]$ as $\CC[H]$-algebra. 
Define $(\mathcal{R}_\cX \otimes \CC[H])'$ the $\CC[H]$-subalgebra of $\mathcal{R}_\cX \otimes \CC[H]$ generated by $\{\vartheta_p\}_{p \in I}$. We need to show that $(\mathcal{R}_\cX \otimes \CC[H])'=\mathcal{R}_\cX \otimes \CC[H]$. To do so, we follow the proof strategy outlined in \cite[Proposition 6.6]{GHK1}
(see also \cite[Lemma 16.6]{KY23} and \cite[Proposition 4.19]{lai2022mirror}).

By Lemma \ref{Lem: torus_action}, 
the divisor $-F$ induces an action of  $\CC^\star \subset T_0$ on $\mathcal{R}_\cX \otimes \CC[H]$, under which the monomials $t^\beta \vartheta_p$ with  $\beta \in H$ and $p \in C(P)_\ZZ$ are eigenvectors with weight $-\varphi_F(p) - F \cdot \beta \in \ZZ_{\geq 0}$. For every $w, h \in \ZZ_{\geq 0}$, denote by $V_{w,h}$ the set of monomials
$t^\beta \vartheta_p$ of weight $w$ with $p$ of height $h$, that is, with $p \in (hP)_\ZZ$. As the monomials $t^\beta \vartheta_p$ are a linear basis of $\mathcal{R}_\cX \otimes \CC[H]$, it suffices to show that, for every $w, h \in \ZZ_{\geq 0}$, we have $V_{w,h} \subset (\mathcal{R}_\cX \otimes \CC[H])'$.

Fix $w, h \in \ZZ_{\geq 0}$. We prove that every monomial
$t^\beta \vartheta_p$ in $V_{v,w}$ is actually in 
$(\mathcal{R}_\cX \otimes \CC[H])'$ by decreasing induction on $-F \cdot \beta$. First note that, as $(hP)_\ZZ$ is finite, there are finitely many possible values for $-\varphi_F(p)$.
Since $w=-\varphi_F(p)-F\cdot \beta$ is fixed, there are also finitely many possible values for $-F \cdot \beta$ for $t^\beta \vartheta_p$ in $V_{v,w}$. Hence, the result is immediate for large enough values of $-F\cdot \beta$.
For the induction step, as $\{\vartheta_q\}_{q \in C(P)_\ZZ}$ generates $\mathcal{R}_\cX \otimes \CC[H]/J$,  one can write $\vartheta_p =a+b$, with $a \in (\mathcal{R}_\cX \otimes \CC[H])'$ and $b \in J (\mathcal{R}_\cX \otimes \CC[H])$, and with both $a$ and $b$ homogeneous with respect to the $\CC^\star$-action and of height $h$. Hence, we have $t^\beta \vartheta_p =t^\beta a+t^\beta b$, with  $t^\beta a \in (\mathcal{R}_\cX \otimes \CC[H])'$, and $t^\beta b \in V_{w,h}$. 
Actually, as $b \in J (\mathcal{R}_\cX \otimes \CC[H])$, for every monomial $t^{\beta'} \vartheta_q$ in $b$, we have $-F \cdot \beta' \geq 1$. 
Thus, every monomial in $t^\beta b$ is of the form $t^{\beta+\beta'} \vartheta_q$ with $-F\cdot (\beta+\beta')>-F \cdot \beta$. 
By the induction hypothesis, this implies that $t^\beta b \in (\mathcal{R}_\cX \otimes \CC[H])'$, and consequently, $t^\beta \vartheta_p \in (\mathcal{R}_\cX \otimes \CC[H])'$, completing the proof.
\end{proof}

\begin{theorem} \label{The: finitely_generated}
Let $\pi: \cX \rightarrow \Delta$ be a generic quasi-projective open Kulikov degeneration. Then, the mirror algebra 
$\mathcal{R}_\cX$ is a finitely generated $\CC[NE(\cX/\cX^{\mathrm{can}})]$-algebra.
\end{theorem}

\begin{proof}
Since $NE(\cX/\cX^{\mathrm{can}})$ 
has finitely many faces, 
Lemma \ref{Lem: finitely_generated_0} ensures the existence of a finite set $\{\vartheta_p\}_{p\in I}$ such that, for every face $H$ of $NE(\cX/\cX^{\mathrm{can}})$, the set $\{\vartheta_p\}_{p\in I}$ generates $R_\cX \otimes \CC[H^{\mathrm{gp}}]$ as a $\CC[H^{\mathrm{gp}}]$-algebra. Denote by $\mathcal{R}_\cX'$ the subalgebra of $\mathcal{R}_\cX$ generated by $\{\vartheta_p\}_{p\in I}$. Both $\mathcal{R}_\cX$ and $\mathcal{R}_\cX'$ are graded by the height function $h: C(P)_\ZZ \rightarrow \ZZ_{\geq 0}$. Hence, denoting for every $h \in \ZZ_{\geq 0}$ by $\mathcal{R}_{\cX,h}$ and $\mathcal{R}_{\cX,h}'$ the subspaces of $\mathcal{R}_{\cX}$ and $\mathcal{R}_{\cX}'$ of degree $h$, we have $\mathcal{R}_{\cX,h}' \subset \mathcal{R}_{\cX,h}$.
Since $\mathcal{R}_{\cX,h}$ is a free finite rank $\CC[NE(\cX/\cX^{\mathrm{can}})]$-module, with basis $\{\vartheta_p\}_{p\in (hP)_\ZZ}$, the quotient  $\mathcal{R}_{\cX,h}/\mathcal{R}_{\cX,h}'$ forms a finitely generated $\CC[NE(\cX/\cX^{\mathrm{can}})]$-module. 
Moreover, for every face $H$ of $NE(\cX/\cX^{\mathrm{can}})$, we have $\mathcal{R}_{\cX,h}/\mathcal{R}_{\cX,h}' \otimes \CC[H^{\mathrm{gp}}]=0$. 
Since the tori $\Spec\, \CC[H^{\mathrm{gp}}]$ cover $\Spec\, \CC[NE(\cX/\cX^{\mathrm{can}})]$, 
Nakayama's lemma implies that $\mathcal{R}_{\cX,h}/\mathcal{R}_{\cX,h}'=0$. Consequently, $\mathcal{R}_{\cX,h}=\mathcal{R}_{\cX,h}'$ for all $h$, and thus $\mathcal{R}_{\cX}=\mathcal{R}_{\cX}'$ is finitely generated.
\end{proof}

\begin{theorem} \label{thm_line_bundle}
    Let $\pi: \cX \rightarrow \Delta$ be a generic quasi-projective open Kulikov degeneration. Then, the sheaf $\cL_\cX=\cO_\cX(1)$ on the mirror family $\nu_\cX: \cY_\cX \rightarrow S_\cX$ is a relatively ample line bundle, with an isomorphism of graded rings
    \[ \bigoplus_{k \geq 0} H^0(S_\cX, \nu_{\cX,\star}
    \cL_\cX^{\otimes k}) \simeq \mathcal{R}_\cX\,.\]
\end{theorem}

\begin{proof}
The result is true in restriction to the central fiber over $0_\cX \in S_\cX$ by Lemma \ref{lem_central_fiber}. The general result over all of $S_\cX$ follows using the action of the torus $T_0$ as in the proof of Lemma \ref{Lem: finitely_generated_0}.
\end{proof}

Using Theorem \ref{thm_line_bundle}, we define the theta divisor $\mathcal{C}_\cX \in |\cL_\cX|$ as follows. 

\begin{definition} \label{Def: theta_divisor}
Let $\pi: \cX \rightarrow \Delta$ be a generic quasi-projective open Kulikov degeneration with tropicalization $(C(P), C(\sP))$.
The \emph{theta divisor} in the linear system of the relatively ample line bundle $\mathcal{L}_{\cX}$ on the mirror family $\cY_\cX \rightarrow S_\cX$  is defined as
    \begin{equation} \label{Eq: divisor_C}
\mathcal{C}_{\cX}:= \left\{ \sum_{p \in P_\ZZ} \vartheta_p=0 \right\} \in |\mathcal{L}_{\cX}| \,.
\end{equation}
\end{definition}

\subsubsection{The divisor $\mathcal{D}_{\cX}$}
\label{Sec: support}

We first define a divisor $\cD_\cX$ in the mirror family $\cY_\cX \rightarrow S_\cX$. Then, we prove in Theorem 
\ref{Thm: fmat family}
that $(\cY_\cX, \cD_\cX) \rightarrow S_\cX$ is a family of semi-log-canonical surfaces with trivial log-canonical divisor.

Let $\cX \rightarrow \Delta$ be a generic quasi-projective open Kulikov degeneration with tropicalization $(C(P), C(\sP))$.
Recall from
\S\ref{Sec: affinization_open_Kulikov_surface}
that the set $(\partial P)_\ZZ$ of integral point in the boundary $\partial P$ of $P$ is the union of the set $P_\ZZ(0)$ of integral points of $P$ corresponding to  $0$-surfaces and of the set $P_\ZZ(1)$ of integral points of $P$ corresponding to $1$-surfaces, and that we denote by $(v_i)_{i \in \ZZ/N \ZZ}$ the cyclically ordered integral points in $P_\ZZ(0)$. In particular, we have a decomposition of $\partial P$ into intervals:
\[\partial P=\bigcup_{i \in \ZZ/N\ZZ} [v_i,v_{i+1}]\,,\] and we have a corresponding decomposition of the boundary $\partial C(P)$ of the cone $C(P)$ into the cones $C([v_i,v_{i+1}])$ over the intervals $[v_i,v_{i+1}]$:
\[\partial C(P)=\bigcup_{i \in \ZZ/N\ZZ} C([v_i,v_{i+1}])\,.\]
For every $v \in P_\ZZ(1)$, the corresponding $1$-surface $X^v$ has, by Definition
\ref{def_1_surface}, a natural map to $\mathbb{A}^1$, which is generically a $\PP^1$-fibration, and we denote by $[F_v] \in NE(\cX/\cX^{\mathrm{can}})$ the class of the fibers.
For every $i \in \ZZ/N\ZZ$, denote $v_{i,1}, \cdots, v_{i,m_i}$ the ordered integral points in $P_\ZZ(1)$ contained in 
$[v_i, v_{i+1}]$, so that we have the interval decomposition $[v_i, v_{i+1}]=[v_i, v_{i,1}] \cup \cdots \cup [v_{i,m_i}, v_{i+1}]$, and the cone decomposition 
\begin{equation} \label{Eq: cone_decomposition}
C([v_i, v_{i+1}])=C([v_i, v_{i,1}]) \cup \cdots \cup C([v_{i,m_i}, v_{i+1}])\,. 
\end{equation}
Let $\varphi_i$ be a $N_1(\cX/\cX^{\mathrm{can}}) \otimes \RR$-valued 
piecewise linear function on $C([v_i, v_{i+1}])$, linear on each cone of the decomposition \eqref{Eq: cone_decomposition}, with kink $[F_{v_{i,j}}]$ across the ray $\RR_{\geq 0} v_{i,j}$.

\begin{lemma}
\label{Lem: support}
Let $\pi: \cX \rightarrow \Delta$ be a generic quasi-projective open Kulikov degeneration, with tropicalization $(C(P), C(\sP))$, and let $r \in \partial C (P)_\ZZ$. Then, for every $p,q \in C(P)_\ZZ$ and $\beta \in NE(\mathcal{X})$, we have either
\begin{itemize}
\item[i)]
$N_{pqr}^{\beta}(\mathcal{X})=1$ if there exists $i \in \ZZ/N\ZZ$ such that $p,q,r \in C([v_i, v_{i+1}])$, $r=p+q$ inside $C([v_i,v_{i+1}])$, and $\beta=\varphi_i(p)+\varphi_i(q)-\varphi_i(r)$, or
\item[ii)] 
$N_{pqr}^\beta(\cX)=0$ elsewise.
\end{itemize}
\end{lemma}

\begin{proof}
If $p=0$ or $q=0$, the result follows from \cite[Theorem 4.3]{GS2019intrinsic}. If $r=0$, then we have $N_{pqr}^{\beta}(\mathcal{X})=0$ unless $p=q=0$ since the mirror algebra is graded by the height function $h: C(P)_\ZZ \rightarrow \ZZ_{\geq 0}$. 
Thus, we may now assume that $p, q, r$ are all non-zero. Denote by $Z_p$, $Z_q$, and $Z_r$ the strata of $\cX_0$ corresponding to the smallest cones of $C(P)$ containing $p$, $q$, and $r$ respectively.
By Definition \ref{Def: structure constants}, the structure constant $N_{pqr}^\beta(\cX)$ can be 
computed using a moduli space of punctured maps with two marked points mapping to $Z_p$, $Z_q$, and one punctured point mapping to a general point of $Z_r$. The assumption that $r \in \partial C(P)_\ZZ$ implies that $Z_r$ is either a 0-surface or a 1-surface. 
By Definitions \ref{def_0_surface}-\ref{def_1_surface}, there are no compact curves passing through a general point of a 0-surface, and the only compact curve passing through a general point of a 1-surface is a $\PP^1$-fiber of the
corresponding generically $\PP^1$-fibration. Consequently, we have $N_{pqr}^\beta(\cX)=0$, unless $p,q,r$ all belong to a common cone $C([v_i, v_{i+1}])$. 

If $p,q,r$ are all contained within a common cone $C([v_i, v_{i+1}])$, then the contributing punctured maps have image contained in a chain of $\PP^1$-fibers of intersecting $1$-surfaces, and so the invariants $N_{pqr}^\beta(\cX)$ are equal to the corresponding invariants for a toric Calabi--Yau 3-fold whose fan contains $C([v_i, v_{i+1}])$ as a boundary face. 
By \cite{GScanonical}, the invariants $N_{pqr}^\beta(\cX)$ can be computed in terms of broken lines in the canonical scattering diagram. 
In the toric case, the broken lines do not bend and solely receive contributions from the kinks of the piecewise linear function $\varphi_i$ as they cross the rays $\RR_{\geq 0} v_{i,j}$ -- see \cite[\S 2]{ArguzHearth}. This confirms the statement of Lemma \ref{Lem: support} i).  
\end{proof}

\begin{remark}
    Any two choices of piecewise linear
    functions $\varphi_i$ differ by a linear function. Consequently, the curve class $\varphi_i(p)+\varphi_i(q)-\varphi_i(r)$ appearing in Lemma \ref{Lem: support} is independent of the particular choice of $\varphi_i$.
\end{remark}

It follows from Lemma \ref{Lem: support}
that the $\CC[NE(\cX/\cX^{\mathrm{can}})]$-submodule 
\begin{equation} \label{Eq: ideal}
I_\cX := \bigoplus_{p \notin \partial C(P)_\ZZ} \CC[NE(\cX/\cX^{\mathrm{can}})] \,\vartheta_p \subset \mathcal{R}_\cX
\end{equation}
is a graded ideal of $\mathcal{R}_\cX$. We denote by 
\begin{equation} \label{Eq: divisor_D} 
\mathcal{D}_\cX := \mathrm{Proj}\, (\mathcal{R}_\cX/I_\cX) \subset \cY_\cX = \mathrm{Proj}\, \mathcal{R}_\cX
\end{equation}
the corresponding closed subscheme of $\cY_\cX$ over $S_\cX$.

\begin{lemma} 
\label{Lem: boundary_divisor}
Let $\pi \colon \cX \rightarrow \Delta$ be a generic quasi-projective open Kulikov degeneration, with tropicalization $(C(P), C(\sP))$. Then, $\mathcal{D}_\cX \rightarrow S_\cX$ has toric irreducible components and is a flat family of divisors in the fibers of $\cY_\cX \rightarrow S_\cX$. Moreover, the general fiber of $\mathcal{D}_\cX \rightarrow S_\cX$ is a cycle of smooth rational curves $D_{i,i+1}$ in one-to-one correspondence with the intervals $[v_i, v_{i+1}] \subset \partial P$, and the central fiber $D_{0_\cX}$ over $0_\cX \in S_X$ is a cycle of smooth rational curves $D_v$ in one-to-one correspondence with the boundary integral points $v \in P_\ZZ$. In addition, the theta functions $\vartheta_p$ with $p \in \partial C(P)_\ZZ$ naturally restrict to toric monomials on the irreducible components of $\mathcal{D}_\cX$.
\end{lemma}

\begin{proof}
It follows from Lemma \ref{Lem: support} that the irreducible components of \[\mathcal{D}_\cX=\mathrm{Proj}\, \bigoplus_{p \in \partial C(P)_\ZZ} NE(\cX/\cX^{\mathrm{can}}) \, \vartheta_p
\] are \[\mathcal{D}_{\cX, i,i+1}=\mathrm{Proj}\, \bigoplus_{p \in \partial C(P)_\ZZ \cap [v_i, v_{i+1}]} 
NE(\cX/\cX^{\mathrm{can}}) \, \vartheta_p\,,\]
for $i \in \ZZ/N\ZZ$, and that $\mathcal{D}_{\cX,i,i+1} \rightarrow S_\cX$ is the  Mumford toric degeneration defined by the polyhedral decomposition of $C([v_i, v_{i+1}])$ given in Equation 
\eqref{Eq: cone_decomposition}, and the piecewise linear function $\varphi_i$. In particular, the general fiber of $\mathcal{D}_{\cX,i,i+1} \rightarrow S_\cX$ is $\PP^1$, and its central fiber over $0_\cX$ is a chain of $\PP^1$-components indexed by the integral points $v \in [v_i, v_{i+1}] \cap P_\ZZ$. Hence the result follows.
\end{proof}

\begin{theorem}
\label{Thm: fmat family}
Let $\pi: \cX \rightarrow \Delta$ be a 
quasi-projective open Kulikov degeneration. Then, the mirror family
$( \mathcal{Y}_{\cX}, \mathcal{D}_{\cX} ) \rightarrow S_\cX$
is a projective flat family of semi-log-canonical surfaces $(Y_t, D_t)$ such that $K_{Y_t} + D_t=~0$.
\end{theorem}

\begin{proof}
First note that $( \mathcal{Y}_{\cX}, \mathcal{D}_{\cX} ) \rightarrow S_\cX$ is flat since the algebra $\mathcal{R}_\cX$ of theta functions is a free $NE(\cX/\cX^{\mathrm{can}})$-module. On the other hand, by Theorem \ref{The: finitely_generated}, the algebra $\mathcal{R}_\cX$ is finitely generated, and so the morphism $\cY_\cX=\mathrm{Proj} \,\mathcal{R}_\cX  \rightarrow S_\cX$ is projective.

To show that the fiber $(Y_t,D_t)$ over $t \in S_\cX$ is semi-log-canonical and $K_{Y_t}+D_t=0$, consider the smallest toric stratum $\Spec\, \CC[H]$ of $S_\cX$
containing $t$, corresponding to a face $H$ of $NE(\cX/\cX^{\mathrm{can}})$. In particular, $t$ is contained in the dense torus orbit $T_H := \Spec \, \CC[H^{\mathrm{gp}}]$ in $\Spec\, \CC[H]$.
As in the proof of Lemma \ref{Lem: finitely_generated_0}, there is a corresponding birational contraction 
$f_H: \cX \rightarrow \cX_H$ such that $H=NE(\cX/\cX_H)$, and the restriction $\cY_\cX|_{T_H}$ is independent of the choice of the crepant resolution $\cX \rightarrow \cX_H$.
Hence, by Lemma \ref{Lem: nef}, we can assume that there exists a good divisor $F$ on $\cX$ over $\cX_H$. Let $G$ be the face of $H$ spanned by the curves $C$ with $F \cdot C=0$, and denote $J:=H \setminus G$. Consider the one-parameter subgroup $\CC^\star \subset T_0$ associated with $F$ as in \S \ref{sec: torus}. 
Consider the corresponding orbit $\CC^\star \subset T_H$ passing through $t$. Its closure in $\Spec\, \CC[H]$ is an affine line $\mathbb{A}^1$ with $0 \in \Spec\, \CC[G] \simeq \Spec\, \CC[H]/J$. By
Lemma \ref{Lem: torus_action}, the restriction of $(\cY_\cX, \mathcal{D}_\cX)$ to this line is a $\CC^\star$-equivariant family. 
Hence, by the proof of \cite[Lemma 8.33]{GHKK}, it suffices to show Theorem \ref{Thm: fmat family} for the central fiber of this family over $0 \in \Spec\, \CC[G] \simeq \Spec\, \CC[H]/J$. By Lemma \ref{Lem: mumtor}, this central fiber is a fiber of a Mumford toric degeneration. Hence, the result follows from \cite[Theorem 1.2.14]{Valeryannals}.
\end{proof}

\subsection{Polarized mirror over the movable secondary fan}
\label{section_movable}
Let $\cX \rightarrow \Delta$ be a generic quasi-projective open Kulikov degeneration. We constructed in \S\ref{Sec: polarized_mirrors} a mirror family $(\cY_\cX, \cD_\cX +\epsilon \mathcal{C}_\cX) \rightarrow S_\cX$, where $0< \epsilon <\!\!<1$. In this section, we glue together the mirror families corresponding to the various crepant resolutions of $\cX^{\mathrm{can}}$.

By Proposition \ref{Prop: Mori dreams}, the contraction $\cX \rightarrow \cX^{\mathrm{can}}$ is a relative Mori dream space. 
Hence, the corresponding Mori fan $MF(\cX/\cX^{\mathrm{can}})$ is a complete fan in $\Pic(\cX/\cX^{\mathrm{can}})_\RR$, whose maximal dimensional cones are indexed by the finitely many $\QQ$-factorial birational models $f':\cX \dashrightarrow \cX'$ 
of $\cX$ over $\cX^{\mathrm{can}}$, and are given by
$\mathrm{Nef}(\cX'/\cX^{\mathrm{can}})+\langle \mathrm{Ex}(f') \rangle$, where $\mathrm{Nef}(\cX'/\cX^{\mathrm{can}})$ is the relative nef cone of $\cX'$ over $\cX^{\mathrm{can}}$, and 
$\langle \mathrm{Ex}(f') \rangle$ is the non-negative span of the $f'$-exceptional divisors.
If $\langle \mathrm{Ex}(f') \rangle \neq \{0\}$, 
that is, if at least one divisor is contracted, the corresponding cone in $MF(\cX/\cX^{\mathrm{can}})$ is referred to as \emph{bogus}. In other words, a cone is non-bogus if $\cX \dashrightarrow \cX'$ is small, that is, by Theorem \ref{Thm: exceptional_curves}, if $\cX' \rightarrow \cX^{\mathrm{can}}$
is a projective resolution, and $\cX' \rightarrow \Delta$ is an open Kulikov degeneration.
The union of non-bogus cones is the movable cone $\mathrm{Mov}(\cX/\cX^{\mathrm{can}})$.

Following \cite[\S 2.2]{HKY20}, we define the \emph{movable secondary fan} $\mathrm{MovSec}(\cX/\cX^{\mathrm{can}})$ as a coarsening of the restriction $\mathrm{MovF}(\cX/\cX^{\mathrm{can}})$ of the Mori fan to the movable cone $\mathrm{Mov}(\cX/\cX^{\mathrm{can}})$.
Two maximal cones $\mathrm{Nef}(\cX'/\cX^{\mathrm{can}})$ and $\mathrm{Nef}(\cX''/\cX^{\mathrm{can}})$ of 
$\mathrm{MovF}(\cX/\cX^{\mathrm{can}})$, related
by a flop $b: \cX' \dashrightarrow \cX''$
are called \emph{equivalent} if the corresponding piecewise linear functions $\varphi'' - \nu_b \circ \varphi'$ in \S \ref{Sec: constants_flops} is equal to zero.

\begin{lemma} \label{Lem: movable_secondary}
Let $\pi: \cX \rightarrow \Delta$ be a quasi-projective open Kulikov degeneration.
There exists a unique fan $\mathrm{MovSec}(\cX/\cX^{\mathrm{can}})$ with support the movable cone $\mathrm{Mov}(\cX/\cX^{\mathrm{can}})$, whose maximal cones are the unions of equivalent maximal cones of $\mathrm{MovF}(\cX/\cX^{\mathrm{can}})$.
\end{lemma}

\begin{proof}
This follows from the proof of \cite[Theorem 2.12]{HKY20}.
\end{proof}

\begin{lemma} \label{Lem: M1_flops}
Let $\pi: \cX \rightarrow \Delta$ be a quasi-projective open Kulikov degeneration. Two maximal cones $\mathrm{Nef}(\cX'/\cX^{\mathrm{can}})$ and $\mathrm{Nef}(\cX''/\cX^{\mathrm{can}})$ of $\mathrm{Mov}(\cX/\cX^{\mathrm{can}})$ are \emph{equivalent} if and only if $\cX'$ and $\cX''$ are connected by a sequence of multiple M1 flops.
\end{lemma}

\begin{proof}
If $\cX'$ and $\cX''$ are connected by a sequence of multiples M1 flops, then the maximal cones $\mathrm{Nef}(\cX'/\cX^{\mathrm{can}})$ and $\mathrm{Nef}(\cX''/\cX^{\mathrm{can}})$ are equivalent by Lemma \ref{Lem_M1_flop}.

Conversely, if $\mathrm{Nef}(\cX'/\cX^{\mathrm{can}})$ and $\mathrm{Nef}(\cX''/\cX^{\mathrm{can}})$ are equivalent, then, by Lemma \ref{Lem: movable_secondary}, $\mathrm{Nef}(\cX'/\cX^{\mathrm{can}})$ and $\mathrm{Nef}(\cX''/\cX^{\mathrm{can}})$ are both contained in a maximal cone $C$ of $\mathrm{MovSec}(\cX/\cX^{\mathrm{can}})$. As $C$ is convex, there exists a path contained in $C$, starting in $\mathrm{Nef}(\cX'/\cX^{\mathrm{can}})$, ending $\mathrm{Nef}(\cX''/\cX^{\mathrm{can}})$, and only intersecting cones of $MF(\cX/\cX^{\mathrm{can}})$ of codimension zero or one. Therefore, it suffices to prove the result when $\mathrm{Nef}(\cX'/\cX^{\mathrm{can}})$ and $\mathrm{Nef}(\cX''/\cX^{\mathrm{can}})$ share a common codimension one face, that is, when $\cX'$ and $\cX''$ are related by a flop of relative Picard rank one, which by Theorem \ref{Thm: exceptional_curves} is either a multiple M1 or a multiple M2 flop. By Lemma \ref{Lem_M2_flop}, $\cX'$ and $\cX''$ would not be equivalent if they were related by a multiple M2 flop, and so we conclude that $\cX'$ and $\cX''$ are related by a multiple M1 flop.
\end{proof}

\begin{theorem} \label{thm_glue_movable}
Let $\pi: \cX \rightarrow \Delta$ be a generic quasi-projective open Kulikov degeneration. Then, the mirror families $(\cY_{\cX'}, \cD_{\cX'}+\epsilon\, \mathcal{C}_{\cX'}) \rightarrow S_{\cX'}$ and the relatively ample line bundles $\cL_{\cX'}$ corresponding to the projective crepant resolutions $\cX' \rightarrow \cX^{\mathrm{can}}$ uniquely glue into a family $(\cY_{\cX}^{\mathrm{msec}}, \cD_{\cX}^{\mathrm{msec}}+\epsilon\, \mathcal{C}_{\cX}^{\mathrm{msec}}) \rightarrow S^{\mathrm{\mathrm{msec}}}_{\cX}$
and a relatively ample line bundle $\cL_{\cX}^{\mathrm{msec}}$ over the toric variety $S^{\mathrm{\mathrm{msec}}}_{\cX}$ with fan given by the movable secondary fan $\mathrm{MovSec}(\cX/\cX^{\mathrm{can}})$.
\end{theorem}

\begin{proof}
Let $C$ be a maximal cone of $\mathrm{MovSec}(\cX/\cX^{\mathrm{can}})$.
By Lemma \ref{Lem: M1_flops} and Theorem \ref{thm_flop_general}, 
all mirror algebras $\mathcal{R}_{\cX'}$ with $\mathrm{Nef}(\cX'/\cX^{\mathrm{can}}) \subset C$ have the same structure constants $N_{pqr}^\beta(\cX')$. In particular, these structure constants are only non-zero if $\beta \in NE(\cX'/\cX^{\mathrm{can}})$ for every $\cX'$ with $\mathrm{Nef}(\cX'/\cX^{\mathrm{can}}) \subset C$, that is if $\beta \in C^\vee_\ZZ$, where $C^\vee$ is the dual cone of $C$.
Hence, the mirror families $(\cY_{\cX'}, \cD_{\cX'}+\epsilon\, \mathcal{C}_{\cX'}) \rightarrow S_{\cX'}$ and line bundles $\cL_{\cX'}$ glue into a mirror family $(\cY_C, \cD_C+\epsilon\, \mathcal{C}_C) \rightarrow \Spec\, \CC[C^\vee_\ZZ]$ and a line bundle $\cL_{C}$, where $\cY_C := \mathrm{Proj} \, \mathcal{R}_C$ and
\[ \mathcal{R}_C := \bigoplus_{p \in C(P)_\ZZ} \CC[C^\vee_\ZZ] \vartheta_p \,.\]

If $C$ and $C'$ are two maximal cones of $\mathrm{MovSec}(\cX/\cX^{\mathrm{can}})$,
with corresponding sections $\varphi_C$ and $\varphi_{C'}$ as in 
\eqref{eq_varphi}, and related by a flop $b$, then it follows from Theorem \ref{thm_flop_general} that the map 
\begin{align} \label{Eq: Psi}
\psi_{C,C'} : \mathcal{R}_C \otimes_{\CC[C^\vee_\ZZ]} \CC[(C \cap C')^\vee_\ZZ]
&\longrightarrow \mathcal{R}_{C'}\otimes_{\CC[(C')^\vee_\ZZ]}
\CC[(C \cap C')^\vee_\ZZ]\,, \\ \nonumber
 \vartheta_p 
 &\longmapsto t^{\nu_b(\varphi_{C}(p))-\varphi_{C'}(p)} \vartheta_p \,,
\end{align}
is an isomorphism of $\CC[(C \cap C')^\vee_\ZZ]$-algebras. If $C$, $C'$, and $C''$ are three maximal cones of $\mathrm{MovSec}(\cX/\cX^{\mathrm{can}})$, then the cocycle condition $\psi_{C,C''}=\psi_{C,C'} \circ \psi_{C',C''}$ holds. 
Therefore, using the isomorphisms $\psi_{C,C'}$, one can glue the polarized families $(\cY_C, \cD_C, \mathcal{L}_C) \rightarrow \Spec\, \CC[C^\vee_\ZZ]$
into a polarized family 
\begin{equation*} 
\nu_{\mathrm{msec}}: (\cY_{\cX}^{\mathrm{msec}}, \cD_{\cX}^{\mathrm{msec}} \mathcal{L}_{\cX}^{\mathrm{msec}})
\longrightarrow S^{\mathrm{\mathrm{msec}}}_{\cX}
\end{equation*}
over the toric variety $S^{\mathrm{\mathrm{msec}}}_{\cX}$ defined by the movable secondary fan $\mathrm{MovSec}(\cX/\cX^{\mathrm{can}})$.
By Equation \eqref{Eq: Psi}, 
for every $p \in (hP)_\ZZ$, $\vartheta_p$ naturally becomes a section of the line bundle $(\mathcal{L}_{\cX}^{\mathrm{msec}})^{\otimes h} \otimes \nu_{\mathrm{msec}}^\star L_p$ on $\cY_{\cX}^{\mathrm{msec}}$, where $L_p$ is the line bundle on $S^{\mathrm{\mathrm{msec}}}_{\cX}$ defined by the 1-cocycle $t^{\nu_b(\varphi_C(p))-\varphi_{C'}(p)}$.
By Lemma \ref{Lem_M2_flop}, we have always $\nu_b(\varphi_C(p)) -\varphi_{C'}(p)=0$ for $p \in P_\ZZ$, and so the line bundle $L_p$ is trivial for every $p \in P_\ZZ$.
In particular, all the theta functions $\vartheta_p$ with $p \in P_\ZZ$
are sections of $\mathcal{L}_{\cX}^{\mathrm{msec}}$, and so the theta divisors $\mathcal{C}_C$ glue into a theta divisor
    \begin{equation*} 
\mathcal{C}_{\cX}^{\mathrm{msec}}:= \left\{ \sum_{p \in P_\ZZ} \vartheta_p=0 \right\} \in |\mathcal{L}_{\cX}^{\mathrm{msec}}| \,.
\end{equation*}
\end{proof}

\subsection{Torus action and extension over the bogus cones}
\label{sec: torus}

In this section, we describe how to canonically extend the family $(\cY_{\cX}^{\mathrm{msec}},\mathcal{D}^{\mathrm{msec}}_{\cX}+\epsilon\, \mathcal{C}_{\cX}^{\mathrm{msec}}) \rightarrow S_\cX^{\mathrm{msec}}$ given by Theorem \ref{thm_glue_movable}, to a family over the toric stack with stacky fan given by the stacky ``secondary fan'' defined below following \cite{HKY20}.

First, recall that the ``bogus cones'' of the Mori fan $MF(\mathcal{X}/\cX^{\mathrm{can}})$, as discussed at the beginning of \S\ref{section_movable}, are of the form $\mathrm{Nef}(\mathcal{X}'/\cX^{\mathrm{can}})+ \langle \mathrm{Ex}(f') \rangle$, for a birational map $f': \mathcal{X} \dashrightarrow \mathcal{X}'$, where $\mathcal{X}'$ is $\QQ$-factorial. We call two bogus cones $\mathrm{Nef}(\mathcal{X}'/\cX^{\mathrm{can}})+ \langle \mathrm{Ex}(f') \rangle$ and $\mathrm{Nef}(\mathcal{X}''/\cX^{\mathrm{can}})+ \langle \mathrm{Ex}(f'') \rangle$ \emph{equivalent} if the smallest cone of $\mathrm{MovSec}(\cX/\cX^{\mathrm{can}})$ containing $\mathrm{Nef}(\cX'/\cX^{\mathrm{can}})$ equals the smallest cone of $\mathrm{MovSec}(\cX/\cX^{\mathrm{can}})$ containing $\mathrm{Nef}(\cX''/\cX^{\mathrm{can}})$.

\begin{definition} \label{Def: secondary_fan}
The \emph{secondary fan} of $\cX/\cX^{\mathrm{can}}$, denoted by $\mathrm{Sec}(\cX/\cX^{\mathrm{can}})$ is the complete fan in $\mathrm{Pic}(\cX/\cX^{\mathrm{can}})_{\RR}$, obtained by adding to the movable secondary fan $\mathrm{MovSec}(\cX/\cX^{\mathrm{can}})$ the maximal cones given by the unions of equivalent bogus cones of the Mori fan $MF(\cX/\cX^{\mathrm{can}})$. 
\end{definition}

For every bogus cone $\sigma = \mathrm{Nef}(\mathcal{X}'/\cX^{\mathrm{can}})+ \langle \mathrm{Ex}(f') \rangle$, let $M_\sigma$ be the sublattice of $\mathrm{Span}(\sigma) \cap \Pic(\cX/\cX^{\mathrm{can}})$
spanned by $\mathrm{Nef}(\mathcal{X}'/\cX^{\mathrm{can}}) \cap \Pic(\cX/\cX^{\mathrm{can}})$ and by the classes of irreducible components of $\cX_0$ contracted by $\cX \dashrightarrow \cX'$. 
Equivalent bogus cones define the same sublattices, 
and so there exists a well-defined sublattice $M_\sigma$ for every cone $\sigma$ of the secondary fan. The data of the sublattices $M_\sigma$ is a toric stacky data as in 
\cite[Definition 4.1]{tyomkin}, and so
defines a stacky fan $\mathcal{S}ec(\cX/\cX^{\mathrm{can}})$ with underlying fan $\mathrm{Sec}(\cX/\cX^{\mathrm{can}})$.
We denote from now on by $S^{\mathrm{sec}}_{\cX}$ the proper toric variety whose fan is the secondary fan $\mathrm{Sec}(\cX/\cX^{\mathrm{can}})$, and by  $\mathcal{S}^{\mathrm{sec}}_{\cX}$ the proper toric Deligne--Mumford stack with stacky fan $\mathcal{S}ec(\cX/\cX^{\mathrm{can}})$ -- see \cite[\S 4.1]{tyomkin} and \cite{toric_stacks} for the general theory of stacky fans and toric stacks.
In what follows, we show that the family $(\cY_{\cX}^{\mathrm{msec}}, \mathcal{D}^{\mathrm{msec}}_{\cX}+\epsilon\, \mathcal{C}_{\cX}^{\mathrm{msec}})$ extends to a family $(\mathcal{Y}^{\mathrm{sec}}_{\cX}, \mathcal{D}^{\mathrm{sec}}_{\cX}+\epsilon\, \mathcal{C}^{\mathrm{sec}}_{\cX})$ over $\mathcal{S}^{sec}_{\cX}$.

The action of the torus $T_0$ given by Lemma \ref{Lem: torus_action} on the affine toric varieties $S_{\mathcal{X}'}$ corresponding to the projective crepant resolutions $\mathcal{X}' \rightarrow \cX^{\mathrm{can}}$ naturally induces an action of $T_0$ on the toric varieties $S^{\mathrm{\mathrm{msec}}}_{\cX}$ and $S^{\mathrm{sec}}_{\cX}$, and on the toric stack $\mathcal{S}^{\mathrm{sec}}_{\cX}$. It follows from the Equation \eqref{Eq: Psi} that the actions of $T_0$ on the families $\mathcal{Y}_{\cX'} \to S_{\cX'}$ given by Lemma \ref{Lem: torus_action} glue into an action of $T_0$ on $\cY_{\cX}^{\mathrm{msec}} \rightarrow S_{\cX}^{\mathrm{msec}} $.

\begin{theorem} \label{Thm: ext_bogus}
Let $\pi:\cX \rightarrow \Delta$ be a generic quasi-projective open Kulikov degeneration.
    There exists a unique $T_0$ equivariant extension of the family $(\cY_{\cX}^{\mathrm{msec}}, \mathcal{D}^{\mathrm{msec}}_{\cX} + \epsilon\, \mathcal{C}_{\cX}^{\mathrm{msec}}) \rightarrow S^{\mathrm{\mathrm{msec}}}_{\cX}$ to a family $(\mathcal{Y}_{\cX}^{\mathrm{sec}},\mathcal{D}^{\mathrm{sec}}_{\cX} + \epsilon\, \mathcal{C}^{\mathrm{sec}}_{\cX}) \to \mathcal{S}^{\mathrm{sec}}_{\cX}$.  
\end{theorem}

\begin{proof}
This follows by the existence of the action of $T_0$ on $\cY_{\cX}^{\mathrm{msec}} \to S_{\cX}^{\mathrm{\mathrm{msec}}}$ and \cite[Corollary 5.7]{HKY20}.  
\end{proof}

\begin{remark} \label{Rem: bogus}
    By Theorem \ref{Thm: ext_bogus}, the line bundle $\mathcal{L}_{\cX}^{\mathrm{msec}}$ on $\cY_{\cX}^{\mathrm{msec}}$ extends to a line bundle $\mathcal{L}_{\cX}^{\mathrm{sec}}$ on $\mathcal{Y}_{\cX}^{\mathrm{sec}}$. Moreover, by Lemma \ref{Lem: torus_action}, the theta functions $\vartheta_p$ with $p \in P_\ZZ$ extend to sections of $\mathcal{L}_{\cX}^{\mathrm{sec}}$.
    By Definition \ref{Def: secondary_fan}, every cone of the secondary fan $\mathrm{Mov}(\cX/\cX^{\mathrm{can}})$ is of the form $\rho+\sum_{v\in I}[X^v]$, where $\rho$ is a cone of $\mathrm{MovSec}(\cX/\cX^{\mathrm{can}})$ and $I$ is a set of integral points $v \in P_\ZZ$ such that the corresponding irreducible components $X^v$ of $\cX_0$ are either 1-surfaces or 2-surfaces. 
    By \cite[\S 6.4]{HKY20}, in restriction to the toric stratum of $\mathcal{S}^{\mathrm{sec}}_{\cX}$ corresponding to the cone $\rho+\sum_{v\in I}\RR_{\geq 0}[X^v]$ of $\mathrm{Sec}(\cX/\cX^{\mathrm{can}})$, the divisor $\mathcal{C}^{\mathrm{sec}}_{\cX}$ given by Theorem \ref{Thm: ext_bogus} is given by the equation $\sum_{p \in P_\ZZ \setminus I} \vartheta_p =0$.
\end{remark}

\begin{theorem}
\label{Thm: prop_bogus}
Let $\pi: \cX \rightarrow \Delta$ be a generic 
quasi-projective open Kulikov degeneration. Then, the extended mirror family
$( \mathcal{Y}_{\cX}^{\mathrm{sec}}, \mathcal{D}^{\mathrm{sec}}_{\cX}) \rightarrow \mathcal{S}^{\mathrm{sec}}_{\cX}$
is a projective flat family of semi-log-canonical surfaces $(Y_t, D_t)$ such that $K_{Y_t} + D_t = 0$.
\end{theorem}

\begin{proof}
For the fibers of $( \cY_{\cX}^{\mathrm{msec}}, \mathcal{D}^{\mathrm{msec}}_{\cX}) \rightarrow S^{\mathrm{\mathrm{msec}}}_{\cX}$, the result follows from Theorem \ref{Thm: fmat family} applied to the various projective crepant resolutions of $\cX^{\mathrm{can}}$. By Theorem \ref{Thm: ext_bogus}, any fiber of $( \mathcal{Y}_{\cX}^{\mathrm{sec}}, \mathcal{D}^{\mathrm{sec}}_{\cX}) \rightarrow \mathcal{S}^{\mathrm{sec}}_{\cX}$ is related by $T_0$-action, and so is isomorphic, to a 
fiber of $(\cY_{\cX}^{\mathrm{msec}}, \mathcal{D}^{\mathrm{msec}}_{\cX}) \rightarrow S^{\mathrm{\mathrm{msec}}}_{\cX}$, and so the general result follows.
\end{proof}

\section{KSBA stability of the mirror family to an open Kulikov degeneration}
\label{sec: KSBA_stability}

In this section we first prove in Theorem \ref{Thm: restricted flat} that the restriction of the mirror family to the dense torus in the base is KSBA stable. In Theorem \ref{Thm: KSBA stability} we extend this result to show KSBA stability over the entire base.
Finally, in Theorem \ref{thm_non_constant}, we establish a result concerning the non-constancy of the mirror family.

\subsection{Generic KSBA stability of the mirror family}

In this section, we prove in Theorem \ref{Thm: restricted flat} the KSBA stability of the restriction of the extended mirror family $(\cY_\cX^{\mathrm{sec}}, \mathcal{D}_\cX^{\mathrm{sec}}+\epsilon\, \mathcal{C}_\cX^{\mathrm{sec}
}) \rightarrow \mathcal{S}_\cX^{\mathrm{sec}}$ to the dense torus in 
$\mathcal{S}_\cX^{\mathrm{sec}}$.

\begin{theorem}
\label{Thm: restricted flat}
Let $\pi: \cX \rightarrow \Delta$ be a 
generic quasi-projective open Kulikov degeneration. 
For every $t \in \mathcal{S}^{\mathrm{sec}}_{\cX}$, let $(Y_t, D_t+\epsilon\, C_t)$ be the fiber over $t$ of the extended mirror family
$(\mathcal{Y}_{\cX}^{\mathrm{sec}}, \mathcal{D}^{\mathrm{sec}}_{\cX} +\epsilon\, \mathcal{C}^{\mathrm{sec}}_{\cX} ) \rightarrow \mathcal{S}^{\mathrm{sec}}_{\cX}$. Then, the following holds:
\begin{itemize}
\item[i)] For every $t \in \mathcal{S}^{\mathrm{sec}}_{\cX}$, if the surface $Y_t$ is normal, then the pair $(Y_t, D_t)$ is log canonical, and the pair $(Y_t, D_t +\epsilon\, C_t)$ is KSBA stable.
\item[ii) ]If $t$ is contained in the dense torus orbit $T_\cX := \Spec \, 
\CC[N_1(\cX/\cX^{\mathrm{can}})] \subset \mathcal{S}^{\mathrm{sec}}_{\cX}$, then the surface $Y_t$ is normal, the pair $(Y_t, D_t)$ is log canonical, and the pair $(Y_t, D_t +\epsilon\, C_t)$ is KSBA stable.
\end{itemize}
\end{theorem}

\begin{proof}
We first prove i). 
By Theorem \ref{Thm: prop_bogus}, the pair $(Y_t,D_t)$ is semi-log-canonical, and so log canonical since the surface $Y_t$ is assumed to be normal. Hence, 
by \cite[Proposition 7.2]{HKY20}, to prove that $(Y_t, D_t +\epsilon\, C_t)$ is semi-log-canonical, it 
suffices to show that $C_t$ does not contain any 0-dimensional stratum of $D_t$. 
By Remark \ref{Rem: bogus}, the divisor $C_t$ is defined by an equation of the form $\sum_{p \in P_\ZZ \setminus I} \vartheta_p=0$ for some subset $I \subset P_\ZZ$ consisting of integral points corresponding to either 1-surfaces or 2-surfaces in $\cX_0$. By Equations \eqref{Eq: ideal}-\eqref{Eq: divisor_D}, we have $\vartheta_p|_{D_t}=0$ if $p \notin \partial P_\ZZ$. On the other hand, by Lemma \ref{Lem: boundary_divisor}, $D_t$ is a cycle of rational curves, with $0$-dimensional strata $x_v$ in one-to-one correspondence with the integral points $v$ of $\partial P$ corresponding to the 0-surfaces in $\cX_0$. Moreover, the theta functions $\vartheta_p$ with $p \in \partial P$ restrict to toric monomials on the irreducible components of $D_t$. It follows that for every $0$-dimensional stratum $x_v$ of $D_t$, we have $\vartheta_p(x_v) \neq 0$ if $p=v$, and $\vartheta_p(x_v)=0$ if $p \in P_\ZZ \setminus \{v\}$. Hence, we have $\sum_{p \in P_v}\vartheta_p(x_v) \neq 0$, and so $x_v \notin C_t$. This concludes the proof of i).

To prove ii), it suffices to prove that $Y_t$ is normal, since then the other conditions follow by i). 
By Theorem \ref{Thm: fmat family}, the pair $(Y_t, D_t)$ is semi-log-canonical. In particular, the surface $Y_t$ is $S_2$, and so, by Serre's criterion for normality \cite[Thm 8.22A]{Hartshorne}, it suffices to show that $Y_t$ is regular in codimension one. As in the proofs of Lemma \ref{Lem: finitely_generated_0} and Theorem \ref{Thm: fmat family}, we can assume that there exists a good divisor $F$ on $\cX$, inducing a $\CC^\star$-action on $\cY_\cX \rightarrow S_\cX$. The closure of the $\CC^\star$-orbit passing through $t$ is an affine line $\A^1$, with $(\A^1 \setminus \{0\})\subset T_\cX$, and $0 \in \Spec\, \CC[G]\simeq \Spec\, \CC[H]/J$, where $G=F^\perp \cap NE(\cX/\cX^{\mathrm{can}})$ and $J=NE(\cX/\cX^{\mathrm{can}}) \setminus G$. 
The restriction $\cY_\cX|_{\A^1}$ of $\cY_\cX$ to this affine line is constant with fibers $Y_t$ on $\A^1 \setminus \{0\}$. Moreover, by the proof of Lemma \ref{Lem: mumtor}, the central fiber $\cY_{\cX,0}$
of $\cY_\cX|_{\A^1}$ 
is a union of toric varieties with intersection complex $(P, \sP_J)$, where $\sP_J$ is the subdivision of $P$ obtained from $\sP_J$ by omitting the edges $e$ such that the corresponding curves $C_e$ in the double locus of $\cX_0$ satisfy $F \cdot C_e=0$. In particular, the singular locus of $(P, \sP_J)$ is the union of its double curves $D_e$, which are in one-to-one correspondence with the edges $e$ of $\sP_J$, that is, with the double curves 
$C_e$ of $\cX_0$ such that $F \cdot C_e \neq 0$. 
Hence to prove that $Y_t$ is smooth in codimension one, it suffices to show that each of the double curves $D_e$ are generically smoothed out in $\cY_{\cX}|_{\A^1}$.

It follows from the construction of $\cY_\cX \rightarrow S_\cX$ from the canonical scattering diagram, together with \cite[Eq.(2.17)]{GHS}, 
that a neighborhood of $D_e$ in  $\cY_{\cX}|_{\A^1}$ is isomorphic to the affine variety of equation $xy=t^{[C_e]} f_{W,e}$, where $f_{W,e}$ is the function of the canonical scattering diagram attached to the wall supported on the edge $e$ of $P$. 
As $(\A^1\setminus \{0\})$ is contained in $T_\cX$, we have $t^{[C_e]} \neq 0$ on the non-central fibers of $\cY_{\cX}|_{\mathbb{A}^1}$. Moreover, by
Lemma \ref{Lem: finiteness_0}, the function $f_{W,e}$ is generically non-zero on $D_e$. This shows that all the double curves $D_e$ of $\cY_{\cX,0}$ are generically smoothed out in the one-parameter family $\cY_{\cX}|_{\A^1}$, and this concludes the proof that $Y_t$ is normal.
\end{proof}

\subsection{KSBA stability of the mirror family}
\label{sec: KSBA_stability_mirror}

In this section, we prove the KSBA stability of the mirror family $(\mathcal{Y}^{\mathrm{sec}}_{\cX}, \mathcal{D}^{\mathrm{sec}}_{\cX}+\epsilon\, \mathcal{C}^{\mathrm{sec}}_{\cX}) \rightarrow \mathcal{S}^{\mathrm{sec}}_{\cX}$. To do this, we first provide an inductive description of the fibers.

Let $\pi: \cX \rightarrow \Delta$ be a generic quasi-projective open Kulikov degeneration, and let $H$ be a face of $NE(\cX/\cX^{\mathrm{can}})$. 
As $f: \cX \rightarrow \cX^{\mathrm{can}}$ is a relative Mori dream space by Proposition \ref{Prop: Mori dreams}, there is a corresponding birational contraction 
$f_H: \cX \rightarrow \cX_H$ such that $H=NE(\cX/\cX_H)$.
Denote by  
$S_H:= \Spec\, \CC[H]= \Spec\, \CC[NE(\cX/\cX_H)]$ the corresponding stratum of $S_\cX$, with dense torus orbit $T_H :=\Spec \, \CC[H^{\mathrm{gp}}]$.
Let $\sP_H$ be the polyhedral decomposition of $P$ obtained as a coarsening of $\sP$ by removing all edges $e$ of $\sP$ such that the corresponding double curves $C_e$ of $\cX_0$ satisfy $[C_e] \in H$, that is, are contracted by $f_H$.
For every 2-dimensional cell $\sigma$ of $\sP_H$, we denote by $\sP_\sigma$ the polyhedral decomposition of $\sigma$ obtained by restricting $\sP$ to $\sigma$.

Recall that the 2-dimensional cells $\tau$ of $\sP$ are lattice triangles of size one, in one-to-one correspondence with the triple points $x_\tau$ of $\cX_0$. If $\tau$ and $\tau'$ are contained in the same 2-dimensional cell $\sigma$ of $\sP_H$, then the triple points $x_\tau$ and $x_{\tau'}$ are connected by a chain of double curves which are all contracted by $f_H$, and so we have 
$f_H(\tau)=f_H(\tau')$. Hence, for every 2-dimensional face $\sigma$ of $\sP_H$, there is a unique point $x_\sigma \in \cX_H$ such that $f_H(x_\tau)=x_\sigma$ for every 2-dimensional cell $\tau$ of $\sP$ contained in $\sigma$. 

\begin{lemma} \label{Lem: contracted}
Let $\sigma$ be a 2-dimensional cell of $\sP_H$. Then, the following holds:
\begin{itemize}
    \item[i)] For every edge $e\in \sP^{[1]}$ such that $e\subset \sigma$ and $e \not\subset \partial \sigma$, corresponding to the double curve $C_e$ of $\cX_0$, we have $f_H(C_e)=\{x_\sigma\}$.
    \item[ii)] For every vertex $v \in \sP^{[0]}$ such that $v \in \sigma$ and $v \notin \partial \sigma$, corresponding to the irreducible component $X^v$ of $\cX_0$, we have $f_H(X^v)=\{x_\sigma\}$.
\end{itemize}
\end{lemma}

\begin{proof}
First note that i) follows immediately from the definition of $\sP_H$. For ii), as $v \notin \partial \sigma$, it follows from i) that all the irreducible components of $\partial X^v$ are contracted by $f_H$ to $x_\sigma$. This implies that $X^v$ is also contracted by $f_H$ to $x_\sigma$ since, by Theorem \ref{thm: positive}, $X^v$ is a positive log Calabi--Yau surface and so $\partial X^v$ supports a big divisor.
\end{proof}

It follows from Lemma \ref{Lem: contracted} that one can find an affine open subset $U_\sigma$ of $\cX_{H,0}$ containing $x_\sigma$ such that $f_H(C_e) \cap U_\sigma=\emptyset$ for every edge $e \in \sP^{[1]}$
such that $e \not\subset \sigma$. We denote
\[ \cX_{\sigma,0}:= f_H^{-1}(U_\sigma) \subset \cX_0\,.\]
By Lemma \ref{Lem: contracted}, the irreducible components $X^v_\sigma$ of $\cX_{\sigma,0}$ are naturally in one-to-one correspondence with the integral points $v \in \sigma_\ZZ$:
\[\cX_{\sigma,0} =\bigcup_{v \in \sigma_\ZZ} X^v_\sigma \,.
\]

\begin{lemma} \label{Lem: open_Kulikov}
Let $\pi: \cX \rightarrow \Delta$ be a generic quasi-projective open Kulikov degeneration and let $H$ be a face of $NE(\cX/\cX^{\mathrm{can}})$. Then, for every 2-dimensional cell $\sigma$ of $\sP_H$, there exists a generic quasi-projective open Kulikov surface $\cX'_{\sigma,0}$, with dual intersection complex $(\sigma, \sP_\sigma)$, such that 
$\cX_{\sigma,0}$ is a Zariski open dense subset of $\cX'_{\sigma,0}$, containing all compact curves of $\cX'_{\sigma,0}$, and $U_\sigma$ is a Zariski open neighborhood of the central point in the affinization 
$(\cX'_{\sigma,0})^{\mathrm{can}}$.
\end{lemma}

\begin{proof}
First note that $\cX_{\sigma,0}$ is quasi-projective, reduced and normal crossings, as an open subset of the quasi-projective, reduced and normal crossings surface $\cX_0$.
If $v \in \sigma_\ZZ$ and $v \notin \partial \sigma$, then, by Lemma \ref{Lem: contracted} ii), we have $X^v_\sigma=X^v$, and so $X^v_\sigma$ is a 2-surface. If $v \in \partial \sigma$, then it follows from Definitions \ref{def_0_surface}-\ref{def_1_surface} that $X^v_\sigma$ is naturally a Zariski open subset of a 0-surface or a 1-surface if $\dim f_H(X^v_\sigma)=2$ or $\dim f_H(X^v_\sigma)=1$ respectively.
Finally, it follows from \cite[Proposition 3.3]{SB2}
that $\sigma$ is homeomorphic to a disk.
\end{proof}

Let $\cX'_{\sigma,0}$ be the quasi-projective open Kulikov surface given by 
Lemma \ref{Lem: open_Kulikov}.
By Theorem \ref{thm_generic},
we can assume up to locally trivial deformations that $\cX_{\sigma,0}'$ is generic $d$-semistable, and so that $\cX_{\sigma,0}'$ is the central fiber of a quasi-projective open Kulikov degeneration $\cX_\sigma \rightarrow \Delta$ by
Theorem \ref{thm: smoothing exists} and
Lemma \ref{Lem: projective}.
The inclusions $\cX_{\sigma,0} \subset \cX_0$ and $\cX^{\mathrm{can}}_{\sigma,0} \subset \cX_{H,0}$ induce a map $NE(\cX_\sigma /\cX^{\mathrm{can}}_\sigma)
\rightarrow H= NE(\cX/\cX_H)$, and so a toric morphism
\begin{equation}\label{Eq: SH_SX} 
S_H=\Spec\, \CC[H]
\longrightarrow S_{\cX_\sigma} =\Spec\,\CC[NE(\cX_\sigma/\cX^{\mathrm{can}}_\sigma)] \,.\end{equation}

\begin{lemma} \label{Lem: proof_1}
Let $\pi: \cX \rightarrow \Delta$ be a generic quasi-projective open Kulikov degeneration, and let $H$ be a face of $NE(\cX/\cX^{\mathrm{can}})$. 
Then, for every $\beta \in H$ and $p,q,r \in C(P)_\ZZ$, we have $N_{pqr}^\beta(\cX)=0$ unless $p,q,r$ are all contained in a common cone $C(\sigma)$ over a 2-dimensional cell $\sigma$ of $\sP_H$.
\end{lemma}

\begin{proof}
The tropicalization of a punctured log curve $C$ contributing to $N_{pqr}^\beta(\cX)$ as
in Definition \ref{Def: structure constants}
is a connected tropical curve in $C(P)$, with two unbounded edges with asymptotic directions $\RR_{\geq 0}p$ and $\RR_{\geq 0}q$, and with one leg contained in $\RR_{\geq 0}r$. In particular, if $p,q,r$ were not all contained in a common cone $C(\sigma)$, this tropical curve would have to intersect a cone $C(\sigma \cap \sigma')$ over the intersection of  two distinct 2-dimensional faces $\sigma$, $\sigma'$ of $\sP_\sigma$, and so a cone $C(e)$ over an edge $e$ of $\sP_\sigma$. In particular, the curve $C$ would contain an irreducible component with class a multiple of $[C_e]$, where $C_e$ is the double curve of $\cX_0$ corresponding to $e$. Since $e$ is an edge of $\sP_\sigma$, we have $[C_e] \notin H$, and this would imply $\beta \notin H$. Hence, the results follows since we assume that $\beta \in H$. 
\end{proof}

\begin{lemma} \label{Lem: proof_2}
Let $\pi: \cX \rightarrow \Delta$ be a generic quasi-projective open Kulikov degeneration and let $H$ be a face of $NE(\cX/\cX^{\mathrm{can}})$. Then, 
for every 2-dimensional cell $\sigma$ of $\sP_H$, 
for every $p,q,r \in C(\sigma)_\ZZ$ and $\beta \in H$, we have $N_{pqr}^\beta(\cX)=0$ if $\beta \notin NE(\cX_\sigma/\cX^{\mathrm{can}}_\sigma)$, and 
$N_{pqr}^\beta(\cX)=N_{pqr}^\beta(\cX_\sigma)$
for every $\beta \in NE(\cX_\sigma/\cX^{\mathrm{can}}_\sigma)$.
\end{lemma}

\begin{proof}
Let $Z_r$ be the stratum of $\cX_0$ corresponding to the smallest cone of $C(P)$ containing $r$. As $r \in C(\sigma)_\ZZ$, it follows from Lemma \ref{Lem: contracted} that $x_\sigma \in f_H(Z_r)$.
Recall that the Definition \ref{Def: structure constants} of $N_{pqr}^\beta(\cX)$ involves considering punctured log curves passing through a given point $z$ in $Z_r$. As $U_\sigma$ is an open neighborhood of $x_\sigma$, there exists a point $z$ in $Z_r$ such that $f_H(z) \in U_\sigma$. 
Using such a point $z$ to calculate $N_{pqr}^\beta(\cX)$, 
all the contributing punctured log curves are contained in $f_H^{-1}(f_H(z))$ and so in $\cX_{\sigma,0}$, since $\beta \in H$ and so curves of class $\beta$ are contracted by $f_H$. In particular, we necessarily have $\beta \in NE(\cX_\sigma/\cX^{\mathrm{can}}_\sigma)$ if $N_{pqr}^\beta(\cX) \neq 0$. 
Finally, as all these curves are contained in $\cX_{\sigma,0}$, and the log structure on $\cX_{\sigma,0}$ is the restriction of the log structure on $\cX_0$, we have 
$N_{pqr}^\beta(\cX)=N_{pqr}^\beta(\cX_\sigma)$ by 
Lemma \ref{Lem: N_independent}.
\end{proof}

\begin{theorem} \label{Thm: inductive_mirror}
    Let $\pi: \cX \rightarrow \Delta$ be a generic quasi-projective open Kulikov degeneration and let $H$ be a face of $NE(\cX/\cX^{\mathrm{can}})$. Then, the following holds:
    \begin{itemize}
    \item[i)] There is a natural one-to-one correspondence between the irreducible components of $\cY_\cX|_{S_H}$ and the 2-dimensional cells $\sigma$ of $\sP_H$.
    \item[ii)] For every 2-dimensional cell $\sigma$ of $\sP_H$, denote by $\cY_\sigma$ the corresponding irreducible component of $\cY_\cX|_{S_H}$, and by
    $\partial \cY_\sigma$, $\mathcal{D}_\sigma$ and $\mathcal{C}_\sigma$ the restrictions to $\cY_\sigma$ of 
   the double locus of $\cY_\cX|_{S_H}$,
    $\mathcal{D}_\cX|_{S_\cX}$ 
    and $\mathcal{C}_\cX|_{S_\cX}$ respectively. 
    Then, the family $(\cY_\sigma, \mathcal{D}_\sigma + \partial \cY_\sigma+\epsilon\, \mathcal{C}_\sigma) \rightarrow S_H$
    is the base change along $S_H \rightarrow S_{\cX_\sigma}$
    of the mirror family $(\cY_{\cX_\sigma}, \mathcal{D}_{\cX_\sigma}+\epsilon\, \mathcal{C}_{\cX_\sigma})$ to the quasi-projective open Kulikov degeneration $\cX_\sigma \rightarrow \Delta$.
\end{itemize}
\end{theorem}

\begin{proof}
Denoting $J_H:= NE(\cX/\cX^{\mathrm{can}}) \setminus H$,
we have $\cY_\cX|_{S_H}=\mathrm{Proj} \, \mathcal{R}_\cX/J_H$. Explicitly, $\mathcal{R}_\cX/J_H$ is the algebra of theta functions $(\vartheta_p)_{p \in C_\ZZ}$ where the structure constants $N_{pqr}^\beta(\cX)$ with $\beta \notin H$ are set to zero. 
By Lemma \ref{Lem: proof_1}, the structure constants $N_{pqr}^\beta(\cX)$ with $\beta \in H$ are zero unless $p,q,r$ are contained in a common cone $C(\sigma)$ over a 2-dimensional cell $\sigma$ of $\sP_H$. Therefore, for every such $\sigma$, denoting by $\mathcal{R}_\sigma$ 
the $\CC[H]$-algebra
spanned by the theta functions $(\vartheta_p)_{p \in C(\sigma)_\ZZ}$, it follows that $\mathcal{Y}_\sigma := \mathrm{Proj}\, \mathcal{R}_\sigma$ is an union of irreducible components of $\cY_{\cX}|_{S_H}$.
Moreover, denoting by $\partial \cY_\sigma$ the union of $\cY_{\sigma} \cap \cY_{\sigma'}$ with $\sigma' \neq \sigma$, it follows that $\partial \cY_\sigma$ is defined by $\vartheta_p=0$ for every $p \in C(\sigma)_\ZZ$ such that there does not exist $\sigma' \neq \sigma$ with $p \in C(\sigma')$. On the other hand, by the definition of $\mathcal{D}_\cX$ in Equations \eqref{Eq: ideal}-\eqref{Eq: divisor_D}, the restriction $\mathcal{D}_\sigma:=\mathcal{D}_{\cX}|_{\cY_\sigma}$ is defined by $\vartheta_p=0$ for $p \notin (C(\partial \sigma) \cap C(\partial P))_\ZZ$. For every $p \in C(\partial \sigma)_\ZZ$, we have $p \in C(\sigma')$ for some $\sigma' \neq \sigma$ or $p \in (C(\partial \sigma) \cap C(\partial P))_\ZZ$, and so the divisor $\mathcal{D}_\sigma + \partial \mathcal{Y}_\sigma$
is defined by $\vartheta_p=0$ for $p \in C(\sigma)_\ZZ \setminus C(\partial \sigma)_\ZZ$.

By Lemma \ref{Lem: proof_2}, for every $p,q,r \in C(\sigma)_\ZZ$ and $\beta \in H$, we have $N_{pqr}^\beta(\cX)=0$ if $\beta \notin NE(\cX_\sigma/\cX^{\mathrm{can}}_\sigma)$, and 
$N_{pqr}^\beta(\cX) = N_{pqr}^\beta(\cX_\sigma)$ for every $\beta \in NE(\cX_\sigma/\cX^{\mathrm{can}}_\sigma)$. Hence, $\mathcal{Y}_\sigma$ is the base change of $\mathcal{Y}_{\cX_\sigma}$ along $S_H \rightarrow S_{\cX_\sigma}$, compatibly with the theta functions. In particular, it follows from the above description of $\mathcal{D}_\sigma + \partial \mathcal{Y}_\sigma$ in terms of theta functions that $\mathcal{D}_\sigma + \partial \mathcal{Y}_\sigma$ is the base change of $\mathcal{D}_{\cX_\sigma}$, and, by Equation \eqref{Eq: divisor_C}, $\mathcal{C}_\sigma:=\mathcal{C}_{\cX}|_{\cY_\sigma}$ 
is the base change of $\mathcal{C}_{\cX_\sigma}$
along $S_H \rightarrow S_{\cX_\sigma}$.
The toric morphism $S_H \rightarrow S_{\cX_\sigma}$ restricts to a morphism $T_H \rightarrow T_{\cX_\sigma}$ between the dense torus orbits. By Theorem \ref{Thm: restricted flat} ii) applied to the quasi-projective open Kulikov degeneration $\cX_\sigma \rightarrow \Delta$, the restriction $\cY_{\cX_\sigma}|_{T_{\cX_\sigma}}$ is irreducible.
Hence, the restriction $\cY_{\sigma}|_{T_H}$ is also irreducible, and we conclude that $\cY_\sigma$ is an irreducible component of $\cY$.
\end{proof}

\begin{theorem}
\label{Thm: KSBA stability}
Let $\cX \rightarrow \cX^{\mathrm{can}}$ be a generic quasi-projective open Kulikov degeneration. Then, the mirror family
    $(\mathcal{Y}^{\mathrm{sec}}_{\cX}, \mathcal{D}^{\mathrm{sec}}_{\cX} +\epsilon\, \mathcal{C}^{\mathrm{sec}}_{\cX}) \rightarrow \mathcal{S}^{\mathrm{sec}}_{\cX}$ is a family of KSBA stable pairs.
\end{theorem}

\begin{proof}
By Theorem \ref{Thm: prop_bogus}, we already have that
$(\mathcal{Y}_{\cX}^{\mathrm{sec}}, \mathcal{D}^{\mathrm{sec}}_{\cX}) \longrightarrow S^{\mathrm{sec}}_{\cX}$
is a projective flat family of semi-log-canonical surfaces $(Y_t, D_t)$ such that $K_{Y_t} + D_t = 0$.
In particular, the divisor $K_{Y_t}+D_t+\epsilon\, C_t =\epsilon\, C_t$ is ample, and so it only remains to prove that the pairs $(Y_t, D_t+\epsilon\, C_t)$ are semi-log-canonical.
It is enough to prove this over each of the torus orbits covering the toric stack $\mathcal{S}^{\mathrm{sec}}_{\cX}$.

For a torus orbit corresponding to a cone of $\mathrm{MovSec}(\cX/\cX^{\mathrm{can}})$, we can assume, up to replacing $\cX$ by another projective crepant resolution of $\cX^{\mathrm{can}}$, that this cone is a face $H$ of $NE(\cX/\cX^{\mathrm{can}})$. By the description of $\cY_\cX|_{S_H}$ given in
Theorem \ref{Thm: inductive_mirror} i), and by \cite[Proposition 7.2]{HKY20}, it suffices to show that for every 2-dimensional face $\sigma$ of $\sP_H$, the restriction of $(\mathcal{Y}_\sigma, \mathcal{D}_\sigma+\partial \cY_\sigma+\epsilon\, \mathcal{C}_\sigma)$ to $T_H$ is a family of KSBA stable pairs. 
The toric morphism $S_H \rightarrow S_{\cX_\sigma}$ in Equation \eqref{Eq: SH_SX} restricts to a morphism $T_H \rightarrow T_{\cX_\sigma}$. 
Hence, by Theorem \ref{Thm: inductive_mirror} ii),  the restriction of $(\mathcal{Y}_\sigma, \mathcal{D}_\sigma+\partial \cY_\sigma+\epsilon\, \mathcal{C}_\sigma)$ to $T_H$ is a base change of the mirror family $(\cY_{\cX_\sigma}, \mathcal{D}_{\cX_\sigma} +\epsilon\, \mathcal{C}_{X_\sigma})$ restricted to $T_{\cX_\sigma}$. By Theorem \ref{Thm: restricted flat} ii) applied to the quasi-projective open Kulikov degeneration $\cX_\sigma \rightarrow \Delta$, the latter family is a family of KSBA stable pairs, and this concludes the proof in this case.

For a general torus orbit, we can assume, as in Remark \ref{Rem: bogus},
that the corresponding cone is a bogus cone of the form $H+\sum_{v\in I} \RR_{\geq 0}[X^v]$, where $H$ is as described above and $I$ a subset $P_\ZZ$ consisting of integral points $v$ such that $X^v$ is contracted by $f_H$. 
By Theorem \ref{Thm: ext_bogus}, the restriction of  $(\mathcal{Y}^{\mathrm{sec}}_{\cX},\mathcal{D}^{\mathrm{sec}}_{\cX}+\epsilon\, \mathcal{C}^{\mathrm{sec}}_{\cX})$ to such a torus orbit is obtained 
by $T_0$-equivariance from the restriction of $(\mathcal{Y}^{\mathrm{sec}}_{\cX},\mathcal{D}^{\mathrm{sec}}_{\cX}+\epsilon\, \mathcal{C}^{\mathrm{sec}}_{\cX})$ to $T_H$.
Since the base change in Theorem \ref{Thm: inductive_mirror} is compatible with the $T_0$-actions, it suffices to show that, for every 2-dimensional face $\sigma$ of $\sP_H$, the extension of the mirror family $(\cY_{\cX_\sigma}, \mathcal{D}_{\cX_\sigma}+\epsilon \, \mathcal{C}_{\cX_\sigma})$ to the toric stratum corresponding to the bogus cone $\sum_{v \in I_\sigma}\RR_{\geq 0}[X^v_\sigma]$ of $\cX_\sigma$, where $I_\sigma=I \cap \sigma_\ZZ$, is a family of KSBA stable pairs. By Theorem \ref{Thm: ext_bogus}, any fiber $Y_{\sigma,t}$ of this family is related by torus action, and so is isomorphic, to a fiber of 
$\cY_{\cX_\sigma}|_{T_{\cX_\sigma}}$.
By Theorem \ref{Thm: restricted flat} ii), such a fiber is normal, and so the KSBA stability follows by Theorem \ref{Thm: restricted flat}~i).
\end{proof}

\subsection{Non-constancy of the mirror family} \label{Sec:non_constancy}

In this section, we prove that the mirror family $(\mathcal{Y}^{\mathrm{sec}}_{\cX}, \mathcal{D}^{\mathrm{sec}}_{\cX}+\epsilon\, \mathcal{C}^{\mathrm{sec}}_{\cX}) \rightarrow \mathcal{S}^{\mathrm{sec}}_{\cX}$ is non-constant in restriction to one-dimensional toric strata of $\mathcal{S}^{\mathrm{sec}}_{\cX}$.

\begin{theorem} \label{thm_non_constant}
    Let $\cX \rightarrow \cX^{\mathrm{can}}$ be a generic quasi-projective open Kulikov degeneration. Then, the family of KSBA stable pairs
    $(\mathcal{Y}^{\mathrm{sec}}_{\cX}, \mathcal{D}^{\mathrm{sec}}_{\cX} +\epsilon\, \mathcal{C}^{\mathrm{sec}}_{\cX}) \rightarrow \mathcal{S}^{\mathrm{sec}}_{\cX}$ is non-constant in restriction to every one-dimensional toric stratum of the toric variety $\mathcal{S}^{\mathrm{sec}}_{\cX}$.
\end{theorem}

\begin{proof}

Let $\rho_0$ be a codimension one cone of the secondary fan $\mathrm{Sec}(\cX/\cX^{\mathrm{can}})$, corresponding to a 1-dimensional toric stratum $S_{\rho_0}$ of $\mathcal{S}^{\mathrm{sec}}_\cX $. As the fan $\mathrm{Sec}(\cX/\cX^{\mathrm{can}})$ is complete, there are two maximal cones $\rho_1$ and $\rho_2$ containing $\rho_0$, corresponding to the two 0-dimensional toric strata $S_{\rho_1}$ and $S_{\rho_2}$ contained in $S_{\rho_0}$. We will show that at least one of the fibers $(Y_{\rho_1}, D_{\rho_1})$ or $(Y_{\rho_2}, D_{\rho_2})$ of $(\cY, \cD)$ over the points $S_{\rho_1}$ and $S_{\rho_2}$ is not isomorphic to the general fiber $(Y_{\rho_0}, D_{\rho_0})$ of $(\cY, \cD)$ over $S_{\rho_0}$.

Let $\tilde{\rho}_0$ be a codimension one cone of the Mori fan containing $\rho_0$, and $\tilde{\rho}_1$, $\tilde{\rho}_2$ be the two maximal dimensional cones of the Mori fan containing $\tilde{\rho}_0$ and contained in 
$\rho_1$ and $\rho_2$ respectively.
For $i=0,1,2$, we have $\tilde{\rho}_i = 
\mathrm{Nef}(\cZ_i/\cX^{\mathrm{can}})+\sum_{v\in J_i} \RR_{\geq 0}[X^v]$, where $\cZ_i \rightarrow \cX^{\mathrm{can}}$ is a partial crepant resolution, and $X^v$ with $v \in J_i$ are divisors which are contracted by a crepant resolution of $\cZ_i$. 
Moreover, since $\tilde{\rho}_0$ is a codimension one face of $\tilde{\rho}_i$, 
there exists a corresponding birational contraction $\cZ_i \rightarrow \cZ_0$.
This map is either of relative Picard rank one, in which case $J_i=J_0$, or
it is an isomorphism, in which case $J_i = J_0 \sqcup \{V\}$ for some $V\in J_i \setminus J_0$. 
The latter case cannot happen for both $i=1$ and $i=2$:
indeed, if $J_1=J_0 \sqcup \{V\}$, then $[X^V] \in \tilde{\rho}_1$ and $[X^V] \notin \tilde{\rho}_0$, and so $[X^V] \notin \tilde{\rho}_2$.

Hence, without loss of generality, we can assume that $\cZ_1\rightarrow \cZ_0$ is of relative Picard rank one and $J_1=J_0$. Let $\cX' \rightarrow \cZ_1$ be a projective crepant resolution, and denote by $H_1$ the face of $NE(\cX'/\cX^{\mathrm{can}})$ corresponding to $\cX' \rightarrow \cZ_1$, and by $H_0$ the codimension one face of $H_1$ corresponding to $\cZ_1 \rightarrow \cZ_0$. 
By Theorem \ref{Thm: inductive_mirror}, the intersection complex of $(Y_{\rho_1}, D_{\rho_1})$
(resp. $(Y_{\rho_0}, D_{\rho_0})$) is given by the polyhedral decomposition $\sP_{H_1}$ (resp. $\sP_{H_0}$) associated to $H_1$ (resp\,. $H_0$) as in  
\S \ref{sec: KSBA_stability_mirror}. We divide the remainder of the proof into two cases, depending if $\cZ_1\rightarrow \cZ_0$ is an isomorphism in codimension one or not.

\textbf{Case 1}:
Assume that $\cZ_1\rightarrow \cZ_0$ is an isomorphism in codimension one. If $\cZ_1 \rightarrow \cZ_0$ contracts a curve in the double locus of $\cZ_1$, then the corresponding edge in $\sP_{H_1}$ disappears in $\sP_{H_0}$. Hence, $Y_{\rho_1}$, and $Y_{\rho_0}$ are not isomorphic, and the result is proved in this case. On the other hand, if $\cZ_1 \rightarrow \cZ_0$ does not contract any curve in the double locus of $\cZ_1$, then the exceptional locus of $\cZ_1 \rightarrow \cZ_0$ is a rigid interior curve in an irreducible component of $\cZ_1$, whose strict transform in $\cX'$ is necessarily an interior exceptional curve. It follows that $NE(\cZ_0/\cX^{\mathrm{can}}_\sP)$ is contained in the codimension one face of $NE(\cX'/\cX^{\mathrm{can}}_\sP)$
defined by a multiple M1 flop contracting this interior exceptional curve, which is by Lemma \ref{Lem: M1_flops} in contradiction with the assumption that $\rho_0$ is a codimension cone of the secondary fan, and so this last case actually does not occur. 

\textbf{Case 2}: Assume that $\cZ_1 \rightarrow \cZ_0$ is a divisorial contraction.
Since $\cZ_1 \rightarrow \cZ_0$ is of relative Picard rank one, its exceptional locus consists of a single divisor $Z_1^v$ corresponding to a vertex $v$ of $\sP_{H_1}$. If an irreducible component of the boundary of $Z_1^v$ is contracted into $\cZ_0$, 
then the corresponding edge of $\sP_{H_1}$ disappears in $\sP_{H_0}$, and so we also deduce in this case that $Y_{\rho_1}$ and $Y_{\rho_0}$ are not isomorphic. 
If no boundary component of $Z_1^v$ is contracted in $\cZ_0$, then, as in \cite[\S 8.2]{HKY20}, $Z_1^v$ is necessarily a generically $\PP^1$-bundle over $\PP^1$, contracted to its base $\PP^1$ in $\cZ_0$, and its boundary $\partial Z_v^1$ is either an irreducible nodal curve mapping 2:1 to the base, or the union of two sections intersecting in two points. 

If $\partial Z_v^1$ is the union of two irreducible components $D_1$ and $D_2$, consider the corresponding edges $E_1$ and $E_2$ adjacent to $v$ in $\sP_{H_1}$, and choose a 2-dimensional face $\sigma_1$ of $\sP_{H_1}$ such that $E_1 \cup E_2$ are edges of the boundary of $\sigma_1$. Denote by $\sigma_0$ the same face viewed as a 2-dimensional face of $\sP_{H_0}$.
By Theorem \ref{Thm: inductive_mirror}, there is an irreducible component $(Y_{\sigma_1}, \partial Y_{\sigma_1})$ (resp.  $(Y_{\sigma_0}, \partial Y_{\sigma_0})$) of $Y_{\rho_1}$ (resp. $Y_{\rho_0}$)
isomorphic to the general fiber of the mirror family to the open Kulikov degeneration $\cX_{\sigma_1}$ (resp. $\cX_{\sigma_0}$) with tropicalization $(C(\sigma_1), C(\sP_{\sigma_1}))$ (resp. $(C(\sigma_0), C(\sP_{\sigma_0}))$).
The vertex $v$ can be viewed as a boundary vertex of both $\sP_{\sigma_1}$ and $\sP_{\sigma_0}$. Since the $\PP^1$-fibers of $Z^1_v$ are not contracted in $\cZ_1$, the non-compact irreducible component corresponding to $v$ in the central fiber $\cX_{\sigma_1,0}$ of $\cX_{\sigma_1}$  is a $0$-surface. Hence, by Lemma \ref{Lem: boundary_divisor}, it corresponds to a singular point of $\partial Y_{\sigma_1}$. On the other hand, since the $\PP^1$-fibers of $Z^1_v$ are contracted in $\cZ_0$, the non-compact irreducible component corresponding to $v$ in the central fiber $\cX_{\sigma_0,0}$ of $\cX_{\sigma_0}$ is a $1$-surface.
Thus, by  Lemma \ref{Lem: boundary_divisor}, it corresponds to a smooth point of $\partial Y_{\sigma_0}$. Since a singular point of $\partial Y_{\sigma_1}$ becomes a smooth point in $\partial Y_{\sigma_0}$, it follows that $Y_{\rho_1}$ and $Y_{\rho_0}$ are not isomorphic.

Finally, the case where $\partial Z^1_v$ is irreducible follows from the previous case as in \cite[\S 8.2]{HKY20}. Indeed, denoting by $\sigma_1$ (resp. $\sigma_0$) the 2-dimensional face of $\sP_{H_1}$ (resp. $\sP_{H_0}$) containing $v$, and cutting $\sigma_1$ (resp. $\sigma_0$) along the edges of $\sP_{\sigma_1}$  (resp. $\sP_{\sigma_0}$) connecting $v$ to the boundary of $\sigma_1$ (resp. $\sigma_0$), we obtain an integral affine manifold with singularities $(\tilde{\sigma}_1,
\sP_{\tilde{\sigma}_1})$
(resp. $(\tilde{\sigma}_0,
\sP_{\tilde{\sigma}_0})$)
with the same formal properties as $(\sigma_1, \sP_{\sigma_1})$
(resp. $(\sigma_0, \sP_{\sigma_0})$)
in the case where $\partial Z^1_v$ is reducible. There is a natural open Kulikov degeneration $\cX_{\widetilde{\sigma}_1}$ (resp. $\cX_{\widetilde{\sigma}_0}$) with tropicalization $(C(\widetilde{\sigma}_1), 
C(\sP_{\widetilde{\sigma}_1}))$
(resp. $(C(\widetilde{\sigma}_0),
C(\sP_{\widetilde{\sigma}_0}))$), and the irreducible 
component 
$Y_{\tilde{\sigma}_1}$ (resp. $Y_{\tilde{\sigma}_0}$) 
of the general fiber 
of the corresponding mirror family is the normalization of $Y_{\sigma_1}$ (resp. $Y_{\sigma_0}$). As above, a singular point of $\partial Y_{\tilde{\sigma}_1}$ becomes smooth in $\partial Y_{\tilde{\sigma}_0}$. 
Consequently, $Y_{\rho_1}$ and $Y_{\rho_0}$ are also non-isomorphic in this case.
\end{proof}

\section{KSBA moduli spaces of stable log Calabi--Yau surfaces}
\label{sec: KSBA moduli spaces}
Let $(Y,D,L)$ be a generic polarized log Calabi--Yau surface and let $\mathscr{P}$ be a good polyhedral decomposition of a Symington polytope $P$ associated to $(Y,D,L)$ as in \S \ref{subsec: good polyhedral}. Let $\pi_{\mathscr{P}} \colon \mathcal{X}_{\mathscr{P}} \to \Delta$ be the semistable mirror to the maximal degeneration of $(Y,D,L)$ defined by $\mathscr{P}$, and denote by 
\[  (\mathcal{Y}_{\cX_\mathscr{P}}, \mathcal{D}_{\cX_\mathscr{P}}, \mathcal{L}_{\cX_\mathscr{P}}) \longrightarrow S_{\cX_\sP} = \Spec\, \CC [NE(\mathcal{X}_{\mathscr{P}}/ \cX^{\mathrm{can}}_{\mathscr{P}}) ] \]
the polarized mirror family as in Definition \ref{Def:polarized mirror}. In what follows we refer to this family as the ``double mirror'' to $(Y,D,L)$. In this section we first investigate in Theorem \ref{Thm: rescaling} how the double mirror changes if $(P,\sP)$ is rescaled. We then apply this result to show in Theorem \ref{Thm: toric_deg_mirror} that, up to rescaling $(P,\sP)$, the restriction of the double mirror to a particular one-dimensional locus is a toric degeneration.
Building on this, Theorem 
\ref{thm: diffeomorphism_2} establishes that the general fiber of the double mirror is deformation equivalent to $(Y,D,L)$. Consequently, we prove in Theorem \ref{Thm: HKY_conj} that we obtain a surjective and finite map from the toric variety $S^{\mathrm{sec}}_{\mathcal{X}_{\sP}}$ to the coarse KSBA moduli space $M_{(Y,D,L)}$.

\subsection{Base change of open Kulikov degenerations}

For every $k\in \ZZ_{\geq 1}$, let 
$\pi_{k \mathscr{P}}: \mathcal{X}_{k \mathscr{P}} \rightarrow \Delta$ be the base change of the open Kulikov degeneration  $\pi_{\mathscr{P}} \colon \mathcal{X}_{\mathscr{P}} \to \Delta$ along the map
\begin{align*}
    \Delta &\longrightarrow \Delta \\
    t &\longmapsto t^k \,.
\end{align*}
Note that $\pi_{k\mathscr{P}}: \mathcal{X}_{k\mathscr{P}} \rightarrow \Delta$ is not an open Kulikov degeneration in general, as the total space $\mathcal{X}_{k\mathscr{P}}$ is singular.
Instead, $\pi_{k\mathscr{P}}: \mathcal{X}_{k\mathscr{P}} \rightarrow \Delta$ is log smooth for the divisorial log structure on  $\mathcal{X}_{k\mathscr{P}}$ 
induced by the central fiber $\mathcal{X}_{k\mathscr{P},0}:=\pi_{k\mathscr{P}}^{-1}(0)$, and the divisorial log structure on $\Delta$ induced by the origin $0 \in \Delta$. 
Moreover, the dual intersection complex of the central fiber $\mathcal{X}_{k\mathscr{P},0}$  is the scaled polyhedral decomposition $k\mathscr{P}$ on the scaled polytope $k P$. The following result will allow us to construct an open Kulikov degeneration which is a resolution of singularities of $\pi_{k\mathscr{P}}: \mathcal{X}_{k\mathscr{P}} \rightarrow \Delta$.

\begin{lemma}
\label{lem: good refined}
For every $k\in \ZZ_{\geq 1}$, there exists a good polyhedral decomposition $\widetilde{\sP}$ on the scaled polytope $kP$, obtained as a refinement of the decomposition $k\mathscr{P}$ on $kP$.
\end{lemma}

\begin{proof}
As in the proof of Theorem \ref{Thm: good dec exists}, we successively pull all integral points in $kP$ which are not vertices of $k\mathscr{P}$. This gives us a desired good polyhedral decomposition $\widetilde{\sP}$.
\end{proof}

Let $\widetilde{\sP}$ be a good polyhedral decomposition as in Lemma \ref{lem: good refined}, and let $\pi_{\widetilde{\sP}}: \mathcal{X}_{\widetilde{\sP}}\rightarrow \Delta$ be the corresponding open Kulikov degeneration.
As $\widetilde{\sP}$ is a refinement of $k\mathscr{P}$, it induces a log modification,
that is a proper birational log \'etale log morphism as in \cite{AW18}, $\nu: \mathcal{X}_{\widetilde{\sP}} \rightarrow \mathcal{X}_{k\mathscr{P}}$, such that $\pi_{\widetilde{\sP}}= \pi_{k\mathscr{P}} \circ \nu$. Let 
\[ \nu_\star: NE(\mathcal{X}_{\widetilde{\sP}}/\cX^{\mathrm{can}}_{\widetilde{\sP}}) \longrightarrow NE(\mathcal{X}_{\mathscr{P}}/\cX^{\mathrm{can}}_{\mathscr{P}})\]
be the composition of the pushforward of curve classes along $\nu$ with the natural isomorphism
$NE(\mathcal{X}_{k\mathscr{P}}/\cX^{\mathrm{can}}_{k\mathscr{P}}) \simeq NE(\mathcal{X}_{\mathscr{P}}/\cX^{\mathrm{can}}_{\mathscr{P}})$ whose existence follows from the fact that the base change $t \mapsto t^k$ does not change the central fiber, and so $\mathcal{X}_{k\mathscr{P},0} \simeq \mathcal{X}_{\mathscr{P},0}$. We denote by $\rho_\nu$ the scheme morphism
\[ \rho_\nu:  \Spec\, \CC[NE(\mathcal{X}_{\mathscr{P}}/\cX^{\mathrm{can}}_{\mathscr{P}})]
\longrightarrow \Spec\, \CC[NE(\mathcal{X}_{\widetilde{\sP}}/\cX^{\mathrm{can}}_{\widetilde{\sP}}) ] 
\]
induced by $\nu_{\star}$.
The following result compares the polarized mirror family of 
the open Kulikov degeneration $\pi_{\widetilde{\sP}}: \mathcal{X}_{\widetilde{\sP}}\rightarrow \Delta$
with the polarized mirror family  of the open Kulikov degeneration $\pi_{\mathscr{P}}: \mathcal{X}_{\mathscr{P}}\rightarrow \Delta$

\begin{theorem}
\label{Thm: rescaling}
Let $k\in \ZZ_{\geq 1}$ and $\widetilde{\sP}$ be a good polyhedral decomposition of the scaled polytope $kP$, obtained as a refinement of $k\mathscr{P}$. 
Then, 
the mirror family \[(\mathcal{Y}_{\cX_\sP}, \mathcal{D}_{\cX_\sP}) \longrightarrow \Spec\, \CC [NE(\mathcal{X}_{\mathscr{P}}/ \cX^{\mathrm{can}}_{\mathscr{P}}) ]\] of the quasi-projective open Kulikov degeneration $\pi_{\mathscr{P}}: \mathcal{X}_{\mathscr{P}}\rightarrow \Delta$
is the base change 
along the map 
\[ \rho_\nu:  \Spec\, \CC[NE(\mathcal{X}_{\mathscr{P}}/\cX^{\mathrm{can}}_{\mathscr{P}})]
\longrightarrow \Spec\, \CC[NE(\mathcal{X}_{\widetilde{\sP}}/\cX^{\mathrm{can}}_{\widetilde{\sP}}) ] 
\]
of the mirror family \[(\mathcal{Y}_{\cX_{\widetilde{\sP}}}, \mathcal{D}_{\cX_{\widetilde{\sP}}}) \longrightarrow \Spec\, \CC [NE(\mathcal{X}_{\widetilde{\sP}}/ \cX^{\mathrm{can}}_{\widetilde{\sP}}) ]\] of 
the quasi-projective open Kulikov degeneration $\pi_{\widetilde{\sP}}: \mathcal{X}_{\widetilde{\sP}}\rightarrow \Delta$.
Moreover, the pullback of the line bundle $\mathcal{L}_{\cX_{\widetilde{\sP}}}$
by the induced map 
$\tilde{\rho}_\nu: \mathcal{Y}_{\cX_\mathscr{P}} \rightarrow \mathcal{Y}_{\cX_{\widetilde{\sP}}}$
is isomorphic to  $ \mathcal{L}_{\cX_\mathscr{P}}^{\otimes k}$, that is, we have $\tilde{\rho}_\nu^\star \mathcal{L}_{\cX_{\widetilde{\sP}}} \simeq \mathcal{L}_{\cX_\mathscr{P}}^{\otimes k}$.
\end{theorem}

\begin{proof}
It suffices to identify the theta bases and the structure constants of the mirror algebra of $\pi_{\widetilde{\sP}}: \mathcal{X}_{\widetilde{\sP}}\rightarrow \Delta$ with the generators and structure constants of the subalgebra of the mirror algebra of $\pi_{\mathscr{P}}: \mathcal{X}_{\mathscr{P}}\rightarrow \Delta$ consisting of elements of degree divisible by $k$.
To identify the theta bases, one notes that the index set of the theta basis for the mirror algebra of $\pi_{\widetilde{\sP}}: \mathcal{X}_{\widetilde{\sP}}\rightarrow \Delta$ is $C(kP)_\ZZ$, which is exactly the set of integral points of height divisible by $k$ in the cone $C(P)_\ZZ$ indexing the theta basis for the mirror algebra of $\pi_{\mathscr{P}}: \mathcal{X}_{\mathscr{P}}\rightarrow \Delta$.
Hence, it remains to show that, for every $p,q,r \in C(kP)_\ZZ$ and $\beta \in NE(\mathcal{X}_{\mathscr{P}}/\cX^{\mathrm{can}}_{\mathscr{P}})$, we have 
\begin{equation} \label{eq: base change}
N_{pqr}^\beta (\cX_\sP)
=\sum_{\substack{\beta' \in NE(\mathcal{X}_{\widetilde{\sP}}/\cX^{\mathrm{can}}_{\widetilde{\sP}}) \\ 
\nu_{\star} \beta' = \beta}} N_{pqr}^{\beta'}(\cX_{\widetilde{\sP}}) \,.\end{equation}

To prove Equation \eqref{eq: base change}, we first note that, even if the base change $\mathcal{X}_{k\sP}$ of $\mathcal{X}_{\sP}$ is not the total space of an open Kulikov degeneration, it is still log smooth over $\Delta$, and so one can also define invariants $N_{pqr}^{\beta}(\cX_{k \sP})$ for $\mathcal{X}_{k\sP}$.
The base change $t \mapsto t^k$ does not change the central fiber and only scales the log structure. Thus, by \cite[Corollary 6.7]{johnston2024intrinsic}, we obtain
\begin{equation}\label{eq: base change1}
N_{pqr}^\beta(\cX_\sP) = N_{pqr}^\beta(\cX_{k \sP})
\end{equation}
for all $p,q,r \in C(kP)_\ZZ$ and $\beta \in NE(\mathcal{X}_{k\sP}/\cX^{\mathrm{can}}_{k\sP}) \simeq NE(\mathcal{X}_{\sP}/\cX^{\mathrm{can}}_{\sP})$. 
On the other hand, as $\nu: \mathcal{X}_{\widetilde{\sP}} \rightarrow \mathcal{X}_{k\sP}$ is a log modification, it follows from \cite[\S 10]{johnston2022birational} that  
\begin{equation} \label{eq: base change2}
N_{pqr}^\beta(\cX_{k \sP})
=\sum_{\substack{\beta' \in NE(\mathcal{X}_{\widetilde{\sP}}/\cX^{\mathrm{can}}_{\widetilde{\sP}}) \\ 
\nu_{\star} \beta' = \beta}} N_{pqr}^{\beta'}(\cX_{\widetilde{\sP}}) \,,\end{equation}
for all $p,q,r \in C(kP)_\ZZ$ and $\beta \in NE(\mathcal{X}_{k\sP}/\cX^{\mathrm{can}}_{k\sP})$. Since Equation \eqref{eq: base change} follows from Equations \eqref{eq: base change1} and \eqref{eq: base change2}, this concludes the proof.
\end{proof}

\subsection{From semistable mirrors to toric degenerations}
\label{Sec: toric_deg}

Let $(Y,D,L)$ be a polarized log Calabi--Yau surface and $(P,\sP)$ a Symington polytope with polyhedral decomposition defining a maximal degeneration of $(Y,D,L)$ as in \S \ref{sec: maximal degenerations}. Let $\pi_{\sP} : \mathcal{X}_{\sP} \to \Delta$ be the corresponding semistable mirror as in \S\ref{Sec: semistable mirror of log CY}. We obtain below how to construct an integral affine manifold with singularity $(\widetilde{P}, \widetilde{\sP})$ by refining the polyhedral decomposition $\sP$, and splitting the integral-affine singularities of $P$ into focus-focus singularities.

\begin{construction}
\label{GS}
We use the notation introduced in the construction of $(P,\sP)$ in \S \ref{Sec: Symington polytopes}, by cutting off the triangles $\Delta_{ij}$ from the polygon $(\overline{P}, \overline{\sP})$.
As in \cite[Lemma 5.6]{EF21}, the tips of the triangle $\Delta_{ij}$ can be moved along the monodromy invariant directions $L_{ij}$ to split the integral affine singularities of $P$ into distinct focus-focus singularities $p_{ij}'$ at rational points of $P$ -- see Figure \ref{figure10}. 
We denote by $\Delta_{ij}'$ the new triangles with tips $p_{ij}'$, and by $P'$ the resulting integral affine manifold. Let $k \in \ZZ_{\geq 0}$ be large enough such that the points $p_{ij}'$ become integral in the scaled polytope $kP'$.
Then, $kP'$ is obtained from $k \overline{P}$ but cutting the integral triangles $\Delta_{ij}'$.
Hence, by Lemma \ref{lem: good refined}, there exists a good polyhedral decomposition $\widetilde{\sP}$ on $kP$, obtained as a refinement of $k \sP$. Up to increasing $k$ and refining $\widetilde{\sP}$ further, we can also assume
that the monodromy invariant directions $L_{ij}'$ are unions of edges of $\widetilde{\sP}$, and
that no edge of $\widetilde{\sP}$ connects two focus-focus singularities $p_{ij}'$. 
By Lemma \ref{Lem: monod_invariant}, for every $i,j$, there exists an edge $e_{ij}$ of $\widetilde{P}$ adjacent to $p_{ij}'$ and contained in the monodromy invariant direction $L_{ij}'$.
Finally, we define $\widetilde{P}$
as the integral affine manifold with singularities obtained from $P$ by continuously moving the focus-focus singularities from the integral points $p_{ij}'$ to points $\widetilde{p}_{ij}$ in the interior of the edges $e_{ij}$. 

By construction, singularities of $\widetilde{P}$ are all focus-focus, are all contained in the interior of edges of $\widetilde{\sP}$, and every edge of $\widetilde{\sP}$
contains at most one focus-focus singularity. Every integral point $v \in \widetilde{P}$ is smooth for the integral affine structure, and so the star of $\widetilde{\sP}$ at $v$ defines a toric fan $\Sigma_v$. For every $i,j$ such that $v=p_{ij}$, we denote by
$\rho_{ij}$ the ray of the fan $\Sigma_v$ corresponding to 
the edge of $\widetilde{P}$ adjacent to $v$ and contained in the 
line segment connecting $p_{ij}$ to $\widetilde{p}_{ij}$.

\begin{figure}[h]
\center{\includegraphics{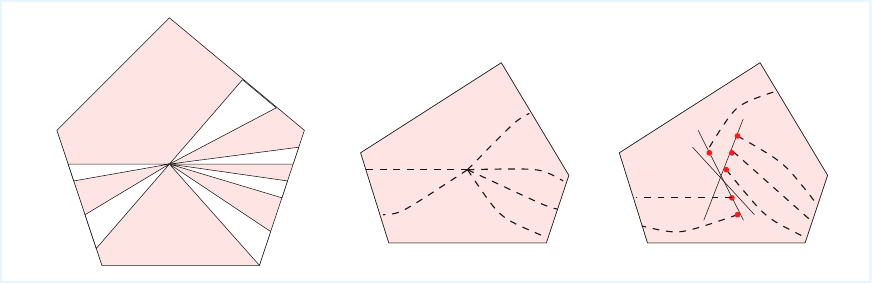}}
\caption{On the left, triangles $\Delta_{ij}$ cut out from $\overline{P}$. In the center, induced cuts in $P$. On the right, deformed cuts $\delta_{ij}$.}
\label{figure10}
\end{figure}

\end{construction}

\begin{lemma} \label{lem: k exists}
For every integral point $v \in (kP)_\ZZ$, the log Calabi--Yau surface $(X^v_{\widetilde{\sP}}, \partial X^v_{\widetilde{\sP}})$ is deformation-equivalent to the surface obtained from the toric surface $(X_{\Sigma_v}, \partial X_{\Sigma_v})$ with fan $\Sigma_v$ by blowing up, for every ray $\rho$ of $\Sigma_v$, $n_\rho$ distinct smooth points on the corresponding toric divisor $D_\rho$, where $n_\rho$ is the number of pairs $(i,j)$ with $p_{ij}=v$ and $\rho_{ij}=\rho$.
\end{lemma}

\begin{proof}
By \S \ref{sec: open Kulikov surfaces from log cy},
the log Calabi--Yau surface  $(X^v_{\widetilde{\sP}}, \partial X^v_{\widetilde{\sP}})$ is a corner blow-up of a  smoothing of the toric surface with fan the star of $v$ in $\overline{P}$.
Using the almost toric description of smoothings of quotient singularities 
-- see \cite[\S 6.1]{Symington} and \cite[Remark 2.7]{Evans2018Markov}, we deduce that $(X^v_{\widetilde{\sP}}, \partial X^v_{\widetilde{\sP}})$ admits an almost toric fibration over a Symington polytope obtained from a toric polytope
by node smoothing. Therefore, by the combinatorial duality between interior blow-ups and node smoothings -- see e.g.\, \cite[p 488]{Eng18}, the pair 
 $(X^v_{\widetilde{\sP}}, \partial X^v_{\widetilde{\sP}})$ is diffeomorphic, and so deformation-equivalent by \cite[Theorem 5.14]{friedman2015geometry}, to the log Calabi--Yau surface described in Lemma \ref{lem: k exists} as interior blow-ups of the toric surface $(X_{\Sigma_v}, \partial X_{\Sigma_v})$.
\end{proof}

We denote 
by $\mathcal{X}_{\widetilde{\sP},0}^{gs}$
the normal crossing surface obtained from $\mathcal{X}_{\widetilde{\sP},0}$ by blowing down all the $(-1)$-curves appearing as exceptional curves of the blow-ups described in Lemma \ref{lem: k exists}.
We show below that $\mathcal{X}_{\widetilde{\sP},0}^{gs}$ is a central fiber of a toric degeneration  $\mathcal{X}_{\widetilde{\sP}}^{gs} \rightarrow \Delta$ in the sense of 
\cite{GSannals}, that is, with toric irreducible components of the central fiber glued along toric divisors, and with the map to $\Delta$ \'etale locally given by a toric morphism near the 0-dimensional strata of the central fiber.

\begin{lemma}
\label{lem: deforming to XGS}
The blowdown morphism 
$p_0: \mathcal{X}_{\widetilde{\sP},0} \to \mathcal{X}_{\widetilde{\sP},0}^{gs}$ is projective and deforms to a projective birational morphism $p: \mathcal{X}_{\widetilde{\sP}} \to \mathcal{X}_{\widetilde{\sP}}^{gs}$, contracting finitely many interior exceptional curves of the central fiber, and fitting into a commutative diagram
\[\begin{tikzcd}
 \mathcal{X}_{\widetilde{\sP}} \arrow[d, "\pi_{\widetilde{\sP}}"] \ar[r,  "p"]
    &  \mathcal{X}_{\widetilde{\sP}}^{gs} \arrow[d, "\pi_{\widetilde{\sP}}^{gs}"]\\
\Delta \ar[r, "\simeq"] & \Delta \end{tikzcd}  \]
where $\pi_{\widetilde{\sP}}^{gs}: \mathcal{X}_{\widetilde{\sP}}^{gs} \to \Delta$ is a toric degeneration in the sense of 
\cite{GSannals}.
\end{lemma}

\begin{proof}
Since $\widetilde{\sP}$ is a good polyhedral decomposition on $kP$ as in \S\ref{subsec: good polyhedral}, there is a strictly convex integral multi-valued piecewise linear function $\varphi$ on $kP$ whose domains of linearity agree with $\widetilde{\sP}$, as in \S \ref{sec: maximal degenerations}. 
Then, we obtain from $\varphi$ a strictly convex integral multi-valued piecewise linear function $\varphi^{gs}$ on $kP$ on the integral affine manifold with singularities $(kP, \widetilde{\sP})^{gs}$ obtained from $(kP, \widetilde{\sP})$ by pushing singularities away from the vertices as described before Lemma \ref{lem: k exists} -- see \cite[\S3]{HDTV} where such piecewise linear functions on integral affine manifolds with singularities pushed away are studied in detail. Therefore, by \cite[Theorem 2.34]{GSannals} since $P$ is a disk with second cohomology zero, $\mathcal{X}_{\widetilde{\sP},0}^{gs}$ is a quasi-projective scheme. Hence, we have an ample line bundle $L_0$ on $\mathcal{X}_{\widetilde{\sP},0}^{gs}$ whose pullback under $\pi_{\widetilde{\sP}}$ is a nef line bundle on $\mathcal{X}_{\widetilde{\sP}}$. 
By Lemma \ref{lem_picard}, this line bundle deforms to a nef line bundle $L$ on $\mathcal{X}_{\widetilde{\sP}}$. Since $\mathcal{X}_{\widetilde{\sP}}$ is a Mori dream space by Proposition \ref{Prop: Mori dreams}, there is a projective birational morphism
$p: \mathcal{X}_{\widetilde{\sP}} \rightarrow \mathcal{X}_{\widetilde{\sP}}^{gs}$ exactly contracting the curves $C$ with $L \cdot C=0$, that is the finitely many interior exceptional curves in the central fiber contracted by $p_0$.

By construction, the irreducible components of the central fiber $\mathcal{X}_{\widetilde{\sP},0}^{gs}$ are the toric varieties $X_{\Sigma_v}$ as in Lemma \ref{lem: k exists} and are glued along toric divisors. Moreover, $\mathcal{X}_{\widetilde{\sP}}^{gs}$ is isomorphic to $\mathcal{X}_{\widetilde{\sP}}$ away from the finitely many ordinary double points where the interior exceptional curves are contracted, and so in particular in the neighborhood of 0-dimensional strata of the central fiber. 
Since $\mathcal{X}_{\widetilde{\sP}} \rightarrow \Delta$ is log smooth, we obtain that $\mathcal{X}_{\widetilde{\sP}}^{gs} \rightarrow \Delta$ is also log smooth near the 0-dimensional strata of the central fiber. Therefore,  $\mathcal{X}_{\widetilde{\sP}}^{gs} \rightarrow \Delta$ is a toric degeneration in the sense of \cite{GSannals}.
\end{proof}

\subsection{Toric degenerations and intrinsic mirrors}

By Lemma \ref{lem: deforming to XGS}, for $k$ large enough, there exists a refinement $\widetilde{\sP}$ of the scaled polyhedral decomposition $k\sP$ such that the corresponding quasi-projective Kulikov degeneration $\cX_{\widetilde{\sP}} \rightarrow \Delta$ admits a birational contraction to a toric degeneration $\cX_{\widetilde{\sP}}^{gs} \rightarrow \Delta$ in the sense of \cite{GSannals}. 
The toric degeneration $\cX_{\widetilde{\sP}}^{gs} \rightarrow \Delta$ 
is actually an example of a positive and simple toric degeneration in the sense of \cite{GSannals}. 
Indeed, by Construction \ref{GS}, its dual intersection complex $(\widetilde{P},\widetilde{\sP})$ is an integral affine manifold with only focus-focus singularities contained in the interior of the edges, and every edge contains at most one focus-focus singularity.

A mirror construction for positive and simple 
toric degenerations of Calabi--Yau varieties has been introduced by Gross--Siebert \cite{GSannals} -- see \cite{Gross} for an exposition in dimension two. 
Unlike the intrinsic mirror construction for log smooth degenerations \cite{GS2019intrinsic, GScanonical} reviewed in \S \ref{sec:mirror algebras} and used in \S \ref{sec: polarized}-\ref{sec: KSBA_stability}, which is based on the enumerative geometry of punctured Gromov--Witten invariants, the mirror construction for toric degenerations is based on a scattering diagram recursively constructed from explicit initial data, and so is mostly combinatorial. In particular, it is often easier to determine the output of the mirror construction for toric degenerations. In this section, we compare the intrinsic mirror family $(\cY_{\cX_{\widetilde{\sP}}}, \cD_{\cX_{\widetilde{\sP}}}, \cL_{\widetilde{\sP}})$ to $\cX_{\widetilde{\sP}} \rightarrow \Delta$ with the mirror to the toric degeneration $\cX_{\widetilde{\sP}}^{gs} \rightarrow \Delta$. Similar comparison results have been obtained in \cite[\S 3.3]{GHK1}, \cite{ArguzHearth}, \cite[\S 6.1]{lai2022mirror}, \cite[\S 2]{AB_FG}.

Let $G$ be the face of $NE(\cX_{\widetilde{\sP}}/\cX^{\mathrm{can}}_{\widetilde{\sP}})$ spanned by the finitely many interior exceptional curves blow down by 
$p: \cX_{\widetilde{\sP}}  \rightarrow \cX_{\widetilde{\sP}}^{gs}$
as in Lemmas \ref{lem: k exists} and \ref{lem: deforming to XGS}. 
Consider the monoid ideal $J_G = NE(\cX_{\widetilde{\sP}}/\cX^{\mathrm{can}}_{\widetilde{\sP}}) \setminus G$, and similarly as in \cite[Definition 3.14]{GHK1} define the \emph{Gross--Siebert locus} as the open torus $T^{gs}=\Spec\, \CC[G^{\mathrm{gp}}]$ in the toric stratum 
$\Spec \, \CC[G]  = \Spec \, \CC[NE(\cX_{\widetilde{\sP}}/\cX^{\mathrm{can}}_{\widetilde{\sP}})/J_G]$ of the toric variety $S_{\cX_{\widetilde{\sP}}}$.

\begin{lemma}
The mirror family $(\cY_{\cX_{\widetilde{\sP}}}, \cD_{\cX_{\widetilde{\sP}}}+\epsilon\, \mathcal{C}_{\cX_{\widetilde{\sP}}})$ is constant with fiber $Y_{0_{\cX_{\widetilde{\sP}}}}$ in restriction to $\Spec \, \CC[G]$, and so to $T^{gs}$.
\end{lemma}

\begin{proof}
As in the proof of Lemma \ref{Lem: mumtor}, all the wall-crossing transformations in the canonical scattering diagrams are trivial modulo $J_G$. Moreover, as no curve of the double locus of $\cX_{\widetilde{\sP},0}$ is contracted by $p$, all the kinks in the canonical scattering diagram are also trivial modulo $J_G$.
\end{proof}

The map $p: \cX_{\widetilde{\sP}}  \rightarrow \cX_{\widetilde{\sP}}^{gs}$ induces a natural inclusion $NE(\cX_{\widetilde{\sP}}/\cX^{\mathrm{can}}_{\widetilde{\sP}}) \subset NE(\cX_{\widetilde{\sP}}^{gs}/\cX^{\mathrm{can}}_{\widetilde{\sP}}) \oplus G^{\mathrm{gp}}$, and so a morphism 
\[ \Spec\, \CC[NE(\cX_{\widetilde{\sP}}^{gs}/\cX^{\mathrm{can}}_{\widetilde{\sP}})] \times T^{gs} \longrightarrow S_{\cX_{\widetilde{\sP}}} \,. \]
Fixing an ample divisor $H$ on $\cX_{\widetilde{\sP}}^{gs}$ defines an affine line \[(\A^1)_H \subset \Spec\, \CC[NE(\cX_{\widetilde{\sP}}^{gs}/\cX^{\mathrm{can}}_{\widetilde{\sP}})]\,.\] 
Also fixing a point $a$ in $T^{gs}$, and denoting $(\A^1)_{H,a} := (\A^1)_H \times \{a\}$, we obtain by restriction a map
\begin{equation} \label{Eq: GS_map}
(\A^1)_{H, a} \longrightarrow S_{\cX_{\widetilde{\sP}}} \,.\end{equation}
On the other hand, the choice of $H$ naturally defines a strictly convex multivalued piecewise linear function $\varphi^{gs}_H$ on the integral affine manifold with singularities $(kP, \widetilde{\sP})^{gs}$ obtained from $(kP, \widetilde{\sP})$ by generically pushing away the singularities from the vertices to irrational points on edges of $\widetilde{\sP}$ as described before Lemma \ref{lem: k exists}. Indeed, by standard toric geometry, the ample divisor $H$ restricted to each of the toric irreducible components of $\cX_{\widetilde{\sP},0}^{gs}$ defines strictly convex piecewise linear functions around the various vertices of $\widetilde{\sP}$. Moreover, the choice of the point $a \in T^{gs}$, viewed as an algebra homomorphism $a: \CC[G^{gp}] \rightarrow \CC$, naturally defines functions $1+a(t^{E_i}) z^{\pm m_i}$ attached to the monodromy invariant directions $\pm m_i$ of the focus-focus singularity associated with the contracted curve $E_i$. 
Taking as input the integral affine $(kP, \widetilde{\sP})^{gs}$, having only focus-focus singularities placed at irrational points on the edges of $\widetilde{\sP}$, the strictly convex multivalued piecewise linear function  $\varphi^{gs}_H$, and the functions $1+a(t^{E_i}) z^{\pm m_i}$ as ``slab functions", the mirror construction for toric degeneration of \cite{GSannals} -- see also \cite[Chapter 6]{Gross}, produces a projective family $\hat{\cY}^{gs}_{H,a} \rightarrow \Spec\, \CC[\![t]\!]$, with a divisor $\hat{\cD}_{H,a}^{gs}$, and an ample line bundle $\hat{\cL}^{gs}_{H,a}$.

\begin{theorem} \label{Thm: toric_deg_mirror}
For every ample divisor $H$ on $\cX^{gs}_{\widetilde{\sP}}$, and every point $a \in T^{gs}$, the mirror family $(\hat{\cY}^{gs}_{H,a}, \hat{\cD}_{H,a}^{gs}, \hat{\cL}^{gs}_{H,a}) \rightarrow \Spec\, \CC[\![t]\!]$ of the toric degeneration  $\cX^{gs}_{\widetilde{\sP}} \rightarrow \Delta$ canonically extends to a family 
$(\cY^{gs}_{H,a}, \cD_{H,a}^{gs}, \cL^{gs}_{H,a}) \rightarrow \A^1$, which is isomorphic to the pullback, along the map 
$(\A^1)_{H, a} \longrightarrow S_{\cX_{\widetilde{\sP}}}$
in \eqref{Eq: GS_map}, of the mirror family $(\cY_{\cX_{\widetilde{\sP}}}, \cD_{\cX_{\widetilde{\sP}}}, \cL_{\cX_{\widetilde{\sP}}}) \rightarrow S_{\cX_{\widetilde{\sP}}}$ of the quasi-projective open Kulikov degeneration $\cX_{\widetilde{\sP}} \rightarrow \Delta$.
\end{theorem}

\begin{proof}
The mirror family of the toric degeneration $\cX^{gs}_{\widetilde{\sP}} \rightarrow \Delta$ can be computed in terms of theta functions defined using broken lines in a scattering diagram $\mathfrak{D}^{gs}$ on $(kP, \widetilde{\sP})^{gs}$, constructed via the Kontsevich-Soibelman algorithm from explicit initial walls coming out of the focus-focus singularities; see \cite[\S 6.3]{Gross} for the construction of the scattering diagram and \cite[Appendix A]{GHS} for a discussion of theta functions. On the other hand, as reviewed in \S\ref{Sec: canonical}, the mirror family of 
$\cX_{\widetilde{\sP}} \rightarrow \Delta$ can be computed in terms of theta functions defined using broken lines in the canonical scattering diagram 
$\mathfrak{D}_{\cX_{\widetilde{\sP}}}$.
Arguing as in \cite[\S 3]{GHK1}, replacing the use of 
\cite{GPS}, which applies to a single log Calabi--Yau surface, by \cite[\S 8]{gross2023remarks}, which applies to a normal crossing union of log Calabi--Yau surfaces such as $\cX_{\widetilde{\sP},0}$, we obtain that $\mathfrak{D}^{gs}$ naturally arises as a deformation of $\mathfrak{D}_{\cX_{\widetilde{\sP}}}$, obtained by pushing the focus-focus singularities away from the vertices of the triangulation $\widetilde{\sP}$. Furthermore, as in \cite[Lemma 3.28]{GHK1}, there exists a natural one-to-one correspondence between broken lines, which induces a canonical identification of the corresponding theta functions. The convergence of the algebra of theta functions for $\hat{\cY}^{gs}_{H,a} \rightarrow \Spec\, \CC[\![t]\!]$ then follows from the convergence of the algebra of theta functions for $\cY_{\cX_{\widetilde{\sP}}} \rightarrow S_{\cX_{\widetilde{\sP}}}$.
\end{proof}

\subsection{The general fiber of the mirror family}

Let $(Y,D,L)$ be a generic polarized log Calabi--Yau surface and let $(P,\sP)$ be an associated Symington polytope with a good polyhedral decomposition as in Construction \ref{Cons: Symington polytope}. Let $\widetilde{\sP}$ be a refinement of $k\sP$ for large $k$ as in Lemma \ref{lem: deforming to XGS}.
Let $\cX_{\widetilde{\sP}}^{gs} \rightarrow \Delta$ be the corresponding toric degeneration given by Lemma \ref{lem: deforming to XGS}, and 
$(\cY^{gs}_{H,a}, \cD_{H,a}^{gs}, \cL^{gs}_{H,a}) \rightarrow \A^1$ the corresponding mirror family given by Theorem \ref{Thm: toric_deg_mirror}, with intersection complex the integral affine manifold with singularities $(kP)^{gs}$ obtained from $kP$ by pushing the focus-focus singularities away from the vertices as in 
\S \ref{Sec: toric_deg}.

The integral affine manifold with singularities $(kP)^{gs}$ is naturally an almost toric base as in \cite{Symington}. Let $(Y^{sym}, D^{sym}, L^{sym}) \rightarrow (kP)^{gs}$ be the corresponding almost toric fibration given by \cite[Corollary 4.2]{Symington}. This fibration is constructed locally as a toric momentum map fibration away from the focus-focus singularities, and using an explicit local model described in \cite[\S 4.2]{Symington} around the focus-focus singularities. In particular, $Y^{sym}$ is naturally a differentiable orbifold, 
with orbifold points in one-to-one correspondence with the vertices of $P$, and locally given by the toric surface singularities with momentum polytopes given by the cones locally describing $P$ near its vertices.

\begin{lemma} \label{lem_diffeo1}
The surface $(Y, D, L^{\otimes k})$ admits an almost toric fibration over $(kP)^{gs}$
and is diffeomorphic as an orbifold to $(Y^{sym}, D^{sym}, L^{sym})$.
\end{lemma}

\begin{proof}
By the Construction \ref{Cons: Symington polytope} of $(P, \sP)$ from a toric model of $(Y,D,L)$,  $(kP)^{gs}$ is a Symington polytope for $(Y,D,L^{\otimes k})$. 
Thus, by \cite[Corollary 4.2]{Symington}, $(Y, D, L^{\otimes k})$
admits an almost toric fibration over 
$(kP)^{gs}$
and so is diffeomorphic as an orbifold to $(Y^{sym}, D^{sym}, L^{sym})$.
\end{proof}

\begin{lemma} \label{lem_diffeo2}
\label{thm: diffeomorphism}
The general fiber $(Y^{gs}_t, D^{gs}_t, L^{gs}_t)$ of 
$(\cY^{gs}_{H,a}, \cD_{H,a}^{gs}, \cL^{gs}_{H,a}) \rightarrow \A^1$
is diffeomorphic as an orbifold to $(Y^{sym}, D^{sym}, L^{sym})$.
\end{lemma}

\begin{proof}
The toric mirror degeneration $(\cY^{gs}_{H,a}, \cD_{H,a}^{gs}, \cL^{gs}_{H,a}) \rightarrow \A^1$ admits a gluing description in terms of local charts -- see \cite[\S 6.2.5]{Gross} for details. We claim that it follows from this description that a general fiber $(Y^{gs}_t, D^{gs}_t, L^{gs}_t)$ admits an almost toric fibration over $(kP)^{gs}$, and so is diffeomorphic as an orbifold to $(Y^{sym}, D^{sym}, L^{sym})$
by \cite[Corollary 4.2]{Symington}. Indeed, away from the finitely many ordinary double points corresponding to focus-focus singularities of $(kP)^{gs}$, the family $(\cY^{gs}_{H,a}, \cD_{H,a}^{gs}, \cL^{gs}_{H,a}) \rightarrow \A^1$ is toric, and the torus fibration over the smooth part of $(kP)^{gs}$ is obtained by locally toric deformation of the torus fibrations on the toric irreducible components of the central fiber induced by the moment maps. Moreover, each of the ordinary double points admits a neighborhood locally described by the equation 
\begin{equation}
\label{eq:xytw}
    xy=t(1+w)
\end{equation}
where the projection to $\A^1$ is given by the $t$-coordinate. There is an explicit torus fibration with a single nodal fiber on the corresponding general fiber of \eqref{eq:xytw}, which is studied in detail in \cite[\S 5.1]{auroux} -- see also \cite[Example 3.20]{CBM}\cite[Example 1.8]{KNreal}. Analogously as in the proof of  \cite[Theorem 3.22]{CBM}, this fibration can be glued in a smooth way to the smooth torus fibration over the smooth locus of $(kP)^{gs}$.
\end{proof}

\begin{theorem}
\label{thm: diffeomorphism_1}
The general fiber $(Y^{gs}_t, D^{gs}_t, L^{gs}_t)$ of 
$(\cY^{gs}_{H,a}, \cD_{H,a}^{gs}, \cL^{gs}_{H,a}) \rightarrow \A^1$
is deformation equivalent to $(Y,D,L^{\otimes k})$.
\end{theorem}

\begin{proof}
Combining Lemma \ref{lem_diffeo1} and Lemma \ref{lem_diffeo2}, we obtain that 
$(Y^{gs}_t, D^{gs}_t, L^{gs}_t)$ and $(Y,D,L^{\otimes k})$ are diffeomorphic as orbifolds. 
Furthermore, the orbifold points of both 
$(Y^{gs}_t, D^{gs}_t, L^{gs}_t)$ and $(Y,D,L^{\otimes k})$
have Zariski neighborhoods
isomorphic to toric surface singularities with momentum polytopes given by the cones describing $P$
near its vertices. Indeed, this follows
for $(Y,D,L^{\otimes k})$
from the Construction \ref{Cons: Symington polytope} of $(P,\sP)$ from a toric model of $(Y,D,L)$, and for  
$(Y^{gs}_t, D^{gs}_t, L^{gs}_t)$ from the explicit gluing construction of the mirror to a toric degenerations by gluing local models -- see \cite[\S 6.2.5]{Gross}.

Hence, the orbifold points of both $(Y^{gs}_t, D^{gs}_t, L^{gs}_t)$ and $(Y,D,L^{\otimes k})$ can be resolved torically in an identical manner to obtain diffeomorphic smooth log Calabi--Yau surfaces $(\widetilde{Y}^{gs}_t, \widetilde{D}^{gs}_t, \widetilde{L}^{gs}_t)$ and $(\widetilde{Y},\widetilde{D},\widetilde{L}^{\otimes k})$. 
Therefore, by \cite[Theorem 5.14]{friedman2015geometry}, the smooth log Calabi--Yau surfaces $(\widetilde{Y}^{gs}_t, \widetilde{D}^{gs}_t)$ and $(\widetilde{Y},\widetilde{D})$ are deformation equivalent. Moreover, for a smooth log Calabi--Yau surface, 
we have $H^1(\mathcal{O})=0$, and so algebraic line bundles are classified up to algebraic isomorphism by their first Chern classes.
Thus, the line bundles $\widetilde{L}^{gs}_t$ and $\widetilde{L}^{\otimes k}$ are isomorphic under the deformation equivalence between 
$(\widetilde{Y}^{gs}_t, \widetilde{D}^{gs}_t)$ and $(\widetilde{Y},\widetilde{D})$. Finally,
Theorem \ref{thm: diffeomorphism_1} follows by contracting back the exceptional divisors introduced during the resolution of the orbifold points.
\end{proof}

\begin{theorem} \label{thm: diffeomorphism_2}
    Let $(Y,D,L)$ be a generic polarized log Calabi--Yau surface and let $(P,\sP)$ be an associated Symington polytope with a good polyhedral decomposition as in Construction \ref{Cons: Symington polytope}. Then, for general $t$ in the open torus $T_{\cX_\sP} \subset S_{\cX_\sP}$, the fiber $(Y_t,D_t,L_t)$ over $t$ of the double mirror $(\cY_{\cX_\sP}, \cD_{\cX_\sP}, \cL_{\cX_\sP}) \rightarrow S_{\cX_{\sP}}$ is deformation equivalent to $(Y,D,L)$.
\end{theorem}

\begin{proof}
By Theorem \ref{Thm: toric_deg_mirror}, there exists a fiber of the double mirror $(\cY_{\cX_{\widetilde{\sP}}}, \cD_{\cX_{\widetilde{\sP}}}, \cL_{\cX_{\widetilde{\sP}}}) \rightarrow S_{\cX_{\widetilde{\sP}}}$ which is a general fiber of one of the mirror families $(\cY^{gs}_{H,a}, \cD_{H,a}^{gs}, \cL^{gs}_{H,a}) \rightarrow \A^1$. 
Hence, by Theorem \ref{thm: diffeomorphism_1}, the general fiber of the double mirror $(\cY_{\cX_{\widetilde{\sP}}}, \cD_{\cX_{\widetilde{\sP}}}, \cL_{\cX_{\widetilde{\sP}}}) \rightarrow S_{\cX_{\widetilde{\sP}}}$ is deformation equivalent to $(Y,D,L^{\otimes k})$.
On the other hand, denoting by $(Y_t, D_t, L_t)$ a general fiber of the double mirror $(\cY_{\cX_\sP}, \cD_{\cX_\sP}, \cL_{\cX_\sP}) \rightarrow S_{\cX_{\widetilde{\sP}}}$, it follows from Theorem \ref{Thm: rescaling} that 
$(Y_t, D_t, L_t^{\otimes k})$ is a fiber of the double mirror of $(\cY_{\cX_{\widetilde{\sP}}}, \cD_{\cX_{\widetilde{\sP}}}, \cL_{\cX_{\widetilde{\sP}}}) \rightarrow S_{\cX_{\widetilde{\sP}}}$, and so is also deformation equivalent to $(Y,D,L^{\otimes k})$. Finally, as $\mathrm{Pic}(Y)$ is torsion-free by 
Lemma \ref{lem: picard_torsion_free}, $L_t^{\otimes k} \simeq L^{\otimes k}$ implies $L_t \simeq L$, and this concludes the proof of Theorem \ref{thm: diffeomorphism_2}.
\end{proof}

\subsection{The Hacking--Keel--Yu conjecture}

Putting together our results in previous sections, we finally prove in Theorem \ref{Thm: HKY_conj} below the two-dimensional case of the Hacking--Keel--Yu conjecture \cite[Conjecture 1.2]{HKY20}.

Let $(Y,D,L)$ be a generic polarized log Calabi--Yau surface and $(P,\sP)$ a Symington polytope with polyhedral decomposition defining a maximal degeneration of $(Y,D,L)$ as in \S \ref{sec: maximal degenerations}. Let $\mathcal{M}_{(Y,D,L)}$ be the closure in the moduli space of stable pairs of the locus of stable pairs deformation equivalent to $(Y,D+\epsilon\, C)$, where $C \in |L|$ and $0< \epsilon <\!\!<1$. It is a proper Deligne--Mumford stack, with a projective coarse moduli space $M_{(Y,D,L)}$ by \cite{MR3779955, MR3671934}.

\begin{lemma} \label{lem_dim_ksba}
Let $(Y,D,L)$ be a generic polarized log Calabi--Yau surface and $(P,\sP)$ a Symington polytope with polyhedral decomposition defining a maximal degeneration of $(Y,D,L)$ as in \S \ref{sec: maximal degenerations}. Then, $M_{(Y,D,L)}$ is an irreducible projective variety of dimension 
\[ \dim M_{(Y,D,L)}=Q+|P_\ZZ|-3\,,\]
where $Q$ is the charge of $(Y,D)$ as in \cite[Definition 1.1]{friedman2015geometry}, that is, the number of interior blow-ups in a toric model of $(Y,D)$, and $|P_\ZZ|$ is the number of integral points in $P$.
\end{lemma}

\begin{proof}
By the Torelli theorem for log Calabi--Yau surfaces 
\cite[Theorem 6.1]{GHKmod},
the moduli stack of log Calabi--Yau surfaces deformation equivalent to $(Y,D)$ is irreducible of dimension $Q-2$. Moreover, given such a log Calabi--Yau pair $(Y',D')$, the space of curves $C' \in |L'|$ such that $(Y', D' +\epsilon\, C')$ is a KSBA stable pair, is open, and so dense in $|L'|$ if non-empty. By the dimension count following \cite[Conjecture 6.3.7]{engel2015proof}, we obtain 
\[ \dim H^0(Y',L') = |P_{\mathbb{Z}}|\,, \quad \text{so} \quad \dim |L'| = |P_{\mathbb{Z}}| - 1\,. \]
Thus, the space of
KSBA stable pairs deformation equivalent to $(Y,D,L)$ is irreducible of dimension $Q+|P_\ZZ|-3$. Consequently, its closure 
$\mathcal{M}_{(Y,D,L)}$ in the moduli space of all KSBA stable pairs is also irreducible of dimension $Q+|P_\ZZ|-3$.
\end{proof}

\begin{theorem} \label{Thm: HKY_conj}
Let $(Y,D,L)$ be a generic polarized log Calabi--Yau surface and $(P,\sP)$ a Symington polytope with polyhedral decomposition defining a maximal degeneration of $(Y,D,L)$ as in \S \ref{sec: maximal degenerations}. Let $\pi_{\sP} : \mathcal{X}_{\sP} \to \Delta$ be the corresponding semistable mirror as in \S\ref{Sec: semistable mirror of log CY}. 
Then, there exists a finite surjective morphism
$S_{\mathcal{X}_{\sP}}^{\mathrm{sec}} \longrightarrow 
M_{(Y,D,L)}$, from the toric variety $S_{\mathcal{X}_{\sP}}^{\mathrm{sec}}$ with fan $\mathrm{Sec}(\mathcal{X}_{\sP}/\cX^{\mathrm{can}}_{\sP})$, to the coarse KSBA moduli space $M_{(Y,D,L)}$.
\end{theorem}

\begin{proof}
By Theorem \ref{Thm: KSBA stability}, 
the toric stack 
$\mathcal{S}^{\mathrm{sec}}_{\cX}$ is the base of a family of KSBA stable log Calabi--Yau surfaces which by Theorem \ref{thm: diffeomorphism_2} are generically deformation equivalent to $(Y,D,L)$. Hence, by the universal property of the KSBA moduli space $\mathcal{M}_{(Y,D,L)}$, there exists a corresponding morphism of stacks $\mathcal{S}^{\mathrm{sec}}_{\cX_\sP} \to \mathcal{M}_{(Y,D,L)}$, and so an induced morphism $S^{\mathrm{sec}}_{\cX_\sP} \to M_{(Y,D,L)}$ at the level of the coarse moduli spaces.
By Theorem \ref{thm_non_constant}, the KSBA stable family over $\mathcal{S}^{\mathrm{sec}}_{\cX}$ is non-constant in restriction to every  1-dimensional toric stratum of $\mathcal{S}^{\mathrm{sec}}_{\cX_\sP}$, hence, no 1-dimensional toric stratum of $\mathcal{S}^{\mathrm{sec}}_{\cX_\sP} $ is contracted to a point in $M_{(Y,D,L)}$. By \cite[Lemma 8.7]{HKY20}, this implies that 
$S_{\mathcal{X}_{\sP}}^{\mathrm{sec}} \to M_{(Y,D,L)}$ is a finite morphism. By Lemma \ref{lem_rank_pic} and Lemma \ref{lem_dim_ksba}, $S^{\mathrm{sec}}_{\cX_\sP} $ and $M_{(Y,D,L)}$ are both proper irreducible varieties, and $\dim \mathcal{S}^{\mathrm{sec}}_{\cX_\sP} \geq Q+|P_\ZZ|-3=\dim M_{(Y,D,L)}$. Hence, the finite morphism $S_{\mathcal{X}_{\sP}}^{\mathrm{sec}} \to M_{(Y,D,L)}$ is necessarily surjective, and this  
 completes the proof of Theorem \ref{Thm: HKY_conj}.
\end{proof}

\begin{corollary} \label{cor_dim}
    The complete toric variety $S_{\cX_\sP}^{\mathrm{sec}}$ with fan the secondary fan $\mathrm{Sec}(\cX_\sP/\cX^{\mathrm{can}}_\sP)$ is projective of dimension $Q+|P_\ZZ|-3$.
\end{corollary}

\begin{proof}
    By Theorem \ref{Thm: HKY_conj}, $S_{\cX_\sP}^{\mathrm{sec}}$ is a finite cover of the coarse KSBA moduli space $M_{(Y,D,L)}$, which is projective by \cite{MR3779955, MR3671934}, and so $S_{\cX_\sP}^{\mathrm{sec}}$ is also projective. Moreover, $\dim S_{\cX_\sP}^{\mathrm{sec}}=\dim M_{(Y,D,L)}=Q+|P_\ZZ|-3$ by Lemma \ref{lem_dim_ksba}.
\end{proof}

\bibliographystyle{plain}
\bibliography{bibliography}

\end{document}